\documentclass[a4paper,12pt]{amsart}

\usepackage{amssymb,xspace,amscd,yhmath,mathdots,txfonts,
mathrsfs, youngtab,stmaryrd,amsthm,amsmath}

\usepackage[matrix,arrow,curve]{xy}\CompileMatrices

\usepackage{xcolor}

\usepackage{color}

\usepackage[colorlinks = true,
linkcolor = blue,
urlcolor  = blue,
citecolor = blue,
anchorcolor = blue]{hyperref}

\usepackage{hyperref}

\usepackage{yfonts}[1998/10/03]

\DeclareMathAlphabet{\mathpzc}{OT1}{pzc}{m}{it}
\usepackage[american]{babel}

\numberwithin{equation}{section}

\theoremstyle{plain}

\newtheorem{thm}{Theorem}[section]

\newtheorem{lem}[thm]{Lemma}

\newtheorem{cor}[thm]{Corollary}

\newtheorem{prop}[thm]{Proposition}

\theoremstyle{definition}
\newtheorem{dfn}{Definition}[section]

\newtheorem{ntz}{Notation}[section]
\newtheorem{exam}[thm]{Example}

\newtheorem{rmk}[thm]{Remark}

\newcommand\Lfs{\mathbf{L}_{\sigmaup}}

\newcommand\et{\mathfrak{e}}

\DeclareMathAlphabet{\mathpzc}{OT1}{pzc}{m}{it}

\newcommand\SO{\mathbf{SO}}
\newcommand\gr{\mathfrak{g}_{\mathbb{R}}}

\DeclareMathOperator{\R}{{\mathbb{R}}}
\DeclareMathOperator{\Ub}{\mathbf{U}}

\DeclareMathOperator{\kt}{{\kappaup}}

\DeclareMathOperator{\Gf}{\mathbf{G}}
\DeclareMathOperator{\Gfs}{\mathbf{G}_{\sigmaup}}
\DeclareMathOperator{\Gfsp}{\mathbf{G}^{\prime}_{\sigmaup^{\prime}}}

\DeclareMathOperator{\supp}{\mathrm{supp}}
\DeclareMathOperator{\ad}{\mathrm{ad}}
\DeclareMathOperator{\Ad}{\mathrm{Ad}}
\DeclareMathOperator{\Lie}{\mathrm{Lie}}

\newcommand\Cd{\mathfrak{C}}

\DeclareMathOperator{\st}{\mathfrak{s}}

\DeclareMathOperator{\N}{\mathbb{N}}
\DeclareMathOperator{\Z}{\mathbb{Z}}

\DeclareMathOperator{\at}{\mathfrak{a}}
\DeclareMathOperator{\pt}{\mathfrak{p}}

\DeclareMathOperator{\ft}{\mathfrak{f}}

\DeclareMathOperator{\im}{{\mathrm{Im}}}
\DeclareMathOperator{\re}{{\mathrm{Re}}}

\DeclareMathOperator{\Pcr}{{\mathscr{P}}}

\DeclareMathOperator{\C}{\mathbb{C}}

\newcommand\Kf{\mathbf{K}}
\newcommand\Kfs{\mathbf{K}_{\sigmaup}}

\newcommand\Qf{\mathbf{Q}}

\newcommand\Aut{\mathpzc{Aut}}

\newcommand\Inv{\mathpzc{I\!{n}{v}}}
\newcommand\Invs{\mathpzc{I\!{n}{v}}^{\tauup}_{\!\stt}(\gr,\hr)}

\newcommand\wt{\mathfrak{w}}

\newcommand\qt{\mathfrak{q}}

\newcommand\qts{\qt\,{\cap}\,\sigmaup(\qt)}

\newcommand\gt{\mathfrak{g}}

\newcommand\ord{\textfrak{L}}

\newcommand\Rad{\mathpzc{R}}

\newcommand\hg{\mathfrak{h}}

\newcommand\Aq{\mathpzc{A}}

\newcommand\vt{\mathfrak{v}}
\newcommand\wb{\mathbf{w}}

\newcommand\Jd{\mathrm{J}}
\newcommand\Sd{\mathrm{S}}

\newcommand\mts{\textswab{m}_{\sigmaup}}

\newcommand\Wf{\mathbf{W}}
\newcommand\SL{\mathbf{SL}}
\newcommand\SU{\mathbf{SU}}
\newcommand\CP{\mathbb{CP}}
\newcommand\RP{\mathbb{RP}}

\newcommand\Sb{\mathbf{S}}
\newcommand\Hc{\mathcal{H}}

\newcommand\Ci{\mathcal{C}^{\infty}}
\newcommand\diag{\mathrm{diag}}

\newcommand\bt{\mathfrak{b}}

\newcommand\nt{\mathfrak{n}}

\newcommand\gl{\mathfrak{gl}}

\newcommand\slt{\mathfrak{sl}}
\newcommand\su{\mathfrak{su}}

\newcommand\rt{\mathfrak{r}}

\newcommand\td{\mathfrak{t}}

\newcommand\Af{\mathbf{A}}

\newcommand\Eq{\mathpzc{E}}

\newcommand{\Bz}{\mathpzc{B}}

\newcommand\lt{\textswab{l}}
\newcommand\lts{\textswab{l}_{\sigmaup}}

\newcommand\id{\mathrm{id}}

\newcommand\Bt{\mathfrak{B}}

\newcommand\vq{\mathpzc{v}}

\newcommand\T{\mathrm{T}}
\newcommand\Hd{\mathrm{H}}

\newcommand\Ft{\mathfrak{F}}
\newcommand\bil{\texttt{b}}

\newcommand\Qt{\mathfrak{Q}}

\newcommand\Hb{\mathbb{H}}
\newcommand\Qq{\mathpzc{Q}}
\newcommand\Pq{\mathpzc{P}}

\newcommand\Lq{\mathcal{L}}

\newcommand\sfV{\textsf{V}}

\newcommand\sfv{\textsf{v}}

\newcommand\Hom{\mathrm{Hom}}

\renewcommand\rq{\mathpzc{r}}

\newcommand\epi{\varepsilon}

\newcommand\Nt{\textfrak{N}}

\newcommand\Kt{\textfrak{K}}

\newcommand\sfa{\textsf{a}}

\newcommand\ttt{\texttt{t}}

\newcommand\Vtt{\texttt{V}}

\newcommand\Pt{\mathfrak{P}}
\newcommand\ktt{\texttt{k}}
\newcommand\stt{\texttt{s}}

\newcommand\pct{\textsf{p}}
\newcommand\qct{\textsf{q}}

\newcommand\sfM{\textsf{M}}

\newcommand\sfF{\textsf{F}}
\newcommand\sfG{\textsf{G}}
\newcommand\sfB{\textsf{B}}
\newcommand\e{\texttt{e}}

\newcommand\hr{\mathfrak{h}_{\mathbb{R}}}
\newcommand\gu{\mathfrak{g}_{\tauup}}

\newcommand\gs{\mathfrak{g}_{\sigmaup}}
\newcommand\gsp{\mathfrak{g}'_{\sigmaup'}}

\newcommand\Qc{\Qq^{c}}
\newcommand\Qn{\Qq^{n}}

\newcommand\qtc{\mathfrak{q}^{c}}

\newcommand\rtt[1]{\texttt{r}_{#1}}

\newcommand\hs{\mathfrak{h}_{\sigmaup}}

\newcommand\ptt{\texttt{p}}

\newcommand\Stt{\texttt{S}}
\newcommand\hst{\mathfrak{h}_{\stt}}
\newcommand\go{\mathfrak{g}_{0}}
\newcommand\ho{\mathfrak{h}_{0}}

\newcommand\codim{\mathrm{codim}}

   \def\DHLhksqrt#1#2{\setbox0=\hbox{$#1\sqrt{#2\,}$}\dimen0=\ht0
     \advance\dimen0-0.2\ht0
     \setbox2=\hbox{\vrule height\ht0 depth -\dimen0}%
     {\box0\lower0.4pt\box2}}
\title{On Contact and Finitely Levi-nondegenerate $CR$ algebras }
\author{S.Marini, C.Medori, M.Nacinovich}
\address{Stefano Marini: Dipartimento di Scienze Matematiche, Fisiche e Informatiche\\ Universit\`a di Parma\\ Parco Area delle Scienze 53/A (Campus), 43124 Parma
 (Italy)} \email{stefano.marini@unipr.it}

\address{Costantino Medori:
Dipartimento di Scienze Matematiche, Fisiche e Informatiche\\ Universit\`a di Parma\\ Parco Area delle Scienze 53/A (Campus), 43124 Parma
 (Italy)} \email{costantino.medori@unipr.it}

\address{Mauro Nacinovich:
Dipartimento di Matematica\\ II Universit\`a di Roma
``Tor Ver\-ga\-ta''\\ Via della Ricerca Scientifica\\ 00133 Roma
(Italy)}
\email{nacinovi@mat.uniroma2.it}


\thanks{The Authors were partially  supported by 
the group G.N.S.A.G.A. of I.N.d.A.M. This research was granted by University of Parma through the action Bando di Ateneo 2025 per la ricerca"
}

\date\today

\begin{document}
\maketitle
\begin{abstract} We study $CR$-manifolds of arbitrary $CR$ codimension,
mainly focusing on Levi and contact-nondegeneracy and depth.
We investigate these and other invariants in the locally homogeneous case,
developing a comprehensive theory which establishes 
correspondences with related 
properties of the associated $CR$-algebras and, 
in the parabolic case, with the combinatorics of their
cross-marked painted
root diagrams.
\end{abstract} 
\tableofcontents
\section*{Introduction}
The Levi form is a fundamental invariant of $CR$ geometry.  Its nondegeneracy 
was shown in \cite{Tan70} to be a sufficient condition to apply Cartan's 
method to investigate 
automorphisms 
and the equivalence problem 
of $CR$ structures. Generalised  Levi forms  
(see e.g. \cite{BHLN})
were concocted to extend these results and
describe other $CR$ invariants.
\par
The main goal of  this paper is investigating 
three  such 
fundamental invariants 
in the case 
of locally homogeneous $CR$-manifolds
of arbitrary $CR$-codimension: \emph{Levi} and 
\emph{contact $k$-nondegeneracy}, and \emph{depth}. \par

Levi $k$-nondegeneracy  
refers to the minimal number of Lie brackets with anti-holomorphic tangent vector fields
needed to move a holomorphic vector field out of the complexified 
contact $CR$ distribution
and is gauged by
the nested sequence of 
kernels of higher order Levi forms
(see e.g. \cite{Freeman1977, MMN21}). 
For  
a point $\pct$  
of 
a real-analytic generic $CR$-manifold
$\sfM$,
Levi nondegeneracy is equivalent to \textit{holomorphic nondegeneracy}, 
that is, 
to 
the absence of any nontrivial germ $X_{(\pct)}$
of holomorphic vector field with holomorphic coefficients  
in the complex ambient space $\tilde{\sfM}$ 
tangent to $\sfM$ 
at $\pct$ 
(see \cite[Th.~1.5.1, p. 330]{BER61999}).
This is a natural  obstruction for locally representing 
$\sfM$ 
as a product of a $CR$ manifold of smaller dimension and a nontrivial complex manifold 
and 
implies 
that the group of germs of $CR$-automorphisms of $\sfM$ at $\pct$ is finite dimensional. 
\par
Levi $k$-nondegeneracy is 
used in the literature as a key invariant 
for studying  
$CR$ equivalence  
and  groups of $CR$ automorphisms, 
see e.g. \cite{IsaevZaitsev2013, KolarKoss2022, Medori_Nacinovich_2001,
MedoriSpiro2014, MedoriSpiro2015, 
MerkerPocchiola2020, ZelenkoSykes2023}. \par

Contact $k$-nondegeneracy 
and depth
are  
basic 
concepts in the wider
context of 
generalised  
contact manifolds. 
They measure  non-in\-te\-gra\-bil\-ity of the contact distribution.
For a point $\pct$ of a generalised contact manifold $\sfM$ and 
a given  tangent vector field of the contact distribution $\Hc(\sfM)$,
not vanishing at $\pct$, 
the first  counts 
the minimal number of Lie brackets 
with vector fields of $\Hc(\sfM)$ 
needed to obtain a section whose value at $\pct$ is not 
taken by any element of~$\Hc(\sfM)$.
The second counts the number of brackets of elements of $\Hc(\sfM)$
that are
needed to generate the  full space of tangent vector fields to $\sfM$. 
\par 
For $CR$-manifolds, 
this applies specifically to the underlying real $CR$ contact distribution
and has noteworthy applications to properties of $CR$ functions
and distributions
(see e.g. \cite{AHNP08a, HN2000, NP2015}).\par
Contact $k$-nondegeneracy turns  
out to be useful in contexts where, for example, 
the local symmetry group of  
the underlying \textit{contact pair}
$(\sfM,\Hd\sfM)$ 
coincides 
with that of 
local 
CR-automorphisms of  the $CR$-triple $(\sfM,\Hd\sfM,\Jd)$ 
(this is the case of \textit{accidental $CR$-manifolds} \cite[Definition 2.8]{hill2025accidentalcrstructures}).
\par
In general, the contact order of a 
$CR$-manifold 
is less than or equal to  its 
order of 
Levi-nondegeneracy. 
Strict inequality 
may occur 
when the contact order is greater or equal to  $2$, 
while 
contact order $1$ 
implies 
Levi 
$1$-nondegeneracy.
Although this invariant is a  
weaker condition than 
Levi-nondegeneracy,  
it  
proves to be useful in the study of the $CR$-equivalence problem 
(see \cite[\S{2}]{KruSan2025}) and 
implies  
finite dimensionality 
of the $CR$-au\-to\-mor\-phism group 
in the locally homogeneous case (see \cite{Marini_Medori_Nacinovich_2020,NMSM}). \par

For 
locally homogeneous $CR$-manifolds, 
it is often convenient to translate 
issues of $CR$ geometry 
into the study of 
suitable 
properties of related 
real Lie algebras as, for instance,  
 in 
\cite{Doubrov2021, Greg2021, sykes2025homogeneous}.
To  
investigate 
locally homogeneous $CR$-manifolds, 
associated $CR$-algebras were introduced in \cite{MN05}.
They are  
 pairs consisting of a real Lie algebra and a complex 
Lie 
subalgebra of its complexification,
encoding 
the Lie algebra of 
the Lie group of  
$CR$-symmetries and the $CR$ structure 
at a base point, respectively. 
\par
In this paper, 
we analyze 
various 
$CR$-invariants  
by using 
$CR$-algebras 
as a tool. 
This approach leads to a comprehensive theory
establishing  
 a correspondence 
 between  $CR$ geometry of the manifold
 and 
algebraic properties of these objects.  
 In the semisimple case, this may be encoded 
 in the combinatorics of their root systems. Prerequisites of Lie theory,
 concepts and notation from \cite{MMN23} will be consistently used
in this research.
\par\vspace{3pt}
The paper is structured as follows.
\par
In 
\S{1} we recall from the literature basic notions of $CR$ geometry,
along with those of Levi $k$-nondegeneracy, 
contact $k$-nondegeneracy and depth.\par
In 
\S{2} we introduce  $CR$-algebras and concepts and notation of
that will be used throughout the rest of the paper (see ~\cite{MN05}).\par
In 
\S{3} we 
use  
$CR$-algebras to 
characterise  
contact 
and Levi nondegeneracy 
in the case of 
(locally) $\Gfs$-homogeneous $CR$-manifolds. 
In particular, we 
construct 
$\Gfs$-equivariant fibrations, 
obtaining quotient $CR$-manifolds that are \textit{fundamental}, \textit{Levi} 
and \textit{contact-nondegenerate},
and describe relations between these fibrations (cf. also \cite{AMN2013}).
\par
In 
\S{4} 
we restrain our consideration to the $CR$ structures of the orbits of real forms 
in complex flag manifolds. 
These objects were first studied by J.A. Wolf in~\cite{Wolf69}.   
In their associated $CR$-algebras $(\gs,\qt)$ the symmetry algebra $\gs$ and
its complexification $\gt$ are semisimple and  
the complex subalgebra $\qt$
describing their $CR$ structure is \textit{parabolic}.
By using suitable Cartan subalgebras, we obtain root systems
in which real orbits may be characterised by suitable 
modified Dynkin diagrams. The root involution induced by the real form
\textit{paints}  roots by different colours and joins  pairs of roots
by arrows, while the parabolic $\qt$
is described by crossing all roots whose root-space is not contained in $\qt$.
Choosing appropriate systems of simple roots, we obtain  modified Dynkin diagrams
where colours, arrows and crosses completely determine $\gs$ and $\qt$.
\par 
In this framework, cross-marked Satake diagrams are well suited to describe
closed orbits and cross-marked Vogan diagrams to describe open orbits.
Intermediate orbits may require to employ Cartan subalgebras of $\gs$
different from the maximally vector and the maximally compact ones
and lead to more general cross-marked diagrams. Moreover, different
systems of simple roots are admissible.  
Two special choices are more useful and lead to \emph{S-} and
\emph{$V$-diagrams},
generalising  cross-marked Satake and Vogan diagrams 
(see~\cite{AMN06b, AMN2013, MMN23}). 
 The   
 $\Gfs$-equivariant fibrations of \S{3}
are 
graphically described 
by using these diagrams, which are also used to compute contact depth
and to study polarization and holomorphic foliations.
\par
In the final  
\S{5},
we 
estimate orders of Levi and contact-nondegeneracy, pursuing 
the work of \cite{MMN21,mmn2022}.
For parbolic $CR$-algebras $(\gs,\qt)$ 
having finite orders and 
in which $\gs$ is the real form of a simple
complex Lie algebra $\gt$, we show that,  
when 
$\mathfrak{g}$ is of type 
$\textsc{A}_{\ell}$, $\textsc{B}_{2}$, or $\textsc{C}_{\ell}$, 
the Levi-order of 
$(\mathfrak{g},\mathfrak{q})$ is at most two. If $\mathfrak{g}$ is of one of the types 
$\textsc{B}_{\ell\geq 3}$, $\textsc{D}_{\ell{\geq}4}$, $\textsc{E}_{6}$, $\textsc{E}_{7}$, $\textsc{E}_{8}$,
$\textsc{F}_{4}$, $\textsc{G}_{2}$, 
then the Levi-order of $(\gs,\qt)$ is at most three, 
and for each of these types there exist $CR$-algebras with Levi-orders one, two, three.
\par 
Minimal orbits have Levi-order at most two 
(cf. \cite{MMN21, mmn2022}). We show that this bound applies to a larger class,
that we call \emph{weakly integrable}, 
clarifying and extenfing the notion of \textit{minimal type}
introduced  in~\cite{MMN21}.


\section{Levi and contact-nondegeneracy of $CR$ manifolds}
 \begin{dfn}
Let $\sfM$ be a smooth manifold of real dimension $2n{+}d$. \par
A \emph{Cauchy-Riemann  structure} of type $(n,d)$ on $\sfM$
is a rank $n$ 
smooth complex  subbundle
$\T^{0,1}\sfM$ of its complexified tangent bundle $\C{\otimes}\T{\sfM},$ satisfying 
\begin{enumerate}
\item  $\T^{0,1}\sfM\cap\overline{\T^{0,1}\sfM}=\{0\}$;
\item  $[\Gamma^{\infty}(\sfM,\T^{0,1}\sfM),\Gamma^{\infty}(\sfM,\T^{0,1}\sfM)]\subseteq
 \Gamma^{\infty}(\sfM,\T^{0,1}\sfM).$
\end{enumerate}
The  pair $(\sfM, \T^{0,1}\sfM)$ is  
 called a $CR$ manifold, $n$ its \emph{$CR$-dimension}, $d$ its \emph{$CR$-codimension}.
 We set $\T^{1,0}\sfM\,{=}\,\overline{\T^{0,1}\sfM}$.
 \end{dfn}
 \subsection{Levi order and Levi-nondegeneracy} 
Strict Levi-nondegeneracy at a point $\pct$ of $\sfM$ means that, for every
complex vector field
$Z\,{\in}\,\Gamma^{\infty}(\sfM,\T^{0,1}\sfM)$ with $Z_{\pct}{\neq}\,0$ we can find 
$Z'\,{\in}\,\Gamma^{\infty}(\sfM,\T^{0,1}\sfM)$ such that 
$[Z,\bar{Z}']_{\pct}\,{\notin}\,\T_{\pct}^{1,0}\sfM\,{\oplus}\,\T_{\pct}^{0,1}\sfM$.
In 
\cite[\S{13}]{MN05},  
a weaker notion of \textit{Levi-nondegeneracy} was introduced. 
\begin{dfn}\label{d1.2}
Let $L\,{\in}\,\Gamma^\infty(\sfM,\T^{1,0}\sfM)$.
Its \emph{Levi-order} at $\pct\,{\in}\,\sfM$  
is defined by\footnote{Higher order commutators of
vector fields are recursively defined by $[X_{1},\hdots,X_{m-1},X_{m}]\,{=}\,[[X_{1},\hdots,X_{m-1}],X_{m}]$.} 
\begin{equation}\label{e4.9}
{k}_{\pct}(L)=\inf\left\{m\,{\in}\,\Z_{+}\left| \begin{gathered} 
 \exists\,\bar{Z}_1,\,\hdots ,\,\, \bar{Z}_m\,{\in}\,\Gamma^\infty(\sfM,\T^{0,1}\sfM)\;\; s.t.\\ 
 [L\,,\,\,\bar{Z}_1\,,\,\,\hdots\, ,
\,\,\bar{Z}_m]_{\pct}\notin \T_{\pct}^{1,0}
{\sfM}\oplus
\T_{\pct}^{0,1}{\sfM}
\end{gathered}\right.\right\}.
\end{equation}
\par
We say that $\sfM$ is \emph{Levi nondegenerate} at $\pct$ if all $L\,{\in}\,\Ci(\sfM,\T^{1,0}\sfM)$ 
with $L_{\pct}\,{\neq}\,0$ have finite Levi-order at $\pct$. The positive integer, or ${+}\infty$,
\begin{equation}
 k_{\pct}(\sfM)=\sup\big\{{k}_{\pct}(L)\mid L\,{\in}\,\Gamma^\infty(M,\T^{1,0}\sfM),\;
 L_{\pct}\,{\neq}\,0\big\} 
\end{equation}
is called the \emph{Levi-order} of $\sfM$ at $\pct$.
 We say that $\sfM$ is \emph{ Levi-$k$-nondegenerate}
at points having exactly finite Levi order $k$.
Strict Levi-nondegeneracy means Levi-$1$-nondegeneracy. 
\end{dfn}
The weaker nondegeneracy condition in Definition\,\ref{d2.2}
is motivated by
the following\footnote{For the notion of $CR$-map and $CR$-submersion the reader may
refer to \cite[\S{1.1}]{AMN2013}.
The map $\piup\,{:}\,\sfM\,{\to}\,\sfM'$ is $CR$ iff $\piup_{*}(\T^{0,1}\sfM)\,{\subseteq}\,\T^{0,1}\sfM'$ and is a
$CR$-fibration iff, moreover,  $\piup$ and $\piup_{*}$ are onto and $\piup_{*}(\T^{0,1}\sfM)\,{=}\,\T^{0,1}\sfM'$.}
\begin{prop}Let $\sfM$ and $\sfM^{\,\prime}$ be $CR$-manifolds. Assume that
$\sfM^{\,\prime}$ is locally embeddable and that
there is a $CR$-fibration $\piup\,{:}\,\sfM\,{\to}\,\sfM^{\,\prime}$, with totally complex
fibres of positive dimension. Then $\sfM$ is Levi-degenerate.\end{prop}
\begin{proof}
Let $f$ be a smooth $CR$-function\footnote{This means that $f\,{\in}\,\Ci(U',\C)$ and $df(\qct')$ 
vanishes on $\T^{0,1}_{\qct'}\sfM'$ for all $\qct'\,{\in}\,U'$.}
 defined on a neighborhood $U'$ of a point
$\pct'$ of $\sfM'$. Then $\pi^*f$ is a $CR$-function in $U\,{=}\,\piup^{-1}(U')$,
that is constant along the fibres of $\piup$. Then, if 
$L\,{\in}\,\Gamma^\infty(\sfM,T^{1,0}\sfM)$ is tangent to the fibres of $\piup$
in $U$, we obtain that 
$[L\,,\,\,\bar{Z}_1\,,\,\,\hdots\, ,
\,\,\bar{Z}_m]\left( \pi^*f\right)\,{=}\,0$ for every choice of 
$\bar{Z}_1,\,\hdots ,\,\, \bar{Z}_m\,{\in}\,\Gamma^{\infty}(\sfM,T^{0,1}\sfM)$.
Assume by contradiction that $\sfM$ is Levi-nondegenerate
at some $\pct$ with $\piup(\pct)\,{=}\,\pct'$. Then for
some choice of  
$\bar{Z}_1,\,\hdots ,\,\, \bar{Z}_m\in\Gamma^\infty(\sfM,T^{0,1}\sfM)$
we would have 
$$\vq_{\pct}\,{=}\,[L\,,\,\,\bar{Z}_1\,,\,\,\hdots\, ,
\,\,\bar{Z}_m]_{\pct}\,{\notin}\, T_{\pct}^{1,0}\sfM\,{\oplus}\,
T_{\pct}^{0,1}\sfM.$$ 
Since the fibres of $\piup$ are totally complex,
$\piup_*(\vq_{p})\,{\neq}\, 0$. By the assumption that $\sfM^{\,\prime}$ is locally embeddable
at $\pct'$, the real parts of the 
(locally defined) $CR$-functions give local coordinates at $\pct'$ for
$\sfM^{\,\prime}$ and therefore
there is a $CR$-function $f$ defined on a neighborhood $U'$ 
of $\pct'$ with $\vq_\pct(\piup^*f)=\piup_*(\vq_{\pct})(f)\,{\neq}\, 0$. This gives a contradiction,
proving our statement.
\end{proof}
\subsection{Contact order, contact-nondegeneracy}
Let $\sfM$ be {a}
$CR$-manifold. 
The real parts of its complex vector fields of type $(1,0)$ are 
smooth sections of its \emph{analytic tangent space} {$\Hd\sfM$}, 
defining a distribution of real vector fields on $\sfM$, that we 
denote by $\Hc(\sfM)$ and call its
\textit{generalised contact structure}.
\begin{dfn}
The \emph{contact order} of {$X\,{\in}\,\Gamma^{\infty}(\sfM,\Hd\sfM)$} 
at a point $\pct$ is 
\begin{equation}
 k^{c}_{\pct}(X)
 =\inf\left\{m\,{\in}\,\Z_{+}\left| \begin{gathered}
 \exists\,X_{1},\hdots,X_{m}\,{\in}\,\Gamma^{\infty}(\sfM,\Hd\sfM) \\
s.t.\; [X,X_{1},\hdots,X_{m}]_{\pct}\,{\notin}\,\Hd_{\pct}\sfM\end{gathered}\right.\right\}.
\end{equation}
We call \emph{contact order} of $\sfM$ at $\pct$ the value
\begin{equation}
 k^{c}_{\pct}(\sfM)=\sup\{k_{\pct}^{c}(X)\,{\mid}\, X\,{\in}\,\Gamma^{\infty}(\sfM,\Hd\sfM),\;X_{\pct}{\neq}\,0\}
\end{equation}  
and
say that $\sfM$ is \emph{contact-nondegenerate}
at points having finite contact order. We say that $\sfM$ is \emph{contact-nondegenerate} if it is such
at all points. \end{dfn}
\begin{prop}
 Let $Z\,{\in}\,\Gamma^{\infty}(\sfM,\T^{0,1}\sfM)$ with $Z_{\pct}{\neq}\,0$. The contact order of $\re(Z)$
 at $\pct$ 
 is the minimal positive integer $m$ such that 
\begin{equation}\label{1.5}
\begin{cases}
 \exists \, Z_{1},\hdots,Z_{m}\in\Gamma^{\infty}(\sfM,\T^{1,0}\sfM)\cup\Gamma^{\infty}(\sfM,\T^{0,1}\sfM)\\
 \text{s.t.}\;\; [Z,Z_{1},\hdots, Z_{m}]_{\pct}\notin\T^{1,0}_{\pct}\sfM\oplus\T^{0,1}_{\pct}\sfM.\end{cases}
\end{equation} 
\end{prop} 
\begin{proof}
 Write $Z\,{=}\,X\,{-}\,\Jd{X}$
 and $Z_{i}{=}\,X_{i}{\pm}i\Jd{X_{i}}$, with $X,X_{i}{\in}\,\Gamma^{\infty}(\sfM,\Hd\sfM)$, 
 where we denoted by $\Jd$ the partial complex structure on $\Hd\sfM$. If the second line
 of \eqref{1.5} holds true, by writing the higher order commutator $[Z,Z_{1},\hdots,Z_{m}]$ 
 as a sum of higher order commutators 
with the first term equal to 
 $X$ or $\Jd{X}$ and the other $m$ terms to  $X_{1},\hdots,X_{m},\Jd{X}_{1},\hdots,\Jd{X}_{m}$,
 we deduce that either $X$ or $\Jd{X}$ has contact order at $\pct$ less than or equal to $m$.
 By the formal integrability of the $CR$ structure, $X$ and $\Jd(X)$ have the same contact order
at $\pct$.   Vice versa, if $X_{1},\hdots,X_{\ell}\in\Gamma^{\infty}(\sfM,\Hd\sfM)$ and
$[\re(Z),X_{1},\hdots,X_{\ell}]_{\pct}{\notin}\Hd_{\pct}\sfM$, then by setting $\re(Z)\,{=}\,\tfrac{1}{2}(Z{+}\bar{Z})$
and $L_{i}{=}\,\tfrac{1}{2}(X_{i}{-}i\Jd{X}_{i})$, $L_{\ell+i}{=}\,\tfrac{1}{2}(X_{i}{+}i\Jd{X}_{i})$,
we deduce that we can find 
$Z_{1},\hdots,Z_{\ell}{\in}\,\{Z_{1},\hdots,Z_{\ell},Z_{\ell+1},\hdots,Z_{2\ell}\}$ 
such that either $[Z,Z_{1},\hdots,Z_{\ell}]_{\pct}$ or $[\bar{Z},Z_{1},\hdots,Z_{\ell}]_{\pct}$
does not belong to $\T^{1,0}_{\pct}\sfM\oplus\T^{0,1}_{\pct}\sfM$. Therefore  either
$[Z,Z_{1},\hdots,Z_{\ell}]_{\pct}$ or $[Z,\bar{Z}_{1},\hdots,\bar{Z}_{\ell}]_{\pct}$
does not belong to $\T^{1,0}_{\pct}\sfM\oplus\T^{0,1}_{\pct}\sfM$, showing that the positive integer
for which \eqref{1.5} is satisfied is less than or equal to the contact order of $\re(Z)$.
This completes the proof.
\end{proof} 
\begin{dfn}
 If $Z\,{\in}\,\Gamma^{\infty}(\sfM,\T^{1,0}\sfM)$, we call \emph{contact order} of $Z$ at $\pct$, and
 denote by $k^{c}_{\pct}(Z)$, the contact order at $\pct$ of its real part. 
\end{dfn}
For $Z\,{\in}\,\Gamma^{\infty}(\sfM,\T^{1,0}\sfM)$ and $\pct\,{\in}\,\sfM$
we have $ k^{c}_{\pct}(\re(Z))\,{\leq}\, k_{\pct}(Z)$ 
 and hence  Levi-nondegeneracy
 is a stronger condition than contact-nondegeneracy. 
 \subsection{Contact depth}
Define by recurrence 
\begin{equation} 
\begin{cases}
 \Hc^{(0)}(\sfM)=0,\; \Hc^{(1)}(\sfM)=\Hc(\sfM),\\
 \Hc^{(\nuup)}(\sfM)=\Hc(\sfM)+{\sum}_{1\leq{h}<\nuup}[\Hc(\sfM),\Hc^{(h)}(\sfM)], & \nuup\geq{2}.
\end{cases}
 \end{equation}
 \begin{dfn}\label{d1.4}
We say that a tangent vector  $\vq\,{\in}\,\T_{\pct}\sfM$ has \emph{finite contact depth} 
if we can find an integer $\nuup\,{\geq}\,0$ and $X\,{\in}\,\Hc^{(\nuup)}(\sfM)$ such that $Z_{\pct}{=}\,\vq$.
We call the smallest $\nuup$ with this property  its \emph{contact depth}.\par
A $CR$ manifold $\sfM$ all whose tangent vectors have finite depth is called \emph{effective}
and the upper bound of the contact depths of its tangent vectors is called its \emph{contact depth}.
\end{dfn}
Likewise, we may define by recurrence 
\begin{equation} 
\begin{cases}
 \hat{\Hc}^{(0)}(\sfM)=0,\;\; \hat{\Hc}^{(1)}(\sfM)=\Gamma^{\infty}(\sfM,\T^{1,0}\sfM)
 \oplus \Gamma^{\infty}(\sfM,\T^{0,1}\sfM),\\
 \hat{\Hc}^{(\nuup)}(\sfM)
 =\hat{\Hc}(\sfM)+{\sum}_{1\leq{h}<\nuup}[\hat{\Hc}(\sfM),\hat{\Hc}^{(h)}(\sfM)], \quad\qquad \nuup\geq{2}.
\end{cases}
 \end{equation}
 
\begin{prop}
 The contact depth of a tangent vector $\vq\,{\in}\,\T_{\pct}\sfM$ is the smallest integer $\nuup\,{\geq}\,0$
 such that $\vq\,{=}\,Z_{\pct}$ for some $Z\,{\in}\,\hat{\Hc}^{(\nuup)}(\sfM)$. \par
 The contact depth of $\sfM$ equals the smallest positive  integer $q$ such that
 $\C\T_{\pct}\sfM\,{=}\,\{Z_{p}{\mid}\,Z\,{\in}\,\hat{\Hc}^{(q)}(sfM)\}$ for all $\pct\,{\in}\,\sfM$. \qed
\end{prop}
Contact depth is an important invariant of 
$CR$-manifolds, being involved in  regularity estimates 
of solutions {to}
tangential Cauchy-Riemann equations (cf. \cite{AHNP08a}).

\section{The notion of $CR$-algebra}\label{s2}

In this section we establish  
concepts and notation that will
be used throughout the rest of the paper. \par 

\subsection{Motivation} Let $\sfM$ be a homogeneous space of
a real Lie group $\Gf_{\sigmaup}$.
A $\Gf_{\sigmaup}$-invariant $CR$-structure on $\sfM$ is completely determined by the datum of 
the complex space $\T^{0,1}_{\pct_{0}}\sfM$ of $(0,1)$-type tangent vectors at any point $\pct_0$ of
$\sfM$, since $\,\sfa_{*}\T^{0,1}_{\pct_0}\sfM\,{=}\,\T^{0,1}_{\sfa\,{\cdot}\,\pct_0}\sfM\,$ for all 
$\sfa\,{\in}\,\Gf_{\sigmaup}$. To a point $\pct_0$ in $\sfM$ we associate the 
smooth submersion $\piup_{\pct_{0}}\,{:}\,\sfa\,{\to}\,\sfa\,{\cdot}\,\pct_0$ of $\Gf_{\sigmaup}$ onto $\sfM$.
The inverse image of $\T^{0,1}_{\pct_0}\sfM$ 
by the complexification of its differential  at the identity 
is a complex
subspace $\qt$ of the complexification $\gt$ of
the Lie algebra $\gt_{\sigmaup}$ of $\Gf_{\sigmaup}$. The left-invariant complex vector fields 
on $\Gf_{\sigmaup}$ 
corresponding to the elements of $\qt$ generate a complex distribution $\mathcal{Q}$ 
on $\Gf_{\sigmaup}$ which
is $\piup_{\pct_0}$-related to the distribution of $(0,1)$ vector fields on $\sfM$.
The formal integrability of the partial complex structure of $\sfM$ is equivalent to the fact
that $[\mathcal{Q},\mathcal{Q}]\,{\subseteq}\,\mathcal{Q}$ and hence to the fact that 
$\qt$ is a complex Lie subalgebra of $\gt$. 
\begin{dfn}
 We call the pair $(\gt_{\sigmaup},\qt)$ the \emph{$CR$-algebra} of $\sfM$ at $\pct_{0}$.
\end{dfn}
\par\smallskip
If $\Gf_{\sigmaup}$  admits a complexification $\Gf$, containing  
a 
{closed} 
complex Lie subgroup $\Qf$ with Lie algebra $\qt$, and $\Qf\,{\cap}\,\Gf_{\sigmaup}$
contains and  has the same
identity component of 
the isotropy subgroup 
$${\Gf_{\sigmaup}}_{\pct_{0}}\,{=}\,\{\sfa\,{\in}\,\Gf_{\sigmaup}\,{\mid}\,\sfa{\cdot}\pct_{0}\,{=}\,\pct_{0}\},$$
 then the commutative diagram 
\begin{equation}\label{CRimmersion}\vspace{4pt}
 \xymatrix{\Gf_{\sigmaup}\; \ar@{^{(}->}[r] \ar_{{\piup_{\pct_{0}}}}[d] & \;\Gf \ar^{\piup_{\Qf}}[d]\\
 \sfM \;\ar@{^{(}->}[r]&\; \sfF}
\end{equation}
defines a smooth immersion of $\sfM$ into $\sfF\,{\coloneqq}\,\Gf{\slash}\Qf$,
and the $CR$-structure on $\sfM$ is inherited from the complex structure of $\sfF$.
The embedding $\sfM\,{\hookrightarrow}\,\sfF$ is \emph{generic}: this means that,
 denoting by $\Jd_{\sfF}$ the complex structure of $\sfF$, 
$\T_{\pct}\sfM\,{+}\,\Jd_{\sfF}\T_{\pct}\sfM\,{=}\,\T_{\pct}\sfF$ for all $\pct$ in $\sfM$.
\par
In general, we can associate to the pair $(\gt,\qt)$ a \emph{germ} of 
\textit{locally homogeneous} complex manifold
$\sfF_{(\pct_{0})}$ and the \emph{germ} of  an analytic embedding $\sfM_{(\pct_{0})}\,{\hookrightarrow}\,
\sfF_{(\pct_{0})}$ defining the $CR$-structure of $\sfM$ near $\pct_{0}$. \par
Thus, the pair $(\gt_{\sigmaup},\qt)$ yields an
\textit{infinitesimal} description characterising $\sfM$ 
and its \emph{generic $CR$-embedding} into a complex manifold  
\textit{near the point $\pct_{0}$}.\par 
These pairs, carrying relevant information about the \textit{local} $CR$ geometry of $\sfM$,
where introduced and investigated in  \cite{MN05}.
\subsection{General notions} \label{s2.2}
In this subsection we review the notion of $CR$-\textit{al\-ge\-bra} and \textit{fundamentality} and 
\textit{nondegeneracy
conditions}. 
The reader may find more details in \cite{AMN06, AMN2013, NMSM, MN05}. 
\par
We will indicate by $\gt$ a 
complex Lie algebra,
by $\sigmaup$ an anti-$\C$-linear involution of $\gt$, by $\gs$ the real form consisting of its
fixed points.
\begin{dfn} \label{d2.2} A $CR$-algebra is a pair $({\gt_{\sigmaup}},\qt)$ 
consisting of a real Lie algebra ${\gt_{\sigmaup}}$ and a complex Lie subalgebra
$\qt$ of its complexification $\gt$.
We call 
\begin{equation*} 
\begin{array}{|c  c  l|} 
\hline
\gt_{\sigmaup} &\text{its} &\text{\emph{symmetry algebra}},\\[3pt]
\hline
\qt & \text{its} & \text{\emph{lifted $CR$-structure}},\\[3pt] 
\hline
{\qt\,{\cap}\,\gs}
 & \text{its} &\textit{{isotropy subalgebra}},\\[3pt]
 \hline
 \qt{/}(\qts) & \text{its} &
\textit{$CR$-structure},\\[3pt]  
\hline
 (\qt\,{+}\,\sigmaup(\qt))\cap{\gt_{\sigmaup}} &\text{its} &\text{\emph{lifted contact module}},\\[3pt]
 \hline 
 \left((\qt\,{+}\,\sigmaup(\qt))\cap\gs\right)/(\qt\,{\cap}\,\gs)
&\text{its} 
&\text{\emph{contact module}.}\\[3pt]
 \hline
\end{array}
\end{equation*}

We say that the $CR$-algebra $(\gt_{\sigmaup},\qt)$ is 
\begin{equation*}
 \begin{array}{| c   l |} \hline
\textit{trivial} &\text{if}\;\; \qt=\gt,\\ \hline
 \textit{effective}  
& \text{if}\;\;\nexists \,\text{an ideal $\at_{\sigmaup}$ of $\gs$ contained in $\qt\,{\cap}\,\gs$}
,\\ \hline
 \textit{integrable} &\text{if}\;\;
 {\qt+\sigmaup(\qt) \;\text{is a Lie subalgebra of $\gt$},}
 \\ \hline
 \textit{totally complex} &\text{if}\;\; \gt=\qt\,{+}\,\sigmaup(\qt),\\ \hline
 \textit{fundamental} &\text{if}\;\; 
 \qt\,{+}\,\sigmaup(\qt)\text{ generates } \gt  
 \\ \hline
\textit{totally real} &\text{if}\;\; \sigmaup(\qt)\,{=}\,\qt,\\ \hline
{\textit{Levi-degenerate}} &
{\text{if}\;\;\exists\, \text{a Lie subalgebra $\qt'$ of $\gt$  s.t. 
$\qt\subsetneqq\qt'\subseteq\qt\,{+}\,\sigmaup(\qt)$},}\\
\hline
{\textit{contact-degenerate}} & \text{if}\;\;
{
\begin{cases}
 \exists\,\text{a Lie subalgebra $\at_{\sigmaup}$
 of $\gs$ {s.t.}}\\
 \text{$\qt\cap\at_{\sigmaup}\subsetneqq\at_{\sigmaup}\subseteq(\qt{+}\sigmaup(\qt)){\cap}\gs$},\\
 [\at_{\sigmaup}, (\qt{+}\sigmaup(\qt))\,{\cap}\,\gs]\subseteq\at_{\sigmaup}.
\end{cases}}
\\
\hline\end{array}
\end{equation*}
\end{dfn}

\begin{rmk}  
Levi-nondegeneracy implies  contact-nondegeneracy and 
totally real $CR$-algebras are 
{trivially} Levi-nondegenerate. 
 We point out that
 {``weakly nondegenerate''} 
 was used in 
 \cite{AMN06, AMN2013, NMSM, MN05},
 instead of  {``Levi-non\-de\-gen\-er\-ate''}
  and 
 ``ideal'' was used for ``contact'' nondegerate in \cite{MN05}. In fact, when $(\gs,\qt)$ is fundamental, the
 condition that $[\at_{\sigmaup}, (\qt{+}\sigmaup(\qt))\,{\cap}\,\gs]\,{\subseteq}\,\at_{\sigmaup}$
implies that $\at_{\sigmaup}$ is an ideal of $\gs$.
\end{rmk}
 The quotients $\gt{/}\qt$ and 
 ${\gs/(\qt\,{\cap}\,\gs)}$
 may be viewed as infinitesimal descriptions
 of germs of complex and real homogeneous spaces. \par
 If $\mts$ is a real Lie subalgebra
 of $\gt_{\sigmaup},$ the choice of a $\qt\subseteq\gt$ 
 { with $\qt\,{\cap}\,\gs\,{=}\,\mts$}
yields a \textit{$\gt_{\sigmaup}$-covariant
 $CR$-structure} on $\gt_{\sigmaup}{/}\mts$
 and 
\begin{equation}\label{e2.2a} \Qt(\gt,\gs,\mts)=
 \{\qt\subseteq\gt
 \mid {[\qt,\qt]\subseteq\qt,\;\qt\cap\gs=\mts\}}
\end{equation}
is the set of \emph{$\gt_{\sigmaup}$-invariant lifted $CR$ structures} on $\gt_{\sigmaup}{/}\mts$. 
If $\qt,\qt' 
$ {belong to} $\Qt(\gt,\gs,\mts)$
and $\qt\,{\subseteq}\,\qt'$ we say that 
$\qt'$ is \emph{stronger than} $\qt$ and $\qt$ 
\emph{weaker than~$\qt'$}.
 \par

\begin{dfn} For a $CR$-algebra $({\gt_{\sigmaup}},\qt)$ we set  
\begin{equation} \label{e2.3}
\begin{cases}
 \dim_{\R}(\gs,\qt)=\dim_{\C}(\gt)-\dim_{\C}(\qt\cap\sigmaup(\qt)), & [\textit{real dimension}],\\
 \dim_{\C}(\gs,\qt)=\dim_{\C}(\qt)\,{-}\,\dim_{\C}(\qt\,{\cap}\,\sigmaup(\qt)), &[\textit{$CR$-dimension}],\\
 \codim_{\R}(\gs,\qt)=\dim_{\C}(\gt)-\dim_{\C}(\qt+\sigmaup(\qt)), &[\textit{$CR$-codimension}].
\end{cases}
\end{equation}
\end{dfn}
Since 
\begin{equation*}
 2\dim_{\C}(\qt)=\dim_{\C}(\qt)\,{+}\,\dim_{\C}(\sigmaup(\qt))=\dim_{\C}(\qt{+}\sigmaup(\qt))\,{+}\,\dim_{\C}
 (\qt\,{\cap}\,\sigmaup(\qt)),
\end{equation*}
we have 
\begin{equation}\label{e2.4}
\begin{cases}
 \dim_{\R}({\gt_{\sigmaup}},\qt)=2\dim_{\C}({\gt_{\sigmaup}},\qt)\,{+}\,\codim_{\R}({\gt_{\sigmaup}},\qt),\\
 \dim_{\C}(\gt)-\dim_{\C}(\qt)=\dim_{\C}({\gt_{\sigmaup}},\qt)+\codim_{\R}({\gt_{\sigmaup}},\qt).
 \end{cases}
\end{equation}
 
\par\smallskip

\begin{rmk}
If $\sfM$ is a  
$\Gf_{\sigmaup}$-homogeneous
$CR$-manifold 
with $CR$-algebra $({\gt_{\sigmaup}},\qt)$ at $\pct_{0}$, then  
$\dim_{\C}({\gt_{\sigmaup}},\qt)$,
$\codim_{\R}({\gt_{\sigmaup}},\qt)$ and $\dim_{\R}({\gt_{\sigmaup}},\qt)$ 
are its $CR$-dimension, its $CR$-codimension and its real dimension,  respectively.
The second equality in \eqref{e2.4}
means that the embedding $\sfM_{(\pct_{0})}\,{\hookrightarrow}\,\sfF_{(\pct_{0})}$ is
\emph{generic}.
 \end{rmk}


A $CR$-algebra is \emph{totally real} 
iff its $CR$-dimension is zero and \emph{totally complex} iff its $CR$-codimension is $0$.
Totally real and totally complex $CR$-al\-ge\-bras are integrable.
\begin{exam}\label{ex2.3}
Consider   
$\mathbb{CP}^2$  as the Grassmanian of complex lines in $\C^3$,  i.e. the compact homogeneous space 
 $\sfF{=}\SL_{3}(\C)/\Qf$ where 
\begin{equation*}
 \Qf\coloneqq\left.\left\{ \left( 
\begin{smallmatrix}
  z_{00}&z_{01}& z_{02}\\
 0& z_{11}&z_{12}\\
 0 & z_{21}& z_{22}
\end{smallmatrix}\right) \right| z_{i,j}\in\C,\; z_{0,0}(z_{11}z_{22}-z_{12}z_{21})=1\right\}.
\end{equation*}
One  example of the  notion of  
\emph{espaces g\'en\'eralis\'ees} of \'E.Cartan
is given by the $3$-dimensional sphere 
$\Stt^{3}\,{=}\,\{[z_0{:}z_1{:}z_2]\,{\in}\,\CP^2\,{\mid}\,z_{0}\bar{z}_{2}\,{+}\,
z_{1}\bar z_{1}\,{+}\,z_{2}\bar{z}_{0}\,{=}\,0\}$
  on which the group $\SU(1,2)$  acts transitively.    
The sphere $\Stt^{3}$ is endowed with an $\SU(1,2)$-invariant $CR$-structure,
 inherited from its embedding  \eqref{CRimmersion} in the complex projective plane.
 Its $CR$-algebra at $\piup(\e_{0})$ is $(\su(1,2),\qt)$ with 
\begin{equation*}
 \su(1,2)\,{=}\,\left.\left\{\left( 
\begin{smallmatrix}
\lambdaup & \zetaup&i\stt\\
 z&\bar{\lambdaup}{-}\lambdaup&{-}\bar{\zetaup}\\
 i\ttt&{-}\bar{z}&{-}\bar{\lambdaup}
\end{smallmatrix}\right)\right| \ttt,\stt\,{\in}\,\R,\; \lambdaup,z,\zetaup\,{\in}\,\C
\right\}
\end{equation*}
and $\qt$ equal to the Lie algebra of $\Qf$.
We have $\dim_{\C} \qt{=}6$ and $\dim_{\C}(\qt\,{+}\,\sigmaup(\qt))\,{=}\,7$.
Since 
\begin{equation*}
 \qt\,{+}\,\sigmaup(\qt)=\left.\left\{ \left( 
\begin{smallmatrix}
 z_{0,0}& z_{0,1}&z_{0,2}\\
 z_{1,0}&z_{1,1}&z_{1,2}\\
 0 & z_{2,1}&z_{2,2}
\end{smallmatrix}\right)\right| z_{i,j}\,{\in}\,\C,\; z_{0.0}{+}z_{1,1}{+}z_{2,2}{=}0\right\}
\end{equation*}
is not a Lie subalgebra of $\slt_{3}(\C)$, Levi-nondegeneracy is obvious. In fact $(\su(1,2),\qt)$
is strictly Levi-nondegenerate. \end{exam} 
\begin{dfn}
 Given two $CR$-algebras $(\gs',\qt')$ and $(\gs'',\qt'')$, the pair $(\gs'\,{\oplus}\,\gs'',\qt'\,{\oplus}\,\qt'')$
 is a $CR$-algebra, which is called their \emph{direct sum}. In an analogous way we may define the
 direct sum of any family of $CR$-algebras.
\end{dfn}

\section{Contact and Levi order}\label{s3}
 Contact and Levi-nondegeneracy 
can be  gauged by using \textit{higher order commutators} (cf. \cite{MN05}).
\subsection{Higher order commutators}
In this preliminary
subsection we establish some lemmas on higher order commutators.
More general results  can be found 
e.g. in \cite{Al2016,BL92,Reu93}.
Here $\gt$ is any Lie algebra over a field of characteristic zero. 

\textit{Higher
order commutators} of elements of  $\gt$ are recursively defined by 
\begin{equation*} 
\begin{cases}
 [X_{1}]=X_{1},\\
 [X_{1},X_{2}]\;\text{is the Lie algebra product of $X_{1}$ and $X_{2}$ in $\gt$,}\\
 [X_{1},\hdots,X_{h-1},X_{h}]=[[X_{1},\hdots,X_{h-1}],X_{h}],\;\;\text{if $h\,{\geq}\,3$.}
\end{cases}
\end{equation*}
\par

\begin{lem}\label{l2.1}
 Let $X_{1},\hdots,X_{m}$ be elements of a Lie algebra $\gt$. 
 The Lie subalgebra of $\gt$ generated by $X_{1},\hdots,X_{m}$
is the linear span of the
 higher order commutators 
\begin{equation*}
 [X_{i_{1}},\hdots,X_{i_{h}}],\;\;\text{with}\;\; 1\leq {i}_{1},\hdots,i_{h}\leq{m}.
\end{equation*}
\end{lem} 
\begin{proof} It suffices to consider the situation where $\gt$ is the free Lie algebra
generated by independent objects $X_{1},\hdots,X_{m}$. We recall that $\gt$ is $\Z$-graded
in such a way that 
a higher order commutator $ [X_{i_{1}},\hdots,X_{i_{h}}]$ has degree $h$.
We want to show that elements of degree $h$ are linear combinations of commutators of
order $h$ of $X_{1},\hdots,X_{m}$. We first prove that the Lie product of two commutators
of degree $h,k$, respectively, are sums of commutators of degree $h{+}k$. We can assume that
$h{\leq}k$ and argue by recurrence on $h$. In fact the case $h\,{=}\,1$ is trivial. 
If $h\,{>}\,1$ and the statement holds true when the  lesser degree is ${\leq}\,h{-}1$,
we use the identity 
\begin{align*}
 &[[X_{j_{1}},\hdots,X_{j_{k}}],[X_{i_{1}},\hdots,X_{i_{h}}]]=
  [[X_{j_{1}},\hdots,X_{j_{k}}],[[X_{i_{1}},\hdots,X_{i_{h-1}}],X_{i_{h}}]]\\
  &{=}\,  [[[X_{j_{1}},\hdots,X_{j_{k}}],[X_{i_{1}},\hdots,X_{i_{h-1}}]],X_{i_{h}}]
  +  [[X_{i_{1}},\hdots,X_{i_{h-1}}], [X_{j_{1}},\hdots,X_{j_{k}},X_{i_{h}}]].
\end{align*}
\par
By the recursive assumption, $[[X_{j_{1}},\hdots,X_{j_{k}}],[X_{i_{1}},\hdots,X_{i_{h-1}}]]$ is 
a linear combination of higher order commutators of degree $h\,{+}\,{k}\,{-}\,1$
and thus the first summand 
$ [[[X_{j_{1}},\hdots,X_{j_{k}}],[X_{i_{1}},\hdots,X_{i_{h-1}}]],X_{i_{h}}]$ is 
a linear combination of higher order commutators of degree
$h{+}k$. 
Again by the recursive assumption, the second summand 
$[[X_{i_{1}},\hdots,X_{i_{h-1}}], [X_{j_{1}},\hdots,X_{j_{k}},X_{i_{h}}]]$
is a linear combination of higher order commutators of degree $h{+}k$. \par
The case of a Lie product of three of more higher order commutators reduces, by iterating the
argument for two,  
to sums
of products of pairs of commutators.
\end{proof}
\begin{lem}\label{l3.2a}
 Let $r\,{\geq}\,2$ and $X_{1},\hdots,X_{r}\,{\in}\,\gt$. If $[X_{1},\hdots,X_{r}]\,{\neq}\,0$, then
 we can find $j$, with $1{\leq}j{<}r$, such that $[X_{j},X_{r}]\,{\neq}\,0$. 
\end{lem} 
\begin{proof} Set 
$
 X_{i,j}= 
\begin{cases}
 X_{j}, & j\neq{i},\\
 [X_{i},X_{r}], & j=i,
\end{cases}\quad\text{for $1{\leq}i{\neq}{j}{\leq} r{-}1$.}
$
The statement follows from 
\begin{equation*}
  [X_{1},\hdots,X_{r-1},X_{r}]={\sum}_{i=1}^{r-1}[X_{i,1},\hdots,X_{i,r-1}].
\end{equation*}
Indeed, 
The left hand side  being different from zero, at least one summand in the right,
and hence some $X_{i,i}\,{=}\,[X_{i},X_{r}]$ is 
different from zero.
\end{proof}
\subsection{A canonical filtration} Let $(\gt_{\sigmaup},\qt)$ be a $CR$-algebra. Its
lifted contact module and isotropy subalgebra  \label{contct}
can be inserted into a sequence of linear subspaces of $\gt_{\sigmaup}$. We set  
 \begin{equation*} 
\begin{cases}
 \Ft_{  0}(\gs,\qt)=\qt
 \cap\gt_{\sigmaup}, \qquad\qquad\qquad\;\qquad\text{(the isotropy subalgebra)},\\[3pt]
\Ft_{ -1}(\gs,\qt)= 
 (\qt+\sigmaup(\qt))\cap\gt_{\sigmaup}, \qquad\qquad\text{(the lifted contact module)},\\[3pt]
\Ft_{  -h}(\gs,\qt)= \Ft_{  1{-}h}(\gs,\qt) +\left[\Ft_{  1-h}(\gs,\qt), \Ft_{  -1}(\gs,\qt)\right],
  \qquad\quad h\geq{2},\\[3pt]
  \Ft_{  h}(\gs,\qt)=\left\{Z\in \Ft_{  h-1}(\gs,\qt)\left| \left[Z, \Ft_{  -1}(\gs,\qt)\right]
  \subseteq\Ft_{h-1}(\gs,\qt)\right\}\right., \qquad\quad h\geq{1},\\[3pt]
 \Ft_{  {-}\infty}(\gs,\qt)=
 {\bigcup}_{h\in\Z}\Ft_{   h}(\gs,\qt) ,\\[3pt]
 \Ft_{  {+}\infty}(\gs,\qt)={\bigcap}_{h\in\Z}\Ft_{  h}(\gs,\qt).
\end{cases}
\end{equation*}
\begin{prop} Given any
$CR$-algebra $(\gt_{\sigmaup},\qt)$, 
\begin{enumerate}
\item  $\Ft_{  {-}\infty}(\gs,\qt)$ is the Lie subalgebra of $\gt_{\sigmaup}$ 
generated by $\Ft_{  {-}1}(\gs,\qt)$;
\item
the sequence 
\begin{equation}\label{e2.1}\left\{ \begin{aligned}
 \cdots\,{\subseteq}\,\Ft_{  h}(\gs,\qt)\,{\subseteq}\,\Ft_{  h-1}(\gs,\qt)
\, {\subseteq}\,\cdots\,{\subseteq}\,\Ft_{  1}(\gs,\qt) 
 \,{\subseteq}\,
 \Ft_{  0}(\gs,\qt)\qquad \qquad \\ 
 \,{\subseteq}\,
 \Ft_{  -1}(\gs,\qt)\,{\subseteq}\,\Ft_{ -2}(\gs,\qt)\,
 \,{\subseteq}\,\cdots\,{\subseteq}\,\Ft_{  -h}(\gs,\qt)\,{\subseteq}\,
 \Ft_{  -h-1}(\gs,\qt)
 \,{\subseteq}\,\cdots \end{aligned} \right.
\end{equation}
is a $\Z$-filtration of $\Ft_{  {-}\infty}(\gs,\qt)$, i.e.  $\Ft_{-\infty}(\gs,\qt)\,{=}\,{\bigcup}_{h\in\Z}\Ft_{h}(\gs,\qt)$ and
$[\Ft_{  h}(\gs,\qt),\Ft_{  k}(\gs,\qt)]\,{\subseteq}\,
 \Ft_{  h{+}k}(\gs,\qt)$ for all $h,k\,{\in}\,\Z$;
\item for $h\,{\geq}\,0$ each $\Ft_{  h}(\gs,\qt)$ is a Lie subalgebra of $\gt_{\sigmaup}$;
\item
If  $h\,{\geq}\,0$ and $\Ft_{  -h}(\gs,\qt)\,{=}\,\Ft_{  {-}h{-}1}(\gs,\qt)$,
then $\Ft_{  {-}k}(\gs,\qt)\,{=}\,\Ft_{  {-}h}(\gs,\qt)$ for all $k\,{>}\,h$; 
\item If  $h\,{\geq}\,0$ and $\Ft_{  h+1}(\gs,\qt)\,{=}\,\Ft_{  h}(\gs,\qt)$,
then $\Ft_{  k}(\gs,\qt)\,{=}\,\Ft_{  h}(\gs,\qt)$ for all $k\,{>}\,h$. 
\end{enumerate}
\par
The  lower limit
$\Ft_{  +\infty}(\gs,\qt)$ of \eqref{e2.1}
is the maximal ideal of $\Ft_{  {-}\infty}(\gs,\qt)$ contained in the isotropy subalgebra 
$\Ft_{  0}(\gs,\qt)$. Thus 
$(\gt_{\sigmaup},\qt)$ is 
\begin{itemize}
 \item fundamental if and only if $\Ft_{  {-}\infty}(\gs,\qt)\,{=}\,\gt_{\sigmaup}$;
 \item effective if  $\Ft_{  {+}\infty}(\gs,\qt)\,{=}\,\{0\}$. 
\end{itemize}
The condition that $\Ft_{  {+}\infty}(\gs,\qt)\,{=}\,\{0\}$ is also necessary for effectiveness if 
$(\gt_{\sigmaup},\qt)$ is fundamental.
\end{prop} 
\begin{proof} By Lemma\,\ref{l2.1}, the elements of  
$\Ft_{ -h}(\gs,\qt)$ are linear combinations of 
 commutators of order less or equal $h$ of elements of 
$\Ft_{ -1}(\gs,\qt)$ and we obtain 
$[\Ft_{ -h}(\gs,\qt),\Ft_{ -k}]\,{\subseteq}\,\Ft_{ {-}h{-}k}$
for $h,k\,{\geq}\,1$.
\par
The inclusion {$[\Ft_{ -h}(\gs,\qt),\Ft_{\,k}(\gs,\qt)]\,{\subseteq}\,
\Ft_{ k{-}h}$ for $h\,{>}\,0$} 
follows then from 
\begin{equation*}\begin{aligned}\;
 [[X_{1},\hdots,X_{p}],Y]\,{=}\,[[X_{1},Y],X_{2},\hdots,X_{p}]\,{+}\,\cdots\,{+}\,
 [X_{1},\hdots,[X_{i},Y],\hdots,X_{p}]\,\\
 +\cdots+[X_{1},\hdots,X_{p-1},[X_{p},Y]],\end{aligned}
\end{equation*} 
by taking $X_{1},\hdots,X_{p}\,{\in}\,\Ft_{ {-}1}(\gs,\qt)$, 
$Y\,{\in}\,\Ft_{ k}(\gs,\qt)$, 
and arguing by recurrence on~$k$.
To prove 
the inclusion 
$[\Ft_{ h}(\gs,\qt),\Ft_{ k}(\gs,\qt)]\,{\subseteq}\,\Ft_{ h{+}k}(\gs,\qt)$
for $h,k\,{\geq}\, 0$,
we argue by recurrence on the sum $h{+}k$. The case $h\,{=}\,k\,{=}\,0$ being
trivial, the inductive assumption yields,  for $Y_{1}\,{\in}\,\Ft_{ h}(\gs,\qt)$,
$Y_{2}\,{\in}\,\Ft_{ k}(\gs,\qt)$, $X\,{\in}\,\Ft_{ {-}1}(\gs,\qt)$, 
\begin{align*}
& [[Y_{1},Y_{2}],X]\,{=}\,[[Y_{1},X],Y_{2}]{+}[Y_{1},[Y_{2},X]] \\
&\qquad\quad \,{\in}\, [\Ft_{ h-1}(\gs,\qt),\Ft_{\, k}(\gs,\qt)]\,{+}
 \,[\Ft_{ h}(\gs,\qt),\Ft_{ k-1}(\gs,\qt)]
 \,{\subseteq}\, \Ft_{ h{+}k{-}1}(\gs,\qt),
\end{align*}
showing that $[Y_{1},Y_{2}]\,{\in}\,\Ft_{ h+k}(\gs,\qt)$.\par
The remaining claims follow 
straightforwardly.
\end{proof}
When $\gt$ if finite dimensional, as we will assume in the following, the filtration 
\eqref{e2.1} stabilizes. 
\begin{dfn} \label{d1.4} A nonzero element $X$ of $\gt_{\sigmaup}$ is said to have  
\textit{depth}, or \textit{degree~$p$} 
if it belongs to  
{$\Ft_{ -p}(\gs,\qt){\setminus}\Ft_{ 1-p}(\gs,\qt)$.}
The smallest positive integer $\nuup$ for which 
{$\Ft_{ -\infty}(\gs,\qt)\,{=}\,\Ft_{ -\nuup}(\gs,\qt)$}
 is called the 
 \emph{depth} 
 of $(\gt_{\sigmaup},\qt)$. \par 
 The smallest nonnegative integer $\deltaup$
 for which either 
{$\Ft_{ \deltaup}(\gs,\qt)$}
 is different from $\{0\}$ and equals 
{$\Ft_{ +\infty}(\gs,\qt)$}
 or
 is such that 
{$\Ft_{ \nuup+1}(\gs,\qt)\,{=}\,\{0\}$}
 is the 
 \emph{heigth}
 of $(\gt_{\sigmaup},\qt)$. 
\end{dfn} 
\begin{rmk}
 A $CR$-algebra $(\gt_{\sigmaup},\qt)$ is called  \emph{Levi-Tanaka} if it is fundamental and effective and 
 \eqref{e2.1} is the filtration associated to
a $\Z$-gradation
 $\gt_{\sigmaup}\,{=}\,{\sum}_{h={-}\nuup}^{\deltaup}{\gt_{\sigmaup}}_{,h}$ of $\gt_{\sigmaup}$.
 For Levi-Tanaka algebras, 
 Levi-nondegeneracy is equivalent to 
{strict Levi-nondegeneracy} 
(see e.g. \cite{MMN20z, MN97,MN98, Tan67, Tan70}).
\end{rmk}

\subsection{Levi-nondegeneracy} 
Levi-nondegeneracy may be described in terms of higher order commutators. 
We begin by proving a general lemma.
\begin{lem} \label{l3.5}
Let $\gt$ be a Lie algebra over a field $\mathbb{K}$.
 If $\qt,\wt$ are Lie subalgebras of $\gt$, then 
\begin{equation}\label{e3.2}
\qt_{\{\wt\}}\coloneqq\{Z\,{\in}\,\qt\,{+}\,\wt\mid \ad(\qt)^{h}(Z)\,{\subseteq}\,\qt\,{+}\,\wt,\,\forall{h}\geq{0}\}
\end{equation}
is the largest Lie subalgebra of $\gt$ {containing $\qt$ and  contained in $\qt{+}\wt$}.
\end{lem} 
\begin{proof}
Set $\st\,{=}\,\qt_{\{\wt\}}\,{\cap}\,\wt$.
Then $\qt_{\{\wt\}}\,{=}\,\st{+}\qt$\, and, 
for all nonnegative
integers~$h$, 
\begin{align*}
 \ad(\qt)^{h}([\qt_{\{\wt\}},\qt_{\{\wt\}}])={\sum}_{i=0}^{h}[\ad(\qt)^{i}(\qt_{\{\wt\}}),\ad(\qt)^{h-i}(\qt_{\{\wt\}})]\qquad 
 \\
 \subseteq [\st+\qt,\st+\qt]\subseteq [\st,\st]+[\st,\qt]+\qt\subseteq \qt+\wt,
\end{align*}
because $[\st,\st]\,{\subseteq}\,\wt$ and $[\st,\qt]\,{\subseteq}\,\qt\,{+}\,\wt$. Thus 
$\qt_{\{\wt\}}$ is a Lie subalgebra of $\gt$. \par
 Since $\ad(\qt)^{h}(\qt')\,{\subseteq}\,\qt'$ for every Lie subalgebra $\qt'$ of $\gt$ containing $\qt$,
 it follows  that $\qt_{\{\wt\}}$ is the largest Lie subalgebra $\qt'$ with
 $\qt\,{\subseteq}\,\qt'\,{\subseteq}\,\qt{+}\wt.$ 
\end{proof} 
\begin{dfn}
We call the Lie subalgebra $\qt_{\{\wt\}}$ defined by
 \eqref{e3.2} the \emph{$\wt$-ex\-ten\-sion of $\qt$ in $\gt$}.
\end{dfn}
\begin{thm} \label{t3.6}
 Let $(\gt_{\sigmaup},\qt)$ be a $CR$-algebra. Then $\qt'\,{=}\,\qt_{\{\sigmaup(\qt)\}}$ is {the}
maximal complex Lie subalgebra of $\gt$
 satisfying
 \begin{equation}\label{e3.3}
 \qt\,{\subseteq}\,\qt'\,{\subseteq}\,\qt{+}\sigmaup(\qt).
 \end{equation}
\end{thm} 
\begin{proof}
 We apply Lemma\,\ref{l3.5}, taking $\wt\,{=}\,\sigmaup(\qt)$.
\end{proof} 
\begin{dfn}
 We call $\qt_{\{\sigmaup(\qt)\}}$, that will be indicated for simplicity by $\qt_{\{\sigmaup\}}$, 
 the \emph{$\sigmaup$-extension} of $\qt$. 
\end{dfn}
\begin{dfn}
 Let $(\gt_{\sigmaup},\qt)$ be a $CR$-algebra and $Z$ an element of $\sigmaup(\qt){\setminus}\qt$.
 A sequence $(Z_{1},\hdots,Z_{k})$ in $\qt$ such that $[Z,Z_{1},\hdots,Z_{k}]\,{\notin}\,\qt\,{+}\,\sigmaup(\qt)$
 is called a \emph{Levi-sequence} for $Z$. Elements $Z$ in $\sigmaup(\qt){\setminus}\qt$ admitting
 a Levi-sequence are said to have \emph{finite Levi-order} and the minimal length $k(Z)$ 
 of
 their Levi-sequences is called their \emph{Levi-order}:
\begin{equation}
 k(Z)=\inf\big\{k\,{\in}\,\N\mid \exists\, Z_{1},\hdots,Z_{k}\in\qt\;\,\text{s.t.}\;\,[Z,Z_{1},\hdots,Z_{k}]
 \notin\qt{+}\sigmaup(\qt)\big\}.
\end{equation}
The \emph{Levi-order} $k(\gs,\qt)$ 
of $(\gs,\qt)$ is the supremum of the Levi orders of elements of $\sigmaup(\qt){\setminus}\qt$.
\end{dfn}
\begin{cor}\label{c2.7}
 A $CR$-algebra $(\gt_{\sigmaup},\qt)$ 
  is 
  Levi-nondegenerate if and only if it has finite Levi order.
 \qed
\end{cor}
\begin{dfn}
 $CR$-algebras 
with finite Levi order $k$ is called  
{\emph{ Levi-$k$-non\-de\-gen\-er\-ate}.}
{Levi $1$-nondegenerate means  \textit{strictly} Levi-nondegenerate.}\end{dfn}
A homogeneous $CR$-manifold $\sfM$ with a Levi-$k$-nondegenerate associated $CR$-algebra
is \emph{Levi-$k$-nondegenerate}.
\begin{lem}\label{l3.8} Assume that 
$Z\,{\in}\,\sigmaup(\qt){\setminus}\qt$ has finite Levi-order. Let  $(Z_{1},\hdots,Z_{k})$
be a minimal length Levi-sequence for $Z$.
Then 
\begin{itemize}
\item $Z_{i}\,{\in}\,\qt{\setminus}\sigmaup(\qt)$ \, for $1{\leq}i{\leq}k$;
\item if $k\,{>}\,1$, then $[Z,Z_{i}]\,{\in}\,(\qt\,{+}\,\sigmaup(\qt)){\setminus}\qt$;
\item $(Z_{i_{1}},\hdots,Z_{i_{k}})$ is a Levi-sequence for $Z$  
for every permutation $i\,{\in}\,\Sb_{k}$.
\end{itemize}
\end{lem}
\begin{proof}
Clearly $Z_{k}{\notin}\,\sigmaup(\qt)$, because $[Z,Z_{1},\hdots,Z_{k-1}]\,{\in}\,\qt{+}\sigmaup(\qt)$
by the minimality assumption:  $[[Z,Z_{1},\hdots,Z_{k-1}],Z_{k}]$ would belong to
$\qt{+}\sigmaup(\qt)$ if $Z_{k}$ belongs to $\qt\,{\cap}\,\sigmaup(\qt)$. 
Moreover, $[Z,Z_{1}]\,{\notin}\,\qt$ because $[[Z,Z_{1}],Z_{2},\hdots,Z_{k}]\,{\notin}\,\qt$.
For $1{\leq}i{\leq}k$,  set
$Z_{i,j}{=}\,Z_{j}$ for  $1{\leq}j{\leq}k{-}1$ and $j\,{\neq}i$, $Z_{i,i}{=}[Z_{i},Z_{k}]$.
Then 
\begin{equation*}
  [Z,Z_{1},\hdots,Z_{k-1},Z_{k}]=[Z,Z_{k},Z_{1},\hdots,Z_{k-1}]+
  {\sum}_{i=1}^{k-1}[Z,Z_{i,1},\hdots,Z_{i,k-1}].
\end{equation*}
Since $Z_{i,j}\,{\in}\,\qt$ for all $1{\leq}i,j{\leq}k{-}1$, all summands 
 $[Z,Z_{i,1},\hdots,Z_{i,k-1}]$ belong to $\qt{+}\sigmaup(\qt)$ by minimality. Thus
 $[Z,Z_{k},Z_{1},\hdots,Z_{k-1}]\,{\notin}\,\qt{+}\sigmaup(\qt)$ and  hence also
 $(Z_{k},Z_{1},\hdots,Z_{k-1})$ is a  Levi-sequence for $Z$.
 By repeating the same reasoning, we obtain that, for each $1{\leq}i{\leq}k{-}1$,
 $(Z_{i+1},\hdots,Z_{k},Z_{1},\hdots,Z_{i})$ is a Levi-sequence for $Z$.
This shows that  $Z_{i}{\notin}\,\sigmaup(\qt)$ for all $1{\leq}i{\leq}k$ and that, 
if $k\,{>}\,1$, then $[Z,Z_{i}]\,{\in}\,(\qt\,{+}\,\sigmaup(\qt)){\setminus}\qt$. 
 \par
 To prove that minimal length Levi-sequences are preserved by permutations, we argue by 
 recurrence on the length of the sequence. If $k(Z)\,{=}\,1$, 
then there is nothing to prove.
 The case $k(Z)\,{=}\,2$ was treated above.
  Assume that $k\,{>}\,2$ and that the statement 
 is true for Levi-sequences of length not exceeding $k{-}1$. 
 The commutator $[Z,Z_{1}]$ is a sum $Z'{+}\,Z''$ with $Z'\,{\in}\,\sigmaup(\qt){\setminus}\qt$
 and $Z''\,{\in}\,\qt$. 
Then 
 $(Z_{2},\hdots,Z_{k})$ is a minimal length sequence for $Z'$ and hence, by the recursive assumption,
 all $(Z_{i_{2}},\hdots,Z_{i_{k}})$ are minimal length Levi-sequences for $Z'$, for all permutations
 $(i_{2},\hdots,i_{k})$ of $(2,\hdots,k)$. This implies  that all $(Z_{1},Z_{i_{2}},\hdots,Z_{i_{k}})$
 are minimal length Levi-sequence for $Z$. Since we already showed that $(Z_{i},\hdots,Z_{k},Z_{1},
 \hdots,Z_{i-1})$ are  minimal length Levi-sequences for $i\,{=}\,2,\hdots,k$,
 we obtain invariance under permutations. 
 \end{proof}
\begin{lem}\label{l3.9}
 Elements of $\sigmaup(\qt){\setminus}\qt$ differing by an element of $\qt\,{\cap}\,\sigmaup(\qt)$ have
 same Levi-sequences. \qed
\end{lem}
\begin{dfn}\label{d3.6}
 We call  
\begin{equation}
 \Kt(\gs,\qt){=}\{Z\,{\in}\,\sigmaup(\qt)\,{\mid}\, [Z,Z_{1},\hdots,Z_{k}]\,{\in}\,\qt{+}\sigmaup(\qt),\,\forall k{\in}\Z_{+},\,
 \forall Z_{1},\hdots,Z_{k}{\in}\qt\}
\end{equation}
the \emph{Levi-degeneracy kernel} of $(\gs,\qt)$.
\end{dfn}
We recall from \cite{MMN21}: 
\begin{prop}
The Levi-degeneracy kernel of a $CR$-algebra $(\gs,\qt)$ is a Lie subalgebra of $\sigmaup(\qt)$ containing
$\qt\,{\cap}\,\sigmaup(\qt)$. 
\end{prop}
\begin{proof}
 We define by recurrence the sequence 
\begin{equation*}\begin{cases}
 \Kt_{0}(\gs,\qt)=\sigmaup(\qt),\\
 \Kt_{p}(\gs,\qt)=\left.\big\{Z\in\Kt_{p-1}(\gs,\qt)\,\right| \, [Z,\qt]\subseteq\qt+\Kt_{p-1}(\gs,\qt)\big\},& p\geq {1}.
 \end{cases}
\end{equation*}
Clearly all $\Kt_{p}(\gs,\qt)$ are complex linear spaces and
 $\Kt(\gs,\qt)\,{=}\,{\bigcap}_{p=0}^{\infty}\Kt_{p}(\gs,\qt)$. The statement of the Proposition is a consequence of: 
\begin{equation*}\tag{$*$}
\textsl{For all $p\,{=}\,0,1,\hdots$, the set $\Kt_{p}(\gs,\qt)$ is a Lie subalgebra of $\sigmaup(\qt)$.}
\end{equation*}
We prove $(*)$ by recurrence. It is trivially true for $p\,{=}\,0$. Assuming that it is true for some $p\,{\geq}0$, we obtain for
$Z_{1},Z_{2}\,{\in}\,\Kt_{p+1}(\gs,\qt)$, 
\begin{align*} &
 [[Z_{1},Z_{2}],\qt]\,{=}\,[Z_{1},[Z_{2},\qt]]]\,{+}\,
[Z_{2},[Z_{1},\qt]]] \\
&\quad \qquad \subseteq [(Z_{1}+Z_{2}),\qt+\Kt_{p}(\gs,\qt))]\subseteq \qt+\Kt_{p}(\gs,\qt)
\end{align*}
because $ [(Z_{1}+Z_{2}),\qt]\,{\subseteq}\, \qt\,{+}\,\Kt_{p}(\gs,\qt)$ 
 by  the definition of $\Kt_{p+1}(\gs,\qt)$ and we  get
 $[\Kt_{p}(\gs,\qt),\Kt_{p}(\gs,\qt)]\,{\subseteq}\,\Kt_{p}(\gs,\qt)$
 the recursive assumption.
\end{proof}
\subsection{Contact-nondegeneracy} 
Contact-nondegeneracy may also be characterised
 by using higher order commutators. 
\begin{thm}
 Let $(\gt_{\sigmaup},\qt)$ be a $CR$-algebra. Then 
\begin{equation*}
\at_0=\left\{X\in
{\Ft_{ -1}(\gs,\qt)}
\mid [X,X_{1},\hdots,X_{p}]\,{\in}\,{\Ft_{ -1}(\gs,\qt)},
\,\forall X_{1},\hdots,X_{p}\,{\in}\,{\Ft_{ -1}(\gs,\qt)}
\right\}
\end{equation*}
is the largest ideal 
of the real Lie algebra ${\Ft_{ -\infty}(\gs,\qt)}$ contained in $\Ft_{ -1}(\gs,\qt)$.
\qed
\end{thm} 
\begin{cor}
 A $CR$-algebra $(\gt_{\sigmaup},\qt)$ 
 is contact-nondegenerate if and only if for every $X\,{\in}\,{\Ft_{ -1}(\gs,\qt)}{\setminus}
 {\Ft_{ 0}(\gs,\qt)}
 $
 we can find $X_{1},\hdots,X_{r}\,{\in}\,{\Ft_{ -1}(\gs,\qt)}$
 such that \begin{equation*}\vspace{-20pt}[X,X_{1},\hdots,X_{r}]\,{\notin}\,{\Ft_{ -1}(\gs,\qt)}
 .\end{equation*}
 \qed 
\end{cor}

\begin{dfn}
 Let $(\gt_{\sigmaup},\qt)$ be a $CR$-algebra.  
The \emph{contact order} $
{
k^{c}\,{\in}\,\N{\cup}\{+\infty\}}$
 of an element $X$ of 
{$\Ft_{ -1}(\gs,\qt){\setminus}\Ft_{ 0}(\gs,\qt)$}
is 
\begin{equation*} k^{c}(X)\,{=}\,\inf\big\{r\,{\in}\,\N\mid \exists\, X_{1},\hdots,X_{r}\in
{\Ft_{ -1}(\gs,\qt)}\;\,\text{s.t.}
 \;\,[X,X_{1},\hdots,X_{r}]
 \notin
 {\Ft_{ -1}(\gs,\qt)}\big\}.
\end{equation*}
 Elements $X$ of $\Ft_{{-}1}(\gs,\qt){\setminus}\Ft_{{0}}(\gs,\qt)$
 with $k^{c}(X)\,{<}\,{+}\infty$ are said to have \emph{finite contact order}. 
 The \emph{contact order} of $(\gt_{\sigmaup},\qt)$ is 
$$k^{c}(\gs,\qt)\,{=}\,\sup\{k^{c}(X)\,{\mid}\,X\,{\in}\,\Ft_{ -1}(\gs,\qt)
{\setminus}\Ft_{ 0}(\gs,\qt)\}.$$
\end{dfn}
\begin{rmk} 
{In general, the contact order of a $CR$ algebra is less than or equal to its Levi order.}
We may have strict inequality 
{when the}
contact order {is}  greater or equal 
$2$, 
while 
$CR$-algebras of contact order  
{ $1$ are also} 
Levi $1$-nondegenerate.  
  In fact Levi-degenerate
$CR$-algebras may have finite contact order (see Example\,\ref{e3.32}).
\end{rmk}
\subsection{Homomorphisms of $CR$-algebras} 
A thorough discussion of 
$CR$-al\-ge\-bra homomorphisms can be found in \cite[\S{2}]{AMN06b} and \cite{AMN2013},
where also the corresponding notions for $CR$-manifolds are discussed.
\par 
 Let $\gt_{\sigmaup}$, 
 {$\gt_{\sigmaup'}'$}
 be real forms {of the complex Lie algebras $\gt$, $\gt'$, corresponding
 to anti-$\C$-linear involutions $\sigmaup$, $\sigmaup'$, respectively.}
A \textit{real} Lie algebra homomorphism 
 {$\phiup_{0}\,{:}\,\gs\,{\to}\,\gt_{\sigmaup'}'$} uniquely 
lifts to a 
\textit{complex} Lie algebra 
homomorphism
$\phiup\,{:}\,\gt\,{\to}\,\gt'$, characterized by  the property that 
\begin{equation}\label{e3.5} 
 \phiup\,{\circ}\,\sigmaup\,{=}\,\sigmaup'\,{\circ}\,\phiup.
\end{equation}
Let $\qt$ and $\qt'$ be complex Lie subalgebras of $\gt,$ $\gt'$, respectively.
\begin{dfn} 
A \emph{$CR$-algebra
homomorphism} 
\begin{equation}\label{e3.6} \phiup^{\sharp}:
 (\gt_{\sigmaup},\qt)\to(\gt_{\sigmaup}',\qt')
\end{equation}
is a real Lie algebra homomorphism 
\begin{equation}\label{e3.10}
 {\phiup_{0}}
 \,{\in}\,\Hom_{\R}(\gs,\gsp),\quad\text{such that}\quad \phiup(\qt)\,{\subseteq}\,\qt',
\end{equation}
for the complexification $\phiup$ of $\phiup_0$. Set
 \begin{equation} \label{e3.11}
\et\,{=}\,\{Z\,{\in}\,\gt\,{\mid}\,\phiup(Z)\,{\in}\,\qt'\,{\cap}\,\sigmaup'(\qt')\},\;\;
\et_{\sigmaup}\,{=}\,\et\,{\cap}\,\gs,\;\;\ft\,{=}\,\qt\,{\cap}\,\et.
\end{equation}
\par 
By \eqref{e3.5}, $\phiup(\gs)$  and $\et_{\sigmaup}$ 
are real forms of $\phiup(\gt)$ and $\et$, respectively.\par \smallskip
The  $CR$-algebras
 $(\phiup_{0}(\gs),\phiup(\qt))$ and $(\et_{\sigmaup},\ft)$
are the \emph{image} and \emph{typical fibre} of $\phiup^{\sharp}$,
respectively.
\end{dfn}
By \eqref{e3.5} and \eqref{e3.10} we also have the inclusions \;
$\phiup(\sigmaup(\qt))\,{\subseteq}\,\sigmaup'(\qt')$
and \par\noindent
{$\phiup_{0}(\qt\,{\cap}\,\gs)\,{\subseteq}\,\qt'{\cap}\,\gsp$.}
 {By passing to quotients, \eqref{e3.6} defines linear maps 
\begin{equation}\label{e3.9}
\,
\dfrac{\gs}{\qt\cap\gs}\xrightarrow{\;\phiup_{0}^{\flat}\;}\dfrac{\gsp}{\qt'\cap\gsp},\;\; \;
\dfrac{\qt}{\qt\cap\sigmaup(\qt)}\xrightarrow{\;\phiup^{\natural}\;}\dfrac{\qt'}{\qt'\cap\sigmaup'(\qt')},\;\;\;
\dfrac{\gt}{\qt}\xrightarrow{\;\phiup^{\flat}\;}\dfrac{\gt'}{\qt'}.
\end{equation}}
\begin{rmk}\label{r3.10}
Let $\Gf$, $\Gf'$ be complex Lie groups with Lie algebras $\gt$, $\gt'$ and 
$\Qf\,{\subseteq}\,\Gf,$ $\Qf'\,{\subseteq}\,\Gf'$ 
closed complex subgroups,
with  Lie algebras $\qt\,{\subseteq}\,\gt,$ $\qt'\,{\subseteq}\,\gt'$.
\par
Let $\Gfs$, $\Gfsp$  be real forms of $\Gf$, $\Gf'$, 
with Lie algebras $\gs$, $\gsp$, which are real forms of $\gt$, $\gt'$, respectively.
Then  
$\Kf_0\,{\coloneqq}\,\Qf\,{\cap}\,\Gfs$
and $\Kf_0'\,{\coloneqq}\,\Qf'\,{\cap}\,\Gfsp$ 
have Lie algebras
$\qt\,{\cap}\,\gs$ and $\qt'\,{\cap}\,\gt'_{\sigmaup'}$. \par 
{ The quotients
$\sfM\,{=}\,\Gfs/\Kf_0$ and  $\sfM'\,{=}\,\Gfsp/\Kf_0'$}
are homogeneous $CR$-man\-i\-folds with $CR$-algebras $(\gs,\qt)$ and $(\gsp,\qt')$
at the base points  $\pct{\simeq}\Kf_0$, $\pct'{\simeq}\Kf'_0$. \par
{Let $\Phi_{0}{:}\Gfs\,{\to}\,\Gfsp$ be a Lie group homomorphism, 
with {$\Phi_{0}(\Kf_{0})\,{\subseteq}\,\Kf'_{0}$.}
The induced map $\Phi_{0}^{\flat}\,{:}\,\sfM\,{\to}\,\sfM'$}
{on manifolds} 
is $CR$ if and only if 
the differential 
{$\phiup_{0}$ of $\Phi_{0}$} at the identity yields 
{a homomorphism
$\phiup^{\sharp}{:}(\gs,\qt)\,{\to}\,(\gsp,\qt')$ of $CR$-algebras.
The maps $\phiup_{0}^{\flat}$, $\phiup^{\natural}$, $\phiup^{\flat}$ correspond to differentials at 
corresponding base points $\pct,\,\pct'$:
the first of the smooth map for the  underling real smooth structures, 
the second of the restriction of its complexification to a
 map of  $\T^{0,1}_{\pct}\sfM$ into $\T^{0,1}_{\pct'}\sfM'$, 
 the third of the holomorphic map  of their ambient homogeneous complex manifolds
 $\Gf/\Qf$, $\Gf'/\Qf'$. It is natural to  classify
 $CR$-algebra homomorphisms fitting with the properties of $\phiup_{0}^{\flat}$, $\phiup^{\natural}$ 
 and $\phiup^{\flat}$. }
{According to \cite{AMN2013} and the definitions below, $\Phi_{0}^{\flat}$}
is a local $CR$-immersion, submersion,  isomorphism, spread, deployment, lift
of $CR$-manifold{s} 
if and only if $\phiup^{\sharp}${, as a $CR$-algebra homomorphism, is} 
a local immersion, submersion, isomorphism, spread, deployment, 
lift\footnote{In case {either $\Qf_{0}$ or $\Qf'_{0}$} is not closed, we may consider the \textit{germs} of 
locally homogeneous real analytic $CR$
manifolds defined by the pairs $(\gt_{\sigmaup},\qt)$, $(\gt_{\sigmaup}',\qt')$ 
(see \cite[\S{13}]{MN05}).}.
\end{rmk} 
 
 {
\begin{lem}
 Let \eqref{e3.6} be a $CR$-algebra homomorphism. Then  
\begin{equation}
 \qt^{\phiup}=\text{Lie subalgebra of $\gt'$ generated by  $\phiup(\qt)\,{+}\,\qt'{\cap}\,\sigmaup'(\qt')$}
\end{equation}
is the smallest Lie subalgebra of $\qt'$ containing 
$\qt'{\cap}\,\sigmaup(\qt')$ and for which $\phiup^{\sharp}\,{:}\,(\gs,\qt)\,{\to}\,(\gsp,\qt^{\phiup})$
is a $CR$-homomorphism. \qed
\end{lem}}
\begin{rmk}
The quotient $\qt^{\phiup}{/}(\qt'{\cap}\,\sigmaup'(\qt'))$,
 in the situation described by Remark\,\ref{r3.10}, 
 represents, at the base point of $\sfM'$, the formally integrable
 distribution generated by
the tangent complex vectors of $\sfM'$
that are $\Phi_{0}^{\flat}$-related to the type $(0,1)$ complex vector fields
 of $\sfM$.
\end{rmk}

  \begin{dfn} A $CR$-algebra homomorphism $\phiup^{\sharp}\,{:}\,(\gs,\qt)\,{\to}\,(\gt'_{\sigmaup'},\qt')$ is
\begin{itemize}
 \item a \emph{local $CR$-algebra immersion} if both $\phiup_{0}^{\flat}$ and 
 $\phiup^{\natural}$ 
 are injective;
\item a \emph{$CR$-algebra immersion} if, moreover, $\ker(\phiup)\,{=}\,\{0\}$. 
\item a \emph{local $CR$-algebra submersion} if $\phiup^{\flat}$ and $\phiup^{\natural}$ are onto;
\item a \emph{$CR$-algebra submersion} if 
moreover
$\phiup(\gt)\,{=}\,\gt'$;
\item a \emph{$CR$-algebra spread} if 
$\qt'\,{=}\,\qt^{\phiup}$;
\item a \emph{$CR$-algebra deployment} if it is a $CR$-spread with $\phiup^{\natural}$ injective;
\item a \emph{$CR$-lift} if it is a $CR$-deployment and $(\go',\qt')$
is totally complex; 
\item a \emph{local $CR$-algebra isomorphism} if it is both
a local $CR$-algebra immersion and a local $CR$-algebra submersion.
\item a \emph{$CR$-algebra isomorphism} if $\phiup_{0}$ is an isomorphism of Lie algebras and
$\phiup(\qt)\,{=}\,\qt'$
\end{itemize}
A local $CR$-algebra isomorphism $\phiup^{\sharp}\,{:}\,(\gs,\qt)\,{\to}\,(\gsp,\qt')$ 
such that $\phiup\,{:}\,\gt\,{\to}\,\gt'$ is injective is called (cf. \cite{MMN20z, Tan70})
\begin{itemize}
 \item a \emph{prolongation} of $(\gs,\qt)$ by $(\gsp,\qt')$;
 \item a \emph{reduction} of $(\gsp,\qt')$ to $(\gs,\qt)$.
\end{itemize}
A $CR$-algebra homomorphism \eqref{e3.6} for which  $\phiup^{\flat}$ is an isomorphism and
$\phiup^{\natural}$ is injective is called a \emph{strengthening} of $(\gs,\qt)$ by $(\gsp,\qt')$
or a \emph{weakening} of $(\gsp,\qt')$ by $(\gs,\qt)$.
\end{dfn}
{We caution that the typical fibre of a homomorphism of effective $CR$-algebras
may fail to be effective. } 
\begin{lem}\label{l3.16} Let $\gs,\vt_{\sigmaup}$ be real forms of complex Lie algebras 
$\gt,\vt$, with $\gt\,{\subseteq}\,\vt$, $\gs\,{\subseteq}\,\vt_{\sigmaup}$ and
$\qt,\wt$ complex Lie subalgebras of $\gt,\vt$, with $\qt\,{\subseteq}\,\wt$.
 \par
 The inclusion $\imath_{0}\,{:}\,\gs\,{\hookrightarrow}\,\vt_{\sigmaup}$ lifts to a 
$CR$-algebra homomorphism $$\imath^{\sharp}\,{:}\,(\gs,\qt)\,{\to}\,(\vt_{\sigmaup},\wt)$$  
 and $(\vt_{\sigmaup},\wt)$ is a prolongation of $(\gs,\qt)$ if and only if 
\begin{equation}\label{e3.11}
 \qt=\wt\cap\,\gt,\quad \wt=\qt+\wt\cap\,\sigmaup(\wt),\quad 
 \vt=\gt+\wt\cap\,\sigmaup(\wt).
\end{equation}
\end{lem} 
\begin{proof} The fact that the maps $\imath_{0}^{\flat}$, $\imath^{\natural}$, $\imath^{\flat}$ 
in \eqref{e3.9} are onto can be described by 
\begin{equation*}\tag{$*$}
 \vt=\gt+\wt\cap\,\sigmaup(\wt),\quad \wt=\qt+\wt\cap\,\sigmaup(\wt),\quad \vt=\gt+\wt,
\end{equation*}
while their injectivity translates into 
\begin{equation*}\tag{$**$}
 \wt\cap\,\sigmaup(\wt)\cap\gt=\qt\cap\sigmaup(\qt),\quad 
 \wt\cap\,\sigmaup(\wt)\cap\qt=\qt\cap\,\sigmaup(\qt),\quad
 \wt\cap\,\gt=\qt.
\end{equation*}
Clearly \eqref{e3.11} is equivalent to the validity of both
$(*)$ and $(**)$.
\end{proof} 
\begin{prop} Keep the notation of Lemma\,\ref{l3.16} and assume that $(\vt_{\sigmaup},\wt)$
is a prolongation of $(\gs,\qt)$.
Then $(\gs,\qt)$ is integrable (resp. totally complex, fundamental, totally real, Levi-degenerate,
contact-degenerate) if and only if $(\vt_{\sigmaup},\wt)$ shares the same property. Moreover,
$(\gs,\qt)$ and $(\vt_{\sigmaup},\wt)$ have the same Levi-order and the same contact order.
\end{prop} 
\begin{proof} We claim that 
\begin{equation*}\tag{$\dag$}
 (\wt{+}\,\sigmaup(\wt))\,{\cap}\,\gt=\qt\,{+}\,\sigmaup(\qt).
\end{equation*}
 Indeed, by \eqref{e3.11} we get  $\wt\,{+}\,\sigmaup(\wt)\,{=}\,\qt\,{+}\,\sigmaup(\qt)\,{+}\,\wt\,{\cap}\,\sigmaup(\wt)$
 and $(\dag)$  follows from $\qt\,{+}\,\sigmaup(\qt)\,{\subseteq}\,\gt$ and
 $\wt\,{\cap}\,\sigmaup(\wt)\,{\cap}\gt\,{=}\,\qt\,{\cap}\,\sigmaup(\qt)$. As a consequence, if $(\vt_{\sigmaup},\wt)$
 is integrable, then also $(\gs,\qt)$ is integrable. Vice versa,  
\begin{align*}
 [\wt,\sigmaup(\wt)]\,{=}\,[\qt+\wt\cap\sigmaup(\wt),\sigmaup(\qt)+\wt\cap\sigmaup(\wt)]
 \subseteq \wt+\sigmaup(\wt)+[\qt,\sigmaup(\qt)]
\end{align*}
shows that $[\wt,\sigmaup(\wt)]\,{\subseteq}\,\wt{+}\,\sigmaup(\wt)$
when $[\qt,\sigmaup(\qt)]\,{\subseteq}\,\qt\,{+}\,\sigmaup(\qt)$ and hence that $(\vt_{\sigmaup},\wt)$
is integrable when $(\gs,\qt)$ is integrable. \par
Again from $(\dag)$ it follows that $(\gs,\qt)$ is totally complex when $(\vt_{\sigmaup},\wt)$
is totally complex, while the vice versa is an easy consequence of the last equality in \eqref{e3.11}. \par
From the second equality in \eqref{e3.11} we deduce that, if $(\gs,\qt)$ is fundamental, then also
$(\vt_{\sigmaup},\wt)$ is fundamental. To prove the vice versa, we observe that, if $(\vt_{\sigmaup},\wt)$
is fundamental, then every element of $\gt$ is a linear combination of higher order commutators of
elements of $\wt{\cup}\,\sigmaup(\wt)$. We prove by recurrence on the integer $m$ that
elements of $\gt$ which are linear combination of commutators of elements of $\wt{\cup}\,\sigmaup(\wt)$
of length at most $m$ are also linear combinations of commutators of 
length at most $m$ of elements
of $\qt\,{\cup}\,\sigmaup(\qt)$. The case $m\,{=}\,0$ is trivial. For $m\,{=}\,1$,
we reduce to the case of a sum $Z_{1}{+}Z_{2}\,{\in}\,\gt$ with $Z_{1},Z_{2}\,{\in}\,\wt{\cup}\,\sigmaup(\wt)$.
We can write $Z_{i}{=}\,Z_{i}'{+}Z_{i}''$ with $Z_{i}\,{\in}\,\qt\,{\cup}\,\sigmaup(\qt)$
and $Z_{i}''{\in}\,\wt{\cap}\,\sigmaup(\wt)$. Then $Z''_{1}{+}Z''_{2}{\in}\wt{\cap}\,\sigmaup(\wt)\,{\cap}\,\gt
\,{=}\,\qt\,{\cap}\,\sigmaup(\wt)$ and  $Z_{1}{+}\,Z_{2}\,{=}\,Z_{1}'{+}Z_{2}'{+}(Z_{1}''{+}Z_{2}'')$  
is a sum of length one commutators of elements of $\qt\,{\cup}\,\sigmaup(\qt)$.
\par
 Let $m\,{\geq}\,2$ and assume that our assertion is true 
for commutators containing a lesser number of terms. It suffices to consider the case of a $Z\,{\in}\,\gt$
which is a sum of length $m$ commutators of elements of $\wt{\cup}\,\sigmaup(\wt)$: 
\begin{equation*}
 Z={\sum}_{h=1}^{r}[Z_{h,1},\hdots,Z_{h,m}],\;\;\; \text{with $Z_{h,i}\,{\in}\,\wt{\cup}\,\sigmaup(\wt)$}.
\end{equation*}
Let us write $Z_{h,i}\,{=}\,Z'_{h,i}\,{+}\,Z''_{h,i}$, with $Z'_{h,i}\,{\in}\,\qt\,{\cup}\,\sigmaup(\qt)$ and
$Z''_{h,i}\,{\in}\,\wt{\cap}\,\sigmaup(\wt)$. We have 
\begin{equation*}
 W={\sum}_{h=1}^{r}[Z'_{h,1},\hdots,Z'_{h,m}]\,{\in}\,\gt,
\end{equation*}
while $Z\,{-}\,W$ is an element of $\gt$ which can be written as a linear combination of
commutators of elements of $\wt{\cup}\,\sigmaup(\wt)$ of length less or equal $m{-}1$. 
In fact, a commutator of length $m$ of elements of $\wt{\cup}\,\sigmaup(\wt)$, in which at least
one element belongs to $\wt{\cap}\sigmaup(\wt)$, is a linear combination of commutators of
length less or equal $m{-}1$ of elements of $\wt{\cup}\,\sigmaup(\wt)$. \par
Similar arguments show that $(\gs,\qt)$ and $(\vt_{\sigmaup},\wt)$ have the same Levi-order
and the same contact order. \par 
If $\wt{=}\,\sigmaup(\wt)$, then also $\qt\,{=}\,\wt{\cap}\,\gt\,{=}\,\sigmaup(\wt)\,{\cap}\,\gt\,{=}\,\sigmaup(\qt)$.
If vice versa $\qt\,{=}\,\sigmaup(\qt)$, then $\wt{=}\,\qt\,{+}\wt{\cap}\sigmaup(\wt)=\sigmaup(\qt)\,{+}
\wt{\cap}\sigmaup(\wt)\,{=}\,\sigmaup(\wt)$. This shows that $(\gs,\qt)$ is totally real if and only if
$(\vt_{\sigmaup},\wt)$ is totally real.\par
The proof is complete.
\end{proof}

\begin{rmk}
 A $CR$-algebra homomorphism is a $CR$-algebra immersion if and only if the typical fibre is trivial,
 i.e. when $\ft\,{=}\,\et$ in \eqref{e3.11}. In this case 
 {$\ker(\phiup_{0})$} is contained in 
 {$\qt\,{\cap}\,\gs$}
 and hence
 $\phiup^{\sharp}$ 
 yields a local $CR$-algebra isomorphism
onto the basis 
 {$(\phiup_{0}(\gs),\phiup(\qt))$} and is  a $CR$-algebra immersion, yielding 
a $CR$-algebra isomorphism onto the basis
when $(\gs,\qt)$ is \textit{effective}.
\end{rmk}
\begin{exam} {As in Example\,\ref{ex2.3}, we take}
$\gt\,{=}\,\slt_{3}(\C)
$ 
{and} 
$\gs\,{=}\,\su(1,2)$ to be the set of
fixed points of the anti-$\C$-linear involution 
\begin{equation*}
 \sigmaup(X)=-K\,X^{*}\,K,\;\;\;\text{with}\;\;\; K= \left( 
\begin{smallmatrix}
  && 1\\
 & 1\\
 1
\end{smallmatrix}\right).
\end{equation*}
Then $(\gs,\qt)$ is the $CR$-algebra of the minimal orbit of $\SU(1,2)$ in $\CP^{2}$,
which is diffeomorphic to the sphere $\Sb^{3}$.\par
The two subalgebras 
\begin{equation*}\begin{aligned}
 \gt'&=\left.\left\{\left( 
\begin{smallmatrix}
  x_{11}&0& 0\\
 x_{21}& x_{22}&0\\
 x_{31} & x_{32}& x_{33}
\end{smallmatrix}\right)\right| x_{ij}\,{\in}\,\C,\; x_{11}{+}x_{22}{+}x_{33}{=}0\right\},\\
 \gt''&=\left.\left\{\left( 
\begin{smallmatrix}
0&{-}x_{21}& {-}x_{31}\\
 x_{21}& 0&{-}x_{32}\\
 x_{31} & x_{32}& 0
\end{smallmatrix}\right)\right| x_{ij}\,{\in}\,\C\right\}
\end{aligned}
\end{equation*}
of $\gt$ are both $\sigmaup$-invariant and have real forms $\gt'_{\sigmaup},\,\gt''_{\sigmaup}$.
Set 
\begin{equation*}
 \qt=\left.\left\{\left( 
\begin{smallmatrix}
  x_{11}&x_{12}& x_{13}\\
 0& x_{22}&x_{23}\\
0 & x_{32}& x_{33}
\end{smallmatrix}\right)\right| x_{ij}\,{\in}\,\C,\; x_{11}{+}x_{22}{+}x_{33}{=}0\right\},\;\;
\qt'=\qt\,{\cap}\,\gt',\;\;\qt''=\qt\,{\cap}\,\gt''.
\end{equation*}
The inclusions of $\gt'_{\sigmaup}$ and $\gt''_{\sigmaup}$ into $\gs$ induce 
local $CR$-isomorphisms 
\begin{equation*}
\varpiup'\,{:}\,(\gt'_{\sigmaup},\qt')\,{\to}\,(\gs,\qt)\;\;\text{and}\;\;
\varpiup''\,{:}\,(\gt''_{\sigmaup},\qt'')\,{\to}\,(\gs,\qt).
\end{equation*}
The Lie subalgebras $\gs'$ and $\gs''$ generate Lie subgroups $\Gf'$, $\Gf''$ of $\SU(1,2)$,
whose orbits in $\CP^{2}$ we denote by $\sfM',$ $\sfM''$.
The local $CR$-algebra homomorphisms $\varpiup',$ $\varpiup''$  
 lift to $CR$-embeddings $\sfM'\,{\hookrightarrow}\,\sfM$, $\sfM''\,{\hookrightarrow}\,\sfM$,
the first onto $\sfM$ minus one point, the second being a $CR$-isomorphism. 
\end{exam} 
\begin{exam} We consider the odd dimensional sphere $\Stt^{2n+1}$ as a real hypersurface in
$\C^{n+1}$, on which 
the group $\SU(n{+}1)$ acts transitively.  \par
Let $\e_{0},\e_{1},\hdots,\e_{n}$ be the canonical orthonormal
basis of $\C^{n+1}$. We associate to $\Stt^{2n+1}$ at $\e_{0}$ the $CR$-algebra
$(\su(n{+}1),\qt)$, where 
\begin{equation*}
 \qt=\left.\left\{ 
\begin{pmatrix}
 0 & \vq^{\intercal}\\
 0 & A
\end{pmatrix}\in\slt_{n+1}(\C)\,\right| \, \vq\,{\in}\,\C^{n},\; A\in\slt_{n}(\C)\right\}.
\end{equation*}
The parabolic Lie subalgebra 
\begin{equation*}
\qt'=\left.\left\{ 
\begin{pmatrix}
 a & \vq^{\intercal}\\
 0 & A
\end{pmatrix}\in\slt_{n+1}(\C)\,\right| \,a\in\C,\;  \vq\,{\in}\,\C^{n},\; A\in\gl_{n}(\C)\right\} 
\end{equation*}
of $\slt_{n+1}(\C)$ is the smallest containing $\qt$ and 
$(\su(n{+}1),\qt')$ is the $CR$-algebra of $\CP^{n}$ at its point $\ptt_{0}$ having  homogeneous
representative $\e_{0}$. \par
The identity map on $\su(n{+}1)$ yields a $CR$-algebra homomorphism
\begin{equation*}
\id^{\sharp}\,{:}\,(\su(n{+}1),\qt)\,{\to}\,(\su(n{+}1),\qt')\end{equation*}
 which is a $CR$-algebra lift: it 
corresponds to the Hopf fibration \begin{equation*}\Stt^{2n+1}\,{\ni}\,(z_{0},z_{1},\hdots,z_{n})\,{\to}\,
(z_{0}:z_{1}:\cdots:z_{n})\,{\in}\,\CP^{n}.\end{equation*}
\par This example illustrates a general geometrical notion 
that was introduced and investigated in
\cite{AMN2013}.
\end{exam}
\begin{exam}
 Let $\gt$ be a finite dimensional complex Lie algebra. Its \emph{complex conjugate} $\bar{\gt}$ is the complex
 Lie algebra obtained by considering on the same underlying real Lie algebra $\gt_{[\R]}$ the scalar multiplication
 $\lambdaup\,{*}\,Z\,{=}\,\bar{\,\lambdaup}\,{\cdot}\,Z$ (for $\lambdaup\,{\in}\,\C$ and $Z\,{\in}\,\gt)$.
 The identity map on $\gt_{[\R]}$ is an anti-$\C$-linear isomorphism of $\gt$ onto $\bar{\gt}$ and of $\bar{\gt}$
 onto $\gt$, that we will indicate by $Z\,{\to}\,\bar{Z}$.
 \par
 The \emph{complexification} $\gt_{[\R]}^{\C}$ of $\gt_{[\R]}$ 
 can be identified with the direct sum $\gt\,{\oplus}\,\bar{\gt}$ and $\gt_{[\R]}$ with the locus 
 of points of $\gt_{[\R]}^{\C}$ which are fixed by  the anti-$\C$-linear involution $(Z_{1},Z_{2})\,{\to}\,(\bar{Z}_{2},\bar{Z}_{1})$. \par
 If $\Gf$ is a complex Lie group with Lie algebra $\gt$ and $\Qf$ a closed subgroup with Lie algebra $\qt$, then
 the pair  $(\gt_{[\R]},\qt\,{\oplus}\,\bar{\gt})$ is the $CR$-algebra at the base point of the homogeneous complex
 manifold $\Gf{/}\Qf$. Anyhow, $(\gt_{[\R]},\qt\,{\oplus}\,\bar{\gt})$ is a \textit{totally complex} $CR$-algebra and,
 if $\gs$ is any real form of $\gt$, then the inclusion $\imath$ 
 of $\gs$ into $\gt_{[\R]}$ yields  a $CR$-map 
 $\imath^{\sharp}\,{:}\,(\gs,\qt)\,{\to}\,(\gt_{[\R]},\at\,{\oplus}\,\bar{\qt})$, which is a $CR$-immersion
 and a local $CR$-isomorphism if and only if $(\gs,\qt)$ is totally complex.  
\end{exam}
\subsection{$\gs$-equivariant $CR$-fibrations}\label{s3.7}
Let $\gt_{\sigmaup}$ be a real Lie algebra, $\qt,\qt'$ 
complex Lie subalgebras of its complexification $\gt$.
If $\qt\,{\subseteq}\,\qt'$,
then the identity map on $\gt_{\sigmaup}$ and the inclusion $\qt\,{\hookrightarrow}\,
\qt'$ define a $CR$-algebra homomorphism  
\begin{equation}\label{e1.9} \varpi\,{=}\,\id^{\sharp}:
(\gt_{\sigmaup},\qt)\to(\gt_{\sigmaup},\qt')\, .\end{equation}
\begin{dfn}(cf.  \cite[\S{5}]{MN05})
We call \eqref{e1.9} a \emph{$\gt_{\sigmaup}$-equivariant $CR$-fibration}
iff  \eqref{e1.9} is a $CR$-submersion, i.e. 
{if}
\begin{equation}\label{e3.13}
 \qt'=\qt'\cap\sigmaup(\qt')+\qt.
\end{equation}
\end{dfn}
{Note} 
that \eqref{e3.13} is always satisfied when
$\qt\,{\subseteq}\,\qt'\,{\subseteq}\,\qt\,{+}\,\sigmaup(\qt)$. 
\begin{rmk} If $\sfM$ and $\sfM'$ are homogeneous $CR$-manifolds of a real Lie group 
{$\Gf_{\sigmaup}$}
with Lie algebra $\gt_{\sigmaup}$ and $(\gt_{\sigmaup},\qt)$, $(\gt_{\sigmaup},\qt')$ their $CR$-algebras at corresponding points
$\pct,\pct'$, then
condition \eqref{e3.13} means that 
the corresponding $\Gf_{\sigmaup}$-equivariant map $f\,{:}\,\sfM\,{\to}\, \sfM'$ is a $CR$-submersion, 
{therefore defining}
a \emph{$CR$-fibration.}
\end{rmk}
\begin{dfn}
We call the $CR$-algebras 
$(\gs,\qt')$ and 
{$(\gs\,{\cap}\,\qt',\qt\,{\cap}\,\sigmaup(\qt'))$}
the \textit{basis} and the \textit{typical fibre} 
of 
\eqref{e1.9}, respectively. \par
By writing  
\begin{equation} 
\begin{CD}
 0 @>>>(\gs'',\qt'')@>>> (\gs,\qt) @>>>(\gs,\qt') @>>> 0
\end{CD}
\end{equation}
we mean that {the arrow $(\gs,\qt)\,{\to}\,(\gs,\qt')$ represents 
a $CR$-algebra homomorphism
satisfying}
\eqref{e3.13}  and that {$\gt_{\sigmaup}''\,{\simeq}\,
\gs\,{\cap}\,\qt'$, $\qt''\,{\simeq}\,\qt\,{\cap}\,\sigmaup(\qt')$ define its} typical fibre.
{Note that the complexification of $\gs\,{\cap}\,\qt'$ is $\qt'\,{\cap}\,\sigmaup(\qt')$
and $\qt\,{\cap}\,\sigmaup(\qt')\,{=}\,\qt\,{\cap}\,(\qt'\,{\cap}\,\sigmaup(\qt'))$.} 
\end{dfn}
\begin{rmk}
When
$\Kf_0$ and ${\Kf_0'}$ are closed subgroups of $\Gfs$ with Lie algebras $\qt\,{\cap}\gs$
and $\qt'{\cap}\,\gs$, representing the isotropies for
 $\sfM$ and $\sfM'$ at corresponding base points, 
the homogeneous space $\Kf_0'{/}\Kf_0$ is the typical fibre. It
is a homogeneous $CR$ manifold, with 
$CR$ algebra $(\gs\,{\cap}\,\qt',\qt'\,{\cap}\,\sigmaup(\qt))$.
\end{rmk} 
\subsection{Effective reduction} 
\begin{thm}
 If $(\gs,\qt)$ is a $CR$-algebra, $\at_{\sigmaup}$ an ideal of $\gs$ contained in $\qts$, with
 complexification $\at,$ and $\piup\,{:}\,\gs\,{\to}\,\gs{/}\!\at_{\sigmaup}$ the canonical projection, 
 then $\piup^{\sharp}\,{:}\,(\gs,\qt)\,{\to}\,(\gs{/}\!\at_{\sigmaup},\;\qt{/}\!\at)$
 is a local $CR$-isomorphism. \par
 Given a $CR$-algebra $(\gs,\qt)$, there is an effective $CR$-algebra $(\gs',\qt')$, unique up to  
 $CR$-algebra isomorphisms,
which is locally $CR$-isomorphic to $(\gs,\qt)$. 
\end{thm} 
\begin{proof} If 
{$\at_{\sigmaup}$ is an ideal of $\gs$ containend in $\Ft_{ 0}(\gs,\qt)$,}
 then its complexification $\at$ is {a $\sigmaup$-invariant ideal of $\gt$}
 and thus,
 by passing to the quotients, $\piup$ defines an isomorphism
 of $\gs{/}(\qt\,{\cap}\,\gs)$ onto $(\gs{/}\!\at_{\sigmaup})/((\qt\,{\cap}\,\gs){/}\!\at_{\sigmaup})$. Its complexification is an
 isomorphism of $\qt{/}(\qt\,{\cap}\,\sigmaup(\qt))$ onto $(\qt{/}\!\at){/}((\qt/\!\at)\,{\cap}\,(\sigmaup(\qt){/}\!\at)$.
 \par
The last statement follows by taking $\at_{\sigmaup}$ equal to the largest ideal of $\gs$
contained in $\qt\,{\cap}\,\gs$.
\end{proof}
\subsection{Fundamental reduction}
\begin{thm} 
\label{t4.4}
Every 
$CR$-algebra $(\gs,\qt)$
admits a unique $\gs$-equivariant $CR$-fibration with 
{totally real basis and fundamental typical fibre.} 
\begin{itemize}
 \item {Its} totally real basis {is} $(\gs,\Lie(\qt\,{+}\,\sigmaup(\qt))$;
 \item {its} fundamental typical fibre {is} 
 $(\gs',\qt)$ with $\gs'\,{=}\,\gs\,{\cap}\,\Lie(\qt\,{+}\,\sigmaup(\qt))$.
\end{itemize}\par
The fundamental fibre is totally complex iff $(\gs,\qt)$ is integrable.
\end{thm}
\begin{proof} Let $(\gs,\qt)$ be a  $CR$
algebra and $\qt'{=}\,\Lie(\qt\,{+}\,\sigmaup(\qt))$. 
Then $\gs'\,{=}\,\qt'\,{\cap}\,\gs$ is a real form
of the complex $\sigmaup$-invariant Lie algebra $\qt'$
and \eqref{e1.9} yields
a $\gs$-equivariant $CR$-fibration with a \textit{totally real} basis $(\gs,\qt')$.
The fibre $(\gs',\qt)$ is clearly fundamental and totally complex iff 
$\Lie(\qt{+}\sigmaup(\qt))\,{=}\,\qt{+}\sigmaup(\qt)$. 
\end{proof}
\subsection{Complex foliations} \label{s3.9}
Let $\sfM$ be a $\Gfs$-homogeneous $CR$-manifold, $\pct\,{\in}\,\sfM$ and  
$(\gs,\qt)$ its $CR$-algebra at $\pct$. 
If $\textfrak{F}$ is a $\Gfs$-invariant foliation of $\sfM$, then the stabiliser in $\Gfs$ 
of the leaf of $\textfrak{F}$ through $\pct$ is a closed Lie subgroup $\Lfs$ of $\Gfs$
containing the isotropy 
subgroup $\Kfs$ and its Lie algebra $\lts$  is a real Lie subalgebra of $\gs$ 
containing $\qt\,{\cap}\,\gs$.
\begin{dfn}\label{d3.11}
 Let $(\gs,\qt)$ be a $CR$-algebra, $\lts$ a Lie subalgebra of $\gs$ and $\lt$ its complexification.
 If $\qt\,{\cap}\,\gs\,{\subseteq}\,\lts$, we say that the $CR$-algebra $(\lts,\lt\,{\cap}\,\qt)$ is 
 the \emph{typical leaf} of a $\gs$-invariant foliation of $(\gs,\qt)$. 
\end{dfn}
\begin{lem}\label{l3.22} Keep the notation of Definition\,\ref{d3.11}.
Let
$(\lts,\lt\,{\cap}\,\qt)$ 
 be the typical leaf of a $\gs$-equivariant
 foliation of the $CR$-algebra $(\gs,\qt)$.
If 
\begin{equation}\label{e3.14}
\qt\cap\sigmaup(\qt)\subseteq\lt\subseteq\qt+\sigmaup(\qt),
\end{equation}
then $(\lts,\qt\,{\cap}\,\lt)$ is totally complex and 
\begin{equation}\label{e3.15}
 \lt=(\lt\cap\qt)+\sigmaup(\lt\cap\qt).
\end{equation}
\end{lem} 
\begin{proof}
 Let $\Vtt$ be a $\C$-linear complement of $\qt\,{\cap}\,\sigmaup(\qt)$ in $\qt$. Then 
\begin{equation*}
  \qt+\sigmaup(\qt)=(\qt\cap\sigmaup(\qt))\oplus\Vtt\oplus\sigmaup(\Vtt),
\end{equation*}
and from \eqref{e3.14} we obtain 
\begin{gather*}
\qt\cap\lt=(\qt\cap\sigmaup(\qt))\oplus(\Vtt\cap\lt),\\
 \lt=(\qt\cap\sigmaup(\qt))\oplus(\Vtt\cap\lt)\oplus\sigmaup(\Vtt\cap\lt),
\end{gather*}
yielding  \eqref{e3.15} and proving our claim. 
\end{proof}

\begin{lem}
 Let $(\gs,\qt)$ be a $CR$ algebra, $\wt$ a complex Lie subalgebra of $\gt$ satisfying 
\begin{gather}
 \qt\cap\sigmaup(\qt)\subseteq\wt\subseteq\qt+\sigmaup(\qt) 
 \intertext{and set}
  \pt=\{Z\in\wt\mid [Z,\sigmaup(\wt)]\subseteq\wt+\sigmaup(\wt)\}.
\end{gather} 
Then 
$\pt$ and $\lt\,{=}\,\pt{+}\,\sigmaup(\pt)$ are complex Lie subalgebras  and 
$\lt$ satisfies \eqref{e3.14}.
\end{lem} 
\begin{proof}If $Z_{1},Z_{2}\,{\in}\,\pt$, then 
\begin{align*}\,
& [[Z_{1},Z_{2}],\sigmaup(\wt)]=[Z_{1},[Z_{2},\sigmaup(\wt)]]+[Z_{2},[Z_{1},\sigmaup(\wt)]] \\
&\qquad\quad\subseteq [Z_{1},\wt+\sigmaup(\wt)]+[Z_{2},\wt+\sigmaup(\wt)]\subseteq\wt+\sigmaup(\wt)
\end{align*}
 shows that $\pt$ is a Lie subalgebra. \par
Clearly \, $
\at\,{=}\,\big\{Z\,{\in}\,\wt\,{+}\,\sigmaup(\wt)\,{\mid}\, [Z,\wt\,{+}\,\sigmaup(\wt)]\,{\subseteq}\,\wt\,{+}\,\sigmaup(\wt)\big\}$
\,  is
a Lie subalgebra of $\gt$ containing $\pt{+}\,\sigmaup(\pt)$.
It suffices to show that $ \pt{+}\,\sigmaup(\pt)\,{=}\,\at$. \par
Take $Z_{1},Z_{2}\,{\in}\,\wt$ with
$Z_{1}{+}\sigmaup(Z_{2})\,{\in}\,\at$. Then 
\begin{align*}
 [Z_{1}+\sigmaup(Z_{2}),\sigmaup(\wt)]\subseteq\wt+\sigmaup(\wt) \Longrightarrow [Z_{1},\sigmaup(\wt)]
 \subseteq\wt+\sigmaup(\wt),\\
  [Z_{1}+\sigmaup(Z_{2}),\wt]\subseteq\wt+\sigmaup(\wt) \Longrightarrow [\sigmaup(Z_{2}),\wt]
 \subseteq\wt+\sigmaup(\wt) \\
 \Longrightarrow [Z_{2},\sigmaup(\wt)]
 \subseteq\wt+\sigmaup(\wt),
\end{align*}
show that $Z_{1},Z_{2}\,{\in}\,\pt$.
This completes the proof.
\end{proof}
\begin{dfn} Let $(\gs,\qt)$ be a $CR$-algebra. Its \emph{Levi-kernel} is the subspace 
\begin{equation}
\kt(\gs,\qt)=\{Z\in\qt\mid [Z,\sigmaup(\qt)]\subseteq\qt+\sigmaup(\qt)\}
\end{equation}
 \end{dfn}
\begin{prop} Let $(\gs,\qt)$ be a $CR$ algebra. 
Its Levi-kernel  $\kt(\gs,\qt)$ and $\kt(\gs,\qt){+}\sigmaup(\kt(\gs,\qt))$
are Lie subalgebras of $\gt$, the latter satisfying
\eqref{e3.14}.\par
An element $Z\,{\in}\,\sigmaup(\qt){\setminus}\qt$ has Levi order larger of equal to  two if and only if
$\sigmaup(Z)\,{\in}\,\kt(\gs,\qt)$. \qed
 \end{prop} 
At infinitesimal level, the study 
of invariant complex foliations of 
homogeneous $CR$-manifolds 
reduces to that of Lie subalgebras satisfying~\eqref{e3.14}.  
\begin{prop}
 Let $(\ft_{\sigmaup},\pt)$ be a prolongation of $(\gs,\qt)$, with $\gs\,{\subseteq}\,\ft_{\sigmaup}$
 and $\qt\,{\subseteq}\,\pt$. If $(\lt_{\sigmaup},\lt\,{\cap}\,\pt)$ is the typical leaf of an $\ft_{\sigmaup}$-invariant
 foliation of $(\ft_{\sigmaup},\pt)$, then 
$(\lt_{\sigmaup}{\cap}\,\gs,\lt\,{\cap}\,\qt)$ is the typical leaf of a $\gs$-invariant
 foliation of $(\gt_{\sigmaup},\qt)$ and the $CR$-algebras $(\lt_{\sigmaup},\lt\,{\cap}\,\pt)$
 and $(\lt_{\sigmaup}{\cap}\,\gs,\lt\,{\cap}\,\qt)$ are locally isomorphic.\qed
\end{prop}
\subsection{Levi-nondegenerate reduction} 
Levi-nondegenerate reductions are an important special example of complex foliations.
\begin{lem}\label{l3.28}
 Assume that $(\gs,\qt)$ is  Levi-degenerate and let $\qt'$ be a complex Lie subalgebra of $\gt$ with 
\begin{equation}\label{e3.20}
 \qt\subsetneqq\qt'\subseteq\qt+\sigmaup(\qt).
\end{equation}
\par
 Then the typical fibre $(\qt'\,{\cap}\,\gs,\qt\,{\cap}\,\sigmaup(\qt'))$ 
 of the $\gs$-equivariant fibration $\varpi\,{:}\,(\gs,\qt)\,{\to}\,(\gs,\qt')$ 
 is totally complex.
\end{lem} 
\begin{proof}
 We apply Lemma\,\ref{l3.22}, taking $\lts\,{=}\,\qt'\,{\cap}\,\gs$.
\end{proof}
The fact  that there is a unique maximal complex Lie subalgebra 
$\qt'$ satisfying \eqref{e3.20}
is the contents of Theorem\,\ref{t3.6}. 
\begin{thm}[{Levi}-nondegenerate reduction]\label{t3.29}
Every fundamental
$CR$-al\-ge\-bra $(\gs,\qt)$ admits a unique $\gs$-equivariant fibration. Its  
Levi-non\-de\-gen\-er\-ate basis is
 $(\gs,\qt_{\{\sigmaup\}})$, where $\qt_{\{\sigmaup\}}$ is the $\sigmaup$-ex\-ten\-sion of $\qt$,
  and  the  typical fibre is 
 totally complex. 
\par
The $CR$-algebra $(\gs,\qt)$ 
is 
Levi-nondegenerate if and only if
$\qt_{\{\sigmaup\}}\,{=}\,\qt.$ 
\end{thm} 
\begin{proof}
 The  claims are consequences of Theorem\,\ref{t3.6} and
 Lemma\,\ref{l3.28}.  
\end{proof}
\subsection{Reductive $CR$-algebras} \label{s.reductive}
We specialise to homogeneous 
$CR$-man\-i\-folds the classical notion of \textit{reductiveness}  for general 
homogeneous
spaces (see \cite[Ch.X]{KN}). 
\begin{dfn}\label{d3.13}
We call \emph{reductive} 
a $CR$-algebra $(\gs,\qt)$ 
whose lifted $CR$-structure $\qt$ admits 
an \emph{invariant complement}, i.e. 
a  $\C$-linear 
complement  $\qt^{c}$ 
in $\gt$
such that  
\begin{align}\label{e3.21}
&[\qt\cap\sigmaup(\qt),\,\qt^{c}]\subseteq\qt^{c},\\
\label{e3.22}
 &\gt=(\qt\cap\sigmaup(\qt))\,\oplus\, (\qt\cap\sigmaup(\qt^{c}))\,
 \oplus\,(\qt^{c}\cap\sigmaup(\qt))\,\oplus\,(\qt^{c}\cap\sigmaup(\qt^{c})) ,
\end{align}
and
a $\Gfs$-homogeneous $CR$-manifold $\sfM$ whose 
 associated $CR$-algebra at $\pct$  is reductive, with an $\Ad(\Kfs)$-invariant
  $\qt^{c}$.
\end{dfn}
\begin{rmk} As an easy consequence of \eqref{e3.21} we obtain 
\begin{equation*} 
\begin{cases}
 [(\qt\cap\sigmaup(\qt)),\,(\qt\cap\sigmaup(\qt^{c}))]\subseteq\qt\cap\sigmaup(\qt^{c}),\\
 [(\qt\cap\sigmaup(\qt)),\,(\qt^{c}\cap\sigmaup(\qt)]\subseteq\qt^{c}\cap\sigmaup(\qt),\\
 [(\qt\cap\sigmaup(\qt)),\,(\qt^{c}\cap\sigmaup(\qt^{c}))]\subseteq\qt^{c}\cap\sigmaup(\qt^{c}),
\end{cases}
\end{equation*} 
so that \eqref{e3.22} is a decomposition of $\gt$ into a direct sum of $\qt\,{\cap}\,\sigmaup(\qt)$-modules.
\par
If $(\gs,\qt)$ is the $CR$-algebra at $\pct$ of a $\Gfs$-homogeneous $CR$-manifold $\sfM$,
then the canonical projection identifies $(\qt\,{\cap}\,\sigmaup(\qt^{c}))$ with $\T^{0,1}_{\pct}\sfM$,
$(\qt^{c}\,{\cap}\,\sigmaup(\qt))$ with $\T^{1,0}_{\pct}\sfM$ and
$(\qt^{c}\,{\cap}\,\sigmaup(\qt^{c}))$ with the fibre on $\pct$ of the complexification of the Reeb distribution
of the underlying generalised contact structure of $\sfM$.
\end{rmk}

\begin{exam}
Being \textit{reductive} depends on the choice of symmetries. 
 A very simple example is provided by the sphere $\Sb^{3}$ in $\CP^{2}$. If we take 
 $\Gfs\,{=}\,\SU(1,2)$ then $\gt=\slt_{3}(\C)$, with 
\begin{equation*}
\gs=\left.\left\{ \left( 
\begin{smallmatrix}
\lambdaup& \zetaup& i \stt\\
z & i \ktt &{-}\bar{\zetaup}\\
i\ttt &{-}\bar{z}&{-}\bar{\lambdaup}\end{smallmatrix}\right) \right| 
\begin{smallmatrix}
 \lambdaup,\zetaup,z\in\C,\\
 \stt,\ttt,\ktt\in\R,\\
 \lambdaup-\bar{\lambdaup}+i\ktt=0
\end{smallmatrix}\right\},\quad
 \qt=\left.\left\{ \left( 
\begin{smallmatrix}
 a_{1,1}& a_{1,2}&a_{1,3}\\
 0 & a_{2,2}&a_{2,3}\\
 0 &a_{3,2}&a_{3,3}\end{smallmatrix}\right) \right| a_{1,1}{+}a_{2,2}{+}a_{3,3}\,{=}\,0
\right\}.
\end{equation*}
One easily checks that it is not reductive for the action of $\SU(1,2)$. \par
A maximal compact subgroup of $\SU(1,2)$ acts transitively on $\Sb^{3}$. Thus we can make a
\textit{reduction} by taking
$\Gf'_{\sigmaup}{=}\Gfs\,{\cap}\,\SU(3)\,{\simeq}\,\Ub(2)$,  $\gt'\,{\simeq}\,\gl_{2}(\C)$, with 
\begin{equation*}
 \gt'{=}\left.\left\{ \left( 
\begin{smallmatrix}
\lambdaup& \zetaup & {-}\thetaup\\
z & {-}2\lambdaup& z \\
\thetaup & \zetaup & \lambdaup 
\end{smallmatrix}\right)\right| \begin{smallmatrix}
\lambdaup,z,\zetaup,\thetaup\,{\in}\,\C
\end{smallmatrix}\right\},\,
\gt'_{\sigmaup}{=}\left.\left\{ \left( 
\begin{smallmatrix}
 i\stt & {-}\bar{z} & {-}i\ttt\\
 z & {-}2i\stt & z \\
 i\ttt & {-}\bar{z} &i\stt 
\end{smallmatrix}\right)\right| 
\begin{smallmatrix}
 z\in\C,\\ \stt,\ttt\in\R
\end{smallmatrix}\right\},\,
 \qt'{=}\left.\left\{ \left( 
\begin{smallmatrix}
\lambdaup& \zetaup & 0\\
0 & {-}2\lambdaup& 0 \\
0 &\zetaup  & \lambdaup 
\end{smallmatrix}\right)\right| \begin{smallmatrix}
\lambdaup,\zetaup\,{\in}\,\C
\end{smallmatrix}\right\}.
\end{equation*}
The matrices in $\gt'{\cap}\,\sigmaup(\gt')$ are in this case diagonal and $(\gt'_{\sigmaup},\qt')$ is reductive.
\end{exam}
\begin{lem}\label{l3.33}
 Let $(\gs,\qt)$ be a reductive $CR$-algebra, with invariant complement $\qt^{c}$. Then 
\begin{equation} \label{e3.23} \begin{cases}
\qt=(\qt\cap\sigmaup(\qt))\oplus(\qt\cap\sigmaup(\qt^{c})),\;\; 
\sigmaup(\qt)=(\qt\cap\sigmaup(\qt))\oplus(\qt^{c}\cap\sigmaup(\qt)),\\
(\qt{+}\sigmaup(\qt))\cap\qt^{c}=\sigmaup(\qt)\cap\qt^{c},\;\; 
 (\qt{+}\sigmaup(\qt))\cap\sigmaup(\qt^{c})=\qt\cap\sigmaup(\qt^{c}). 
 \end{cases}
\end{equation}
\end{lem} 
\begin{proof} From \eqref{e3.22} we get the first line of \eqref{e3.23} and 
the direct sum decomposition 
\begin{equation*}
 \qt+\sigmaup(\qt)=\qt\,{\cap}\,\sigmaup(\qt)\,\oplus\,\qt^{c}{\cap}\,\sigmaup(\qt)\,\oplus\,
 \qt\,{\cap}\,\sigmaup(\qt^{c}),
\end{equation*}
yielding the second line of  \eqref{e3.23}.
\end{proof}
\begin{lem}\label{l3.10}
  Let $(\gs,\qt)$ a reductive  
  $CR$-algebra, with invariant complement $\qt^{c}$. 
An element of $\sigmaup(\qt){\setminus}\,\qt$ 
having finite Levi-order admits a minimal length Levi-sequence with
elements in $\qt\,{\cap}\,\sigmaup(\qt^{c})$. 
\end{lem} 
\begin{proof} Assume that
$Z\,{\in}\,\sigmaup(\qt){\setminus}\qt$ has finite Levi-order $k$ and let
 $(Z_{1},\hdots,Z_{k})$ be a minimal length Levi-sequence in $\qt$
 for $Z$, having a maximal number of terms
 in $\qt\,{\cap}\,\sigmaup(\qt^{c})$. Assume by contradiction that the sequence contains
 an element non belonging to $\qt\,{\cap}\,\sigmaup(\qt^{c})$. By the invariance under
 permutations, we may assume that this element is $Z_{k}$. Decompose $Z_{k}$ into a sum
 $Z_{k}{=}Z_{k}'{+}Z_{k}''$ with
 $Z_{k}'\,{\in}\,\qt\,{\cap}\sigmaup(\qt^{c})$ and $Z_{k}''\,{\in}\,\qt\,{\cap}\,\sigmaup(\qt)$.
 Then $$[Z,Z_{1},\hdots,Z_{k}]\,{=}\,[Z,Z_{1},\hdots,Z'_{k}]{+}[Z,Z_{1},\hdots,Z''_{k}]$$ 
 implies that $(Z_{1},\hdots,Z_{k-1},Z_{k}')$ is a Levi-sequence for $Z$, since 
 $[Z,Z_{1},\hdots,Z''_{k}]$, by the minimality of $k$, belongs to $\qt\,{+}\,\sigmaup(\qt)$. 
 This contradicts the maximality
 of the number of terms in $\qt\,{\cap}\,\sigmaup(\qt^{c})$ of the Levi-sequence 
 $(Z_{1},\hdots,Z_{k})$. Thus all its elements belong to $\qt\,{\cap}\,\sigmaup(\qt^{c})$. 
\end{proof}
\begin{rmk}
 By Lemmas\,\ref{l3.9} and \ref{l3.10}, while discussing Levi-non\-deg\-en\-er\-acy of reductive
 $CR$-algebras, we can limit ourselves to computing Levi-orders of elements
 of $\sigmaup(\qt)\,{\cap}\,\qt^{c}$ and taking Levi-sequences in $\qt\,{\cap}\,\sigmaup(\qt^{c})$. 
\end{rmk} 
\begin{prop}
 Let $(\gs,\qt)$ be a reductive $CR$-algebra, with invariant complement $\qt^{c}$. Then 
\begin{equation}
 \Nt(\qt,\qt^{c})=\{Z\in\qt\mid [Z,\qt^{c}]\subseteq\qt^{c}\}
\end{equation}
is a Lie subalgebra of $\qt$
containing $\qt\,{\cap}\,\sigmaup(\qt)$.
\par 
If $Z\,{\in}\,\Nt(\qt,\qt^{c}){\setminus}\sigmaup(\qt)$, then  the Levi order of $\sigmaup(Z)$ is either
$1$ or $\infty$.
\end{prop} 
\begin{proof}
If $Z_{1},Z_{2}\,{\in}\,\Nt(\qt,\qt^{c})$, then 
\begin{equation*}\,
 [[Z_{1},Z_{2}],\qt^{c}]\subseteq [Z_{1},[Z_{2},\qt^{c}]]+[Z_{2},[Z_{1},\qt^{c}]]\subseteq
 [Z_{1},\qt^{c}]+[Z_{2},\qt^{c}]\subseteq\qt^{c}.
\end{equation*}
This shows that $\Nt(\qt,\qt^{c})$ is a Lie subalbebra of $\qt$. \par
If $Z\,{\in}\,\Nt(\qt,\qt^{c}){\setminus}\sigmaup(\qt)$ and  $\sigmaup(Z)$ has finite Levi order, then
we can find some $Z_{1}\,{\in}\,\qt\,{\cap}\,\sigmaup(\qt^{c})$ such that $[\sigmaup(Z),Z_{1}]\,{\notin}\,\qt$.
Since 
\begin{equation*}
[\sigmaup(Z),Z_{1}]\,{=}\,\sigmaup([Z,\sigmaup(Z_{1})])\,{\subseteq}\,\sigmaup([Z,\qt^{c}])\,
\subseteq\sigmaup(\qt^{c}),\end{equation*} 
we obtain that $[\sigmaup(Z),Z_{1}]\,{\in}\,\sigmaup(\qt^{c}){\setminus}\qt$ and hence, by \eqref{e3.23}, 
$[\sigmaup(Z),Z_{1}]\,{\notin}\,\qt\,{+}\,\sigmaup(\qt)$, showing that $\sigmaup(Z)$ has Levi order one.
\end{proof} 
\begin{exam}(see \cite{MMN21}) \label{ex3.37}
We construct examples of reductive $CR$-algebras by taking  as $\gt$ 
Abelian extensions of $\slt_{2}(\C)$.
The vector fields
$Z\,{=}\,\bar{z}\frac{\partial}{\partial{z}}$, 
$\bar{Z}\,{=}\,z\frac{\partial}{\partial\bar{z}}$,
$H\,{=}\,[Z,\bar{Z}]\,{=}\,\bar{z}\frac{\partial}{\partial{\bar{z}}}\,{-}\,z\frac{\partial}{\partial{z}}$,
generate a Lie algebra $\st$ 
isomorphic to $\slt_{2}(\C)$. The space $\sfV$ of homogeneous polynomials
of degree $2n$ is an $\st$-module for the action of the elements of $\st$ as partial differential operators
and the monomials $\sfv_{h}\,{=}\,z^{n+h}\bar{z}^{n-h}$, for $h\,{\in}\,\Z$ and $|h|\,{\leq}\,n$,
form a basis of $\sfV$ with  
\begin{equation*}
 H(\sfv_{h})=2h\,\sfv_{h},\;\; Z(\sfv_{h})=(n+h)\,\sfv_{h-1},\;\; \bar{Z}(\sfv_{h})=(n-h)\,\sfv_{h+1},\;\; -n\,{\leq}\,h\,{\leq}\,n.
\end{equation*}
We take $\gt\,{=}\,\st\,{\oplus}\,\sfV$, with 
\begin{equation*}
 [(W,\sfv),(W',\sfv')]=[W,W']+W(\sfv')-W'(\sfv),\;\;\forall W,W'\in\st,\;\forall \sfv,\sfv'\in\sfV.
\end{equation*}
We fix the conjugation $\sigmaup(W,\sfv)\,{=}\,(\bar{W},\bar{\sfv})$ and the complex Lie subalgebra
\begin{equation*}\qt\,{=}\,\langle\bar{Z},H,\sfv_{1},\sfv_{2},\hdots,\sfv_{n}\rangle\end{equation*}
of $\gt$, defining the 
$CR$-algebra $(\gs,\qt)$. Since $\sigmaup(\bar{Z})\,{=}\,Z$, $\sigmaup(H)\,{=}\,{-}H$,
$\sigmaup(\sfv_{h})\,{=}\,\sfv_{{-}h}$, we obtain that $\qt\,{\cap}\,\sigmaup(\qt)\,{=}\,\langle{H}\rangle$.
The subspace $\qt^{c}\,{=}\,\langle{Z},\sfv_{0},\sfv_{-1},\hdots,\sfv_{-n}\rangle$ is an invariant
complement. We have \begin{equation*}
\qt\,{\cap}\,\sigmaup(\qt^{c})\,{=}\,\langle\bar{Z},\sfv_{1},\hdots,\sfv_{n}\rangle,\;\,
\qt^{c}\,{\cap}\,\sigmaup(\qt)\,{=}\,\langle{Z},\sfv_{-1},\hdots,\sfv_{-n}\rangle,\,\;  \qt^{c}\,{\cap}\,\sigmaup(\qt^{c})\,{=}\,
\langle\sfv_{0}\rangle.\end{equation*}
The element $Z$ has Levi-order $1$, since $[Z,\sfv_{1}]\,{=}\,(n{+}1)\sfv_{0}$.  
For each $1{\leq}h{\leq}n$, the element $\sfv_{-h}$ has Levi order $h$, since 
\begin{equation*}
 [\sfv_{-h},\underset{h\;\text{times}}{\underbrace{\bar{Z},\hdots,\bar{Z}}}]=(-1)^{h}\,\frac{(n+h)!}{n!}\,\sfv_{0}.
\end{equation*}
Thus $(\gs,\qt)$ is a $CR$ algebra of $CR$-dimension $(n{+}1)$, $CR$-codimension $1$, fundamental
and Levi-nondegenerate of finite Levi-order $n$ (see also \cite{Marini_Medori_2025} for local equation associated to this $CR$ algebra).
\end{exam}

\section{Parabolic $CR$-algebras}\label{s4}
 Let $\gt$ be a complex Lie algebra. Its maximal solvable subalgebras are 
 all conjugate and are called its
 \emph{Borel subalgebras}.
 Lie subalgebras containing a Borel subalgebra of $\gt$ are called \emph{parabolic}.
 We will denote by $\Pt(\gt)$ \index{$\Pt(\gt)$} (resp. $\Bt(\gt)$)
 \index{$\Pt(\gt,\hg)$}
 the set of parabolic (resp. Borel)
 Lie subalgebras of $\gt$ and by $\Pt(\gt,\hg)$ (resp. $\Bt(\gt,\hg)$) \label{parbor}
 those containing a given 
 subalgebra $\hg$ of $\gt$.\par
If $\rt$ is the solvable radical and \, $\st\,{\simeq}\,\gt{/}\rt$\,  a semisimple Levi factor of $\gt$,
then Borel and parabolic subalgebras of $\gt$ are pull-backs of Borel and parabolic
subalgebras of $\st$ by the natural projection $\piup:\gt\,{\to}\,\st$. 
\par 
\begin{dfn}{We call \emph{parabolic} a $CR$ algebra $(\gs,\qt)$
whose lifted $CR$-structure $\qt$ is parabolic in $\gt$.}
\end{dfn}

We show in this section that,
for the special class
of \emph{parabolic $CR$-al\-ge\-bras}, 
by using convenient root systems, 
notions and results of \S\ref{s2} and \S\ref{s3}
may  be  translated into
properties  of
suitable systems of simple roots.  These can be encoded in additional marks on
the corresponding Dynkin diagrams.\par\smallskip
We denote by
$\Rad(\gt,\hg)$ the root system associated to a semisimple complex Lie algebra $\gt$ and
its Cartan subalgebra~$\hg$. To simplify notation, we will usually write $\Rad$ instead of $\Rad(\gt,\hg)$. 
Let $\hr$ be the real form of $\hg$ consisting 
of the elements on which
the roots of $\Rad$ are real valued, so that $\Rad\,{\subset}\,\hr^{*}$.  The dual of the restriction to
$\hr$ of the Killing form of $\gt$ endows 
$\hr^{*}$ by a natural Euclidean structure. We denote by 
\begin{equation*}
 (\xiup\,|\,\etaup),\;\; \|\,\xiup\,\|=\sqrt{(\xiup\,|\,\xiup)} ,\;\;\forall \xiup,\etaup\,{\in}\,\hr^{*}
\end{equation*}
its scalar product and norm and by 
\begin{equation*}
 \rtt{\etaup}(\xiup)=\xiup-\langle\xiup\,|\,\etaup\rangle,\;\;\text{with}\;\; \langle\xiup\,|\,\etaup\rangle=
 \dfrac{2(\xiup\,|\,\etaup)}{\|\,\etaup\,\|^{2}},\;\;\forall \xiup,\etaup\,{\in}\,\hr^{*},
\end{equation*}
the reflection in the hyperplane orthogonal to the vector $\etaup$. The pairing $\langle\,\cdot\,|\,\cdot\,\rangle$
is linear with respect to the first term, while, for $\etaup_{1},\etaup_{2}\,{\in}\,\hr^{*}$
and $\etaup_{1}\,{+}\,\etaup_{2}\,{\neq}\,0$, we get 
\begin{equation*}
 \langle\xiup\,|\,\etaup_{1}\,{+}\,\etaup_{2}\rangle=
 \frac{\|\etaup_{1}\|^{2}}{\|\etaup_{1}{+}\etaup_{2}\|^{2}}\langle\xiup\,|\,\etaup_{1}\rangle
 +\frac{\|\etaup_{2}\|^{2}}{\|\etaup_{1}{+}\etaup_{2}\|^{2}}\langle\xiup\,|\,\etaup_{2}\rangle
, \;\;\;\forall\xiup\in\hr^{*}.
\end{equation*}

\subsection{Closed, parabolic and horocyclic sets} \label{s1.2}
In this and the following two subsections we summarize preliminary 
results and notation for  regular
and parabolic sets of roots and  corresponding complex subalgebras. 
For a more detailed discussion we refer to  \cite{MMN23}.
\begin{dfn} \!\![cf. \cite[Ch.VI,\S{1.7}]{Bou68}]\;\;
A subset $\Qq$ in $\Rad$ is said to be:
\begin{itemize}
 \item \emph{closed}, \; if \; $\alphaup,\betaup\,{\in}\,\Qq\,,$ $\alphaup{+}\betaup\,{\in}\,\Rad$ $\Longrightarrow$
 $\alphaup{+}\betaup\,{\in}\,\Qq\,$;
 \item \emph{parabolic}, \; if $\Qq$ is closed and $\Qq\,{\cup}\,({-}\Qq)\,{=}\,\Rad$;
 \item \emph{horocyclic}, \; if $\Rad{\setminus}\Qq$ is parabolic;
 \item  \emph{generated by} 
 $\Eq$, \; if $\Eq\,{\subseteq}\,\Rad$ and $\Qq\,{=}\,\Z_{+}[\Eq]\,{\cap}\,\Rad$. 
\end{itemize}
\end{dfn}

\begin{ntz}
Let $\Qq$ be any subset of $\Rad$. We set 
\begin{equation}\label{e1.12}\begin{cases}
 \Qq^{r}{=}\{\alphaup\,{\in}\,\Qq\,{\mid}\,{-}\alphaup\,{\in}\,\Qq\},\\
 \Qq^{n}{=}\{\alphaup\,{\in}\,\Qq\,{\mid}\,{-}\alphaup\,{\notin}\,\Qq\},\\
  \Qc{=}\Rad{\setminus}\Qq\,,\\
  \Qq^{\vee}{=}\Qq^{r}\,{\cup}\,\Qc.
  \end{cases}
\end{equation}
\end{ntz}
\par
If $\Qq$ is parabolic, then $\Qq^{n}$ and $\Qc$ are closed and horocyclic,
$\Qc\,{=}\,{-}\Qq^{n}$ and $\Qq^{\vee}$ is 
\emph{parabolic}. Denote by $\Pcr(\Rad)$ the set of parabolic subsets of $\Rad$. 
\begin{ntz} The \textit{Weyl chambers} of $\Rad$ are the connected components of the open set of
\textit{regular} elements of $\hr$. Denote by $\Cd(\Rad)$ the set of Weyl chambers
of $\Rad$. Having fixed $C$ in $\Cd(\Rad)$, we get a decomposition of $\Rad$
into the disjoint union $\Rad^{+}(C)
\,{\cup}\,\Rad^{-}(C)$ of  roots which are \textit{positive} and \textit{negative}
for the lexicographic order defined by $C$:
 $\Rad^{+}(C)$ (resp. $\Rad^{-}(C)$)
 consists of the roots which are positive (resp. negative) on one
(and hence on all) elements of $C$.  
Roots of $\Rad^{+}(C)$ which are not sums of two $C$-positive roots are called \emph{simple}.
The simple positive roots in $\Rad^{+}(C)$ are a basis of $\hr^{*}$, that we denote by $\Bz(C)$.
Any root $\alphaup$ uniquely decomposes into a sum 
\begin{equation*}
 \alphaup={\sum}_{\betaup\in\Bz(C)}k_{\alphaup,\betaup}\betaup
\end{equation*}
in which the coefficients $k_{\alphaup,\betaup}$ are integers, all nonnegative if $\alphaup\,{\in}\,\Rad^{+}(C)$,  
all nonpositive if $\alphaup\,{\in}\,\Rad^{-}(C)$. We define 
\begin{equation*}
 \supp_{C}(\alphaup)=\big\{\betaup\,{\in}\,\Bz(C)\,{\mid}\, k_{\alphaup,\betaup}\,{\neq}\,0\big\}.
\end{equation*}
\end{ntz}
The sets 
$\Rad^{+}(C)$ and $\Rad^{-}(C)$ are minimal parabolic 
and maximal horocyclic with respect to inclusion.  
Any closed subset containing a parabolic one is parabolic. Proper parabolic sets are
intersections of $\Rad$ with \textit{closed} half-spaces of $\hr^{*}$, while
horocyclic sets are  intersections 
of $\Rad$ with  \textit{open} half-spaces in $\hr^{*}$ (cf. Prop.\ref{p2.5} below).
\subsection{$\hg$-regular subalgebras} \label{s1.3}
We consider in this subsection subalgebras of $\gt$
having  a nice description in terms of  roots (cf. \cite[Ch.6]{GOV94}). 
\begin{dfn} Let $\hg$ be a Cartan subalgebra of $\gt$. A Lie subalgebra $\qt$ of 
$\gt$ is called 
 \emph{$\hg$-regular} if 
\begin{equation}
 \label{reg}
 [\hg,\qt]\subseteq\qt.
\end{equation}\par
\emph{Regular} subalgebras $\qt$ of $\gt$ are those which are $\hg$-regular for some
Cartan subalgebra $\hg$ of $\gt$.
\end{dfn}
\begin{prop}
 Let $\Qq$ be a closed system of roots and $\td$ a linear subspace of $\hg$ containing 
$\{H_{\alphaup}\,{\mid}\,\alphaup\,{\in}\,\Qq^{r}\}.$ Then 
\begin{equation}\label{e1.19}
 \qt\,{=}\,\qt^{r}\,{\oplus}\,\qt^{n},\;\;\text{with}\;\;\; \qt^{r}\,{=}\,\td\,{\oplus}\,{\sum}_{\alphaup\,{\in}\,\Qq^{r}}\gt^{\alphaup},\;
 \;\qt^{n}\,{=}\,{\sum}_{\alphaup\in\Qq^{n}}\gt^{\alphaup}
\end{equation}
is a Levi-Chevalley decomposition of an $\hg$-regular subalgebra $\qt,$ with nilradical
$\qt^{n}$ and uniquely determined 
$\hg$-regular reductive Levi factor~$\qt^{r}$.\par
Vice versa, every $\hg$-regular subalgebra $\qt$ of $\gt$ is described by \eqref{e1.19}
with 
\begin{equation}
 \td\,{=}\,\hg\,{\cap}\,\qt,\quad\Qq\,{=}\,\{\alphaup\,{\in}\,\Rad\,{\mid}\,\gt^{\alphaup}\,{\subseteq}\,\qt\}.
\end{equation}
\par
 Every Lie subalgebra  containing a Cartan subalgebra of $\gt$ is regular.
\qed
\end{prop}
\begin{rmk} Let $\qt\,{\in}\,\Pt(\gt)$.
Note that $\qt^{n}$, being its nilradical, is independent of the choice of $\hg$.
If $\hg$ has been fixed and 
 $\qt$ is $\hg$-regular, then the subalgebras $\gt^{r}$, $\gt^{c}$ and 
 $\qt^{\vee}{\coloneqq}\,\qt^{r}\,{\oplus}\,\qt^{c}$ are  determined by the data of both
 $\qt$ and $\hg$. \par
The set $\Bt(\qt)$ of Borel sublagebras of $\qt$ coincides with the set
 of Borel subalgebras of $\gt$ which are contained in $\qt$.
\end{rmk}
\begin{prop}
 Parabolic subalgebras of $\gt$ are regular. \par
 The $\hg$-regular Borel subalgebras of $\gt$ are those of the form 
\begin{equation*}
 \bt_{C}=
 \hg\oplus{\sum}_{\alphaup\,{\in}\,\Rad^{+}(C)}\gt^{\alphaup}\quad\text{for some $C\,{\in}\,\Cd(\Rad)$}.
\end{equation*} \par
The $\hg$-regular parabolic subalgebras of $\gt$ are those of the form 
\begin{equation*}
 \qt_{\Qq}\,{=}\,\hg\,{\oplus}\,{\sum}_{\alphaup\in\Qq}\gt^{\alphaup}
\end{equation*}
for a parabolic subset $\Qq$ in $\Rad$ and vice versa, if {$\hg$ is a Cartan
subalgebra of $\gt$ contained in $\qt\,{\in}\,\Pt(\gt)$, }
then 
$\Qq\,{=}\,\{\alphaup\,{\in}\,\Rad\,{\mid}\,\gt^{\alphaup}\,{\subset}\,\gt\}$ is a parabolic
subset of~$\Rad$. \qed
\end{prop}
\subsection{Description of parabolic subsets and subalgebras} \label{s4.3}
\begin{prop}
 The following are equivalent: 
\begin{enumerate}
 \item $\Qq$ is parabolic;
 \item $\Qq$ is closed and contains $\Rad^{+}(C)$ for some 
$C\,{\in}\,\Cd(\Rad)$.\qed
\end{enumerate}
\end{prop} 
\begin{dfn}\label{d4.4}
 We call  \emph{admissible for $\Qq$}  a $C\,{\in}\,\Cd(\Rad)$ with
 $\Rad^{+}(C)\,{\subseteq}\,\Qq$. Let
\begin{equation}\label{e4.5}
 \Cd(\Rad,\Qq)=\{C\in\Cd(\Rad)\mid \Rad^{+}(C)\,{\subseteq}\,\Qq\}
\end{equation}
be the collection  of admissible Weyl chambers for $\Qq$. For $C\,{\in}\,\Cd(\Rad)$, 
{we}
set
\begin{equation}\label{e4.6}
 \Phi_{C}=\Qq^{n}\cap\Bz(C),\quad \Phi^{\vee}_{C}=\Bz(C){\setminus}\Phi_{C}.
\end{equation}
\end{dfn} 
\begin{lem} \label{l2.5}
Let $\Qq\,{\in}\Pcr(\Rad)$ and $C\,{\in}\,\Cd(\Rad)$.  
The following are equivalent 
\begin{equation*}\vspace{-18pt}
\begin{aligned}
( 1)  &\qquad C\in\Cd(\Rad,\Qq),\\
(2)  &\qquad \Qq=\Qq_{\;\Phi_{C}}\,{\coloneqq}\,\Rad^{+}(C)\cup
 \{\alphaup\,{\in}\,\Rad^{-}(C)\,{\mid}\,\supp_{C}(\alphaup)\,{\cap}\,
 \Phi_{C}\,{=}\,\emptyset\},\\ 
(3)  &\qquad \Qq=\Rad^{+}(C)\cup\{\alphaup\,{\in}\,\Rad^{-}(C)\mid \supp_{C}(\alphaup)\,{\subseteq}\,
 \Phi_{C}^{\vee}\},\\
 (4) &\qquad \Qq^{n}\subseteq\Rad^{+}(C).
\end{aligned}\end{equation*}\qed
\end{lem}
\begin{rmk}\label{r1.7}
Let $C\,{\in}\,\Cd(\Rad)$ and $\Phi_{C}$ be any subset of $\Bz(C)$. The set $\Qq_{\;\Phi_{C}}$
defined by $(2)$ of Lemma\,\ref{l2.5} is parabolic. 
We will use the notation 
\begin{equation} \label{e2.5}\left\{
\begin{array}{l  l}
\Qq_{\;\Phi_{C}}^{r}\,{=}\,\{\alphaup\,{\in}\,\Rad\,{\mid}\,\supp_{C}(\alphaup)\,{\cap}\,\Phi_{C}\,{=}\,\emptyset\},
&\quad  \qt_{\Phi_{C}}^{r}\,{=}\,\hg\,{\oplus}\,{\sum}_{\alphaup\in\Qq^{r}_{\;\Phi_{C}}}\gt^{\alphaup},\\
\Qq_{\;\Phi_{C}}^{n}\,{=}\,\{\alphaup\,{\in}\,\Rad^{+}(C)\,{\mid}\,\supp_{C}(\alphaup)\,{\cap}\,\Phi_{C}\,{\neq}\,\emptyset\}, &\quad  \qt_{\Phi_{C}}^{n}\,{=}\,{\sum}_{\alphaup\in\Qq^{n}_{\;\Phi_{C}}}\gt^{\alphaup},\\
\qt_{\Phi_{C}}\,{=}\,\qt^{r}_{\Phi_{C}}\,{\oplus}\,\qt^{n}_{\Phi_{C}}, & \quad
\qt_{\Phi_{C}}\,{=}\,\qt^{r}_{\Phi_{C}}\,{\oplus}\,\qt^{n}_{\Phi_{C}}, 
\end{array}\right.
\end{equation}
and call $\qt_{\Phi_{C}}$ its 
corresponding \textit{parabolic Lie subalgebra}. Note that 
\begin{equation}
 \qt_{\Phi_{C}}=\qt_{\Phi_{C}}^{r}\oplus\qt_{\Phi_{C}}^{n},
 \;\;\text{with}\;\; \qt^{r}_{\Phi_{C}}=\hg\oplus{\sum}_{\alphaup\in\Qq^{r}_{\;\Phi_{C}}}
 \gt^{\alphaup},\;\; \qt^{n}_{\Phi_{C}}={\sum}_{\alphaup\in\Qq^{n}_{\;\Phi_{C}}}
 \gt^{\alphaup}
\end{equation}
is its Levi-Chevalley decomposition, with \emph{nilradical} $\qt^{n}_{\Phi_{C}}$, 
{and}
$\hg$-regular 
\emph{reductive Levi factor}
{$\qt_{\Phi_{C}}^{r}$}. 
{We get $\Qq^{r}_{\Phi_{C}}\,{=}\,\Z[\Phi^{\vee}_{C}]\,{\cap}\,\Rad$.}
\par
Every parabolic
subalgebra $\qt$ of $\gt$ containing $\hg$ can be
described in this way, by choosing a $C\,{\in}\,\Cd(\Rad)$
with $\bt_{C}\,{\subseteq}\,\qt$
and   $\Phi_{C}\,{=}\,\{\alphaup\,{\in}\,\Bz(C)\,{\mid}\,\gt^{\alphaup}\,{\subseteq}\,\qt^{n}\}$.
\end{rmk}
\begin{lem} \label{l4.7}
Let $C_{1},C_{2}$ be admissible Weyl chambers for $\Qq\,{\in}\Pcr(\Rad)$.
 Then we can find a sequence $(\betaup_{1},\hdots,\betaup_{p})$ in $\Qq^{r}$ such that 
\begin{equation*}
 \rtt{\betaup_{1}}\circ\cdots\circ\rtt{\betaup_{p}}(\Rad^{+}(C_{1}))=\Rad^{+}(C_{2}).
\end{equation*}
where  $\rtt{\beta_i}$ {is}
the root reflection with respect to $\beta_i$.
\end{lem} 

\begin{proof}
 We argue by recurrence on the number $m$ 
 of roots in $\Rad^{+}(C_{2}){\setminus}\Rad^{+}(C_{1})$.
 If $m$ is zero, then $C_{1}\,{=}\,C_{2}$ and there is nothing to prove.
Assume now that 
$m\,{>}\,0$ and that the statement holds true when $\#(\Rad^{+}(C_{2}){\setminus}\Rad^{+}(C_{1}))
 {<}m$.
 Since $C_{1}\,{=}\,C_{2}$ when $\Bz(C_{1})\,{\subset}\,\Rad^{+}(C_{2})$, the basis 
$\Bz(C_{1})$ is not contained in $\Rad^{+}(C_{2})$. A root 
 $\betaup\,{\in}\,\Bz(C_{1})$ not belonging to $\Rad^{+}(C_{2})$ belongs to $\Qq^{r}$
and
 \begin{equation*}
 \rtt{\betaup}(\Rad^{+}(C_{1}))\,{=}\,(\Rad^{+}(C_{1}){\setminus}\{\betaup\})\,{\cup}\,\{{-}\betaup\}
 \end{equation*}
 is the set $\Rad^{+}(C')$, for an admissible 
 Weyl chamber $C'$ for $\Qq$, for which 
 \begin{equation*}\#(\Rad^{+}(C_{2}){\setminus}\Rad^{+}(C'))
 \,{=}\,m{-}1.\end{equation*}
 By the recursive assumption there are $\betaup_{1},\hdots,\betaup_{\ell}$ in $\Qq^{r}$ such that 
\begin{equation*}\rtt{\betaup_{1}}\circ\cdots\circ\rtt{\betaup_{\ell}}\circ\rtt{\betaup}(\Rad^{+}(C_{1}))=
 \rtt{\betaup_{1}}\circ\cdots\circ\rtt{\betaup_{\ell}}(\Rad^{+}(C'))=\Rad^{+}(C_{2}).
\end{equation*}
The proof is complete.
 \end{proof}
 \begin{ntz}
 {We denote by $\Af(\Rad)$ the automorphism group of $\Rad$, consisting of
 the linear isometries of $\hr^{*}$ leaving $\Rad$ invariant and by
 $\Wf(\Rad)$ the Weyl group of $\Rad$, which is the normal
 subgroup of $\Af(\Rad)$ generated by the reflections
 $\rtt{\alphaup}$ with respect to the roots $\alphaup$ of $\Rad$.}
 \end{ntz}
\begin{cor}\label{c2.4} The subgroup 
 {$\Wf(\Rad,\Qq)$ of $\Wf(\Rad)$} generated by the 
 {reflections} 
 $\rtt{\betaup}$
 with $\betaup\,{\in}\,\Qq^{r}$, acts simply transitively on the set
 $\Cd(\Rad,\Qq)$ of admissible Weyl chambers for~$\Qq$.\qed
\end{cor}
\par\smallskip

A {$\Z$-gradation}
of $\Rad$ (or, equivalently, of $\Z[\Rad]$) is any group homomorphism 
in $\Hom(\Z[\Rad],\Z)$.
\begin{prop}\label{p2.5}
For each parabolic subset $\Qq$ in $\Rad$ there is a unique
 $\Z$-gradation $\chiup_{\Qq}$
 of $\Rad$ such that 
\begin{equation} \label{e4.5} \begin{cases}
  \Qq^{r}{=}\{\alphaup\,{\in}\,\Rad\,{\mid} \chiup_{\Qq}(\alphaup){=}0\},\\
  \Qq^{n}{=}\{\alphaup\,{\in}\,\Rad\,{\mid}
 \chiup_{\Qq}(\alphaup){>}0\},\\
 \Qc{=}\{\alphaup\,{\in}\,\Rad\,{\mid}\chiup_{\Qq}(\alphaup){<}0\},\\
\big\{\alphaup\,{\in}\,\Rad\,{\mid}\,\chiup_{\Qq}(\alphaup)\,{=}\,{1}\big\}\;\;\text{generates}\;\;\Qq^{n}.
\end{cases}
\end{equation}
\end{prop} 
\begin{proof} Fix a Weyl chamber $C$  in $\Cd(\Rad,\Qq)$  and let
$\Phi_{C}\,{=}\,\Bz(C)\,{\cap}\,\Qq^{n}$. Then 
\begin{equation*} \tag{$*$}
\begin{cases}
 \Qq^{r}\,{=}\,\{\alphaup\,{\in}\,\Rad\,{\mid}\,\supp_{C}(\alphaup)\,{\cap}\,\Phi_{C}\,{=}\,\emptyset\},\\
 \Qq^{n}\,{=}\,\{\alphaup\,{\in}\,\Rad^{+}(C)\,{\mid}\,\supp_{C}(\alphaup)\,{\cap}\,\Phi_{C}\,{\neq}\,\emptyset\}.
\end{cases}
\end{equation*}
Let us define $\chiup_{\Qq}\,{\in}\,\Hom(\Z[\Rad],\Z)$ on the basis $\Bz(C)$ by setting 
\begin{equation*}\chiup_{\Qq}(\alphaup)= 
\begin{cases}
 1, & \forall \alphaup\,{\in}\,\Phi_{C},\\
 0, & \forall \alphaup\,{\in}\,\Bz(C)\,{\cap}\,\Qq^{r}.
\end{cases}
\end{equation*}
By $(*)$ we obtain \eqref{e4.5}. Uniqueness follows from Lemma\,\ref{l4.7}.\end{proof}
\begin{cor}\label{c2.10}
 Let $\Qq\,{\in}\Pcr(\Rad)$. The degree $\chiup_{\Qq}(\deltaup_{C})$ of the largest root 
 $\deltaup_{C}$ in $\Rad^{+}(C)$ is the same for all $C\,{\in}\,\Cd(\Rad,\Qq)$.
\end{cor} 
\begin{proof}
 If $C_{0}, C_{1}\,{\in}\,\Cd(\Rad,\Qq)$, by Lemma\,\ref{l4.7}
 we can find  $\betaup_{1},\hdots,\betaup_{p}\,{\in}\,\Qq^{r}$ such that 
 $\Rad^{+}(C_{1})\,{=}\,\rtt{\betaup_{1}}{\circ}\cdots{\circ}\rtt{\betaup_{p}}(\Rad^{+}(C_{0}))$.
 Then 
\begin{equation*}
 \deltaup_{C_{1}}=\rtt{\betaup_{1}}{\circ}\cdots{\circ}\rtt{\betaup_{p}}(\deltaup_{C_{0}})=
 \deltaup_{C_{0}}+c_{1}\betaup_{1}+\cdots+c_{p}\betaup_{p},
\end{equation*}
with $c_{i}\,{\in}\,\Z$ for $1{\leq}i{\leq}p$,
 implying that $\chiup_{\Qq}(\deltaup_{C_{1}})\,{=}\,\chiup_{\Qq}(\deltaup_{C_{0}})$, because
 $\chiup_{\Qq}(\betaup_{i})\,{=}\,0$ for $1{\leq}i{\leq}p$. 
\end{proof}
\begin{lem}
 Let $\Qq$ be any parabolic subset of $\Rad$. Set 
\begin{equation}
 \rhoup_{\Qq}=\tfrac{1}{2}{\sum}_{\alphaup\in\Qq}\alphaup=\tfrac{1}{2}{\sum}_{\alphaup\in\Qq^{n}}\alphaup.
\end{equation}
If $C$ is any admissible Weyl chamber for $\Qq$, then 
\begin{equation*}
 \chiup_{\Qq}(\alphaup)=\langle\rhoup_{\Qq}\,|\,\alphaup\rangle,\;\;\forall \alphaup\,{\in}\,\Bz(C).
\end{equation*}
\end{lem} 
\begin{proof}
 If $\betaup\,{\in}\,\Phi_{C}^{\vee}$, then $\rtt{\betaup}(\Qq^{n})\,{=}\,\Qq^{n}$. Thus 
 $\rtt{\betaup}(\rhoup_{\Qq})\,{=}\,\rhoup_{\Qq}$ implies that $\langle{\rhoup_{\Qq}}|\,\betaup\rangle\,{=}\,0$. 
 If $\betaup\,{\in}\,\Phi_{C}$, then $\rtt{\betaup}(\Qq^{n})\,{=}\,(\Qq^{n}{\setminus}\,\{\betaup\}\,{\cup}\,\{{-}\betaup\}$
 implies that 
$\rtt{\betaup}(\rhoup_{\Qq})\,{=}\,\rhoup_{\Qq}{-}\,\betaup$, yielding 
 $\langle{\rhoup_{\Qq}}|\,\betaup\rangle\,{=}\,1$.
\end{proof}

\subsection{Adapted Cartan subalgebras}\label{s4.4}
Here and 
in the following we will consistently employ  the notation of~\cite{MMN23}.
\begin{dfn}   \label{d4.5}
Cartan subalgebras of $\gs$ contained in  
{$\qt\,{\cap}\,\gs$} are said to be \emph{adapted to $(\gs,\qt)$}. A \emph{parabolic $CR$-triple}
$(\gs,\qt,\ho)$ is the datum of a parabolic $CR$ algebra $(\gs,\qt)$ and an adapred
Cartan subalgebra.
\end{dfn}
Parabolic $CR$-algebras $(\gs,\qt)$ admit adapted Cartan subalgebras
(see e.g. \cite[Proposition 3.2]{AMN06b}, \cite[Thm.2.6]{Wolf69}). \par
The complexification $\hg$ of an adapted Cartan subalgebra defines a root system
$\Rad\,{=}\,\Rad(\gt,\hg)$. The elements of $\hg$ on which the roots of  
$\Rad$
take real values
yield a $\sigmaup$-invariant
real form $\hr$ of $\hg$ and $\Rad$ is a subset of the dual $\hr^{*}$ of $\hr$.
The restriction of $\sigmaup$ to $\hr$ defines by duality
an involution $\stt$ on $\hr^{*}$. The Cartan subalgebra 
$\hg\,{\cap}\,\gs$ decomposes into the direct sum $\hs^{+}\,{\oplus}\,i\hs^{-}$,
where $\hs^{\pm}{=}\{H\,{\in}\,\hr\,{\mid}\,\sigmaup(H){=}{\pm}H\}$. In fact
this set only depends on the \textit{restriction} of $\sigmaup$ to $\hr$ and thus
on the involution $\stt$ of $\Rad$. This is our reason to
indicate in the following
by $\hst$ the Cartan subalgebra $\hg\,{\cap}\,\gs$ of $\gs$. In fact 
different choices of adapted Cartan subalgebras yield isomorphic root systems $\Rad$
and equivalent involutions $\stt$ of $\Rad$ may correspond to inequivalent 
anti-$\C$-linear involutions $\sigmaup$ of $\gt$, while inequivalent involutions $\stt$ of $\Rad$
may correspond to the same anti-$\C$-linear involution $\sigmaup$ of $\gt$.
In fact the involutions $\stt$ can be used to classify the conjugacy classes of
adapted Cartan subalgebras to $(\gs,\qt)$.
\par

Vice versa, 
 fix 
a Cartan subalgebra $\hg$ and a 
Chevalley system $(X_{\alphaup},H_{\alphaup})_{\alphaup\in\Rad}$ of the complex 
semisimple Lie algebra $\gt$. 
This determines
split and compact real forms $\gr$ and $\gu$ of $\gt$, together with  
their Cartan subalgebras $\hst$,
$\hr$ and $i\hr$, which are real forms of $\hg$.
We restrain to anti-$\C$-linear involutions $\sigmaup$ of $\gt$ leaving $\gr,\,\gu,\,\hr$ invariant:
they induce involutions $\stt$ of the root system $\Rad$ 
and vice versa all involutions $\stt$ of $\Rad$ lift to elements of the set $\Inv^{\tauup}(\gr,\hr)$
of anti-$\C$-linear
involutions of $\gt$ 
leaving $\gr,\,\gu,\,\hr$ invariant (cf. \cite[\S{4}]{MMN23}), whose collection we indicate by
$\Invs$.
In doing so, we don't loose generality,
because {\textsl{every anti-$\C$-linear involution of $\gt$ is conjugated in $\Aut({\gt})$
with an element
$\sigmaup$ of $\Invs$}}
 (\cite[Theorem 4.2]{MMN23}) and by
\cite[Thm.2.29]{MMN23}
\textsl{we can, modulo isomorphisms, 
restrain our consideration to $CR$-algebras $(\gs,\qt)$, with
an $\hg$-regular parabolic $\qt$ and $\sigmaup\,{\in}\,\Invs$.}
 The canonical Cartan subalgebra $\hg\,{\cap}\,\gs$ 
 is the real form $\hst$ described above.\par
  We point out that 
 we need to keep $\sigmaup$ and $\stt$ distinct, as
 non equivalent anti-$C$-linear involutions 
 $\sigmaup$ of $\gt$ may correspond to a same involution $\stt$ and non equivalent involutions
 $\stt$ of $\Rad$ to
a same anti-$\C$-linear involution $\sigmaup$.
\par\smallskip
We get a partition $\Rad\,{=}\,\Rad_{\,\;\circ}^{\stt}\,{\cup}\,\Rad_{\,\;\bullet}^{\stt}
\,{\cup}\,\Rad_{\,\;\star}^{\stt}$ 
of roots into \textit{real, imaginary and complex}, depending
on being kept fixed, changed into their opposite by $\stt$, or neither of the two, respectively.  
When $\sigmaup\,{\in}\,\Invs$ is fixed, imaginary roots further decompose
into compact {($\Rad^{\sigmaup}_{\,\;\bullet}$)}
and non compact {($\Rad^{\sigmaup}_{\,\;\oast}$)}. 
\par\smallskip
Let $(\gs,\qt,\hst)$ be a parabolic $CR$-triple. The datum of the adapted Cartan subalgebra $\hst$ 
yields,  in the root system $\Rad$ of the complexification $\hg$ of $\hst$,
the parabolic subset  
\begin{equation}\label{e4.10}
 \Qq=\{\alphaup\,{\in}\,\Rad\mid \gt^{\alphaup}\subset\qt\}.
\end{equation} 
\begin{dfn}
 We say that the parabolic set in \eqref{e4.10} is associated to the parabolic $CR$-triple
 $(\gs,\qt,\hst)$ and call $(\Qq\,,\hst)$ its \emph{associated pair}.
\end{dfn}

\begin{rmk}[Cross-marked diagrams]
 Let $(\gs,\qt,\hst)$ be a parabolic $CR$-triple, $(\Qq\,,\stt)$ its associated pair. 
 Fix  an admissible
 Weyl chamber $C$ for $\Qq$. To keep track of the involution $\stt$ of $\Rad$, we 
paint the nodes of the Dynkin diagram according to the nature of the corresponding
root of $\Bz(C)$: we use ``$\medcirc$'' for the real, ``$\medbullet$'' for the imaginary,
 ``$\oplus$'' and ``$\ominus$'' for the complex with
$C$-positive and $C$-negative $\stt$-conjugate, respectively. If we want to emphasise 
the actual lifting $\sigmaup$ of $\stt$, we substitute ``$\oast$'' for ``$\medbullet$'' 
when the node corresponds to a non compact imaginary root.
\par
Eventually, 
the parabolic $\qt$ is determined by crossing the nodes corresponding to roots
in $\Phi_{C}{\coloneqq}\,\Bz(C)\,{\cap}\,\Qq^{n}$.
\end{rmk}
\begin{exam} A  parabolic $CR$-algebra may have different 
not isomorphic 
adapted
$CR$-algebras, as we show in discussing the example below. \par
The flag manifold $\sfF$ of $\SL(4,\C)$, consisting of pairs $(\ell_1,\ell_2),$ with
$\ell_{i}$ a linear subspace of dimension $i$ in $\C^{4}$ and $\ell_{1}\,{\subset}\,\ell_{2}$,
is a smooth complex projective variety of  dimension $5$.
The different orbits 
of the real form $\SU(2,2)$
 in $\sfF$
are characterised by the signature of the restriction to $\ell_1$ and to $\ell_2$ of  a
hermitian symmetric form of signature $(2,2)$ (they are listed below in the first two columns
of Table \ref{e6.2}).
\par
{There}
are three {conjugacy classes of} 
Cartan subalgebras of $\su(2,2)$.  
Namely, we can take 
\begin{align*}
 \hg_{\stt_{1}}&=\{\diag(\lambdaup_{1},\lambdaup_{2},{-}\bar{\,\lambdaup}_{2},{-}\bar{\,\lambdaup}_{1})\mid
 \lambdaup_{1},\lambdaup_{2}\,{\in}\,\C,\;\im(\lambdaup_{1}{+}\lambdaup_{2})\,{=}\,0\},\\
  \hg_{\stt_{2}}&=\{\diag(\lambdaup_{1},i\lambdaup_{2},i\lambdaup_{3},{-}\bar{\,\lambdaup}_{1})\mid
 \lambdaup_{1}\,{\in}\,\C,\;\lambdaup_{2},\lambdaup_{3}\,{\in}\,\R,\;2\im(\lambdaup_{1}){+}\lambdaup_{2}{+}
 \lambdaup_{3}\,{=}\,0\},\\
   \hg_{\stt_{3}}&=\{\diag(i\lambdaup_{1},i\lambdaup_{2},i\lambdaup_{3},i\lambdaup_{4})\mid
\lambdaup_{i}\,{\in}\,\R,\;\lambdaup_{1}{+}\lambdaup_{2}{+}
 \lambdaup_{3}{+}\lambdaup_{4}\,{=}\,0\}.
\end{align*}
In fact, $\hg_{\stt_{i}}$ is the space of diagonal matrices {of the real form}
\begin{equation*}
 \gt_{\sigmaup_{i}}\,{=}\,\{X\,{\in}\,\slt_{4}(\C)\,{\mid}\,X^{*}K_{i}{+}\,K_{i}X\,{=}\,0\},
\end{equation*}
where $\sigmaup_{i}(X)\,{=}\,{-}K_{i}X^{*}K_{i}$, for 
\begin{equation*}
 K_{1}=\left( 
\begin{smallmatrix}
 &&& 1\\
 &&1\\
 &1\\
 1
\end{smallmatrix}\right),\;\; K_{2}=\left( 
\begin{smallmatrix}
 &&& 1\\
 &1\\
 &&{-}1\\
 1
\end{smallmatrix}\right),\;\;
K_{3}=\left( 
\begin{smallmatrix}
1\\
 &1\\
 &&{-}1\\
 &&&{-}1
\end{smallmatrix}\right).
\end{equation*}

We have 
\begin{equation*}\begin{gathered}
 \dim_{\R}(\hg^{+}_{\stt_{i}})=
\begin{cases}
 2, & i=1,\\
 1, & i=2,\\
 0, & i=3,
\end{cases}
\qquad \dim_{\R}(\hg^{-}_{\stt_{i}})= 
\begin{cases}
 1, & i=1,\\
 2, & i=2,\\
 3, & i=3,
\end{cases}
\end{gathered}
\end{equation*}
where $\hg_{\R}=\hg^{+}_{\stt_{i}}\oplus\hg^{-}_{\stt_{i}} $
,   with $\hg^{\pm}_{\stt_{i}}=\{H\in\hg_{\R}\,|\,\stt_i(H)=\pm H \}$. 
 \par\smallskip
The ten orbits of $\SU(2,2)$ in $\sfF$ are listed in the following array. In the second to last column
we write, for each orbit, the adapted Cartan subalgebras of their $CR$-algebras.
In the last column we comment on the 
{\textit{type}}.
\begin{equation*}\label{e6.2}
 \begin{array}{| c | c | c | c | l |}
 \hline
 \ell_{1} & \ell_{2} & \text{$CR$-type} & \text{adapted $\hg_{\sigmaup}$} &\qquad\text{type}\\
 \hline
 0 & 0,0 & 1,4 & \hg_{\stt_{1}}&(1,4)\;\;\text{(minimal)}\\
 \hline
 0 & 0,+ & 2,3 & \hg_{\stt_{2}}&(2,3)\;\;\text{(intermediate)} \\
 \hline
  0 & 0,- & 2,3 & \hg_{\stt_{2}}&(2,3)\;\;\text{(intermediate)}\\
  \hline
   0 & +,- & 4,1 & \hg_{\stt_{1}},\,\hg_{\stt_{2}}&(4,1)\;\;\text{(intermediate)}\\
 \hline
    + & 0,+ & 4,1 & \hg_{\stt_{1}},\,\hg_{\stt_{2}}&(4,1)\;\;\text{(intermediate)}\\
    \hline
     - & 0,- & 4,1 & \hg_{\stt_{1}},\,\hg_{\stt_{2}}&(4,1)\;\;\text{(intermediate)}\\
 \hline
    + & +,+ & 5,0 & \hg_{\stt_{3}}&(5,0)\;\;\text{(open)}\\
 \hline
    + & +,- & 5,0 & \hg_{\stt_{3}}&(5,0)\;\;\text{(open)}\\
 \hline
    - & -,- & 5,0 & \hg_{\stt_{3}}&(5,0)\;\;\text{(open)}\\
 \hline
    - & +,- & 5,0 & \hg_{\stt_{3}}&(5,0)\;\;\text{(open)}\\
 \hline
\end{array}
\end{equation*}
\end{exam}\medskip
In the case of parabolic $CR$-algebras,
by using parabolic $CR$-triples, notions and results of \S\ref{s2.2} and \S\ref{s3.7},
 can be reformulated in terms of the associated pairs
$(\Qq\,,\stt)$.

\par 
This approach makes clear 
that several \emph{local} properties of the 
$CR$-mani\-fold corresponding to  a parabolic $CR$-algebra 
only depend 
on the involution
$\stt$ of $\Rad$, and are shared by all 
{$(\gs,\qt)$ with} $\sigmaup$ in $\Invs$. 
\begin{dfn}
 Let $\Qq$ be a parabolic set in $\Rad$ and $\stt$ an involution of $\Rad$. We say that the
 pair $(\Qq\, ,\stt)$ is 
 \begin{equation*}
 \begin{array}{| c   l |} \hline
 \textit{trivial} &\text{if}\;\; \Qq=\Rad,\\ \hline
 \textit{integrable} &\text{if}\;\; \Qq\cup\stt(\Qq)\;\;\text{is closed},\\ \hline
 \textit{fundamental} &\text{if}\;\; \Qq\cup\stt(\Qq)\;\,\text{generates $\Rad$},
\\ \hline
  \textit{totally complex} &\text{if}\;\; \Qq\cup\stt(\Qq)=\Rad,\\ \hline
 \textit{totally real} &\text{if}\;\; \stt(\Qq)=\Qq\, ,\\ \hline
 \textit{$1$-nondegenerate} &\text{if}\;\; \forall\,\alphaup\,{\in}\,\Qq{\setminus}\stt(\Qq),\,\;
 (\alphaup\,{+}\,\stt(\Qq))\,{\cap}\,\Rad\not\subseteq\Qq\,{\cup}\,\stt(\Qq),\\ \hline 
 {\textit{Levi-nondegenerate}}
 &\text{if}\; 
\Qq'\,{\in}\Pcr(\Rad) \;\text{and}\; \Qq\,{\subseteq}\,\Qq'\,{\subseteq}\,\Qq\,{\cup}\,\stt(\Qq){\Rightarrow}
 \Qq'\,{=}\,\Qq\, ,
  \\ \hline
\end{array}
\end{equation*}
\end{dfn}
It follows from the relationship between parabolic subalgebras and parabolic sets of roots that
the proposition below holds true.
\begin{prop} Let $\Qq$ be a parabolic set in $\Rad$, $\qt$ the corresponding parabolic subalgebra
of $\gt$, 
$\stt$ an involution of   $\Rad$
and $\sigmaup\,{\in}\,\Invs$ a lifting of $\stt$ to {an anti-$\C$-linear involution of}
$\gt.$ \par
The parabolic $CR$-algebra
$(\gs,\qt)$ is trivial (resp. integrable, fundamental, totally complex, totally real, $1$-nondegenerate,
finitely 
Levi-nondegenerate) iff the pair $(\Qq\, ,\stt)$ is
trivial (resp. integrable, fundamental, totally complex, totally real, $1$- nondegenerate,
Levi-nondegenerate).\qed
 \end{prop}
\subsection{Levi-Chevalley decomposition of parabolic $CR$-algebras}
Let $(\gs,\qt)$ be a parabolic $CR$-algebra. Fix an adapted Cartan subalgebra $\hst$
and let $\Rad$ be the root system of $\gt\,{=}\,\C\,{\otimes}_{\R}\gs$ for the complexification $\hg$
of $\hst$. 
{Let $\stt$ be the involution on $\Rad$ induced by $\sigmaup$
and $\Qq$ the parabolic subset of $\Rad$ associated to $\qt$. We have 
Levi-Chevalley decompositions} 
\begin{gather*}\qt=\qt^{n}\oplus\qt^{r},\quad\text{with}\;\;
 \qt^{n}{=}{\sum}_{\alphaup\in\Qq^{n}}\gt^{\alphaup},\;\; 
 \qt^{r}{=}\hg\,{\oplus}{\sum}_{\alphaup\in\Qq^{r}}\gt^{\alphaup},
 \\
 \sigmaup(\qt)=\sigmaup(\qt^{n})\oplus\sigmaup(\qt^{r}),\quad\text{with}\;\;
 \sigmaup(\qt^{n}){=}{\sum}_{\alphaup\in\Qq^{n}}\gt^{\stt(\alphaup)},\;\; \sigmaup(\qt^{r}){=}\hg
 \,{\oplus}{\sum}_{\alphaup\in\Qq^{r}}\gt^{\stt(\alphaup)},
\end{gather*}
{in which}
 $\qt^{n},\,\sigmaup(\qt^{n})$ and $\qt^{r},\,\sigmaup(\qt^{r})$ are the nilradicals and 
{the $\hg$-regular reductive Levi factors}  
 of $\qt$ and $\sigmaup(\qt)$, respectively. 
 We get an analogous decomposition for the complexification of the isotropy.
 \begin{prop} Let $(\gs,\qt)$ be a parabolic $CR$-algebra.
  The complexification 
 $\qt\,{\cap}\,\sigmaup(\qt)$ of its isotropy subalgebra $\gs\cap\qt$ 
 is a regular Lie subalgebra of $\gt,$ having the Levi-Chevalley decomposition 
\begin{equation}\vspace{-16pt}
\begin{cases}
 \qt\cap\sigmaup(\qt)=(\qt\cap\sigmaup(\qt))^{r}\oplus(\qt\cap\sigmaup(\qt))^{n},\;\;\text{with}\\
 (\qt\cap\sigmaup(\qt))^{r}\,{=}\,\qt^{r}\,{\cap}\,\sigmaup(\qt^{r}),\,\quad\text{({$\hg$-regular}
  reductive Levi factor)},
 \\
(\qt\cap\sigmaup(\qt))^{n}\,{=}\, 
\qt^{r}{\cap}\,\sigmaup(\qt^{n})\,{\oplus}\,\qt^{n}{\cap}\,\sigmaup(\qt^{r})\,{\oplus}\,
\qt^{n}{\cap}\,\sigmaup(\qt^{n})
\;\;
\text{(nilradical)}.
\end{cases}
\end{equation}\qed
\end{prop}
\begin{prop}
 Let $(\gt_{\sigmaup},\qt)$ be a parabolic $CR$-algebra. We can find a nilpotent horocyclic Lie
 subalgebra 
 $\qtc$ of $\gt$ such that 
\begin{equation}
 \gt=\qt\oplus\qtc
\end{equation}
and we get 
a decomposition {of $\gt$} 
into a sum of complex Lie subalgebras 
\begin{equation}
 \gt=(\qt\cap\sigmaup(\qt))\oplus(\qt\cap\sigmaup(\qtc))\oplus(\qtc\cap\sigmaup(\qt))\oplus
 (\qtc\cap\sigmaup(\qtc)).
\end{equation}
\end{prop} 
\begin{proof} 
We fix a 
Cartan subalgebra 
{$\hst$ adapted to $(\gs,\qt)$} and consider the pair $(\Qq\,,\stt)$ of the parabolic $CR$-triple
$(\gs,\qt,\hst)$.
 Then $\Qc\,{=}\,\Rad{\setminus}\Qq$ is horocyclic  and $\qt^{c}\,{=}\,{\sum}_{\alphaup\in\Qc}
 {\gt^{\alphaup}}$ has the required properties. 
\end{proof}
We described in
Proposition\,\ref{p2.5} the natural $\Z$-grading $\chiup_{\Qq}$ of $\Rad$
associated to a $\Qq\,{\in}\Pcr(\Rad)$.
Having introduced the involution $\stt$, it is natural to consider in $\Rad$ the 
bi-grading $(\chiup_{\Qq},\chiup_{\Qq}\,{\circ}\,\stt)$. For a pair of integers $p,q$ set 
\begin{equation}\label{e2.14}\begin{cases}
 \Qq^{p,q}=\{\alphaup\,{\in}\,\Rad\,{\mid}\,\chiup_{\Qq}(\alphaup)\,{=}\,p,\,\chiup_{\Qq}(\stt(\alphaup))\,{=}\,q\},\\
 \qt^{p,q}\,{=}\,
\begin{cases}
 {\sum}_{\alphaup\,{\in}\,\Qq^{p,q}}\gt^{\alphaup}, & (p,q)\neq(0,0),\\
 \hg\oplus{\sum}_{\alphaup\,{\in}\,\Qq^{0,0}}\gt^{\alphaup}, & (p,q)=(0,0).
\end{cases}
\end{cases}
\end{equation}
Then, e.g.  $\Qq\,{=}\,{\bigcup}_{p\geq{0}}\Qq^{p,q}$, $\stt(\Qq)\,{=}\,{\bigcup}_{q\geq{0}}\Qq^{p,q}$,
$\Qc\,{\cap}\,\stt(\Qc)\,{=}\,{\bigcup}_{p<0,\,q<0}\Qq^{p,q}$. 

\begin{rmk}\label{r4.16}
The restriction of the complexification of differentials of the smooth submersion in the commutative diagram \eqref{CRimmersion} 
 yields isomorphisms (that can be further detailed by using
\eqref{e2.14})
\begin{equation*}
 \qt\cap\sigmaup(\qtc)\simeq \qt/(\qt\,{\cap}\,\sigmaup(\qt)),\quad \qtc\cap\sigmaup(\qtc)\simeq
 \gt/(\qt+\sigmaup(\qt)).
\end{equation*}
We recall that, $(\gs,\qt)$ being the $CR$-algebra at a point $\pct_{0}$ of a homogeneous $CR$-manifold
$\sfM$, the quotient $\qt/(\qt\,{\cap}\,\sigmaup(\qt))$ represents the space $T^{0,1}_{\pct_{0}}\sfM$ 
and $\gt/(\qt\,{+}\,\sigmaup(\qt))$ is the complexification 
{of $\T_{\pct_{0}}\sfM{/}\Hd_{\pct_{0}}\sfM$.}
The Lie subalgebra {$\qtc\,{\cap}\,\gs$}
represents the stalk at $\pct_{0}$ of a
Reeb distribution for the contact structure underlying the $CR$-structure of $\sfM$.\end{rmk}


\subsection{Reduction to simple symmetry  algebras} 
{We can reduce the}
study of effective parabolic
$CR$-algebras  to the cases where their symmetry algebra 
is \textit{simple} (see e.g. \cite{AMN06,AMN06b}).
\begin{prop}
Let $(\gs,\qt)$ be an effective parabolic $CR$-algebra
and let $\gs ={\gs}_{1}\oplus \cdots \oplus{\gs}_{\ell}$ be the
decomposition of $\gs$ into the direct sum of its simple ideals.
Then, denoting by ${\gt}_{j}$ the complexification of ${\gs}_{j},$
we get:
\begin{itemize}
\item[(i)] $\qt=\qt_{1}\oplus \cdots \oplus\qt_{\ell}$
where $\qt_{j}=\qt\cap{\gt}_{j}$ for $j=1,\hdots,\ell$;
\item[(ii)] for each $j=1,\hdots,\ell$, $({\gs}_j,\qt_{j})$
is an effective parabolic $CR$-algebra;
\item[(iii)] $(\gs,\qt)$ is contact (resp. Levi, 
{strictly} Levi)
nondegenerate if and only if for 
{every}
$j{=}1,\hdots,\ell$, the
$CR$-algebra $ ({\gs}_j,\qt_{j})$ is  contact (resp. Levi, 
{strictly} Levi)
nondegenerate;
\item[(iv)] $(\gs,\qt)$ is isomorphic to the direct sum of the $CR$-algebras $({\gs}_j,\qt_{j})$;
\item[(v)] $(\gs,\qt)$ is fundamental (resp. integrable, totally real, totally complex)
if and only if for 
{every}
$j\,{=}\,1,\hdots,\ell$, the
$CR$-algebra $ ({\gs}_j,\qt_j)$ is fundamental (resp. integrable, totally real, totally complex). \qed
\end{itemize}\end{prop}

\subsection{Contact-degeneracy} 

\begin{prop} 
A fundamental 
parabolic $CR$-algebra $(\gs,\qt)$ with $\gs$ simple
is either totally complex or contact-nondegenerate. 
\end{prop}
\begin{proof}
Indeed
a fundamental %
$(\gs,\qt)$ is contact-nondegenerate iff 
{$\Ft_{ -1}(\gs,\qt)$}
does not
contain a nontrivial ideal of $\gs$. When $\gs$ is simple,
this is equivalent to the fact that $(\qt{+}\sigmaup(\qt))\,{\neq}\,\gt$, i.e.
that $(\gs,\qt)$ is not totally complex.
\end{proof}

\subsection{$S$-fit and $V$-fit chambers} \label{s5.3}  
While drawing representing diagrams for semisimple real forms corresponding to different involutions $\stt$
of the root system, we used in \cite{MMN23} 
special Weyl chambers, that we called of type $S$
(from \textit{Satake}: are those in which positive $\stt$-complex roots have positive conjugate)
and of type $V$ (from \textit{Vogan}:  are those in which positive 
$\stt$-complex roots have negative conjugate). When we need to take into account also
a parabolic subset $\Qq$ in $\Rad$, we need to restrict to admissibe Weyl chamber which resemble
as much as possible to $S$\!{-} or $V$\!-chambers.
\par\smallskip
We recall (Definition\,\ref{d4.4}) that 
a Weyl chamber $C$
 in $\Cd(\Rad)$ is \emph{admissible} for a parabolic set $\Qq$ in $\Rad$ if $\Rad^{+}(C)\,{\subseteq}\,\Qq\, .$ 
\par
Properties of a parabolic  $(\gs,\qt)$ may be related 
to the behaviour under the involution $\stt$ of the simple positive roots
of an admissible Weyl chamber   for the parabolic set $\Qq$ of $\qt$.
{We  use the notation: 
\begin{equation*}
 \Bz_{\circ}^{\stt}(C)=\Bz(C)\cap\Rad^{\stt}_{\,\;\circ},\;\;
  \Bz_{\bullet}^{\stt}(C)=\Bz(C)\cap\Rad^{\stt}_{\,\;\bullet},\;\;
  \Bz_{\star}^{\stt}(C)=\Bz(C)\cap\Rad^{\stt}_{\,\;\star},\;\;
\end{equation*}
and  
$\Phi_{C}(\Qq)\,{=}\,\Qq^{n}\cap\Bz(C)$ (Formula \eqref{e4.6}).}

\begin{dfn} (cf. \cite[Lemma\,4.3 and Def.4.4]{AMN06b})
Let $\stt$ be an involution of $\Rad$ and $\Qq$ a parabolic set in  $\Rad$. 
An admissibe Weyl chamber $C\,{\in}\,\Cd(\Rad,\Qq)$ is called 
\begin{itemize}
  \item $S$-fit for $(\Qq\, ,\stt)$ \; iff\;\; $ \stt(\alphaup)\,{\in}\,\Rad^{+}(C),\;\forall\alphaup\,{\in}\,
 \Bz_{\star}^{\stt}(C){\setminus}\Phi_{C}(\Qq)
 $,
  \item $V$-fit for $(\Qq\, ,\stt)$ \; iff\;\; $ \stt(\alphaup)\,{\in}\,\Rad^{-}(C),\;\forall\alphaup\,{\in}\,
 \Bz_{\star}^{\stt}(C){\setminus}\Phi_{C}(\Qq)$.
 \end{itemize}
\par
If $\sigmaup\,{\in}\,\Invs$ and $\Qq$ is the set of roots of the parabolic subalgebra $\qt$ of
$\gt$, we say that $C$ is $S$-fit (resp. $V$-fit) for the parabolic $CR$-algebra $(\gs,\qt)$ iff it is 
$S$-fit (resp. $V$-fit) for $(\Qq\, ,\stt)$.
 \end{dfn}
 
\begin{prop} For every  involution $\stt$ and parabolic set $\Qq$ in $\Rad$
there are both $S$\!{-} and $V$-fit
Weyl chambers for $(\Qq\, ,\stt)$. 
\end{prop} 
\begin{proof} Let $C\,{\in}\,\Cd(\Rad,\Qq)$ be a chamber for which the number of roots in 
\begin{equation*}
\Aq(C)= \{\alphaup\in\Qq^{r}\,{\cap}\,\Rad^{+}(C)\,{\mid}\,\stt(\alphaup)\,{\in}\,\Rad^{-}(C)\}
\end{equation*}
is minimal. We claim that $C$ is $S$-fit. Assume by contradiction that there is 
$\betaup\,{\in}\,\Bz^{\stt}_{\star}(C){\setminus}\Phi_{C}$ such that $\stt(\alphaup)\,{\in}\,\Rad^{-}(C)$.
Since $\betaup\,{\in}\,\Qq^{r}$, the chamber $C'$ with
\begin{equation*}
 \Rad^{+}(C')=\rtt{\betaup}(\Rad^{+}(C))=(\Rad^{+}(C){\setminus}\{\betaup\})\cup\{{-}\betaup\}
\end{equation*}
is, by Lemma\,\ref{l4.7}, 
 admissible for $\Qq$ 
and then $\Aq(C')\,{=}\,\Aq(C){\setminus}\{\betaup\}$ contradicts minimality. 
\par\smallskip
Analogously, a $C'\,{\in}\,\Cd(\Rad,\Qq)$ for which the number of roots in 
\begin{equation*}
 \Aq'(C')=\{\alphaup\,{\in}\,\Qq^{r}\,{\cap}\,\Rad^{+}(C')\mid \stt(\alphaup)\,{\in}\,\Rad^{+}(C')\}
\end{equation*}
is minimal is $V$-fit for $(\gs,\qt).$
\end{proof}
\begin{lem}\label{l4.20}
 Let $C\,{\in}\,\Cd(\Rad)$. Then 
\begin{align*} \begin{array}{c c c}
 C \;\text{is $S$-fit for $(\Qq\, ,\stt)$} & \Longleftrightarrow & \stt(\alphaup)\,{\in}\,\Rad^{+}(C),\;
 \forall\alphaup\,{\in}\,\Qq^{r}\,{\cap}\,\Rad^{+}(C)\,{\cap}\,\Rad^{\stt}_{\,\;\star},\\
 C \;\text{is $V$-fit for $(\Qq\, ,\stt)$} & \Longleftrightarrow & \stt(\alphaup)\,{\in}\,\Rad^{-}(C),\;
 \forall\alphaup\,{\in}\,\Qq^{r}\,{\cap}\,\Rad^{+}(C)\,{\cap}\,\Rad^{\stt}_{\,\;\star}.
 \end{array}
\end{align*}
\end{lem} 
\begin{proof}
 Since 
{$\Bz_{\star}^{\stt}(C){\setminus}\Phi_{C}\,{\subseteq}\,
 \Qq^{r}\,{\cap}\,\Rad^{+}(C)\,{\cap}\,\Rad^{\stt}_{\,\;\star},$}
 the implications ``$\Leftarrow$'' are trivial. Let us prove the vice versa. Let 
 \begin{equation*}
 \alphaup={\sum}_{\betaup\in\Bz(C)\setminus\Phi_{C}}k_{\alphaup,\betaup}\betaup\in
 \Qq^{r}\,{\cap}\,\Rad^{+}(C)\,{\cap}\,\Rad^{\stt}_{\,\;\star}.
\end{equation*}
Since $\Rad_{\,\;\circ}^{\stt}$ and $\Rad_{\,\;\bullet}^{\stt}$ are 
root subsystems of $\Rad$ orthogonal to each other, 
{and}
real and imaginary simple roots 
{are}
strongly orthogonal,
$\supp_{C}(\alphaup)$ is not contained in any of them. 
Hence it
contains elements of
$\Bz_{\star}^{\stt}(C)$. Therefore
$\supp_{C}(\stt(\alphaup))$ contains elements of $\Bz_{\star}^{\stt}(C),$
which, in its decomposition as a linear combination of elements of $\Bz(C)$,
appear with positive coefficients when $C$ is $S$-fit, with negative coefficients
when $C$ is a $V$-fit. 
The statement follows from this observation.
\end{proof}
\begin{lem}\label{l4.21}
 Let $C\,{\in}\,\Cd(\Rad)$.  
\begin{align*} \begin{array}{c c c}
\text{if $C$ is $S$-fit for $(\Qq\, ,\stt)$} & \text{then} & \begin{cases}
\stt(\Qq^{n})\subseteq\Qc\,{\cup}\,\Rad^{+}(C),\\
\stt(\Qq^{r}\cap\Rad^{+}(C))\subset\Qq\, ,
\end{cases}\\
\quad
\\
\text{If $C$ is  $V$-fit for $(\Qq\, ,\stt)$} & \text{then} & \begin{cases}
\stt(\Qq^{n})\subseteq\Qq^{n}\,{\cup}\,\Rad^{-}(C),\\
\stt(\Qq^{r}\cap\Rad^{-}(C))\subset\Qq\, .
\end{cases}
 \end{array}
\end{align*}
\end{lem} 
\begin{proof} We have
\begin{equation*}
\stt(\Qq^{n}\,{\cap}\,(\Rad_{\,\;\circ}^{\stt}\,{\cup}\,\Rad_{\,\;\bullet}^{\stt}))
\,{\subseteq}\, 
\Qq^{n}\cup\Qc\;\;\text{and}\;\; \Rad^{-}(C)\,{\subseteq}\,\Qq^{\vee}\,{=}\,\Qc\,{\cup}\,
 \Qq^{r}.
\end{equation*}
 Consider a root $\alphaup$ in $\Qq^{n}\,{\cap}\,\Rad_{\,\;\star}^{\stt}$. \par
 If $C$ is $S$-fit for $(\Qq\, ,\stt)$ and $\stt(\alphaup)\,{\in}\,\Rad^{-}(C)$, then  $\stt(\alphaup)\,{\in}\Qc$.
 Indeed, if by contradiction 
$\stt(\alphaup)\,{\in}\,\Qq^{r}$, then, 
 since $\stt(\alphaup)\,{\in}\,\Rad_{\,\;\star}^{\stt}$, by Lemma\,\ref{l4.20}
 this would imply that $\alphaup\,{\in}\,\Rad^{-}(C)$, yielding a contradiction. 
 This completes the proof of the first inclusion and,   
{to}
 check
the second one, it suffices to consider roots $\alphaup\,{\in}\,
\Qq^{r} \,{\cap}\,\Rad^{+}(C)\,{\cap}\,\Rad_{\,\;\bullet}^{\stt}$, for which $\stt(\alphaup)\,{=}\,{-}\alphaup
\,{\in}\,\Qq^{r}\,{\subseteq}\,\Qq$.\par\smallskip
A similar argument applies when $C$ is $V$-fit for $(\gs,\qt)$. Indeed, in this case, if 
$\alphaup\,{\in}\,\Qq^{n}\,{\cap}\,\Rad^{\stt}_{\,\;\star}$ and $\stt(\alphaup)\,{\in}\,\Rad^{+}(C)$,
then $\stt(\alphaup)\,{\in}\,\Qq^{n}$, because $C$-positive complex roots in $\Qq^{r}$ have negative
$\stt$-conjugate. For the second inclusion again it suffices to consider real and imaginary roots in $\Qq^{r}$,
for which the statement is clear. The proof is complete.
\end{proof}
\begin{lem} \label{l4.22}
Let $\stt$ be an involution of $\Rad,$
 $\Qq\,{\in}\Pcr(\Rad)$, $C_{S}$ and $C_{V}$ an $S$-fit and a $V$-fit 
Weyl chamber for $(\Qq\, ,\stt)$, respectively. Then 
\begin{equation}\label{e2.12}\begin{cases}
\Qq^{n}\,{\cap}\,\stt(\Qc)\,{=}\,\{\alphaup\,{\in}\,\Qq^{n}\,{\mid}\,\stt(\alphaup)\,{\in}\,\Rad^{-}(C_{S})\},\\
 \Qq^{n}\,{\cap}\,\stt(\Qq^{n})\,{=}\,\{\alphaup\,{\in}\,\Qq^{n}\,{\mid}\,\stt(\alphaup)\,{\in}\,\Rad^{+}(C_{V})\},\\
  \Rad^{+}(C_{S})\cap\stt(\Rad^{-}(C_{S}))=(\Rad^{+}(C_{S})\cap\Rad^{\stt}_{\,\;\bullet})\cup(\Qq^{n}\cap\stt(\Qq^{c})),\\
 \Rad^{+}(C_{V})\cap\stt(\Rad^{+}(C_{V}))=(\Rad^{+}(C_{V})\cap\Rad^{\stt}_{\,\;\circ})\cup(\Qq^{n}\cap\stt(\Qq^{n})).
 \end{cases}
\end{equation}
\end{lem} 
\begin{proof} Let $\alphaup\,{\in}\,\Qq^{n}$. If $\stt(\alphaup)\,{\in}\,\Qq^{r}$, then $\alphaup\,{\in}\,
\Rad^{\stt}_{\,\;\star}$ and 
$\stt(\alphaup)\,{\in}\,\Rad^{+}(C_{S})\,{\cap}\,\Rad^{-}(C_{V})$,
because $C_{S}$ and $C_{V}$ are $S$-fit and 
$V$-fit, respectively. Hence, 
$\stt(\alphaup)\,{\in}\,\Rad^{+}(C_{V})$
iff $\stt(\alphaup)\,{\in}\,\Qq^{n}$, and { $\stt(\alphaup)\,{\in}\,\Rad^{-}(C_{S})$}
iff 
$\stt(\alphaup)\,{\in}\,\Qc$. The last two lines follow from  Lemma\,\ref{l4.20}.
\end{proof} 
The $\chiup_{\Qq}$-degree of the largest root $\deltaup_{C}$ is the same for all  
$C\,{\in}\,\Cd(\Rad,\Qq)$. This may not be the case for {the}
$\chiup_{\Qq}$-degree of their $\stt$-conjugate.
\begin{prop}
 If $C_{S}$ and $C_{V}$ are an $S$\!{-} and a $V$-fit chamber for $(\Qq\,,\stt)$, respectively, then 
\begin{equation}
\chiup(\stt(\deltaup_{C_{V}}))\leq \chiup(\stt(\deltaup_{C}))\leq \chiup(\stt(\deltaup_{C_{S}})),\;\;
\forall C\,{\in}\,\Cd(\Rad,\Qq).
\end{equation}
\end{prop} 
\begin{proof}
 {If $C,C'\,{\in}\,\Cd(\Rad,\Qq)$, by 
 Lemma\,\ref{l4.7} we can 
 obtain $\deltaup_{C}$ by applying to $\deltaup_{C'}$ 
 a sequence of reflections with respect to roots in $\Qq^{r}$. Hence
 we can find a minimal length sequence $(\betaup_{1},\hdots,\betaup_{r})$ in $\Qq^{r}$
 such that 
\begin{equation*}\tag{$*$}
 \deltaup_{C'}{-}{\sum}_{i=1}^{h}\betaup_{i}\,{\in}\,\Rad\;\forall\, 1{\leq}h{\leq}{r}
 \;\text{and}\;\; \deltaup_{C}=
 \deltaup_{C'}{-}{\sum}_{i=1}^{r}\betaup_{i}.
\end{equation*}
We claim that $\betaup_{i}\,{\in}\,\Rad^{+}(C')$ for all $i$. This is true for $i\,{=}\,1$, because
$\deltaup_{C'}$ is the largest root in $\Rad^{+}(C')$. Assume by contradiction that there
is a smallest $p$, with $1{<}p{\leq}r$, for which $\betaup_{p}\,{\in}\,\Rad^{-}(C)$. 
Let $(X_{\alphaup},H_{\alphaup})_{\alphaup\,{\in}\,\Rad}$ be a Chevalley system for $(\gt,\hg)$.
Then $(*)$ translates into the fact that 
\begin{equation*}
 [X_{\deltaup_{C'}},X_{-\betaup_{1}},\hdots,X_{-\betaup_{r}}]=
 [[[X_{\deltaup_{C'}},X_{-\betaup_{1}},\hdots,X_{-\betaup_{p-1}}],X_{-\betaup_{p}}],X_{-\betaup_{p+1}},
 \hdots,X_{-\betaup_{r}}]
\end{equation*}
is different from zero and proportional to $X_{\deltaup_{C}}$.
Since
$(\deltaup_{C'}{-}\,\betaup_{p})$ is not a root, we have 
\begin{align*}
 [[X_{\deltaup_{C'}},X_{-\betaup_{1}},\hdots,X_{-\betaup_{p-1}}],X_{-\betaup_{p}}]=
 [X_{\deltaup_{C'}},[X_{-\betaup_{1}},X_{-\betaup_{p}}],X_{-\betaup_{2}},\hdots,X_{-\betaup_{p-1}}]\\
 +[X_{\deltaup_{C'}},X_{-\betaup_{1}},\hdots,[X_{-\betaup_{i}},X_{-\betaup_{p}}],\hdots,X_{-\betaup_{p-1}}]\\
 +[X_{\deltaup_{C'}},X_{-\betaup_{1}},\hdots,X_{-\betaup_{p-2}},[X_{\betaup_{p-1}},X_{-\betaup_{p}}]].
\end{align*}
At least one  summand in the right hand side is different from $0$. Hence there is an index  $j$, 
with $1{\leq}j{<}p$,  for which $\betaup_{j}{+}\betaup_{p}$ is a root, and,
 by taking $\betaup'_{i}{=}\,\betaup_{i}$
for $i{\neq}j$ and
$1{\leq}i{<}p$, $\betaup'_{j}{=}\,\betaup_{j}{+}\,\betaup_{p}$,
 and $\betaup_{h}'{=}\,\betaup_{h+1}$
for $p{\leq}h{\leq}r{-}1$
we obtain 
a shorter sequence $(\betaup'_{1},\hdots,\betaup'_{r-1})$ in $\Qq^{r}$, with 
\begin{equation*}
 \deltaup_{C'}{-}{\sum}_{i=1}^{h}\betaup'_{i}\in\Rad,\;\forall 1{\leq}h{\leq}r{-}1,\;\;
 \deltaup_{C}=\deltaup_{C'}{-}{\sum}_{i=1}^{r-1}\betaup'_{i},
\end{equation*}
contradicting our choice of $r$.
This proves our claim. } \par
{
If $C'\,{=}\,C_{S}$ is  $S$-fit, then $\chiup_{\Qq}(\stt(\betaup_{i}))\,{\geq}\,0$ for all
$\betaup_{i}\,{\in}\,\Rad^{+}(C_{S})$, because the value is $0$ if $\betaup_{i}\,{\in}\,
\Qq^{r}{\cap}(\Rad^{\stt}_{\,\;\circ}{\cup}\,\Rad^{\stt}_{\,\;\bullet})$ and nonnegative if
$\betaup_{i}\,{\in}\,\Qq^{r}{\cap}\,\Rad^{\stt}_{\,\star}{\cap}\,\Rad^{+}(C)$,
because  by Lemma\,\ref{l4.20} $\stt(\betaup_{i})\,{\in}\,\Rad^{+}(C_{S})$ in this case. Hence 
\begin{equation*}
 \chiup_{\Qq}(\stt(\deltaup_{C}))=\chiup_{\Qq}(\deltaup_{C_{S}})-{\sum}_{h=1}^{p}\chiup_{\Qq}(\stt(\betaup_{h}))
 \leq \chiup_{\Qq}(\deltaup_{C_{S}}).
\end{equation*}}
\par
{Analogously, if $C'\,{=}\,C_{V}$ is $V$-fit, all $\chiup_{\Qq}(\stt(\betaup_{i}))$
are nonpositive, because they are $0$ for $\betaup_{i}\,{\in}\,
\Qq^{r}{\cap}(\Rad^{\stt}_{\,\;\circ}{\cup}\,\Rad^{\stt}_{\,\;\bullet})$ and nonpositive if
$\betaup_{i}\,{\in}\,\Qq^{r}{\cap}\,\Rad^{\stt}_{\,\star}{\cap}\,\Rad^{+}(C)$,
because, again  by Lemma\,\ref{l4.20}, $\stt(\betaup_{i})\,{\in}\,\Rad^{-}(C_{S})$ in this case.
This proves 
that $\chiup(\stt(\deltaup_{C}))\,{\geq}\,\chiup_{\Qq}(\stt(\deltaup_{C_{V}}))$ if $C_{V}$ is $V$-fit.}
\end{proof}
\subsection{Open and integrable orbits} (see \cite[\S{7}]{Wolf69}) 
Recall that 
a $CR$-algebra $(\gs,\qt)$ is \textit{totally complex} when $\qt\,{+}\,\sigmaup(\qt)\,{=}\,\gt$.
For the corresponding parabolic set $\Qq$ and involution $\stt$, 
this means
that $\Qq\,{\cup}\,\stt(\Qq)\,{=}\,\Rad$. 
{This is equivalent to $\Qc\,{\cap}\,\stt(\Qc)\,{=}\,\emptyset$ and
thus, since $\Qq^{n}\,{=}\,\{{-}\alphaup\,{\mid}\,\alphaup\,{\in}\,\Qc\}$, we obtain} 
\begin{prop}\label{p5.10} 
The pair $(\Qq\, ,\stt)$  is totally complex if and only if 
 \begin{equation}\label{e2.13} \vspace{-18pt}
 \Qq^{n}\,{\cap}\,\stt(\Qq^{n})\,{=}\,\emptyset.
 \end{equation}\qed
\end{prop} 
For the corresponding parabolic subalgebra $\qt$ and a lift $\sigmaup\,{\in}\,\Invs$, 
\eqref{e2.13} is equivalent to $\qt^{n}\,{\cap}\,\sigmaup(\qt^{n})\,{=}\,\{0\}$. Thus we get 
\begin{cor} \label{c6.10}
If $\bt$ is Borel and $(\gs,\bt)$ effective and totally complex, then 
$\bt$ and $\sigmaup(\bt)$ are opposite and $\hst\,{=}\,\bt\,{\cap}\,\gs$ 
is a maximally compact Cartan subalgebra of~$\gs$.
 \qed
\end{cor}
\begin{prop}
 If $(\gt_{\sigmaup},\qt)$ is a totally complex effective parabolic $CR$-algebra, then 
 {$\qt\,{\cap}\,\gs$} contains a
 maximally compact Cartan subalgebra of~$\gs$.
\end{prop} 
\begin{proof} Let $\sfB$ and $\sfF$ be complex flag manifolds of $\gt$,  the first 
consisting
of the Borel and the second 
of the conjugates of $\qt$ in $\gt$.
 Let us fix an $\at\,{\in}\,\sfB$ contained in $\qt$. The adjoint action of a  
 complex semisimple Lie group  $\Gf$ with Lie algebra $\gt$, being transitive on $\sfF$ and $\sfB$,
 defines a projection 
 $\piup\,{:}\,\sfB\,{\to}\,\sfF$ by 
\begin{equation*}
 \piup(\Ad(a)(\at))\,{=}\,\Ad(a)(\qt),\;\;\forall\,a\,{\in}\,\Gf.
\end{equation*}
 By assumption the orbit $\sfM$ of a real form 
{$\Gfs$} of $\Gf$ through $\qt$ 
 in $\sfF$
 is open. Its inverse image $\piup^{-1}(\sfM)$ is a union of finitely many orbits of 
{$\Gfs$}
 in $\sfB$.
 At least one of them is open and 
 then by Cor.\ref{c6.10}
 has $CR$-algebra $(\gt_{\sigmaup},\bt)$, 
with $\sigmaup(\bt)$ 
opposite of $\bt$, and 
{$\bt\,{\cap}\,\sigmaup(\bt)$} 
is a maximally compact Cartan subalgebra
 of $\gt_{\sigmaup}$ contained in 
{$\qt\,{\cap}\,\gs$.}
\end{proof}
\begin{prop} 
Let $\Qq\,{\in}\Pcr(\Rad)$ and 
{$\stt$ an involution of $\Rad$.}
 The following are equivalent: 
\begin{enumerate}
 \item $(\Qq\, ,\stt)$ is totally complex;
 \item if $C$ is a $V$-fit Weyl chamber for $(\gs,\qt)$, then 
 $\stt(\Qq^{n})\,{\subseteq}\,\Rad^{-}(C).$
\item if $C$ is a $V$-fit Weyl chamber for $(\gs,\qt)$, then 
$ \stt(\Phi_{C})\,{\subseteq}\,\Rad^{-}(C).
$
\end{enumerate}
\end{prop} 
\begin{proof} $(1)\,{\Rightarrow}\,(2)$. If $C$ is $V$-fit for $(\Qq\, ,\stt)$, then
by Lemma\,\ref{l4.21}, we have $\stt(\Qq^{n})\,{\subseteq}\,
\Qq^{n}\,{\cup}\,\Rad^{-}(C)$ and therefore (2) follows from $\Qq^{n}\,{\cap}\,\stt(\Qq^{n})\,{=}\,\emptyset$.
\par
Clearly $(2){\Rightarrow}(3)$.\par
{If $(3)$ is satisfied, then
$\stt(\alphaup)\,{\in}\,\Rad^{-}(C)$ for all $\alphaup\,{\in}\,\Bz^{\stt}_{\star}(C)$.} This means
that $C$ is a $V$-chamber for $\stt$ (cf.\,\cite[\S{2.6}]{MMN23}). In particular, 
$\Bz_{\circ}^{\stt}(C)$ is a basis of $\Rad_{\,\;\circ}^{\stt}$ and $\Rad_{\,\;\circ}^{\stt}\,{\subseteq}\,\Qq^{r}$.
Thus $\Qq^{n}\,{\subset}\,
(\Rad_{\,\;\bullet}^{\stt}\,{\cup}\,
\Rad_{\,\;\star}^{\stt})$ and therefore $\stt(\Qq^{n})\,{\subseteq}\,\Rad^{-}(C)$
because 
\begin{equation*}\stt(\Rad^{+}(C)\,{\cap}\,(\Rad_{\,\;\bullet}^{\stt}\,{\cup}\,
\Rad_{\,\;\star}^{\stt}))\,{=}\,\Rad^{-}(C)\,{\cap}\,(\Rad_{\,\;\bullet}^{\stt}\,{\cup}\,
\Rad_{\,\;\star}^{\stt}).\end{equation*}
{This implies that $\Qq^{n}{\cap}\,\stt(\Qq^{n})\,{=}\,\emptyset$ and thus}
$(\Qq\, ,\stt)$ is totally complex.
\end{proof}
We denote by $\Inv_{\C}^{\vee}(\gt)$ the set of anti-$\C$-linear involution of the complex
Lie algebra $\gt$.
\begin{thm}
 Let $\qt$ be a parabolic subalgebra of $\gt,$ $\qt^{n}$ its nilradical, 
{$\sigmaup$ an anti-$\C$-linear involution of $\gt$.}
The following are equivalent 
\begin{enumerate}
 \item $(\gs,\qt)$ is integrable, i.e. 
 $\qt\,{+}\,\sigmaup(\qt)$ is a Lie subalgebra of $\gt$;
 \item $\qt\,{+}\,\sigmaup(\qt)$
 is the normaliser of $\qt^{n}\,{\cap}\,\sigmaup(\qt^{n})$ in $\gt$;
 \item $\qt\,{+}\,\sigmaup(\qt)$ is a Lie subalgebra of $\gt$ and $\qt^{n}\,{\cap}\,\sigmaup(\qt^{n})$
 its nilradical;
 \item $\qt^{n}\,{\cap}\,\sigmaup(\qt^{n})$ is the nilradical 
 of its normaliser in $\gt$;
 \item $\qt$ is contained in the normaliser 
 of $\qt^{n}\,{\cap}\,\sigmaup(\qt^{n})$ in $\gt$. 
\end{enumerate}
\end{thm} 
\begin{proof} Assume that $(1)$ holds true. 
Since $\qt\,{+}\,\sigmaup(\qt)$ is parabolic, 
its nilradical is equal to
its orthogonal for the Killing form, which is the intersection
$\qt^n\,{\cap}\,\sigmaup(\qt^{n})$. Hence $(1)$ implies
$(3)$ and $(5)$.\par 
If $(5)$ holds true, then the normaliser $\at$ of  
$\qt^n\,{\cap}\,\sigmaup(\qt^{n})$ in $\gt$,
containing $\qt$, is a parabolic
subalgebra of $\gt$. Since $\qt^n\,{\cap}\,\sigmaup(\qt^{n})$
is $\sigmaup$-invariant, $\at$ contains
$\qt\,{+}\,\sigmaup(\qt)$. By construction the nilradical
$\nt$ 
of $\at$ contains $\qt^n\,{\cap}\,\sigmaup(\qt^{n})$
and hence coincides with it, because $\at^\perp{\subseteq}\,
\qt^n\,{\cap}\,\sigmaup(\qt^{n})$. This shows that $(5)$
implies $(4)$, $(3)$, $(2)$ and $(1)$. The proof is complete.
\end{proof}

 \subsection{Levi-nondegenerate 
 reduction}  
Levi-nondegenerate reductions 
 of parabolic $CR$-algebras
 are better
 described by employing
  $V$-fit Weyl chambers (cf. \cite[Thm.6.4]{AMN06b}).
\begin{lem}\label{l4.30} Let 
$\Qq\,{\in}\Pcr(\Rad)$, $C\,{\in}\,\Cd(\Rad,\Qq)$, 
$\Phi_{C}\,{=}\,\Bz(C)\,{\cap}\,\Qq^{n}$ and
$\stt$ an involution
of $\Rad$. 
 If $C$ is $V$-fit and 
 {$\alphaup_{0}\,{\in}\,\Phi_{C}$,}
then 
\begin{equation}\label{e2.19} \Qq_{\,\;(\Phi_{C}{\setminus}\{\alphaup_{0}\})_{C}}
 \subseteq\Qq\cup\stt(\Qq)\;\;\Longleftrightarrow\;\; 
 {\stt(\alphaup_{0})\in\Rad^{-}(C).}
\end{equation}
\end{lem} 
\begin{proof} Assume that $\stt(\alphaup_{0})\,{\in}\,\Rad^{-}(C).$ 
A root $\alphaup$ in $\Qq_{\,\;(\Phi_{C}{\setminus}\{\alphaup_{0}\})_{C}}\!{\setminus}\Qq$
is $C$-negative {and has $\supp_{C}(\alphaup)\,{\cap}\,\Phi_{C}\,{=}\,\{\alphaup_{0}\}$. Then
\begin{equation*}
 \alphaup={-}k_{\alphaup_{0}}\alphaup_{0}-{\sum}_{\betaup\in\Phi_{C}^{\vee}}k_{\betaup}\betaup,
\end{equation*}
with nonnegative integral coefficients $k_{\betaup}$ and $k_{\alphaup_{0}}{>}\,0$.
By the assumption that $C$ is $V$-fit, $\stt(\betaup)\,{\in}\,\Rad^{-}(C)$ for all nonreal
$\betaup$ in $\Phi^{\vee}_{C}$. Hence $\stt(\alphaup)\,{\in}\,\Rad^{+}(C)\,{\subseteq}\,\Qq$.
} \par
{If, vice versa, $\Qq_{\,\;(\Phi{\setminus}\{\alphaup_{0}\})_{C}}
 \,{\subseteq}\,\Qq\cup\stt(\Qq)$, then ${-}\alphaup_{0}$ belongs to $\stt(\Qq)$.}
If $\alphaup_{0}$ has $C$-positive $\stt$-conjugate,
then ${\pm}\stt(\alphaup_{0})$ both belong to $\Qq$. 
Since $\stt(\alphaup_{\alphaup_{0}})\,{\notin}\,\Rad_{\,\;\circ}^{\stt},$
we get $\alphaup_{0}\,{=}\,\stt(\stt(\alphaup_{0}))\,{\in}\,\Rad^{-}(C)$, because $C$ is $V$-fit
for $(\Qq\, ,\stt)$. This is a contradiction, showing that in fact $\stt(\alphaup_{0})\,{\in}\,\Rad^{-}(C)$.
The proof is complete.
\end{proof}
\begin{thm}\label{t4.30}
 Let 
 $\Qq\,{\in}\Pcr(\Rad)${, $C\,{\in}\,\Cd(\Rad,\Qq)$, 
$\Phi_{C}\,{=}\,\Bz(C)\,{\cap}\,\Qq^{n}$ and
$\stt$ an involution
of $\Rad$.} 
If $C$ is $V$-fit for $(\Qq\, ,\stt)$, { set 
\begin{equation*}
 \Phi_{C}^{\stt,+}=\{\alphaup\in\Phi_{C}\mid \stt(\alphaup)\in\Rad^{+}(C)\}.
\end{equation*} Then}
$\Qq_{\,\;\Phi_{C}^{\stt,+}}$ is the largest closed subset of $\Rad$ which
contains $\Qq$ and is contained in $\Qq\,{\cup}\,\stt(\Qq)$. \par
In particular, for a $V$-fit $C$ the following are equivalent:
\begin{itemize}
\item $(\Qq\, ,\stt)$ 
is 
Levi-nondegenerate;
\item 
$\stt(\Phi_{C})\,{\subseteq}\,\Rad^{+}(C)$;
\item  $\stt(\Phi_{C})\,{\subseteq}\,\Qq^{n}$.
\end{itemize}
\end{thm}
\begin{proof} The first statement follows by iterating the argument of 
Lemma\,\ref{l4.30}, after noticing that, if 
{$\alphaup\,{\in}\,\Phi_{C}\,{\cap}\,\stt(\Rad^{-}(C))$,}
then a chamber $C$ which is
$V$-fit for $(\Qq_{\;\Phi_{C}},\stt)$ is still adapted to $\Qq_{\;(\Phi_{C}\setminus\{\alphaup\})_{C}}$
and $V$-fit for $(\Qq_{\;(\Phi_{C}\setminus\{\alphaup\})_{C}},\stt)$.\par
The second is a consequence of the first and the fact that a parabolic $\Qq'$ containing $\Qq$
is of the form $\Qq'\,{=}\,\Qq_{\;\Psi_{C}}$ for a subset 
$\Psi_{C}$ of $\Phi_{C}$.
The equivalence of
the last two items follows from Lemma\,\ref{l4.21}.
\end{proof}
\begin{cor} Let {$\Qq$, $C$, $\stt$, $\Phi_{C}$, $\Phi_{C}^{\stt,+}$}
be as in the statement of  Theorem\,\ref{t4.30}. 
\par
If $\sigmaup\,{\in}\,\Invs$, 
then 
$(\gs,\qt_{\Phi_{C}^{\stt,+}})$ is the basis of the 
Levi-nondegenerate reduction of $(\gs,\qt_{\Phi_{C}})$ 
and $C$ is $V${-}fit for $(\gs,\qt_{\Phi_{C}^{\stt,+}})$. \par
\par
In particular, $(\gs,\qt_{\Phi_{C}})$ is 
Levi-nondegenerate iff $\stt(\Phi_{C})\,{\subseteq}\,\Rad^{+}(C)$.
\qed
\end{cor} 
\begin{exam}{
 The typical fibre of the  Levi-nondegenerate reduction of a parabolic $CR$-algebra may fail to be
 parabolic. }\par
 {
 Consider, e.g., in the full complex flag manifold 
 $\sfF$ of $\SL_{3}(\C)$,
 the orbit $\sfM$ of $\SU(1,2)$, consisting of 
 positive lines contained in  isotropic planes. To describe 
 its $CR$-algebra
 and its Levi-nondegenerate reduction, 
 we may consider on $\Rad\,{=}\,\{{\pm}(\e_{i}{-}\e_{j})\,{\mid}\,
 1{\leq}i{<}j{\leq}3\}$
 the conjugation $\stt$  defined by 
\begin{equation*}
 \stt(\e_{i})\,{=}\,{-}\e_{1},\;\;\stt(\e_{2})\,{=}\,{-}\e_{3},\;\; \stt(\e_{3})\,{=}\,{-}\e_{2}.
\end{equation*}}
{
Let $\qt$ be the Borel subalgebra $\bt_{C}$,
with $\Rad^{+}(C)\,{=}\,\{\e_{i}{-}\e_{j}\,{\mid}\,
1{\leq}i{<}j{\leq}3\}.$} \par
{ 
Then $\qt_{\{\sigmaup(\qt)\}}$ is the parabolic subalgebra associated to  
$ \Qq'\,{=}\,\Rad^{+}(C){\cup}\{\e_{2}{-}\e_{1}\}$
and
$
 \Qq'{\cap}\,\stt(\Qq')\,{=}\,\{\e_{2}{-}\e_{1},\,\e_{1}{-}\e_{3},\,\e_{2}{-}\e_{3}\}
$
shows that $\qt'\,{\cap}\,\sigmaup(\qt')$ is a Borel subalgebra of $\slt_{3}(\C)$. 
Then $\qt\,{\cap}\,\sigmaup(\qt')$ 
is not parabolic in $\qt'\,{\cap}\,\sigmaup(\qt')$,
because it is different from it.}
\end{exam}
\begin{exam} \label{e3.32}
Let $\Rad\,{=}\{{\pm}\e_{i},\,i=1,2,3,\, {\pm}\e_{i}{\pm}\e_{j},\, 1{\leq}i{<}j{\leq}3\}$ 
be an irreducible root system of type $\textsc{B}_{3}$.
 We consider the pair $(\Qq\,,\stt)$ described by the cross-marked  diagram 
 
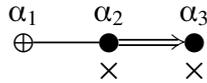
\begin{equation*} 
 \xymatrix @M=0pt @R=4pt @!C=8pt{   \alphaup_{1}& \alphaup_{2} & \alphaup_{3}
\\
 \oplus \ar@{-}[r] & \medbullet   \ar@{=>}[r] & \medbullet
  \\
&\times&\times
}
\end{equation*}
where $\alphaup_{1}{=}\e_{1}{-}\e_{2}$, $\alphaup_{2}{=}\e_{2}{-}\e_{3}$, $\alphaup_{3}\,{=}\e_{3}$
and the involution $\stt$ is given by 
\begin{equation*}
 \e_{1}\leftrightarrow\e_{1},\;\;\e_{2}\leftrightarrow{-}\e_{2},\;\; \e_{3}\leftrightarrow{-}\e_{3}.
\end{equation*}
We obtain a $V$-chamber 
 \begin{equation*} 
 \xymatrix @M=0pt @R=4pt @!C=8pt{   \alphaup'_{1}& \alphaup'_{2} & \alphaup'_{3}
\\
 \ominus \ar@{-}[r] & \oplus  \ar@{=>}[r] & \medbullet
  \\
&\times&\times
}
\end{equation*}
by choosing $\Bz(C')\,{=}\{\alphaup_{1}'{=}\e_{2}{-}\e_{1},\,
\alphaup'_{2}{=}\e_{1}{-}\e_{3},\, \alphaup'_{3}{=}\e_{3}\}$. This shows that $(\Qq\,,\stt)$,
and any corresponding $CR$-algebra, is Levi-degenerate. Its Levi-non\-de\-gen\-er\-ate reduction 
$(\Qq',\stt)$ has cross-marked diagram 
 \begin{equation*} 
 \xymatrix @M=0pt @R=4pt @!C=8pt{   \alphaup_{1}& \alphaup_{2} & \alphaup_{3}
\\
 \oplus \ar@{-}[r] & \medbullet   \ar@{=>}[r] & \medbullet
  \\
&\times
}
\end{equation*}
Since a $CR$-algebra $(\gs,\qt)$ is contact-nondegenerate if and only if its Levi-nondegenerate
reduction is contact-nondegenerate, $CR$-algebras corresponding to $(\Qq\,,\stt)$ are
Levi-degenerate, but contact-nondegenerate.
\end{exam}
\subsection{Fundamental reduction} Keep the notation of the previous
subsections. For $C\,{\in}\,\Cd(\Rad,\Qq)$ we set 
\begin{equation}
 \Bz^{\stt,+}(C)=\Bz(C)\cap\stt(\Rad^{+}(C)),\;\;\Bz^{s,-}(C)=\Bz(C)\cap\stt(\Rad^{-}(C)).
\end{equation}

\begin{lem}\label{l4.70}
 Let $\Qq\,{\in}\Pcr(\Rad)$, 
 $C\,{\in}\,\Cd(\Rad,\Qq)${, $\alphaup_{0}{\in}\,\Phi_{C}$ and $\stt$ an involution of~$\Rad$. 
 Set $\Phi_{C}^{\stt,-}{=}\,\Phi_{C}\,{\cap}\,\Bz^{\stt,-}(C)$. }
Then 
\begin{equation}\label{e4.52}
 \Qq\cup\stt(\Qq)\subseteq\Qq_{\,\;\{\alphaup_{0}\}_{C}} \;
 \Leftrightarrow\; \begin{cases}\stt(\alphaup_{0})\in\Qq^{n}\;\,\text{and}\,\;\\
  \alphaup_{0}\,{\notin}\,
 {\bigcup}_{\betaup\,{\in}\,\Phi^{\vee}_{C}\cup\,\Phi_{C}^{\stt,-}}\supp_{C}(\stt(\betaup)).
 \end{cases}
\end{equation}
 \end{lem}
 \begin{proof}
 The 
 left hand side of \eqref{e4.52} is equivalent to 
$\stt(\Qq){\subseteq}\Qq_{\,\;\{\alphaup_{0}\}_{C}}$ and hence to 
the opposite inclusion  
 $\Qq_{\,\;\{\alphaup_{0}\}_{C}}^{n}
 {\subseteq}\,\stt(\Qq^{n})$ of their horocyclic parts. \par
 Assume that $\Qq_{\,\;\{\alphaup_{0}\}_{C}}^{n}
 {\subseteq}\,\stt(\Qq^{n})$. In particular,  $\stt(\alphaup_{0})\,{\in}\,\Qq^{n}$.
Let us  prove 
that, if 
$\betaup\,{\in}\, \Bz(C)$ and  $\alphaup_{0}\,{\in}\,\supp_{C}(\stt(\betaup))$,
then $\stt(\betaup)\,{\in}\,\Rad^{+}(C)$. Indeed, 
{if} 
 $\stt(\betaup)\,{\in}\,\Rad^{-}(C)$, then
\begin{equation*}
 {-}\stt(\betaup)\,{\in}\,\Qq^{n}_{\,\;\{\alphaup_0\}_{C}}\Rightarrow
  {-}\stt(\betaup)\,{\in}\,\stt(\Qq^{n}) \Leftrightarrow {-}\betaup\,{\in}\,\Qq^{n},
\end{equation*}
{yields}
a contradiction, because $\betaup$ is $C$-positive. Finally, if $\stt(\betaup)\,{\in}\,
\Rad^{+}(C)$, then $\stt(\betaup)\,{\in}\,\Qq_{\,\;\{\alphaup_{0}\}_{C}}^{n}\,{\subseteq}\,\stt(\Qq^{n})$
shows that $\betaup\,{\in}\,\Phi_{C}^{\stt,+}{\,{=}\,\Phi_{C}\,{\cap}\,\Bz^{\stt,+}(C)}$.
\par\smallskip
To prove the vice versa, we assume by contradiction 
that there is a root $\alphaup$ in $\Qq^{n}_{\,\;\{\alphaup_{0}\}_{C}}
 {\setminus}\stt(\Qq^{n})$, with $\alphaup\,{\neq}\,\alphaup_{0}$. We consider the two cases
 that may occur.
\smallskip\par 
If 
 ${\stt(\alphaup)\,{\in}\,\Rad^{+}(C)}$, then  
\begin{equation*}
 \stt(\alphaup)={\sum}_{\betaup\,{\in}\,\Phi_{C}^{\vee}}k_{\stt(\alphaup),\betaup}\betaup \;
 \Longrightarrow \; \alphaup_{0}\,{\prec}_{\;C}\,\alphaup=
 {\sum}_{\betaup\,{\in}\,\Phi_{C}^{\vee}}k_{\stt(\alphaup),\betaup}\stt(\betaup)
\end{equation*}
 implies that $\alphaup_{0}$ belongs to the support of $\stt(\betaup)$ for some 
 $\betaup\,{\in}\,\Phi_{C}^{\vee}$.
 \smallskip\par
 If 
 ${\stt(\betaup)\,{\in}\,\Rad^{-}(C)}$, then  
\begin{gather*}
 \stt(\alphaup)={\sum}_{\betaup\,{\in}\,\Bz(C)}k_{\stt(\alphaup),\betaup}\betaup,\;\;\text{with\;
 $k_{\stt(\alphaup),\betaup}\,{\leq}\,0$\; for all $\betaup\,{\in}\,\Bz(C)$\; and}
\\
 \alphaup_{0}\prec_{\;C}\alphaup =
{\sum}_{\betaup\,{\in}\,\Bz(C)}k_{\stt(\alphaup),\betaup}\stt(\betaup),\;\;\text{with\;
 $k_{\stt(\alphaup),\betaup}\,{\leq}\,0$\; for all $\betaup\,{\in}\,\Bz(C)$}
\end{gather*}
 implies that $\alphaup_{0}$ belongs to the support of $\stt(\betaup)$ for 
 some  $\betaup$ having a $C$-neg\-a\-tive $\stt$-conjugate, hence belonging to 
 $\Phi_{C}^{\vee}\,{\cup}\,\Phi_{C}^{\stt,-}$. 
 The proof is complete.
\end{proof}
\begin{thm}\label{t2.31}
Let $\Qq\,{\in}\Pcr(\Rad)$, 
 $C\,{\in}\,\Cd(\Rad,\Qq)$ and $\stt$ an involution of
 $\Rad$. With $\Phi_{C}^{\stt,-}{=}\,\Phi_{C}\,{\cap}\,\Bz^{\stt,-}(C)$,
 $\Phi_{C}^{\stt,+}{=}\,\Phi_{C}\,{\cap}\,\Bz^{\stt,+}(C)$, set 
\begin{equation}
 \Psi_{C}=\left\{\alphaup\,{\in}\,\Phi_{C}^{\stt,+}\left| \stt(\alphaup)\,{\in}\,\Qq^{n},\; 
 \alphaup\notin{\bigcup}_{\betaup\,{\in}\,\Phi_{C}^{\vee}\cup\,\Phi_{C}^{\stt,-}}\supp_{C}(\stt(\betaup))\right\}\right.
\end{equation}
Then $(\Qq_{\,\;\Psi_{C}},\stt)$ is the totally real basis of the fundamental reduction of $(\Qq\, ,\stt)$. \par
The pair $(\Qq\, ,\stt)$ is fundamental if and only if $\Psi_{C}\,{=}\,\emptyset$.
\end{thm} 
\begin{proof}
 The fundamental reduction of $(\Qq\, ,\stt)$ is of the form $(\Qq_{\,\;\Psi_{C}},\stt)$ for some
 $\Psi_{C}\,{\subseteq}\,\Phi_{C}$. This $\Qq_{\,\;\Psi_{C}}$ is the intersection of all
 $\Qq_{\,\;\{\alphaup\}_{C}}$ containing $\Qq\,{\cup}\,\stt(\Qq)$. Hence the statement is
 a straightforward consequence of Lemma\,\ref{l4.70}.
\end{proof}
%
\begin{exam} 
Consider  the involution  $\stt$ 
on $\Rad(\textsc{A}_{3})$
defined by 
\begin{equation*}
 \e_{1}\,{\leftrightarrow}{-}\e_{3},\; \e_{2}\,{\leftrightarrow}\,{-}\e_{2},\;
\e_{4}\,{\leftrightarrow}\,{-}\e_{4},
\end{equation*}
yielding, on the canonical basis $\Bz(C)\,{=}\,\{\alphaup_{i}{=}\e_{i}{-}\e_{i+1}\,{\mid}\,1=1,2,3\}$,
\begin{equation*}
 \stt(\alphaup_{1})\,{=}\,\alphaup_{2},\;\;\stt(\alphaup_{2})\,{=}\,\alphaup_{1},\; 
 \stt(\alphaup_{3})\,{=}\,{-}(\alphaup_{1}{+}\alphaup_{2}{+}\alphaup_{3}),
\end{equation*}
and the parabolic 
$\Qq\,{=}\,\{\alphaup\,{\in}\,\Rad\,{\mid}\,(\alphaup|2\e_{1}{+}\e_{2})\,{\geq}\,0\}$. \par
The Weyl chamber 
$C$ with $\Bz(C)\,{=}\,\{\alphaup_{i}{=}\e_{i}{-}\e_{i+1}\,{\mid}\,1{\leq}i{\leq}3\}$
 is  $V$-fit for $(\Qq\, ,\stt)$, yielding the  
cross-marked diagram 
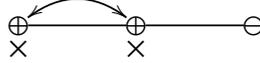
\vspace{10pt}
\begin{equation*}
   \xymatrix @M=0pt @R=2pt @!C=20pt{
\oplus \ar@/^10pt/@{<->}[r] \ar@{-}[r]&\oplus \ar@{-}[r]& \ominus\\
\times & \times
}
\end{equation*}
with $\Phi_{C}\,{=}\,\Phi_{C}^{\stt,+}\,{=}\,\{\alphaup_{1},\alphaup_{2}\}. $ 
This equality tells us that $(\Qq\, ,\stt)$ is 
Levi-non\-de\-gen\-er\-ate. 
Since $\alphaup_{3}\,{\in}\,\Phi_{C}^{\vee}$ and 
$\alphaup_{1},\alphaup_{2}$ belong to the support of $\stt(\alphaup_{3})$,
the pair $(\Qq\, ,\stt)$ is also fundamental.  
\end{exam}

\begin{rmk}
Since a $CR$-algebra and its 
Levi-nondegenerate reduction
have the same fundamental reduction, 
while characterising fundamental reduction 
we may restrain  
to Levi-nondegenerate
parabolic $CR$-algebras. \end{rmk}
\begin{prop}\label{p2.33}
{Let $\Qq\,{\in}\Pcr(\Rad)$,  
 $C\,{\in}\,\Cd(\Rad,\Qq)$ and $\stt$ an involution of
 $\Rad$.} 
 We assume that $(\Qq\, ,\stt)$ is  Levi
 nondegenerate
 and 
  $C $ is $V$-fit  for $(\Qq\, ,\stt)$. 
Then the fundamental reduction $(\Qq_{\,\;\Psi_{C}},\stt)$ of $(\Qq\, ,\stt)$ has 
\begin{equation}\label{e4.50}
 \Psi_{C}=\left\{\alphaup\,{\in}\,
 \Phi_{C}
 \left| \stt(\alphaup)\,{\in}\,\Qq^{n},\; 
 \alphaup\notin{\bigcup}_{\betaup\,{\in}\,\Phi_{C}^{\vee}}\supp_{C}(\stt(\betaup))\right\}\right. .
\end{equation}
\end{prop} 
\begin{proof} This follows from Theorem\,\ref{t2.31}
because $\Phi_{C}^{\stt,-}\,{=}\,\Phi_{C}\,{\cap}\,\Bz^{\stt,-}(C)\,{=}\,\emptyset$.
\end{proof}
\begin{cor} \label{c2.34}
Let 
$\Qq\,,\,C,\,\stt,\,\Psi_{C}$
be as in the statement of Theorem\,\ref{t2.31},
 $\sigmaup\,{\in}\,\Invs$.
Then 
the parabolic $CR$-algebra $(\gs,\qt_{\Psi_{C}})$ is the basis of the 
fundamental reduction of $(\gs,\qt_{\Phi_{C}})$.\par
The last one is fundamental if and only if $\Psi_{C}\,{=}\,\emptyset$.\qed
\end{cor} 
\begin{prop}
 The $CR$-algebra of the typical fibre of the fundamental reduction of a parabolic $CR$-algebra
 is isomorphic to a parabolic $CR$-algebra. 
\end{prop} 
\begin{proof}
 We keep the notation of Corollary\,\ref{c2.34}. If $\qt_{\Psi_{C}}\,{=}\,\gt$, the basis is the  
 trivial $CR$-algebra, the typical fibre is $(\gs,\qt)$ 
 and we have nothing to prove. Assume that  $\qt_{\Psi_{C}}\,{\subsetneqq}\,\gt$.
 Then $\gt_{\Psi_{C}}$ is the complexification of a real parabolic subalgebra $\gs'$ of $\gs$.
Its radical $\rt'$ is the complexification of the radical $\rt'_{\sigmaup}$ of $\gs'$ and 
is contained in the radical of $\qt$. 
The typical
fibre of the fundamental reduction 
$\id^{\sharp}\,{:}\,(\gs,\qt)\,{\to}\,(\gs,\qt_{\Psi_{C}})$ has $CR$-algebra
$(\gs',\qt)\simeq((\gs'/\rt'_{\sigmaup}),(\qt/{\rt'}))$, the last one being a parabolic $CR$-algebra. 
\end{proof}
\begin{rmk}\label{r2.37}
A cross marked diagram of $((\gs'/\rt'_{\sigmaup}),(\qt/{\rt'}))$
can be obtained from that of $(\gs,\qt)$ by erasing the nodes corresponding to roots in $\Psi_{C}$
and the issuing lines and the remaining connected parts of the diagram not containing crossed nodes. 
\end{rmk}
\begin{exam} An example of the procedure outlined above in Remark\,\ref{r2.37} is represented by: 
 \begin{gather*} 
  \xymatrix @M=0pt @R=2pt @!C=8pt{ \medcirc \ar@{-}[r] \ar@/^7pt/@{<->}[r] 
&\medcirc
\\
\times}
\\
  \xymatrix @M=0pt @R=2pt @!C=4pt{   \medcirc \ar@{-}[r] \ar@/^18pt/@{<->}[rrrrr] &
  \medcirc \ar@{-}[r] \ar@/^14pt/@{<->}[rrr] & \medcirc \ar@{-}[r]
\ar@/^10pt/@{<->}[r] &\medcirc
 \ar@{-}[r] &\medcirc\ar@{-}[r] &\medcirc\\
&\times&\times && \times} 
\quad\xrightarrow{\quad\qquad}\quad
  \xymatrix @M=0pt @R=2pt @!C=4pt{   \medcirc \ar@{-}[r] \ar@/^18pt/@{<->}[rrrrr] &
  \medcirc \ar@{-}[r] \ar@/^14pt/@{<->}[rrr] & \medcirc \ar@{-}[r]
\ar@/^10pt/@{<->}[r] &\medcirc
 \ar@{-}[r] &\medcirc\ar@{-}[r] &\medcirc\\
&\times& &&\times} 
\end{gather*}
\end{exam}
\begin{rmk}  Theorem\,\ref{t4.30} and Proposition\,\ref{p2.33} yield a recipe for manufacturing examples of 
Levi-nondegenerate  fundamental 
$CR$-algebras: having fixed any Weyl chamber $C$ 
and {an involution $\stt$ of $\Rad$,} 
we can take any $\Phi_{C}$ satisfying 
\begin{equation*}
 \{\betaup\,{\in}\,\Bz_{\star}^{\stt}(C)\,{\mid}\,\stt(\betaup)\,{\in}\,\Rad^{+}(C)\}
 \subseteq \Phi_{C} \subseteq \{\betaup\,{\in}\,\Bz(C)\,{\mid}\,\stt(\betaup)\,{\in}\,\Rad^{+}(C)\}.
\end{equation*}
\par
 If $\sigmaup\,{\in}\,\Invs$, then
$C$ is a $V$-fit chamber for $(\gs,\qt_{\Phi_{C}})$, which
is  Levi-nondegenerate
 (we recall that, by Definition\,\ref{d2.2},
totally real $CR$-al\-ge\-bras
are  finitely 
Levi-nondegenerate).
Fundamental examples are next obtained by the procedure explained in Remark\,\ref{r2.37}.
\end{rmk}
\begin{exam}
 Consider in an irreducible root system $\Rad$
of type $\textsc{B}_{4}$ the 
Borel subalgebra $\Qq\,{=}\,\Rad^{+}(C)$ for the
Weyl chamber $C$ with 
\begin{equation*}
 \Bz(C)=\{\alphaup_{i}\,{=}\,\e_{i}{-}\e_{i+1}\mid i=1,2,3\}\cup\{\alphaup_{4}\,{=}\,\e_{4}\}
\end{equation*}
and the involution $\stt$ defined by 
\begin{equation*}
 \stt(\e_{i})= 
\begin{cases}
 \e_{i}, & i=1,2,\\
 {-}\e_{4}, & i=3,\\
 {-}\e_{3}, & i=4.
\end{cases}
\end{equation*}
The corresponding diagram is:
 
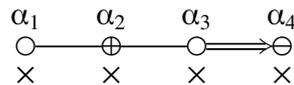
\begin{equation*} 
 \xymatrix @M=0pt @R=4pt @!C=8pt{   \alphaup_{1}& \alphaup_{2} & \alphaup_{3}
 & \alphaup_{4}\\
 \medcirc \ar@{-}[r] & \oplus \ar@{-}[r] & \medcirc
  \ar@{=>}[r] & \ominus
  \\
\times&\times&\times&\times
}
\end{equation*}
The basis of the 
Levi-nondegenerate reduction 
is obtained by  dropping from $\Phi_{C}\,{=}\,\Bz(C)$ the root $\alphaup_{4}$
with $C$-negative $\stt$-conjugate and has diagram 
 
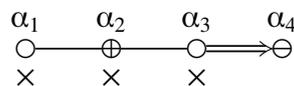
\begin{equation*} 
 \xymatrix @M=0pt @R=4pt @!C=8pt{   \alphaup_{1}& \alphaup_{2} & \alphaup_{3}
 & \alphaup_{4}\\
 \medcirc \ar@{-}[r] & \oplus \ar@{-}[r] & \medcirc
  \ar@{=>}[r] & \ominus\\
\times&\times&\times&}
\end{equation*}
We have $\Phi^{\vee}_{C}\,{=}\,\{\alphaup_{4}\}$ and 
$\stt(\alphaup_{4})\,{=}\,{-}\e_{3}\,{=}\,{-}(\alphaup_{3}{+}\alphaup_{4})$. Hence
$\Psi_{C}\,{=}\,\{\alphaup_{1},\alphaup_{2}\}$
and the basis of the fundamental reduction has diagram 
 \begin{equation*} 
 \xymatrix @M=0pt @R=4pt @!C=8pt{   \alphaup_{1}& \alphaup_{2} & \alphaup_{3}
 & \alphaup_{4}\\
 \medcirc \ar@{-}[r] & \oplus \ar@{-}[r] & \medcirc
  \ar@{=>}[r] & \ominus\\
\times&\times&&}
\end{equation*}
Taking any $\sigmaup\,{\in}\,\Invs$, we obtain the sequence of reductions 
\begin{equation*}
 (\gs,\qt_{\{\alphaup_{1},\alphaup_{2},\alphaup_{3},\alphaup_{4}\}_{C}})
 \xrightarrow{\text{Levi-nondegenerate}}
  (\gs,\qt_{\{\alphaup_{1},\alphaup_{2},\alphaup_{3}\}_{C}})
 \xrightarrow{\text{fundamental}}
  (\gs,\qt_{\{\alphaup_{1},\alphaup_{2}\}_{C}}).
\end{equation*}
The typical fibre of the fundamental reduction has diagrams 
\begin{equation*}
\xymatrix @M=0pt @R=4pt @!C=8pt{    \alphaup_{3}
 & \alphaup_{4}\\
 \medcirc
  \ar@{=>}[r] & \ominus\\
\times} \qquad
\begin{matrix}
 \quad\\
 \quad\\
 \text{or}\\
 \quad
\end{matrix}\qquad\quad
\xymatrix @M=0pt @R=4pt @!C=8pt{    \alphaup'_{3}
 & \alphaup'_{4}\\
 \medbullet
  \ar@{=>}[r] & \oplus\\
\times}
\end{equation*}
where in the second one $\alphaup'_{3}{=}\e_{3}{+}\e_{4}$, $\alphaup'_{4}{=}{-}\e_{4}$.
\end{exam}
\subsection{Depth of fundamental  parabolic $CR$-algebras} 
As we noted before, 
contact depth of $CR$ manifolds is important to investigate 
various properties of their $CR$ functions
(cf. \cite{AHNP08a, NP2015}). In the homogeneous case, it coincides with the depth
of associated $CR$ algebras.
\par
In this subsection we study contact depth of homogeneous $CR$ manifolds
associated to parabolic $CR$ algebras. \par
If $\Gf$ is a semisimple complex Lie group
with Lie algebra $\gt$, $\Qf$ its parabolic subgroup with Lie algebra $\qt$ and
$\Gf_{\sigmaup}$ a real form of $\Qf$, then the contact depth of the orbit $\sfM$ of
$\Gf_{\sigmaup}$ through the base point of $\Gf{\! /}\Qf$  equals 
the \textit{$H$-index},  
introduced in  
Definition\,\ref{d4.9} below, of 
an associated pair 
$(\Qq\,,\stt)$.
\subsubsection{Depth caracterisation}
{Keep the notation of the previous subsections}. 

By applying  Lemma\,\ref{l2.1} to elements  
of a Chevalley system 
$(X_{\alphaup},H_{\alphaup})_{\alphaup\in\Rad}$, we get 
\begin{prop} 
 The parabolic $CR$-algebra $(\gs,\qt)$ is fundamental iff for every root $\alphaup$ 
\begin{equation}\label{e4.21}
\exists\,\betaup_{1},\hdots,\betaup_{r}\,{\in}\,\Qq\cup\stt(\Qq)\;\;
 \text{such that}\;\; \alphaup_{h}={\sum}_{i=1}^{h}\betaup_{i}\in\Rad\;\;\text{and}\;\;\alphaup_{r}=\alpha.
\end{equation}
\end{prop} 
\begin{proof} Both $\qt\,{+}\,\sigmaup(\qt)$ and its complement
$\qt^{c}\,{\cap}\,\sigmaup(\qt^{c})$ are $\hg$-modules and hence decompose into direct sums
of subspaces of $\hg$ and  
root spaces. 
Thus, by Lemma\,\ref{l2.1},  
$(\gs,\qt)$ is fundamental if and only if for each root $\alphaup$ we can find 
$\betaup_{1},\hdots,\betaup_{r}\,{\in}\,\Qq\,{\cup}\,\stt(\Qq)$ such that 
$\betaup_{1}{+}\cdots{+}\betaup_{r}\,{=}\,\alphaup$ and 
$[X_{\betaup_{1}},\hdots,X_{\betaup_{r}}]\,{\neq}\,0$. 
These conditions are
equivalent to 
\eqref{e4.21}. 
\end{proof} 
\begin{dfn}\label{d4.9}
We call a string 
$(\betaup_{1},\hdots,\betaup_{r})$ satisfying \eqref{e4.21}
 an \emph{$H$-sequence} for $\alphaup$ with respect to $(\Qq\, ,\stt)$.
The minimal length of  an $H$-sequence for $\alphaup$ is called its \emph{$H$-index} and
will be denoted by $\nuup(\alphaup)$.
The supremum of $H$-indices of roots of $\Rad$ is
the
 $H$-\emph{index} $\nuup\,{=}\,\nuup(\Qq\,,\stt)$ 
 of~$(\Qq\, ,\stt)$. \par
To simplify notation,  we will drop explicit reference to $(\Qq\,,\stt)$ when this would not cause
 confusion.
\end{dfn}
Roots in $\Qq\,{\cup}\,\stt(\Qq)$ have $H$-index $1$.
We set $\nuup(\alphaup)\,{=}\,\infty$ for a root $\alphaup$ which is not a sum of roots
in $\Qq\,{\cup}\,\stt(\Qq)$. 

\begin{prop}
 The depth of the parabolic $CR$-algebra $(\gs,\qt)$ equals the $H$-index of  
 its associated pair $(\Qq\, ,\stt)$. \qed
\end{prop}  
Lemma\,\ref{l3.2a}, applied to a Chevalley  system for $\Rad$, yields:
\begin{cor}\label{c2.46}
  If 
 $\betaup_{1},\hdots,\betaup_{r}$ are roots such that  ${\sum}_{i=1}^{j}\betaup_{i}\,{\in}\,\Rad,$\, 
 for $1{\leq}j{\leq}r$, 
 then we can find $h,$ with $1{\leq}h{<}r$, such that $\betaup_{h}{+}\betaup_{r}$
 is a root and, with $\betaup'_{i}{=}\betaup_{i}$ for $i{\neq}h$, $\betaup'_{h}{=}\betaup_{h}{+}\betaup_{r},$
we have ${\sum}_{i=1}^{p}\betaup'_{i}\,{\in}\,\Rad$ for $1{\leq}p{\leq}r{-}1$.
  \qed
\end{cor} 

\begin{lem}\label{l2.46} \textsc{(1)}
 Let $\alphaup\,{\in}\,\Qc\,{\cap}\,\stt(\Qc)$ and $(\betaup_{1},\hdots,\betaup_{r})$
 an $H$-sequence of minimal length for $\alphaup$. Then
 $\betaup_{i}\,{\in}\,(\Qq\,{\cap}\,\stt(\Qc))\,{\cup}\,(\Qc\,{\cap}\,\stt(\Qq))$ for  $i\,{=}\,1,\hdots,r$.
\par \textsc{(2)}
 If $\betaup\,{\in}\,\Qq\,{\cap}\,\stt(\Qq)$ and $\alphaup\,{+}\,\betaup$ is a root, then
 $\nuup(\alphaup\,{+}\,\betaup)\,{\leq}\,\nuup(\alphaup)$. 
\end{lem} 
\begin{proof} \textsc{(1)}
Since $\alphaup\,{\in}\,\Qc\,{\cap}\,\stt(\Qc)$, its $H$-index $r\,{=}\,\nuup(\alphaup)$
is at least two. Assume by contradiction that there is a smallest index $i$ for which 
$\betaup_{i}\,{\in}\,\Qq\,{\cap}\,\stt(\Qq)$. If $i{=}1,$ 
then $\betaup_{1}{+}\betaup_{2}\,{\in}\,\Qq\,{\cup}\,\stt(\Qq)$
and $(\betaup_{1}{+}\betaup_{2},\betaup_{3},\hdots,\betaup_{r})$ is an $H$-sequence 
for $\alphaup$ of length $r{-}1$, contradicting minimality.
If $i{>}1$, by Corollary\,\ref{c2.46} we can find $j,$ with $1{\leq}j{<}i$, such that 
$\betaup_{j}{+}\betaup_{i}$ is still a root and, since $\betaup_{j}{+}\betaup_{i}\,{\in}\,\Qq\,{\cup}\,\stt(\Qq)$,
the sequence $(\betaup'_{1},\hdots,\betaup'_{r-1})$ with
\begin{equation*}
 \betaup'_{h}= 
\begin{cases}
 \betaup_{h}, &\text{if $h{\neq}j$ and $h{<}i$},\\
 \betaup_{j}{+}\betaup_{i}, &\text{if $h=j$,}\\
 \betaup_{h+1}, & \text{if $i{\leq}h{\leq}r-1$,}
\end{cases}
\end{equation*}
is an $H$-sequence of length $r{-}1$ for $\alphaup$, yielding again a contradiction. 
\par \noindent
\textsc{(2)}
Let $(\betaup_{1},\hdots,\betaup_{r})$ be a minimal $H$-sequence for $\alphaup$ and
$\betaup\,{\in}\,\Rad$ be such that $\alphaup\,{+}\,\betaup\,{\in}\,\Rad$. 
Then there is at least one index $j$ for which $\betaup_{j}{+}\betaup$ is a root. If
$\betaup\,{\in}\,\Qq\,{\cap}\,\stt(\Qq)$, then also
$\betaup_{j}{+}\betaup$ belongs to $\Qq\,{\cup}\,\stt(\Qq)$ and 
the $r$-tuple $(\betaup'_{1},\hdots,\betaup'_{r})$ with $\betaup'_{i}\,{=}\,\betaup_{i}$ for $i\,{\neq}\,j$
and $\betaup'_{j}{=}\,\betaup_{j}{+}\,\betaup$ is an $H$-sequence for $(\alphaup\,{+}\,\betaup)$. 
\end{proof}
\begin{prop} \label{p2.42}
Let $\alphaup\,{\in}\,\Qc\,{\cap}\,\stt(\Qc)$ be a root of finite $H$-index  
and let $(\betaup_{1},\hdots,\betaup_{r})$ be a minimal-length sequence of roots 
with 
\begin{equation}\label{e4.53}
 \betaup_{1}+\cdots+\betaup_{r}=\alphaup,\quad \betaup_{i}\,{\in}\,(\Qq\cap\stt(\Qc))\cup(\Qc\cap
 \stt(\Qq)).
\end{equation}
Then we can find a permutation $(i_{1},\hdots,i_{r})$ of $1,\hdots,r$
such that $(\betaup_{i_{1}},\hdots,\betaup_{i_{r}})$ is an $H$-sequence for $\alphaup$.
\end{prop} 
\begin{proof}
 The statement is trivial when $r\,{=}\,2$. Assume that $r\,{>}\,2$ and that the statement is true
 for roots of $\Qc\,{\cap}\,\stt(\Qc)$ which can be expressed as sums of less than
 $r$ roots of $(\Qq\,{\cap}\,\stt(\Qc))\,{\cup}\,(\Qc\,{\cap}\,
 \stt(\Qq))$. Since 
\begin{equation*}
 0<\|\alphaup\|^{2}={\sum}_{i=1}^{r}(\alphaup|\betaup_{i}),
\end{equation*}
there is at least an index $i$, with $1{\leq}i{\leq}r$, for which $(\alphaup|\betaup_{i})\,{>}\,0$.
We can assume this happens for $i\,{=}\,r$. Then $(\alphaup{-}\betaup_{r})$
is a root in $\Qc\,{\cap}\,\stt(\Qc)$, because the $H$-index of $\alphaup$ is larger than $2$ and,
by the recursive assumption, we can find a permutation $(i_{1},\hdots,i_{r-1})$ of $1,\hdots,r{-}1$
such that $(\betaup_{i_{1}},\hdots,\betaup_{i_{r-1}})$ is an $H$-sequence for $(\alphaup{-}\betaup_{r})$
and hence $(\betaup_{i_{1}},\hdots,\betaup_{i_{r-1}},\betaup_{r})$ is an $H$-sequence for~$\alphaup$.
\end{proof}

\begin{prop}\label{p2.44}
 A parabolic $CR$-algebra $(\gs,\qt)$ and its 
 Levi-nondegerate reduction $(\gs,\qt')$
 have the same  depth. 
\end{prop} 
\begin{proof}
Since $\qt'{+}\,\sigmaup(\qt')\,{=}\,\qt\,{+}\,\sigmaup(\qt)$, we have $\Ft_{ -h}(\gs,\qt')
 \,{=}\,\Ft_{{-}h}(\gs,\qt)$
 for all positive integers $h$ (we used the notation of p.\pageref{contct}).
 This implies the statement. 
\end{proof}
\begin{lem} \label{l2.53}
Let $\Rad$ be an irreducible root system,
$C$ a Weyl chamber and $\deltaup_C$ the largest root in $\Rad^{+}(C)$. If 
$\alphaup\,{\in}\,\Rad^{+}(C){\setminus}\{\deltaup_C\}$, then 
$(\deltaup_C{-}\,\alphaup)$ is either a $C$-positive root or the sum of two 
$C$-positive roots. \par
If $\deltaup_C{-}\,\alphaup$ is not a 
root
 and 
\begin{equation*}
 \deltaup_C=\alphaup+\betaup_{1}+\betaup_{2},\;\;\text{with}\;\; \betaup_{1},\betaup_{2}\in\Rad^{+}(C)\,\;
 \text{and}\; \betaup_{1}{+}\betaup_{2}\notin\Rad,
\end{equation*}
then both $\alphaup{+}\betaup_{1}$ and $\alphaup{+}\betaup_{2}$ are roots and $\betaup_{1},\betaup_{2}$
have the following 
properties\footnote{{By $\betaup_{1}\,{\Perp}\,\betaup_{2}$ we mean that
the two roots are \textit{strongly orthogonal}, i.e. that $\betaup_{1}{\neq}\,{\pm}\,\betaup_{2}$
and that neither $\betaup_{1}{+}\,\betaup_{2}$, nor $\betaup_{1}{-}\,\betaup_{2}$ ia a  root.}}:
\begin{itemize}
 \item If $\Rad$ is of type $\textsc{A}_{n-1}$, $\textsc{D}_{n}$, $\textsc{E}_{6}$, $\textsc{E}_{7}$,
$ \textsc{E}_{8}$, then $\betaup_{1}\Perp\betaup_{2}$.
\item If $\Rad$ is of type $\textsc{B}_{n}$, then: \begin{enumerate}
\item if $\alphaup$ is short, then $|\betaup_{1}|\,{\neq}\,
|\betaup_{2}|$ and $\betaup_{1}\Perp\betaup_{2}$,
\item if $\alphaup$ is long, then either $\betaup_{1}$ and $\betaup_{2}$ equal a same
short root, or  $\betaup_{1},\, \betaup_{2}$ are strongly orthogonal long roots.
\end{enumerate}
\item If $\Rad$ is of type $\textsc{C}_{n}$, then 
\begin{enumerate}
 \item if $\alphaup$ is long, then $\betaup_{1}$ and $\betaup_{2}$ equal a same
short root,
\item if $\alphaup$ is short, then $\betaup_{1},\, \betaup_{2}$ are long and $\betaup_{1}{-}\betaup_{2}
\,{\in}\,\Rad.$
\end{enumerate}
\item If $\Rad$ is of type $\textsc{F}_{4}$, then 
\begin{enumerate}
 \item if $\alphaup$ is long, then either $\betaup_{1},\, \betaup_{2}$ equal a same short root, or
 $|\betaup_{1}|\,{\neq}\,|\betaup_{2}|$ and $\betaup_{1}\Perp\betaup_{2}$,
 \item if $\alphaup$ is short, either $|\betaup_{1}|\,{\neq}\,|\betaup_{2}|$ and $\betaup_{1}\Perp\betaup_{2}$,
 or $\betaup_{1},\,\betaup_{2}$ are short and $\betaup_{1}{-}\betaup_{2}$ is a root.
\end{enumerate}
\item If $\Rad$ is of type $\textsc{G}_{2}$, then $\alphaup$ is short and $\betaup_{1},\,\betaup_{2}$
equal a same short root. 
 \end{itemize}
\end{lem} 
\begin{proof} In the proof we can fix any  Weyl chamber
and  restrain to the cases where $\alphaup\,{\Perp}\,\deltaup_C$. 
We consider below the different cases.
\par
\noindent $\textsc{A}_{n-1}$. \; With
 $\Rad^{+}(C)\,{=}\,\{\e_{i}{-}\e_{j}\,{\mid}\,1{\leq}i{<}j{\leq}n\}$
and $\deltaup_C\,{=}\,\e_{1}{-}\e_{n}$,
we have
\begin{equation*}
 (\e_{1}{-}\e_{n})-(\e_{i}{-}\e_{j})=
(\e_{1}{-}\e_{i})+(\e_{j}{-}\e_{n}),\\
 \quad
  \forall\, 2{\leq}i{<}j{<}n.
\end{equation*} 
\par\smallskip
\noindent $\textsc{B}_{n}$. \;
With  $\Rad^{+}(C)\,{=}\,\{\e_{i}\,{\mid}\,1{\leq}i{\leq}n\}\,{\cup}\,\{
\e_{i}{\pm}\e_{j}\,{\mid}\,1{\leq}i{<}j{\leq}n\}$ and $\deltaup_C\,{=}\,\e_{1}{+}\e_{2}$, we get
\begin{equation*}\begin{cases} (\e_{1}{+}\e_{2})-\e_{i}= 
\begin{cases}
 (\e_{1}{-}\e_{i})+\e_{2},\\
 \e_{1}+(\e_{2}-\e_{i}),
\end{cases}
& 3{\leq}i{\leq}n,\\[3pt]
(\e_{1}{+}\e_{2})-(\e_{1}{-}\e_{2})=2\cdot\e_{2},\\[3pt]
(\e_{1}{+}\e_{2})-(\e_{i}{\pm}\e_{j})= 
\begin{cases}
 (\e_{1}{-}\e_{i})+(\e_{2}{\mp}\e_{j}),\\
 (\e_{1}{\mp}\e_{j})+(\e_{2}{-}\e_{i}),
\end{cases}& 3{\leq}i{<}j{\leq}n.
 \end{cases}
\end{equation*}
\par\smallskip
\noindent $\textsc{C}_{n}$. \; With $\Rad^{+}(C)\,{=}\,\{2\e_{i}\,{\mid}\,1{\leq}i{\leq}n\}\,{\cup}\,\{
\e_{i}{\pm}\e_{j}\,{\mid}\,1{\leq}i{<}j{\leq}n\}$ and $\deltaup_C\,{=}\,2\e_{1}$, we get
\begin{equation*}\begin{cases} (2\e_{1})-(2\e_{i})=2\cdot(\e_{1}{-}\e_{i}), 
 & 2{\leq}i{\leq}n,\\ 
 2\e_{1}{-}(\e_{i}{\pm}\e_{j})= 
 (\e_{1}{-}\e_{i})+(\e_{1}{\mp}\e_{j}),
& 2{\leq}i{<}j{\leq}n.
 \end{cases}
\end{equation*}
\par\smallskip
\noindent $\textsc{D}_{n}$. \;\; With $\Rad^{+}(C)\,{=}\,\{
\e_{i}{\pm}\e_{j}\,{\mid}\,1{\leq}i{<}j{\leq}n\}$ and $\deltaup_C\,{=}\,\e_{1}{+}\e_{2}$, we get 
\begin{equation*} \begin{cases}
(\e_{1}{+}\e_{2})-(\e_{1}{-}\e_{2})=(\e_{2}{-}\e_{i})+(\e_{2}{+}\e_{i}), \qquad\;\; 3{\leq}i{\leq}n,\\[3pt]
 (\e_{1}{+}\e_{2})-(\e_{i}{\pm}\e_{j})= \begin{cases}
 (\e_{1}{-}\e_{i})+(\e_{2}{\mp}\e_{j}), 
 \\
 (\e_{1}{\mp}\e_{j})+(\e_{2}{-}\e_{i}),  \end{cases}
\quad 3{\leq}i{<}j{\leq}n.
\end{cases}
\end{equation*}
\par\smallskip
In \cite{MMN23},
for describing root systems of type $\textsc{E}$, we found convenient to introduce 
the notation ($\e_{1},\hdots,\e_{8}$ is the canonical basis of $\R^{8}$)
\begin{equation*}
 \zetaup_{\emptyset}={-}\tfrac{1}{2}{\sum}_{i=1}^{8}\e_{i},\quad \zetaup_{i_{1},\hdots,i_{h}}=\zetaup_{\emptyset}+
 {\sum}_{p=1}^{h}\e_{i_{p}},\;\text{for $1{\leq}i_{1},\hdots,i_{h}\leq{8}$}.
\end{equation*}
\par\medskip
\noindent $\textsc{E}_{6}$. \;\; We use, for the root system of type $\textsc{E}_{6}$,
 the description 
\begin{equation*}\Rad(\textsc{E}_{6})\,{=}\{{\pm}(\e_{7}{-}\e_{8})\}\cup
\{{\pm}(\e_{i}{-}\e_{j})\,{\mid}\, 
 1{\leq}i{<}{j}{\leq}6
\}\cup\{{\pm}\zetaup_{i,j,h,7}\,{\mid}\, 
 1{\leq}i{<}j{<}h{\leq}6
\}.
\end{equation*}
Taking the basis of simple roots 
\begin{equation*}
 \Bz(C)\,{=}\,\{\alphaup_{i}\,{=}\,\e_{i}{-}\e_{i+1}\,{\mid}\,1{\leq}i{\leq}5\}
 \cup\{\alphaup_{6}{=}\zetaup_{4,5,6,7}\},
\end{equation*}
we obtain 
\begin{equation*}
 \Rad^{+}(C)=\{\e_{7}{-}\e_{8}\}\,{\cup}\,
 \{\e_{i}{-}\e_{j}\mid 1{\leq}i{<}j{\leq}6\}\cup\{\zetaup_{i,j,h,7}\mid 1{\leq}i{<}j{<}h{\leq}6\},\;
 \deltaup_C{=}\e_{7}{-}\e_{8}.
\end{equation*}
The positive roots of the form $\zetaup_{i,j,h,7},$ with
$1{\leq}i{<}j{<}k{\leq}6$, 
are not orthogonal to 
$\deltaup_C$ and we obtain, for the remaining positive roots,   
\begin{equation*} 
 (\e_{7}{-}\e_{8})-(\e_{i}{-}\e_{j})= \zetaup_{j,h,k,7}+\zetaup_{j,p,q,7}, \;\;\text{with $\{i,j,h,k,p,q\}\,{=}\,
 \{1,2,3,4,5,6\}$}.
\end{equation*}
\par\smallskip
\noindent $\textsc{E}_{7}$. \;\; We use, for the root system of type $\textsc{E}_{7}$,
 the description  
\begin{equation*}
 \Rad(\textsc{E}_{7})\,{=}\,\{{\pm}(\e_{i}{-}\e_{j})\,{\mid}\, 1{\leq}i{<}j{\leq}8\}\cup\{\zetaup_{i,j,h,k}\,{\mid}\,
 1{\leq}i{<}j{<}h{<}k{\leq}8\}.
\end{equation*}
Taking the basis of simple roots 
\begin{equation*}
 \Bz(C)\,{=}\,\{\alphaup_{i}\,{=}\,\e_{i}{-}\e_{i+1}\,{\mid}\,1{\leq}i{\leq}6\}
 \cup\{\alphaup_{7}{=}\zetaup_{4,5,6,7}\},
\end{equation*}
we obtain 
\begin{equation*}
 \Rad^{+}(C)=\{\e_{i}{-}\e_{j}\mid 1{\leq}i{<}j{\leq}8\}\cup\{\zetaup_{i,j,h,k}\mid 1{\leq}i{<}j{<}h{<}k{\leq}7\},
 \; \deltaup_C\,{=}\,\e_{1}{-}\e_{8}.
\end{equation*}
Noting that that $\{\zetaup_{i,j,h,k}\,{\mid}\, 1{\leq}i{<}j{<}h{<}k{\leq}7\}\,{=}\,
\{{-}\zetaup_{i,j,h,8}\,{\mid}\,1{\leq}i{<}j{<}h{\leq}7\}$), we get, for the positive roots which are 
orthogonal to $\deltaup_C$, 
\begin{equation*} 
\begin{cases}
 (\e_{1}{-}\e_{8})-(\e_{i}{-}\e_{j})= 
 (\e_{1}{-}\e_{i})+(\e_{j}{-}\e_{8}),
& 2{\leq}i{<}j{\leq}7,
\\
(\e_{1}{-}\e_{8})-\zetaup_{i,j,h,k}= (\e_{1}{-}\e_{i})+\zetaup_{1,i,p,q},
& \{i,j,h,k,p,q\}\,{=}\,\{2,3,4,5,6,7\}.
\end{cases}
\end{equation*}

\par\smallskip
\noindent $\textsc{E}_{8}$. \;\; We may take 
\begin{equation*}
 \Rad(\textsc{E}_{8})\,{=}\,\{{\pm}\e_{i}{\pm}\e_{j}\,{\mid}\,1{\leq}i{<}j{\leq}8\}\,{\cup}\,
 \{{\pm}\zetaup_{i}\,{\mid}\,1{\leq}i{\leq}8\}\,{\cup}\,\{{\pm}\zetaup_{i,j,k}\,{\mid}\,1{\leq}i{<}j{<}k{\leq}8\}
\end{equation*}
and choose $C$ with 
\begin{equation*} 
\begin{cases}
 \Bz(C)=\{\alphaup_{i}\,{=}\,\e_{i+1}{-}\e_{i}\,{\mid}\, 1{\leq}i{\leq}6\}\cup\{\alphaup_{7}\,{=}\,\e_{1}{+}\e_{2},\;
 \alphaup_{8}\,{=}\,\zetaup_{8}\},\\
 \begin{aligned} \Rad^{+}(C)\,{=}\,\{\e_{j}{\pm}\e_{i}\,{\mid}\,1{\leq}i{<}j{\leq}8\}\,{\cup}\,\{\zetaup_{8}\}\,{\cup}\, \{{-}\zetaup_{i}{\mid}
 1{\leq}i{\leq}7\}\,{\cup}\,
 \{\zetaup_{i,j,8}\,{\mid}\,
 1{\leq}i{<}j{\leq}7\}\qquad\\
 {\cup}\,\{{-}\zetaup_{i,j,k}\,{\mid}\, 1{\leq}i{<}j{<}k{\leq}7\},\end{aligned}\\
 \deltaup_C\,{=}\,\e_{7}{+}\e_{8}.
\end{cases}
\end{equation*} We have 
\begin{equation*} 
\begin{cases}
(\e_{7}{+}\e_{8}){-}(\e_{8}{-}\e_{7})=(\e_{7}{-}\e_{i})+(\e_{7}{+}\e_{i}), & 1{\leq}i{\leq}6,\\
 (\e_{7}{+}\e_{8})-(\e_{j}{\pm}\e_{i})= 
\begin{cases}
 (\e_{7}{-}\e_{j})+(\e_{8}{\mp}\e_{i}),\\
 (\e_{7}{\mp}\e_{i})+(\e_{8}{-}\e_{j}),
\end{cases} & 1{\leq}i{<}j{\leq}6,\\
(\e_{7}{+}\e_{8})-\zetaup_{8}= (\e_{7}{+}\e_{i})+({-}\zetaup_{i}), & 1{\leq}i{\leq}6,\\
(\e_{7}{+}\e_{8})+\zetaup_{7}=(\e_{7}{-}\e_{i})+\zetaup_{i,7,8}, & 1{\leq}i{\leq}6,\\
(\e_{7}{+}\e_{8})-\zetaup_{i,j,8}=(\e_{7}{-}\e_{i})+({-}\zetaup_{j}), & \{i\neq{j}\}\subset\{1,2,3,4,5,6\},\\
(\e_{7}{+}\e_{8})+\zetaup_{i,j,7}=(\e_{7}{+}\e_{i})+\zetaup_{j,7,8}, & \{i\neq{j}\}\subset\{1,2,3,4,5,6\}.
\end{cases}
\end{equation*}
\par\smallskip
\noindent $\textsc{F}_{4}$. \;\; We take for $\textsc{F}_{4}$ the root
system 
\begin{equation}\label{rf4}
 \Rad(\textsc{F}_{4})\,{=}\,\{{\pm}\e_{i}\,{\mid}\,1{\leq}i{\leq}4\}\cup\{{\pm}\e_{i}{\pm}\e_{j}\,{\mid}\,1{\leq}i{<}j{\leq}4\}\cup
 \{\tfrac{1}{2}({\pm}\e_{1}{\pm}\e_{2}{\pm}\e_{3}{\pm}\e_{4})\},
\end{equation}
and we choose the Weyl chamber $C$ for which 
\begin{equation*}\begin{cases}
 \Bz(C)=\{\alphaup_{1}{=}\e_{3}{-}\e_{2},\, \alphaup_{2}{=}\e_{2}{-}\e_{1},\,\alphaup_{3}{=}\e_{1},\,
 \alphaup_{4}{=}\tfrac{1}{2}(\e_{4}{-}\e_{1}{-}\e_{2}{-}\e_{3})\},\\
 \Rad^{+}(C)=\{\e_{i}\,{\mid}\,1{\leq}i{\leq}4\}\cup
 \{\e_{j}{\pm}\e_{i}\,{\mid}\,1{\leq}i{<}j{\leq}4\}\cup\big\{\tfrac{1}{2}({\pm}\e_{1}{\pm}\e_{2}{\pm}\e_{3}{+}
 \e_{4})\big\},\\
 \deltaup_C\,{=}\,\e_{3}{+}\e_{4}.
 \end{cases}
\end{equation*}
We have, for positive roots orthogonal to $\deltaup_C$,  
\begin{equation*} 
\begin{cases}
(\e_{3}{+}\e_{4})-(\e_{4}{-}\e_{3})=2\cdot\e_{3}, \\
 (\e_{3}{+}\e_{4}){-}\e_{i}= 
 (\e_{3}{-}\e_{i})+\e_{4} = 
 \e_{3}+(\e_{4}{-}\e_{i}),
\qquad i=1,2,
\\
(\e_{3}{+}\e_{4})-(\e_{2}{\pm}\e_{1})= 
 (\e_{3}{-}\e_{2})+(\e_{4}{\mp}\e_{1}) 
 =
 (\e_{3}{\mp}\e_{1})+(\e_{4}-\e_{2}),
\\
(\e_{3}{+}\e_{4})-\tfrac{1}{2}({\pm}\e_{1}{\pm}\e_{2}{-}\e_{3}{+}\e_{4})=
\e_{3}+\tfrac{1}{2}({\mp}\e_{1}{\mp}\e_{2}{+}\e_{3}{+}\e_{4})
\end{cases}
\end{equation*}

\par\smallskip \noindent
$\textsc{G}_{2}$. \;\; We take for $\textsc{G}_{2}$ the root
system 
\begin{equation}\label{rg2}
 \Rad(\textsc{G}_{2})
 =\{{\pm}(\e_{i}{-}\e_{j}\,{\mid}\,1{\leq}i{<}j{\leq}3\}\cup\{{\pm}(2\e_{i}{-}\e_{j}{-}\e_{h})\mid \{i,j,k\}{=}\{1,2,3\}\}
\end{equation}
and fix $C$ in such a way that $\Bz(C)\,{=}\,\{\alphaup_{1}{=}(\e_{2}{-}\e_{3}),\,\alphaup_{2}\,{=}\,
(2\e_{3}{-}\e_{1}{-}\e_{2}\}$, so that 
\begin{equation*}
 \Rad^{+}(C)=\{\e_{2}{-}\e_{1},\,\e_{3}{-}\e_{1},\,\e_{2}{-}\e_{3}, 2\e_{2}{-}\e_{1}{-}\e_{3}, 2\e_{3}{-}\e_{1}{-}\e_{2},\,
 \e_{3}{+}\e_{2}{-}2\e_{1}\}, 
\end{equation*}
with $\deltaup_C\,{=}\,\e_{2}{+}\e_{3}{-}2\e_{1}$. The root $\e_{2}{-}\e_{3}$ is the only positive root
orthogonal to $\deltaup_C$ and we have 
\begin{equation*} 
(\e_{2}{+}\e_{3}{-}2\e_{1})-(\e_{2}{-}\e_{3})=2\cdot (\e_{3}{-}\e_{1}).
\end{equation*}
This completes the proof of the Lemma. 
\end{proof} 
\begin{ntz}
 If $\Rad$ is irreducible and $C\,{\in}\,\Cd(\Rad)$, we denote by $\deltaup_{C}$ and
 $\gammaup_{C}$ 
 the largest and lowest roots
  for the lexicographic order induced by $C$.  [$\gammaup_{C}{=}{-}\deltaup_{C}$.]
\end{ntz}

\begin{exam} Lowest roots of admissible Weyl chambers may have different $H$-indices.
Take the simple example with $\Rad\,{=}\,\{{\pm}(\e_{i}{-}\e_{j})\,{\mid}\,1{\leq}i{<}j{\leq}3\}$ 
of type $\textsc{A}_{2}$ and {root system involution $\stt$}
defined by 
\begin{equation*}
 \e_{1}\leftrightarrow{-}\e_{3},\;\; \e_{2}\leftrightarrow{-}\e_{2},
\end{equation*}
and $\Qq\,{=}\,\{\e_{1}{\pm}\e_{i}\,{\mid}\,i{=}2,3\}$.
The chamber $C_{S}$ with $\Bz(C_{S})\,{=}\,\{\alphaup_{i}{=}\e_{i}{-}\e_{i+1}\,{\mid}\,i{=}1,2\}$
is $S$-fit, while $C_{V}$ with $\Bz(C_{V})\,{=}\,
\{\alphaup'_{1}{=}\e_{1}{-}\e_{3},\,\alphaup'_{2}{=}\e_{3}{-}\e_{2}\}$
is $V$-fit. The two cross-marked diagrams are 
\begin{equation*}
   \xymatrix @M=0pt @R=2pt @!C=8pt{   \alphaup_{1}\;\;\;& \;\;\;\alphaup_{2} 
 \\
  \oplus \ar@{-}[r]\ar@/^7pt/@{<->}[r] &  \oplus  & \text{and}\\
\times}\quad
 \xymatrix @M=0pt @R=2pt @!C=8pt{   \alphaup'_{1}& \alphaup'_{2} 
 \\
 \medcirc \ar@{-}[r] &  \ominus  \\
\times}
\end{equation*}
with lowest roots $\gammaup_{C_{S}}{=}\e_{3}{-}\e_{1}$ 
and $\gammaup_{C_{V}}{=}\e_{2}{-}\e_{1}$, respectively.\par 
We have $\nuup(\gammaup_{C_{S}}){=}2$, $\nuup(\gammaup_{C_{V}}){=}1$.
\end{exam}
\begin{prop} \label{p2.57}
Assume that $\Rad$ is irreducible,  
$(\Qq\, ,\stt)$ 
nontrivial and fundamental, 
{and that $C_{V}$ and $C_{S}$ are $V$\!- and $S$\!-fit for $(\Qq\, ,\stt)$, respectively. Then} 
\begin{equation}
 \nuup(\gammaup_{C_{V}})\leq\nuup(\gammaup_{C})\leq\nuup(\gammaup_{S}),\;\;\forall C\,{\in}\,
 \Cd(\Rad\,,\Qq).
\end{equation}
\end{prop} 
\begin{proof} We observe that{, for all $C\,{\in}\,\Cd(\Rad,\Qq)$,}
$\gammaup_{C}\,{\in}\,\Qc$ 
and, in particular, belongs to $\Rad^{-}(C')$ for all $C'\,{\in}\,\Cd(\Rad,\Qq)$.\par 
 Let $C_{S}$ be an $S$-fit Weyl chamber and $C\,{\in}\,\Cd(\Rad,\Qq)$. Assume that 
$\gammaup_{C}{\neq}\,\gammaup_{C_{S}}$. If $\betaup\,{=}\,\gammaup_{C}{-}\gammaup_{C_{S}}
\,{\in}\,\Rad$, we have 
$\chiup_{\Qq}(\betaup)\,{=}\,\chiup_{\Qq}(\gammaup_{C}){-}\chiup_{\Qq}(\gammaup_{C_{S}})\,{=}\,0$
by Corollary\,\ref{c2.10} and thus $\betaup\,{\in}\,\Qq^{r}$.
If $\betaup\,{\in}\,\Rad_{\,\;\circ}^{\stt}\,{\cup}\,\Rad_{\,\;\bullet}^{\stt}$, then 
$\betaup\,{\in}\,\Qq^{r}\,{\cap}\,\stt(\Qq^{r})$ and
$\nuup(\gammaup_{C})\,{=}\,\nuup(\gammaup_{C_{S}})$.
If $\betaup$ is complex, then $\stt(\betaup)$ is still positive, because $C_{S}$ is $S$-fit,
and thus belongs to $\Qq$.
Thus $\betaup\,{\in}\,\Qq\,{\cap}\,\stt(\Qq)$ and,
from $\gammaup_{C}{=}\,\gammaup_{C_{S}}{+}\,\betaup$, 
we get    $\nuup(\gammaup_{C})\,{\leq}\,\nuup(\gammaup_{C_{S}})$.\par
If $\gammaup_{C}{-}\,\gammaup_{C_{S}}$ is not a root, then 
\begin{equation*}
 \gammaup_{C}=\gammaup_{C_{S}}+\betaup_{1}+\betaup_{2},\;\;\text{with $\betaup_{1},\betaup_{2}\,{\in}\,
 \Rad^{+}(C_{S})\subseteq\Qq,\; \betaup_{1}{+}\betaup_{2}\,{\notin}\,\Rad$.}
\end{equation*}
Since $\chiup_{\Qq}(\betaup_{1})$ and $\chiup_{\Qq}(\betaup_{2})$ are both nonnegative and
$\chiup_{\Qq}(\gammaup_{C}){=}\,\chiup_{\Qq}(\gammaup_{C_{S}})$,
we obtain $\betaup_{1},\betaup_{2}\,{\in}\,\Qq^{r}$. This implies that 
 $\betaup_{1},\betaup_{2}$ belong to $\Qq\,{\cap}\,\stt(\Qq)$,
because $C_{S}$ is $S$-fit. Thus
$\gammaup_{C}\,{=}\,\gammaup_{C_{S}}{+}\betaup_{1}{+}\betaup_{2}$ yields
$\nuup(\gammaup_{C})\,{\leq}\,\nuup(\gammaup_{C_{S}})$.
\par\smallskip
The proof that $\nuup(\gammaup_{C_{V}})\,{\leq}\,\nuup(\gammaup_{C})$ if $C_{V}$ is $V$-fit and
$C$ belongs to $\Cd(\Rad,\Qq)$ is similar. \par
 Let $C_{V}$ be a $V$-fit Weyl chamber and $C\,{\in}\,\Cd(\Rad,\Qq)$. Assume that 
$\gammaup_{C}{\neq}\,\gammaup_{C_{V}}$. 
{If $\betaup\,{=}\,\gammaup_{C}{-}\gammaup_{C_{V}}
\,{\in}\,\Rad$, then $\betaup\,{\in}\,\Rad^{+}(C_{V})\,{\cap}\,\Qq^{r}$,  since 
$\chiup_{\Qq}(\betaup)\,{=}\,\chiup_{\Qq}(\gammaup_{C}){-}\chiup_{\Qq}(\gammaup_{C_{V}})\,{=}\,0$.
If $\betaup\,{\in}\,\Rad_{\,\;\circ}^{\stt}\,{\cup}\,\Rad_{\,\;\bullet}^{\stt}$, then 
$\betaup\,{\in}\,\Qq^{r}\,{\cap}\,\stt(\Qq^{r})$ and
$\nuup(\gammaup_{C})\,{=}\,\nuup(\gammaup_{C_{S}})$.
If $\betaup$ is complex, then $\stt(\betaup)\,{\in}\,\Rad^{-}(C)$, because $C_{V}$ is $V$-fit.
Thus
$({-}\betaup)\,{\in}\,\Qq\,{\cap}\,\stt(\Qq)$ and, 
from $\gammaup_{C_{V}}{=}\,\gammaup_{C}{-}\,\betaup$ 
we get    $\nuup(\gammaup_{C_{V}})\,{\leq}\,\nuup(\gammaup_{C})$.
If $\gammaup_{C}{-}\,\gammaup_{C_{V}}$} is not a root, then 
\begin{equation*}
 \gammaup_{C}=\gammaup_{C_{V}}+\betaup_{1}+\betaup_{2},\;\;\text{with $\betaup_{1},\betaup_{2}\,{\in}\,
 \Rad^{+}(C_{V})\subseteq\Qq,\; \betaup_{1}{+}\betaup_{2}\,{\notin}\,\Rad$.}
\end{equation*}
Since $\chiup_{\Qq}(\betaup_{1})$ and $\chiup_{\Qq}(\betaup_{2})$ are both nonnegative and
$\chiup_{\Qq}(\gammaup_{C}){=}\,\chiup_{\Qq}(\gammaup_{C_{V}})$,
we obtain $\betaup_{1},\betaup_{2}\,{\in}\,\Qq^{r}$. Since $C_{V}$ is $V$-fit, 
$({-}\betaup_{1}),\,({-}\betaup_{2})$ belong to $\Qq\,{\cap}\,\stt(\Qq)$. Thus
$\gammaup_{C_{V}}\,{=}\,\gammaup_{C_{S}}{-}\betaup_{1}{-}\betaup_{2}$ implies that
$\nuup(\gammaup_{C_{V}})\,{\leq}\,\nuup(\gammaup_{C})$.
\end{proof}
\begin{rmk}
 In particular, all lowest roots of $V$-fit Weyl chambers have the same $H$-index and
 all lowest roots of $S$-fit Weyl chambers have the same $H$-index.
\end{rmk}

\begin{exam}
 Consider an irreducible root system of type $\textsc{C}_{3}$
\begin{equation*}
 \Rad=\{{\pm}2\e_{i}\mid i=1,2,3\}\cup\{{\pm}\e_{i}{\pm}\e_{j}\,{\mid}\,1{\leq}i{<}j{\leq}3\},
\end{equation*}
 $\e_{1},\e_{2},\e_{3}$ being an orthonormal basis
 of $\R^{3}$ and  $\Qq\,{=}\,\{\alphaup\,{\in}\,\Rad\,{\mid}\,\chiup(\alphaup)\,{\geq}\,0\}$, with
\begin{equation*}
 \chiup(\alphaup)=(\alphaup|\vq), \;\;\text{with}\;\; \vq=\e_{1}{+}\e_{2}.
\end{equation*}
Fix  the
involution $\stt$ of $\Rad$ given by 
\begin{equation*}
\e_{1}\,{\leftrightarrow}\,{-}\e_{3},\quad \e_{2}\,{\leftrightarrow}\,\e_{2}.
\end{equation*}
The canonical Weyl chamber is $V$-fit, with 
corresponding cross-marked diagram 
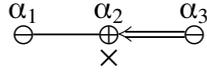
\begin{equation*} 
  \xymatrix @M=0pt @R=2pt @!C=8pt{   \alphaup_{1}& \alphaup_{2} & \alphaup_{3}
 \\
  \ominus \ar@{-}[r] &  \oplus \ar@{<=}[r] & \ominus  \\
&\times}
\end{equation*}
showing that $(\Qq\,,\stt)$ is 
Levi-nondegenerate. It is also fundamental, because 
$\stt(\alphaup_{3})\,{=}\,{-}2\e_{1}\,{=}\,{-}2(\alphaup_{1}{+}\alphaup_{2}){-}\alphaup_{3}$ 
has support equal to $\Bz(C)$. From
\begin{equation*} 
\begin{array} {| c | c | c | c |} \hline
\alphaup & \stt(\alphaup) & \chiup(\alphaup) & \chiup(\stt(\alphaup)) \\
\hline
2\e_{1} & {-}2\e_{3} & 2 & 0 \\
\hline
2\e_{2} & 2\e_{2} & 2 & 2 \\
\hline
2\e_{3} & {-} 2\e_{1} & 0 & {-}2 \,\;\\
\hline
\e_{1}{-}\e_{2}& {-}\e_{2}{-}\e_{3} & 0 & {-}1\,\; \\
\hline
\e_{1}{+}\e_{2}& \e_{2}{-}\e_{3} & 2 & 1 \\
\hline
\e_{1}{-}\e_{3}& \e_{1}{-}\e_{3} & 1 & 1 \\
\hline
\e_{1}{+}\e_{3}& {-}\e_{1}{-}\e_{3} & 1 & {-}1\;\, \\
\hline
\e_{2}{-}\e_{3} & \e_{1}{+}\e_{2} & 1 & 2 \\
\hline
\e_{2}{+}\e_{3} & \e_{2}{-}\e_{1} & 1 & 0 \\
\hline
\end{array}
\end{equation*}
we obtain 
\begin{equation*} 
\begin{cases}
 \Qc\cap\stt(\Qc)=\{{-}2\e_{2},\, {-}\e_{1}{-}\e_{2},\,\e_{3}{-}\e_{1},\, \e_{3}{-}\e_{2} \},\\ 
\Qq\cap\stt(\Qc)=\{2\e_{3},\, \e_{1}{-}\e_{2},\, \e_{1}{+}\e_{3}\},\\
 \Qc\cap\stt(\Qq)=\{{-}2\e_{1},\, {-}\e_{2}{-}\e_{3},\, {-}\e_{1}{-}\e_{3}\}.
 \end{cases}\end{equation*}
 The minimal root $({-}2\e_{1})$ belongs to $\Qc\,{\cap}\,\stt(\Qq)$ and has $H$-index $1$,
 while $({-}2\e_{2})$ has $H$-index $3$. 
 In fact 
 {\small
\begin{equation*}
 {-}2\e_{2}{-}2\e_{3}\notin\Rad,\;\; {-}2\e_{2}{-}(\e_{1}{-}\e_{2})\,{=}\,({-}\e_{1}{-}\e_{2})\,{\in}\,\Qc\,{\cap}\,
 \stt(\Qc),\;\; {-}2\e_{2}{-}(\e_{1}{+}\e_{3})\,{\notin}\,\Rad
\end{equation*}}
proves that $\nuup({-}2\e_{2})\,{>}\,2$ and thus equals $3$ because $\nuup({-}2\e_{2}){\leq}\,\nuup({-}2\e_{1})
{+}2\,{=}\,3.$ \par
The pair $(\Rad^{+}(C),\stt)$, corresponding to the cross-marked diagram 
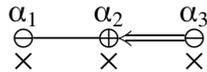
\begin{equation*} 
  \xymatrix @M=0pt @R=2pt @!C=8pt{   \alphaup_{1}& \alphaup_{2} & \alphaup_{3}
 \\
  \ominus \ar@{-}[r] &  \oplus \ar@{<=}[r] & \ominus  \\
\times&\times&\times}
\end{equation*}
has the same $H$-index as $(\Qq\, ,\stt)$. Clearly $C$ is $S$-fit for $(\Rad^{+}(C),\stt)$,
which is  Levi-degenerate. 
\end{exam}

\begin{exam}
 We consider on the root system $\Rad\,{=}\,\{{\pm}(\e_{i}{-}\e_{j})\,{\mid}\,1{\leq}i{<}j{\leq}2m\}$
 of type $\textsc{A}_{2m-1}$, with $m\,{\geq}\,2$,  the involution 
\begin{equation*}
 \stt(\e_{2i-1})=\e_{2i},\;\;\stt(\e_{2i})=\e_{2i-1},\;\; 1{\leq}i{\leq}m, 
\end{equation*}
and the parabolic 
{$\Qq\,{=}\,\{\alphaup\,{\in}\,\Rad\,{\mid}\,(\alphaup\,|\, \epi)\,{\geq}\,0\}$, with
$\epi\,{=}\,\tfrac{1}{2}{\sum}_{i=1}^{m}(\e_{2i-1}{-}\,\e_{2i})$.}
The  $S$-fit and $V$-fit diagrams for $(\Qq\, ,\stt)$ are 
\begin{gather*} 
  \xymatrix @M=0pt @R=3pt @!C=8pt{   \alphaup_{1}& \alphaup_{2} & \alphaup_{3}
 & \alphaup_{4}&  & \alphaup_{2m-3}&\alphaup_{2m-2}&\alphaup_{2m-1}\\
  \medbullet \ar@{-}[r] &  \oplus\ar@{-}[r] & \medbullet \ar@{-}[r] & \oplus \ar@{-}[r]&\cdots\cdots \ar@{-}[r] 
  & \medbullet \ar@{-}[r] & \oplus \ar@{-}[r] & \medbullet
   \\
\times&&\times &&&\times&&\times}\\[4pt]
  \xymatrix @M=0pt @R=3pt @!C=8pt{   \alphaup'_{1}& \alphaup'_{2} & \alphaup'_{3}
 & \alphaup'_{4}&  & \alphaup'_{2m-3}&\alphaup'_{2m-2}&\alphaup'_{2m-1}\\
  \oplus \ar@{-}[r] &  \ominus\ar@{-}[r] & \oplus \ar@{-}[r] & \ominus \ar@{-}[r]&\cdots\cdots \ar@{-}[r] 
  & \oplus \ar@{-}[r] & \ominus \ar@{-}[r] & \oplus
   \\
\times&&\times &&&\times&&\times}
\end{gather*}
where $\alphaup_{i}{=}\,\e_{i}{-}\e_{i+1}$ for $1{\leq}i{<}m$, 
while
$\alphaup'_{1}{=}\,\e_{1}{-}\e_{3}$, $\alphaup'_{2i+1}{=}\,\e_{2i}{-}\e_{2i{+}3}$ for $1{\leq}i{\leq}{m{-}2}$,
$\alphaup'_{2m-1}{=}\,\e_{2m-2}{-}\e_{2m}$ 
and  
$\alphaup'_{2i}{=}\,{-}\alphaup_{2i},$ 
 for $1{\leq}i{<}m$. With $\Rad^{+}(C){=}\{\e_{i}{-}\e_{j}\,{\mid}1{\leq}i{<}j{\leq}2m\}$,
we obtain 
\begin{equation*} 
\begin{cases}
 \Qq=\Rad^{+}(C)\cup\{{-}\alphaup_{2i}\,{\mid}\, 1\,{\leq}\,i\,{<}\,m\},\\
 \begin{aligned}
 \stt(\Qq)=(\Rad^{+}(C){\setminus}\{\alphaup_{2i-1}\,{\mid}\,1\,{\leq}\,i\,{\leq}\,m\})\cup\{{-}\alphaup_{2i-1}\,{\mid}\,
 1\,{\leq}\,{i}\,\leq\,{m}\}\qquad \\
 \cup\; \{{-}(\alphaup_{2i-1}{+}\alphaup_{2i}{+}\alphaup_{2i+1})\,{\mid}\, 1\,{\leq}\,\,i\,{<}\,m\},
 \end{aligned}\\
\Qc\cap\stt(\Qc)=\{\e_{j}-\e_{i}\mid 1{\leq}i{<}j{\leq}2m,\; j\,{-}\,i\,{\geq}\,2\},
\\
 \Qq\cap\stt(\Qc)=\{\alphaup_{2i-1},\,{-}\alphaup_{2i}\mid 1{\leq}i{<}m\},\\
 \Qc\cap\stt(\Qq)=\{{-}\alphaup_{2i-1},\; {-}(\alphaup_{2i-1}{+}\alphaup_{2i}{+}\alphaup_{2i+1})\mid
 1{\leq}i{<}m\}.
 \end{cases}
\end{equation*}
We have $\nuup(\Qq\, ,\stt)\,{=}\,m$ and the $H$-index of the lowest root, {which
is the same for the $S$\!- and the $V$\!-fit chamber, is} 
\begin{equation*}
\nuup(\e_{2m}{-}\e_{1})= 
\begin{cases}
 m{-}1, & \text{if $m$ is even,}\\
 \,\; m, & \text{if $m$ is odd.}
\end{cases}
  \end{equation*}
  Indeed 
\begin{equation*}
 \e_{2m}{-}\e_{1}= 
\begin{cases} \begin{aligned}
{-}{\sum}_{1\leq{i}<\left[\tfrac{m}{2}\right]}\left( (\alphaup_{4i-3}{+}\alphaup_{4i-2}
{+}\alphaup_{4i-1})+\alphaup_{4i})\right)-(\alphaup_{2m-3}{+}\alphaup_{2m-2}{+}\alphaup_{2m-1}),\;\\[-3pt]
\text{if $m$ is even,}\end{aligned}\\[14pt]
\begin{aligned}
{-}{\sum}_{1\leq{i}\leq\left[\tfrac{m}{2}\right]}\left((\alphaup_{4i-3}{+}\alphaup_{4i-2}
{+}\alphaup_{4i-1})+\alphaup_{4i})\right)-\alphaup_{2m-1}, \qquad
 \text{if $m$ is odd.}
\end{aligned}
\end{cases}
\end{equation*}
\par
When $m$ is even, the roots with largest $H$-index $m$ are $(\e_{2m-2}{-}\e_{1})$ and
$(\e_{2}{-}\e_{m-1})$: 
a minimal  $H$-sequence for $(\e_{2m-2}{-}\e_{1})$ 
{being}
\begin{equation*}
 ({-}(\alphaup_{1}{+}\alphaup_{2}{+}\alphaup_{3}), {-}\alphaup_{4}, \hdots,
 {-}(\alphaup_{2m-7}{+}\alphaup_{2m-6}{+}{-}\alphaup_{2m-5}), 
 {-}\alphaup_{2m-4},\,{-}\alphaup_{2m-3},\,{-}\alphaup_{2m-2}).
\end{equation*}
\end{exam} 

\begin{thm}
 Assume that $\Rad$ is irreducible, $\stt$ an involution of $\Rad$ and 
 $\Qq$ a parabolic subset of $\Rad$.
 \begin{enumerate}
 \item Let $C$ be any Weyl chamber in $\Cd(\Rad,\Qq)$.
 Then
 \begin{equation}
 \nuup(\Qq\,,\stt)\leq\nuup(\gammaup_{C})+2.
\end{equation} In particular,  the
  pair $(\Qq\,,\stt)$ if fundamental if and only if the  \mbox{$H$-index} of
  $\gammaup_{C}$ is finite. 
 \item Assume that $(\Qq\,,\stt)$ 
 is fundamental and 
 Levi-nondegenerate. {If $C_{S}$ is $S$-fit for $(\Qq\,,\stt)$,}
then  
\begin{equation}\label{e4.57}
 \nuup(\Qq\,,\stt)\leq \nuup(\gammaup_{C_{S}})+1.
\end{equation}
\end{enumerate}
\end{thm}
\begin{proof} $(1)$ Let $C\,{\in}\,\Cd(\Rad,\Qq)$.
Roots in $\Qc\,{\cap}\,\stt(\Qc)$ are $C$-negative. \par
Take $\alphaup\,{\in}\,\Rad^{-}(C)$. By
 Lemma\,\ref{l2.53} there are the possibilities: 
\begin{equation*} 
\begin{cases}
 \alphaup\,{=}\,\gammaup_{C},\\
 \exists\,\betaup\,{\in}\,\Rad^{+}(C)\;\;\text{s.t.}\;\; \alphaup=\gammaup_{C}{+}\betaup ,\\
  \exists\,\betaup_{1},\betaup_{2}\,{\in}\,\Rad^{+}(C)\;\;\text{s.t.}\;\; \betaup_{1}{+}\betaup_{2}\notin\Rad\;\;
  \text{and}\;\;
  \alphaup
  =\gammaup_{C}{+}\betaup_{1}{+}\betaup_{2} .
\end{cases}
\end{equation*}
Since positive roots belong to $\Qq$, this implies that the $H$-index of $\alphaup$ is
equal to $\nuup(\gammaup_{C})$ in the first, less or equal than $\nuup(\gammaup_{C}){+}1$ in the second,
less or equal than $\nuup(\gammaup_{C}){+}2$ in the third case.
\par\smallskip
$(2)$
By Proposition\,\ref{p2.57}
it suffices to find an adapted Weyl chamber $C$ for which
$\nuup(\Qq,\stt)\,{\leq}\,\nuup(\gammaup_{C}){+}1$. \par
The statement is trivial when $(\Qq\,,\stt)$ has $H$-index less or equal two. \par
Assume that 
$\nuup(\Qq\, ,\stt)\,{\geq}\,3$ and let 
$\alphaup_{0}$ be a root with $\nuup(\alphaup_{0})\,{=}\,\nuup(\Qq\,,\stt)$.
Fix a $V$-fit chamber $C_{V}$. 
If $\nuup(\alphaup_{0})\,{\leq}\,\nuup(\gammaup_{C_{V}}){+}1$, then there is nothing to prove.  
If $\nuup(\alphaup_{0})\,{=}\,\nuup(\gammaup_{V}){+}2$,  then
$\alphaup_{0}{-}\gammaup_{C_{V}}$ is not a root and, by Lemma\,\ref{l2.53} 
\begin{equation*}\tag{$*$}
 \alphaup_{0}=\gammaup_{C_{V}}+\betaup_{1}+\betaup_{2},\;\;\text{with}\;\;
 \begin{cases} \betaup_{1},\betaup_{2}\,{\in}\,\Rad^{+}(C),\;\; \betaup_{1}+\betaup_{2}\notin\Rad,\\
 \gammaup_{C_{V}}{+}\betaup_{1}\,{=}\,\alphaup_{0}{-}\betaup_{2}\in\Rad^{-}(C),\\
  \gammaup_{C_{V}}{+}\betaup_{2}\,{=}\,\alphaup_{0}{-}\betaup_{1}\in\Rad^{-}(C).
 \end{cases}
\end{equation*}
\par 
Moreover, as $\betaup_{1},\betaup_{2}\,{\in}\,\Rad^{+}(C)\,{\subseteq}\,\Qq$,
we get  $\betaup_{1},\,\betaup_{2}\,{\in}\,\stt(\Qc)$ because otherwise
$\nuup(\alphaup_{0})\,{<}\,\nuup(\gammaup_{V}){+}2.$  
From the assumptions that 
$C_{V}$ is $V$-fit and $(\Qq\,,\stt)$ 
Levi-nondegenerate, it follows that 
$\betaup_{1},\betaup_{2}$ belong
to $\Qq^{r}$, 
because roots in $\stt(\Qc)$
have negative $\stt$-conjugate. We assume, as we can,  that $|\betaup_{1}|\,{\geq}\,|\betaup_{2}|$.
\par
Let us consider the Weyl chamber $C$ with
$\Rad^{+}(C)\,{=}\,\rtt{\betaup_{1}}(\Rad^{+}(C_{V}))$. Its lowest root is 
\begin{equation*}
 \gammaup_{C}=\gammaup_{V}-\langle\gammaup\,|\,\betaup_{1}\rangle\betaup_{1}.
\end{equation*}
The coefficient $({-}\langle\gammaup_{V}\,|\,\betaup_{1}\rangle)$ is a positive integer.
If $\betaup_{1}$ is long, then this coefficient is $1$ and from
$\alphaup_{0}\,{=}\,\gammaup_{C}{+}\betaup_{2}$ it follows that $\nuup(\alphaup_{0})\,{\leq}\,
\nuup(\gammaup_{C}){+}1$. \par 
If $\betaup_{1},\betaup_{2}$ are short, and $\Rad$ is not of type $\textsc{G}_{2}$, 
then there are two possibilities: \par
\noindent
1. $\betaup_{1}$ and $\betaup_{2}$ equal the same short root $\betaup$. Then
$\rtt{\betaup}(\Rad^{+}(C_{V}))\,{=}\,\Rad^{+}(C)$ for a $C\,{\in}\,\Cd(\Rad,\Qq)$
with lowest root $\gammaup_{C}\,{=}\,\gammaup_{V}{+}2\betaup\,{=}\alphaup_{0}$
and $\nuup(\alphaup_{0})\,{=}\,\nuup(\gammaup_{C})$. \par\noindent
2. $\betaup_{1},\betaup_{2}$ are short and $\betaup_{2}{-}\betaup_{1}$ is a root.
We note that $(\betaup_{2}{-}\betaup_{1})\,{\in}\,\Qq^{r}$. With
$\rtt{\betaup_{1}}(\Rad^{+}(C_{V}))\,{=}\,\Rad^{+}(C))$ we obtain an admissible Weyl chamber
$C$ with lowest root $\gammaup_{C}\,{=}\,\gammaup_{V}{+}2\betaup_{1}$ and 
$\alphaup_{0}\,{=}\,\gammaup_{C}\,{+}\,(\betaup_{2}{-}\betaup_{1})$ implies that
$\nuup(\alphaup_{0})\,{\leq}\,\nuup(\gammaup_{C}){+}1$. \par\smallskip
If $\Rad$ is of type $\textsc{G}_{2}$ then 
$
 \alphaup_{0}\,{=}\,\gammaup_{V}\,{+}\,2\betaup
$
for a short root $\betaup\,{\in}\,\Qq^{r}$. \par
The Weyl chamber $C$ with $\Rad^{+}(C)\,{=}\,\rtt{\betaup}(\Rad^{+}(C_{V})$
is admissible and has lowest root $\gammaup_{C}\,{=}\,\gammaup_{V}\,{+}\,3\betaup$.
Then $\alphaup_{0}\,{=}\,\gammaup_{C}{+}({-}\betaup)$ yields $\nuup(\alphaup)\,{\leq}\,\nuup(\gammaup_{C})
{+}1$.
The proof is complete.
\end{proof} 
\begin{exam}
 We consider on the irreducible $\Rad\,{=}\,\{{\pm}\e_{i}{\pm}\e_{j}{\neq}0\,{\mid}\,1{\leq}i,j{\leq}4\}$
of type  $\textsc{C}_{4}$ the involution $\stt$ with 
\begin{equation*}
 \e_{1}\,{\leftrightarrow}\,\e_{4},\; \e_{2}\,{\leftrightarrow}{-}\e_{2},\; 
\e_{3}\,{\leftrightarrow}\,{-}\e_{3},
\end{equation*}
and the parabolic set 
$\Qq\,{=}\,\{\alphaup\,{\mid}\,\chiup(\alphaup)\,{\geq}\,0\}$ with $\chiup(\e_{i})\,{=}\, 2$
for $i\,{=}\,1,2$, $\chiup(\e_{3})\,{=}\,1$, $\chiup(\e_{4})\,{=}\,0$. 
The standard Weyl chamber is 
 $S$-fit for $(\Qq\, ,\stt)$, yielding the diagram
\begin{equation*}
  \xymatrix @M=0pt @R=1pt @!C=8pt{   \alphaup_{1}& \alphaup_{2} & \alphaup_{3}
 & \alphaup_{4}\\
  \oplus \ar@{-}[r] &  \medbullet\ar@{-}[r] & \ominus \ar@{<=}[r] & \oplus &\qquad(\text{$S$-fit $C$}) . \\
&\times&\times}
\end{equation*}
Since $\stt(\alphaup_{4})\,{=}2\alphaup_{1}{+}2\alphaup_{2}{+}2\alphaup_{3}{+}\alphaup_{4}$ the pair
$(\Qq\, ,\stt)$ is fundamental. \par
With $\Bz(C')\,{=}\,\{\alphaup'_{1}{=}\e_{2}{-}\e_{1},\, \alphaup'_{2}{=}\e_{1}{-}\e_{3},\,
\alphaup'_{3}{=}\e_{3}{+}\e_{4},\, \alphaup'_{4}{=}{-}\e_{4}\}$ we get a $V$-fit $C'\,{\in}\,\Cd(\Rad,\Qq)$, 
with diagram
\begin{equation*}
  \xymatrix @M=0pt @R=1pt @!C=8pt{   \alphaup'_{1}& \alphaup'_{2} & \alphaup'_{3}
 & \alphaup'_{4}\\
  \ominus \ar@{-}[r] &  \oplus\ar@{-}[r] & \oplus \ar@{<=}[r] & \ominus &\qquad(\text{$V$-fit $C'$}) . \\
&\times&\times}
\end{equation*}
showing that $(\Qq\,,\stt)$ is also
Levi-nondegenerate.
 We obtain 
\begin{equation*}
 \begin{array} {| c | c | c | c ||} \hline
\alphaup & \stt(\alphaup) & \chiup(\alphaup) & \chiup(\stt(\alphaup)) \\
\hline
2\e_{1} & 2\e_{4} & 4 & 0 \\
\hline
2\e_{2} & {-}2\e_{2} & 4 & {-}4\,\; \\
\hline
2\e_{3} & {-}2\e_{3} & 2 & {-}2\,\;\\
\hline
2\e_{4}& 2\e_{1} & 0 & 4 \\
\hline
 \e_{1}{-}\e_{2} & \e_{2}{+}\e_{4} & 0 & 2 \\
 \hline
 \e_{1}{+}\e_{2}&\e_{4}{-}\e_{2} & 4 & {-}2 \,\; \\
 \hline
 \e_{1}{-}\e_{3}&\e_{3}{+}\e_{4} & 1 & 1 \\
 \hline
 \e_{1}{+}\e_{3}&\e_{4}{-}\e_{3} & 3 & {-}1\,\; \\
 \hline \end{array}
  \begin{array} {| c | c | c | c |} \hline
\alphaup & \stt(\alphaup) & \chiup(\alphaup) & \chiup(\stt(\alphaup)) \\
\hline
\e_{1}{-}\e_{4}&\e_{4}{-}\e_{1} & 2 & {-}2\,\; \\
 \hline
 \e_{1}{+}\e_{4}&\e_{1}{+}\e_{4} & 2 & 2 \\
 \hline 
 \e_{2}{-}\e_{3}&\e_{3}{-}\e_{2} & 1 & {-}1 \\
 \hline
 \e_{2}{+}\e_{3}&{-}\e_{2}{-}\e_{3} & 3 & {-}3\,\; \\
 \hline 
 \e_{2}{-}\e_{4}&{-}\e_{1}{-}\e_{2} & 2 & {-}4\,\;\\
 \hline
 \e_{2}{+}\e_{4}&\e_{1}{-}\e_{2} & 2 & 0\\
 \hline 
 \e_{3}{-}\e_{4}&{-}\e_{1}{-}\e_{3} & 1 & {-}3\,\;\\
 \hline
 \e_{3}{+}\e_{4}&\e_{1}{-}\e_{3} & 1 & 1\\
 \hline 
\end{array}
\end{equation*}
Hence:
\begin{equation*} 
\begin{cases}
 \Qc\,{\cap}\,\stt(\Qc)\,{=}\,\{\e_{3}{-}\e_{1},\, {-}\e_{1}{-}\e_{4},\, {-}\e_{3}{-}\e_{4}\}\\
\begin{aligned}
 \Qq\,{\cap}\,\stt(\Qc)\,{=}\,\{ 2\e_{2},\,2\e_{3},\, {-}2\e_{4},\,
 \e_{2}{-}\e_{1},\, \e_{1}{+}\e_{2},\, \e_{1}{+}\e_{3},\, 
 \e_{1}{-}\e_{4},\, \e_{2}{-}\e_{3},\,\\ 
 \e_{2}{+}\e_{3},\, 
 \e_{2}{-}\e_{4},\, \e_{3}{-}\e_{4}\},\end{aligned}\\
 \begin{aligned}
 \Qc\,{\cap}\,\stt(\Qq)\,{=}\,\{ {-}2\e_{1},\,{-}2\e_{2},\,{-}2\e_{3},\, 
 {-}\e_{2}{-}\e_{4},\, \e_{4}{-}\e_{2},\, \e_{4}{-}\e_{3},\,
 \e_{4}{-}\e_{1},\, \qquad\\\e_{3}{-}\e_{2},\, {-}\e_{2}{-}\e_{3},\, 
 {-}\e_{1}{-}\e_{2},\, {-}\e_{1}{-}\e_{3}\}.\end{aligned}
 \end{cases}
\end{equation*}
In this case $\nuup(\gammaup_{C}){=}\nuup({-}2\e_{1})\,{=}\,1$, $\nuup(\gammaup_{C'}){=}\nuup({-}2e_{2}){=}1$ 
and $\nuup(\Qq\, ,\stt)\,{=}\,2$. Indeed 
\begin{equation*} 
\begin{cases}
 \e_{3}{-}\e_{1}=({-}2\e_{1})+(\e_{1}{+}\e_{3}),\\
 {-}\e_{1}{-}\e_{4}=({-}2\e_{1}){+}(\e_{1}{-}\e_{4}),\\
 {-}\e_{3}{-}\e_{4}=({-}2\e_{4})+(\e_{4}{-}\e_{3}).
\end{cases}
\end{equation*}
\end{exam} 
Dropping the assumption that $\Rad$ is irreducible, we have 
\begin{thm}
 Let $\Rad\,{=}\,
 {\bigcup}_{h=1}^{p}{\Rad_{\; h}}$ be the decomposition of $\Rad$ into the disjoint union of 
 its irreducible components. \par
 Given $C\,{\in}\,\Cd(\Rad)$, denote by $\gammaup_{C,h}$  the $C$-lowest
 root in $\Rad_{\; h}$. \par
 Consider a pair $(\Qq\,,\stt)$. 
\begin{enumerate}
 \item Let $C$ be any Weyl chamber in $\Cd(\Rad,\Qq)$.
 Then
 \begin{equation}
 \nuup(\alphaup)\leq\nuup(\gammaup_{C,h})+2,\;\; \forall h=1,\hdots,p, \;\forall \alphaup\,{\in}\,\Rad_{\,h}
\end{equation} In particular,  the
  pair $(\Qq\,,\stt)$ if fundamental if and only if the  \mbox{$H$-index} of all 
  $\gammaup_{C,h}$, with $h\,{=}\,1,\hdots,p$,   is finite. 
 \item If $C_{V},\,C_{S}\,{\in}\,\Cd(\Rad,\Qq)$ are $V$-fit and $S$-fit, respectively, then 
\begin{equation}
 \nuup(\gammaup_{C_{V},h})\leq\nuup(\gammaup_{C,h})\leq\nuup(\gammaup_{C_{S},h}),\;\;
 \forall h=1,\hdots,p,\; \forall C\in\Cd(\Rad,\Qq).
\end{equation}
\item  Assume that $(\Qq\,,\stt)$ is fundamental and 
Levi-nondegenerate.
If $C_{S}$ is any $S$-fit Weyl chamber for $(\Qq\,,\stt)$, then 
\begin{equation}\vspace{-20pt}
 \nuup(\alphaup)\leq\nuup(\gammaup_{C_{S},h})\,{+}\,1,\;\;\forall \alphaup\,{\in}\,\Rad_{\;h}.
\end{equation}\qed
\end{enumerate}
\end{thm}
\medskip
\begin{exam}
 Let $\Rad$ consist of two copies of $\textsc{B}_{2}$ and take $C\,{\in}\,\Cd(\Rad)$ with 
\begin{equation*}
 \Bz(C)=\{\alphaup_{1}{=}\e_{1}{-}\e_{2},\,\alphaup_{2}{=}\e_{2}\}\cup
 \{\alphaup'_{1}{=}\epi_{1}{-}\epi_{2},\,\alphaup'_{2}{=}\epi_{2}\},
\end{equation*}
$\e_{1},\e_{2},\epi_{1},\epi_{2}$ being an orthonormal basis of $\R^{4}$. \par
We have $\gammaup_{C,1}\,{=}\,{-}\e_{1}{-}\e_{2},$ $\gammaup_{C,2}\,{=}\,{-}\epi_{1}{-}\epi_{2}$.
Let 
\begin{equation*}
 \Qq\,{=}\,\{\e_{1},\,\e_{2},\,{\pm}(\e_{1}{-}\e_{2}),\, \e_{1}{+}\e_{2}\}\cup
 \{\epi_{1},\,{\pm}\epi_{2}, \epi_{2}{\pm}\epi_{1}\}
\end{equation*}
and define an involution $\stt$ by 
\begin{equation*}
 \e_{1}\leftrightarrow {-}\epi_{2},\quad \e_{2}\leftrightarrow \epi_{1}.
\end{equation*}
We associate to the pair $(\Qq\,,\stt)$ the cross-marked diagram 
\begin{equation*}
 \xymatrix @M=0pt @R=1pt @!C=8pt{\alphaup_{1} & \alphaup_{2}&\alphaup'_{1}&\alphaup'_{2}\\
 \ominus\ar@{=>}[r]& \oplus & \oplus \ar@{=>}[r] & \ominus & \qquad\quad\;\;\text{($C$ is $V$-fit).}\\
 &\times & \times} 
\end{equation*}
The pair $(\Qq\,,\stt)$ is fundamental and 
Levi-nondegenerate. Being 
\begin{equation*}
 \begin{array} {| c | c | c | c |} \hline
\alphaup & \stt(\alphaup) & \chiup(\alphaup) & \chiup(\stt(\alphaup)) \\
\hline
\e_{1} & {-}\epi_{2} & 1 & 0\\
\hline
\e_{2} & \epi_{1} & 1 & 1\\
\hline
\e_{1}{-}\e_{2}& {-}\epi_{1}{-}\epi_{2}& 0 & {-}1\,\;\\
\hline
\e_{1}{+}\e_{2}&\epi_{1}{-}\epi_{2}& 2 & 1\\
\hline
\end{array} 
\end{equation*}
we obtain 
\begin{equation*} 
\begin{cases}
 \Qc\cap\stt(\Qc)=\{{-}\e_{2},\;{-}\e_{1}{-}\e_{2},\; {-}\epi_{1},\; \epi_{2}{-}\epi_{1}\},\\
 \Qq\,{\cap}\,\stt(\Qc)\,{=}\, \{\e_{1}{-}\e_{2},\; \epi_{2}\},\\
 \Qc\,{\cap}\,\stt(\Qq)\,{=}\,\{{-}\e_{1},\; {-}\epi_{1}{-}\epi_{2}\}.
\end{cases}
\end{equation*}
We have $\nuup(\Qq\,,\stt)\,{=}\,3$ and $\nuup(\gammaup_{C,1})\,{=}\,3$, $\nuup(\gammaup_{C,2})\,{=}\,1$.
Indeed 
\begin{equation*}\begin{cases}
 {-}\e_{2}{=}(\e_{1}{-}\e_{2}){+}\e_{1},\;\; {-}\epi_{1}{=}\epi_{2}{+}({-}\epi_{1}{-}\epi_{2}) & \text{($H$-index $2$)},\\
 {-}\e_{1}{-}\e_{2}{=}(\e_{1}{-}\e_{2})\,{+}2{\cdot}({-}\e_{1}),\;\;
 \epi_{2}{-}\epi_{1}{=}({-}\epi_{1}{-}\epi_{2}){+}2{\cdot}\epi_{2}, & \text{($H$-index $3$)}.
 \end{cases}
\end{equation*}
\end{exam}


 \subsection{Variation of the $CR$-structure} 
 The quotients $\gt{/}\qt$ and $\gs{/}\qts$ may be viewed as infinitesimal descriptions
 of complex and real smooth homogeneous spaces. In particular, if $\mts$ is a real Lie subalgebra
 of $\gs,$ the  complex Lie subalgebras $\qt$ of $\gt$ with $\qt\,{\cap}\,\gs\,{=}\,\mts$
define the \textit{$\gs$-covariant}
 $CR$-structures on $\gs{/}\mts$. They form the set $\Qt(\gt,\gs,\mts)$ of \eqref{e2.2a},
that we consider here in the special context of
parabolic $CR$-structures.\par
 We recall from
\cite[\S{5}]{AMN06b}: 
\begin{thm} \label{t4.62} 
Let $\sfF,\sfF^{\,\prime}$ be  flag manifolds of a
semisimple complex Lie group $\Gf$
and $\sfM,\sfM^{\,\prime}$ orbits
of its real form $\Gfs,$  with $CR$-algebras $(\gs,\qt)$, $(\gs,\qt')$
at their base points $\pct_{0},\pct_{0}'.$ If
\begin{equation*}
 \qts\,{\subseteq}\,{\qt'}\cap\sigmaup(\qt'),
\end{equation*}
then:
\begin{itemize}
 \item  There is a smooth $\Gfs${-}equivariant submersion 
\begin{equation}\label{4.33}\vspace{-4pt}
 \piup:\sfM\,{\to}\,\sfM^{\,\prime},\;\;\text{with}\;\; \piup(\pct_{0})=\pct_{0}',
\end{equation}
which is a diffeomorphism when  $\qts\,{=}\,{\qt'}\,{\cap}\,\sigmaup(\qt')$. 
\item If a maximally noncompact Cartan subalgebra of $\qt\,{\cap}\,\gs$ is also maximally noncompact in 
${\qt'}\,{\cap}\,\gs$,
then the fibres of \eqref{4.33} are connected.
\item The map \eqref{4.33} is  $CR$  iff $\qt\,{\subseteq}\,\qt'$.
\item The map \eqref{4.33} is a $CR$-submersion iff $\qt'\,{=}\,\qt\,{+}\,\qt'\,{\cap}\,\sigmaup(\qt')$. \qed
\end{itemize}
 \end{thm}
 Let $\mts$ be a real Lie subalgebra of $\gs$ and 
 \begin{equation}
 \Pt(\gt,\gs,\mts)=\{\qt\,{\in}\,\Pt(\gt)
 \,{\mid}\,\qts\,{=}\,\mts\}
\end{equation}
the subset of parabolic subalgebras in  $\Qt(\gt,\gs,\mts)$.
 In view of Thm.\ref{t4.62} all orbits of  $\Gfs$ having
 $CR$-algebra $(\gs,\qt)$ with $\qt$ in $\Pt(\gt,\gs,\mts)$  
are diffeomorphic as smooth manifolds. 
\par
We recall that, 
if $\qt,\qt'\,{\in}\,\Qt(\gt,\gs,\mts)$, we say that \emph{$\qt'$ defines a stronger $CR$-structure
than $\qt$}, or that \emph{$\qt$ defines a weaker $CR$-structure
than $\qt'$}, if $\qt\,{\subset}\,\qt'$. \par
Any $CR$-structure which is stronger than a parabolic one
is parabolic.
\subsubsection{Polarization} \label{s.pol} Polarization yields 
weakest parabolic $CR$-struc\-tures.
\begin{lem}\label{l5.9}
 Let $\qt,\wt$ be parabolic subalgebras of $\gt$. Then 
\begin{equation}
 \qt_{[\wt]}=(\qt\,{\cap}\,\wt)+\qt^{n}
\end{equation}
is the smallest parabolic subalgebra containing $\qt\,{\cap}\,\wt$ and contained in $\qt$.\par
Its nilradical is 
\begin{equation}
 \qt_{[\wt]}^{n}=\qt^{n}+\wt^{n}\,{\cap}\,\qt.
\end{equation}
\end{lem} 
\begin{proof}
 The intersection $\qt\,{\cap}\,\wt$ contains a 
 Cartan subalgebra $\hg$ of $\gt$
 (see e.g. \cite[Ch.8,\S{3},Prop.10]{Bou82}). Hence we can describe $\qt$ and $\wt$
 by using their parabolic sets of roots $\Qq$ and $\mathpzc{W}$ in $\Rad\,{=}\,\Rad(\gt,\hg)$:
\begin{equation*}
 \Qq\,{=}\,\{\alphaup\,{\in}\,\Rad\mid \chiup_{\qt}(\alphaup)\geq{0}\},\;\;
 \mathpzc{W}\,{=}\,\{\alphaup\,{\in}\,\Rad\mid \chiup_{\wt}(\alphaup)\geq{0}\}.
\end{equation*}
If $0\,{<}\,\epsilonup\,{<}\,(\min_{\alphaup\in\Qq^{n}}|\chiup_{\qt}(\alphaup))|/
(\max_{\alphaup\,{\in}\,\Rad}|\chiup_{\wt}(\alphaup)|)$, then $ \qt_{[\wt]}$
is the parabolic subalgebra associated to the parabolic set 
\begin{equation*}
 \{\alphaup\,{\in}\,\Rad\mid \chiup_{\qt}(\alphaup)\,{+}\,\epsilonup\,{\cdot}\,\chiup_{\wt}(\alphaup)\,{\geq}\,0\}
 =(\Qq\cap\mathpzc{W})\cup\Qq^{n}.
\end{equation*}
The nilradical of every parabolic subalgebra $\qt'$ contained in $\qt$ contains $\qt^{n}$ 
and therefore, if we add the condition that $\qt\,{\cap}\,\wt$ is contained in $\qt'$, we obtain that
$\qt'\,{=}\, \qt_{[\wt]}$ is the minimal parabolic subalgebra with these properties.
\end{proof}
\begin{dfn}(cf. \cite[\S{9}]{Wolf69})
 Let $\sigmaup\,{\in}\,\Inv^{\vee}_{\C}(\gt)$. 
 We call $\sigmaup$-polarized a parabolic 
 subalgebra of $\gt$ admitting a $\sigmaup$-invariant reductive Levi factor.\par
A parabolic $CR$-algebra $(\gs,\qt)$ is called \emph{polarized} if $\qt$ is $\sigmaup$-polarized.
\end{dfn} 
\begin{ntz}
 For a parabolic $\qt$ of $\gt$ and $\sigmaup\,{\in}\,\Inv^{\vee}_{\C}(\gt)$
 we set $\qt_{[\sigmaup]}{=}\,\qt_{[\sigmaup(\qt)]}$.
\end{ntz}
\begin{rmk}
 Any Borel subalgebra of $\gt$ is $\sigmaup$-polarized for every $\sigmaup$. 
\end{rmk}
\begin{prop}\label{p5.11}
 Let $(\gs,\qt)$ be a parabolic $CR$-algebra. Then $\qt_{[\sigmaup]}$ is the smallest
 parabolic subalgebra in $\Pt(\gt,\gs,\qt\,{\cap}\,\gs)$
and the largest $\sigmaup$-polarized
 parabolic subalgebra of $\qt$.
\end{prop} 
\begin{proof}
Having chosen a Cartan subalgebra $\hst$ of  $\qt\,{\cap}\,\gs$, its complexification $\hg$ is  a Cartan subalgebra
of both $\qt$ and $\sigmaup(\qt)$. By Lemma\,\ref{l5.9} we obtain 
a Levi-Chevalley decomposition
\begin{equation}\label{e4.38a}
 \qt_{[\sigmaup]}
 =\qt^{r}\,{\cap}\,\sigmaup(\qt^{r})\,{\oplus}\,\left(\qt^{n}\,{\oplus}\,(\sigmaup(\qt^{n})\,{\cap}\qt^{r})
 \right)
\end{equation}
in which the first addendum is a $\sigmaup$-invariant
reductive Levi factor. This shows that $(\gs,\qt_{[\sigmaup]})$ is polarized.
Any $\sigmaup$-invariant reductive Levi factor of a $\sigmaup${-}polarized 
parabolic subalgebra of $\qt$ is contained in 
a reductive Levi factor of $\qt\,{\cap}\,\sigmaup(\qt)$: this implies the maximality of $\qt_{[\sigmaup]}$.
Every parabolic subalgebra in $\Pt(\gt,\gs,\qt\,{\cap}\,\gs)$ is $\hg$-regular 
and contains the reductive subalgebra $\qt^{r}{\cap}\,\sigmaup(\qt^{r})$. Thus \eqref{e4.38a}
shows that $\qt_{[\sigmaup]}$ is minimal in $\Pt(\gt,\gs,\qt\,{\cap}\,\gs)$.
\end{proof} 

\begin{dfn}
 We call $(\gs,\qt_{[\sigmaup]})$ the $\sigmaup$-\emph{polarization} of $(\gs,\qt)$. 
\end{dfn} 
\begin{rmk} By Remark\,\ref{r4.16},
 the $CR$-structure of $(\gs,\qt_{[\sigmaup]})$ is isomorphic to the Lie subalgebra
\begin{equation*}
 \qt_{[\sigmaup]}\cap\sigmaup(\qt_{[\sigmaup]}^{c})=\qt^{n}\cap\sigmaup(\qt^{c})
 \end{equation*}
 and the complexification of the Reeb subalgebra to 
\begin{equation*}
 \qt_{[\sigmaup]}^{c}\cap\sigmaup(\qt_{[\sigmaup]}^{c})
 =\qt^{c}\cap\sigmaup(\qt^{c})\oplus\qt^{r}\cap\sigmaup(\qt^{c})
 \oplus\qt^{c}\cap\sigmaup(\qt^{r}).
 \end{equation*}
\end{rmk}
Polarization may be computed by using the root presentation of \S\ref{s4.3}.
It is convenient in this context to reformulate the notions in terms of parabolic sets and
involutions of the root system. 
\begin{dfn} Let $\stt$ be an involution of the root system $\Rad$ and $\Qq\,{\subseteq}\,\Rad$.
We say that $(\Qq\, ,\stt)$ is \emph{$\stt$-polarized} if $\Qq^{r}\,{=}\,\stt(\Qq^{r})$. 
For $\Qq\,{\subseteq}\,\Rad$ we set 
\begin{equation}\Qq_{\,\;[\stt]}\,{\coloneqq}\,\big(\Qq\,{\cap}\stt(\Qq)\big)\,{\cup}\,\Qq^{n}\end{equation}
and call  $(\Qq_{\,\;[\stt]},\stt)$ the \emph{polarized} of $(\Qq\, ,\stt)$. 
 \end{dfn} 
\begin{lem} Let $\Qq\,{\in}\Pcr(\Rad)$, $\qt$ the corresponding parabolic subalgebra of $\gt$ and 
$\stt$ an involution of $\Rad$.
 The following are equivalent: 
\begin{enumerate}
 \item $(\Qq\, ,\stt)$ is polarized;
 \item we can find $\sigmaup\,{\in}\,\Invs$ such that $(\gs,\qt)$ is polarized;
 \item for all $\sigmaup\,{\in}\,\Invs$, the $CR$-algebra $(\gs,\qt)$ is polarized.\qed
\end{enumerate}
\end{lem}
\begin{prop}
 Let $\Qq\,{\in}\Pcr(\Rad)$, 
 $C\,{\in}\,\Cd(\Rad,\Qq)$, ${\Phi_{C}}\,{=}\,\Bz(C)\,{\cap}\,\Qq^{n}$ and $\stt$ an involution of
 $\Rad$. 
The 
following are equivalent 
\begin{align}
 \label{e4.38}& \text{$(\Qq\,,\stt)$ is polarized},\\
  \label{e4.39}& \supp_{C}(\stt(\alphaup))\cap\Phi_{C}\,{=}\,\emptyset,\;\;\forall\alphaup\,{\in}\,\Phi^{\vee}_{C},\\
  \label{e4.40} & \stt(\Qq^{n})\subset\Qq^{n}\cup\Qq^{c}.
\end{align}
\end{prop} 
\begin{proof} 
We have \eqref{e4.38}$\Leftrightarrow$\eqref{e4.39}$\Leftrightarrow$\eqref{e4.40},
because all these conditions are equivalent to $\stt(\Qq^{r})\,{=}\,\Qq^{r}$.
\end{proof}
\begin{prop}\label{p4.70}
 Let $(\gs,\qt)$ be an effective parabolic $CR$-algebra and $\Rad$ the root system of the
 complexification of an adapted Cartan subalgebra $\hst$. 
\begin{enumerate}
 \item An $S$-fit chamber for $(\Qq\, ,\stt)$ is also $S$-fit for its polarized $(\Qq_{\,\;[\stt]},\stt)$;
 \item let $C$ be any $S$-fit chamber for $(\Qq\, ,\stt)$ and $\Phi_{C}\,{=}\,\Bz(C)\,{\cap}\,\Qq^{n}$. 
 Then $\Qq_{\,\;[\stt]}\,{=}\,\Qq_{\,\;\Psi(C)}$ with 
\begin{equation*}
 \Psi_{C}\,{=}\,\Phi_{C}\cup\{\alphaup\,{\in}\,\Phi^{\vee}_{C}\mid \supp_{C}(\stt(\alphaup))\,{\cap}\,
 \Phi_{C}\,{\neq}\,\emptyset\}.
\end{equation*}
\end{enumerate}
\end{prop} 
\begin{proof} Let $\Qq\, , \, \mathpzc{P}$ be the parabolic sets of $\qt,$
$\qt_{[\sigmaup]}$, respectively. By \eqref{e4.39} we have 
\begin{equation*}
 \mathpzc{P}^{n}=\Qq^{n}\cup\{\alphaup\,{\in}\,\Qq^{r}\,{\mid}\,\stt(\alphaup)\,{\in}\,\Qq^{n}\}.
\end{equation*}
We note  that $\mathpzc{P}\,{\cap}\,\Rad_{\,\;\bullet}^{\stt}\,{\subseteq}\,\Qq^{n}$. It follows that, 
if $C$ is $S$-fit for $(\gs,\qt)$, then $\{\alphaup\,{\in}\,\Qq^{r}\,{\mid}\,\stt(\alphaup)\,{\in}\,\Qq^{n}\}
\,{\subseteq}\,\Rad^{+}(C)$, proving $(1)$. Finally we observe that $(2)$ is a straightforward
consequence of $(1)$.
\end{proof}

\begin{rmk} The 
condition that $\stt(\Phi_{C})\,{\cap}\,\Qq^{r}\,{=}\,\emptyset$ is necessary, but not sufficient for
polarization. This condition is in fact satisfied by  the $CR$-algebra $(\gs,\qt)$, with
$\gs\,{\simeq}\,\su(1,3)$, 
described by the cross-marked
$\Sigma$-diagram \par\smallskip
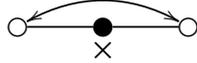
\begin{equation*}
  \xymatrix @M=0pt @R=2pt @!C=8pt{ \medcirc \ar@{-}[r] \ar@/^10pt/@{<->}[rr] & \medbullet
 \ar@{-}[r] &\medcirc\\
 &\times}
\end{equation*}\par\smallskip\noindent
which is not polarized because  $\stt(\Bz(C){\setminus}\Phi_{C})\,{\subseteq}\,\Qq^{n}$. 
By Prop.\ref{p4.70} its polarization $(\gs,\qt_{[\sigmaup]})$ 
has $\Sigma$-diagram 
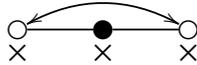
\begin{equation*}
  \xymatrix @M=0pt @R=2pt @!C=8pt{ \medcirc \ar@{-}[r] \ar@/^10pt/@{<->}[rr] & \medbullet
 \ar@{-}[r] &\medcirc\\
 \times&\times&\times}
\end{equation*}
\par 
\smallskip
The map $\ell_{2}\,{\to}\,(\ell_{2}{\cap}\ell_{2}^{\perp},\,\ell_{2},\, \ell_{2}{+}\ell_{2}^{\perp})$ 
is a diffeomorphism of the manifold  $\sfM^{3,1}$ of isotropic $2$-planes and 
the minimal orbit $\sfM^{1,5}$ of $\SU(1,3)$ in the full complex flag manifold of $\SL_{4}(\C)$.
\end{rmk}
\begin{prop}
 A polarized parabolic $CR$-algebra $(\gs,\qt)$ is either totally real
 or Levi-degenerate.  
\end{prop}
\begin{proof} Let $\hs$ be a Cartan subalgebra adapted to $(\gs,\qt)$, $\Rad$ the root system of
its complexification $\hg$, $(\Qq\,,\stt)$ the associated pair 
and $C$ a $V$-fit Weyl chamber.
If all roots in $\Phi_{C}\,{=}\,\Qq^{n}\,{\cap}\,\Bz(C)$ have a $C$-positive $\stt$-conjugate,
then by \eqref{e4.40} we obtain $\stt(\Qq^{n})\,{=}\,\Qq^{n}$ and $(\gs,\qt)$ is totally real.
Otherwise there is at least a simple root in $\Phi_{C}$ having a $C$-negative $\stt$-conjugate,
and $(\gs,\qt)$ is Levi-degenerate by Theorem\,\ref{t4.30}.
The proof is complete.
\end{proof}

\subsubsection{Strenghtening} Every $CR$ structure in $\Pt(\gt,\gs,\qt\,{\cap}\,\gs)$ can be
strengthened to a maximal one (see \cite[Proposition 5.9]{AMN06b}). 
\begin{dfn}
The $CR$-structure of 
$(\gs,\qt)$ is
 \textit{maximal} if 
 \begin{equation*}\big( \qt'\,{\in}\,\Pt(\gt,\gs,\qt\,{\cap}\,\gs),\;\; \qt\subseteq\qt' \big)\;
 \Longrightarrow \; \qt'=\qt.
\end{equation*}
The pair $(\Qq\, ,\stt)$, with $\Qq\,{\in}\Pcr(\Rad)$ and $\stt$
an involution of $\Rad$,
is \emph{maximal} if 
\begin{equation}
 \big( \Qq'\in\Pcr(\Rad),\; \Qq\subseteq\Qq',\; \Qq'\cap\,\stt(\Qq')=\Qq\cap\,\stt(\Qq)\big)\,
 \Longrightarrow \Qq'=\Qq\, .
\end{equation}
\end{dfn}
\begin{prop}
 Let $\Qq\,{\in}\Pcr(\Rad)$ be the set of roots of a parabolic $\qt\,{\in}\,\Pt(\gt)$,
 $\stt$~an involution of $\Rad$ and $\sigmaup\,{\in}\,\Invs$. Then: 
\begin{itemize}
 \item we can find $\qt'\,{\in}\,\Pt(\gt)$ such that $\qt\,{\subseteq}\,\qt'$ and 
 $(\gs,\qt')$ is $CR$-maximal; 
 \item we can find $\Qq'\,{\in}\Pcr(\Rad)$ such that $\Qq\,{\subseteq}\,\Qq'$ and
 $(\Qq',\stt)$ is $\stt$-maximal;
 \item $(\gs,\qt)$ is maximal if and only if $(\Qq\,,\stt)$ is maximal. \qed
\end{itemize}
\end{prop}
\begin{thm} \label{t4.74}
Let $\Qq\,{\in}\,\Pcr(\Rad)$, $\stt$ an involution of $\Rad$ and $C$ an $S$-fit chamber
for $(\Qq\,,\stt)$. The pair $(\Qq\,,\stt)$ is maximal if and only if 
\begin{equation}\label{e4.43}
 \supp_{C}(\stt(\alphaup))\cap\Phi_{C}\subseteq\{\alphaup\},\;\;\forall\alphaup\,{\in}\,\Phi^{+}(C).
\end{equation}
\end{thm} 
\begin{proof}
 A parabolic $\Pq$ containing $\Qq$ is equal to $\Qq_{\;\Psi_{C}}$ for a subset $\Psi_{C}$ of
 $\Phi_{C}\,{=}\,\Bz(C)\,{\cap}\,\Qq^{n}$. 
  Let $\alphaup\,{\in}\,\Phi_{C}$ and set $\Psi_{C}\,{=}\,\Phi_{C}{\setminus}
 \{\alphaup\}$. \par If either $\alphaup\,{\in}\,\Phi^{-}(C)$,  or $\alphaup\,{\in}\,\Phi_{C}^{\stt,+}$ and satisfies
 \eqref{e4.43}, 
 then 
$({-}\alphaup)$ belongs to $(\Qq_{\,\Psi_{C}}{\cap}\,\stt(\Qq_{\,\Psi_{C}})){\setminus}(\Qq\,{\cap}\,\stt(\Qq))$.\par
Assume vice versa that 
$\supp_{C}(\stt(\alphaup))\,{\cap}\,\Phi_{C}$, contains a root $\alphaup'$ distinct from~$\alphaup$. 
 A root $\xiup\,{\in}\,\Qq_{\;\Psi_{C}}\!{\setminus}\Qq$ has the form 
\begin{equation*}
 \xiup\,{=}\,k_{\xiup,\alphaup}\alphaup+{\sum}_{\betaup\in\Phi^{\vee}_{C}}k_{\xiup,\betaup}\betaup,\;\;
 \text{with $k_{\xiup,\alphaup}\,{<}\,0$, $k_{\xiup,\betaup}\,{\leq}\,0$ for all $\betaup$ in $\Phi^{\vee}_{C}$.}
\end{equation*}
Since we assumed that $C$ is $S$-fit, the $\stt$-conjugate of complex roots in $\Phi^{\vee}_{C}$
are $C$-positive. Thus  $\stt(\xiup)$ is $C$-negative and its support contains $\alphaup'$.
Therefore $\stt(\Qq_{\;\Psi_{C}}{\setminus}\Qq)$ is disjoint from $\Qq_{\;\Psi_{C}}$,
showing that $\Qq_{\;\Psi_{C}}{\cap}\,\stt(\Qq_{\;\Psi_{C}})\,{=}\,\Qq{\cap}\,\stt(\Qq)$
and hence that $(\Qq,\stt)$ is not maximal. 
\end{proof}
\begin{rmk}
In general, starting with a given $(\Qq\,,\stt)$, a maximal $(\Qq'\,,\stt)$ satisfying
$\Qq'{\cap}\,\stt(\Qq')\,{=}\,\Qq\,{\cap}\,\stt(\Qq)$ and $\Qq\,{\subseteq}\,\Qq'$
is not uniquely determined. 
Theorem\,\ref{t4.74} suggests 
a recursive construction of maximal parabolic sets $\Pq$ 
with $(\Pq\,{\cap}\,\stt(\Pq))\,{=}\,(\Qq\,{\cap}\,\stt(\Qq))$.
Indeed, if we start with a $C$-fit chamber for $(\Qq,\stt)$ and take a $\Qq_{\,\Psi_{C}}$
with $\Psi_{C}{=}\,\Phi_{C}\!{\setminus}\{\alphaup\}$ for an $\alphaup\,{\in}\,\Phi_{C}^{\stt,+}$ not
satisfying \eqref{e4.43}, then $C$ is also $S$-fit for $(\Qq_{\,\Psi_{C}},\stt)$. 
Thus we can construct
a sequence $\Phi_{C}^{(h)}$ of subsets of 
$\Phi_{C}$, with $\Phi_{C}^{(0)}\,{=}\,\Phi_{C}$ and, when
$(\Qq_{\;\Phi^{(h)}_{C}}\,,\stt)$ is not maximal, defining $\Phi^{(h+1)}_{C}$ equal
to $\Qq_{\;\Phi^{(h)}_{C}}{\setminus}\{\alphaup_{h}\}$ for an $\alphaup_{h}\,{\in}\,{\Phi_{C}^{(h),+}}$
with $\supp_{C}(\stt(\alphaup_{h}))\,{\cap}\,\Phi^{(h)}_{C}\,{\not\subseteq}\,\{\alphaup_{h}\}$. 
\end{rmk}

\begin{exam} We consider the minimal parabolic $CR$-algebra described by the cross-marked
Satake diagram 
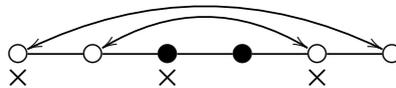
\vspace{15pt}
\begin{equation*}
   \xymatrix @M=0pt @R=2pt @!C=4pt{   \medcirc \ar@{-}[r] \ar@/^18pt/@{<->}[rrrrr] &
  \medcirc \ar@{-}[r] \ar@/^14pt/@{<->}[rrr] & \medbullet \ar@{-}[r]
 &\medbullet
 \ar@{-}[r] &\medcirc\ar@{-}[r] &\medcirc\\
\times&&\times && \times} 
\end{equation*}
with set of simple roots $\Bz(C)\,{=}\,\{\alphaup_{i}\,{=}\,\e_{i}{-}\e_{i+1}\,{\mid}\,1{\leq}i{\leq}6\}$ and
$\stt(\e_{i})\,{=}\,{-}\e_{8-i}$ for $i\,{=}\,1,2,6,7$,\, $\stt(\e_{i})\,{=}\,{-}\e_{i}$ for $i\,{=}\,3,4,5$. \par
We have $\Phi_{C}\,{=}\,\{\alphaup_{1},\alphaup_{3},\alphaup_{5}\}$. Since $\stt(\alphaup_{5})\,{=}\,
\alphaup_{2}{+}\alphaup_{3}{+}\alphaup_{4}$ contains in its $C$-support the root $\alphaup_{3}$
of $\Phi_{C}$, different from $\alphaup_{5}$, we can \textit{strengthen} the $CR$-structure
by \textit{dropping} $\alphaup_{5}$, 
stepping to $\Psi_{C}\,{=}\,\{\alphaup_{1},\alphaup_{3}\}$. $CR$-structures
corresponding to $(\Qq_{\,\;\Psi_{C}},\stt)$ 
are maximal, since
$\stt(\alphaup_{1})\,{=}\,\alphaup_{6}\,{\in}\,\Qq^{r}_{\,\;\Psi_{C}}$ and
$\stt(\alphaup_{3})\,{=}\,{-}\alphaup_{3}$. 
\end{exam}
\begin{exam} The two cross-marked diagrams for $\Rad(\textsc{B}_{6})$:
 
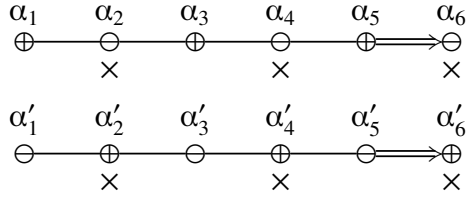
\begin{gather*}
 \xymatrix @M=0pt @R=4pt @!C=8pt{   \alphaup_{1}& \alphaup_{2} & \alphaup_{3}
 & \alphaup_{4}& \alphaup_{5} & \alphaup_{6}\\
  \oplus \ar@{-}[r] &  \ominus\ar@{-}[r] & \oplus \ar@{-}[r] & \ominus \ar@{-}[r] & \oplus
  \ar@{=>}[r] & \ominus\\
&\times&&\times&&\times}
\\[5pt]
 \xymatrix @M=0pt @R=4pt @!C=8pt{   \alphaup'_{1}& \alphaup'_{2} & \alphaup'_{3}
 & \alphaup'_{4}& \alphaup'_{5} & \alphaup'_{6}\\
  \ominus \ar@{-}[r] &  \oplus\ar@{-}[r] & \ominus \ar@{-}[r] & \oplus \ar@{-}[r] & \ominus
  \ar@{=>}[r] & \oplus\\
&\times&&\times&&\times}
\end{gather*}
with 
\begin{equation*}
\begin{array}{l l l l l l}
 \alphaup_{1}{=}\e_{1}{-}\e_{2} & 
  \alphaup_{1}{=}\e_{2}{-}\e_{3} & 
   \alphaup_{1}{=}\e_{3}{-}\e_{4} & 
  \alphaup_{1}{=}\e_{4}{-}\e_{5} & 
   \alphaup_{1}{=}\e_{5}{-}\e_{6} & 
  \alphaup_{6}=\e_{6} ,\\
  \alphaup'_{1}{=}\e_{2}{-}\e_{1} & 
  \alphaup'_{1}{=}\e_{1}{-}\e_{4} & 
   \alphaup'_{1}{=}\e_{4}{-}\e_{3} & 
  \alphaup'_{1}{=}\e_{3}{-}\e_{6} & 
   \alphaup'_{1}{=}\e_{6}{-}\e_{5} & 
  \alphaup'_{6}=\e_{5} ,
\end{array}
 \end{equation*}
 correspond to an $S$-fit and $V$-fit Weyl chamber  for a pair $(\Qq\, ,\stt)$ with 
\begin{equation*}\begin{cases}
\Qq=\{\alphaup\in\Rad\mid (\alphaup\mid 3\e_{1}{+}3\e_{2}{+}2\e_{3}{+}2\e_{4}{+}\e_{5}{+}\e_{6})
\,{\geq}\,0\},\\
 \stt(\e_{i})= 
\begin{cases}
 {+}\e_{i}, & i=1,3,5,\\
 {-}\e_{i}, & i=2,4,6.
\end{cases}
\end{cases}
\end{equation*}
The second diagram shows that $(\Qq\, ,\stt)$ is 
Levi-nondegenerate, the first that it is maximal. 
\end{exam}
\subsection{Holomorphic foliations} Let $\sfF$ be a flag manifold of a  complex
semisimple Lie group $\Gf$, $\sfM$ the orbit in $\sfF$ of a real form $\Gfs$ of $\Gf$
and $(\gs,\qt)$ the $CR$ algebra at its point $\pct$. The lifted $CR$ structure $\qt$
is the complex parabolic Lie subalgebra  of the stabiliser $\Qf$ of $\pct$
for the action of $\Gf$. 
If $\textfrak{F}$ is a $\Gfs$-invariant foliation of $\sfM$, then 
the stabiliser in $\Gfs$ 
of the leaf of $\textfrak{F}$ through $\pct$ is a closed Lie subgroup $\Lfs$ of $\Gfs$
containing the isotropy 
subgroup $\Kfs$. Its Lie algebra $\lts$  is a real Lie subalgebra of $\gs$ 
containing $\qt\,{\cap}\,\gs$.
We showed  in \S\ref{s3.9}  that 
the leaves of $\textfrak{F}$ are totally complex  if and only if 
the complexification $\lt$ of
$\lts$ satisfies condition \eqref{e3.14}. Here we specialise to the case of \textit{parabolic}
$CR$-algebras. 
Some results below are inspired by \cite{AMN06b, Wolf69}.

\begin{lem}\label{l4.78}
If $(\gs,\qt)$ be a parabolic $CR$-algebra, then the nilpotent radical $\lt^{n}$ of a
complex Lie subalgebra $\lt$ of $\gt$
 satisfying \eqref{e3.14}  contains $\qt^{n}\,{\cap}\,\sigmaup(\qt^{n})$. 
\end{lem} 
\begin{proof}
 Let $\hst$ be an adapted Cartan subalgebra of $(\gs,\qt)$ and $\Rad$ the root system of its complexification $\hg$.
 Since $(\qt\,{\cap}\,\sigmaup(\qt))\,{\subseteq}\,\lt$, the subalgebra $\lt$ is regular for $\hg$ and hence
$\lt\,{=}\,\hg\,{\oplus}\,{\sum}_{\alphaup\,{\in}\,\Lq}\,\gt^{\alphaup}$ for a closed subset $\Lq$ of $\Rad$
and $\lt^{n}\,{=}\,{\sum}_{\alphaup\in\Lq^{n}}\gt^{\alphaup}$, with $\Lq^{n}\,{=}\,\{\alphaup\,{\in}\,\Lq\,{\mid}\,
{-}\alphaup\,{\notin}\,\Lq\}$.
We have $\Qq^{n}\,{\cap}\,\stt(\Qq^{n})\,{\subseteq}\,\Lq^{n}$, because
$\Lq\,{\cap}\,\Qq^{c}\,{\cap}\,\stt(\Qq^{c})\,{=}\,\emptyset$.
This implies that $\qt^{n}\,{\cap}\,\sigmaup(\qt^{n})\,{\subseteq}\,\lt^{n}$.
\end{proof}
\begin{lem}\label{l4.83}
 Let $(\gs,\qt)$ be a parabolic $CR$-algebra and $\lt$ a complex Lie subalgebra of $\gt$
 satisfying \eqref{e3.14}. 
 The normaliser
\begin{equation}
 \Nt_{\gt}(\lt^{n})=\big\{Z\in\gt\mid [Z,\lt^{n}]\subseteq
\lt^{n}\big\}
\end{equation}
of its nilpotent radical  is a Lie subalgebra of $\gt$ satisfying \eqref{e3.14}.
\end{lem} 
\begin{proof} 
Let $\hst$ be an adapted Cartan subalgebra of $(\gs,\qt)$ and $\Rad$ the root system of its complexification $\hg$.
We keep the notation in the proof of Lemma\,\ref{l4.78}. We observe that $ \Nt_{\gt}(\lt^{n})$ is regular for
$\hg$ and denote by $\Lq'$ the set of roots $\alphaup$ for which $\gt^{\alphaup}\,{\subseteq}\,\Nt_{\gt}(\lt^{n})$.
We claim that $\Lq'$ cannot contain any $\alphaup$ of $\Qq^{c}{\cap}\,\stt(\Qq^{c})$, because
$\gt^{-\alphaup}\,{\subseteq}\,\qt^{n}{\cap}\,\sigmaup(\qt^{n})\,{\subseteq}\,\lt^{n}$ 
and $\{0\}\,{\neq}\,[\gt^{\alphaup},\gt^{-\alphaup}]
\,{\subseteq}\,\hg$ is disjoint from $\lt^{n}$.
\end{proof} 
\begin{lem}\label{l4.80}
 Let $(\gs,\qt)$ be a parabolic $CR$-algebra and $\lt$ a complex Lie subalgebra of $\gt$
 satisfying \eqref{e3.14}. Then we can find a Lie subalgebra $\at$ of $\gt$, satisfying \eqref{e3.14},
with nilpotent radical $\at^{n}$ satisfying
\begin{equation*}\tag{$*$}
 \lt\subseteq\at,\quad \at\,{=}\,\Nt_{\gt}(\at^{n})=\Nt_{\gt}(\at).
\end{equation*}
\end{lem}
\begin{proof}
Since the nilpotent radical is a characteristic ideal (see e.g. \cite[3.8.3]{Varad}), 
for every complex Lie subalgebra $\at$ of $\gt$ 
the normaliser 
 $\Nt_{\gt}(\at)$ of $\at$ is contained in the normaliser $\Nt_{\gt}(\at^{n})$ of its nilpotent radical $\at^{n}$.
 Starting from $\at_{0}\,{=}\,\lt$, let us define by recurrence 
\begin{equation*} 
 \at_{h}=\Nt_{\gt}(\at_{h-1}^{n}),\quad\text{for $h\geq{1}$.}
\end{equation*}
Since $\at_{0}\,{\subseteq}\,\at_{1}\,{\subseteq}\,\cdots$ is an increasing sequence of Lie subalgebras
of $\gt$, their sum is a Lie subalbebra $\at$ of $\gt$ satisfying $(*)$.
\end{proof}
\begin{prop}\label{p4.81}
If $(\gs,\qt)$ is a parabolic $CR$-algebra and $\lt$ a complex Lie subalgebra of $\gt$
 satisfying \eqref{e3.14}, then we can find a parabolic complex  Lie subalgebra $\pt$ of
 $\gt$ satisfying \eqref{e3.14} and containing $\lt$.
\end{prop} 
\begin{proof} By Lemma\,\ref{l4.80} we can assume that $\lt\,{=}\,\Nt_{\gt}(\lt^{n})\,{=}\,\Nt_{\gt}(\lt)$. 
Fix an admissible
Cartan subalgebra $\hst$ for $(\gs,\qt)$, let $\Rad$ the set of roots of its complexification $\hg$ and
$\Lq\,{\subseteq}\,\Rad$ the closed set of roots $\alphaup$ with $\gt^{\alphaup}\,{\subset}\,\lt$.
Set 
\begin{equation*}
 \deltaup_{\Lq}={\sum}_{\alphaup\in\Lq}\,{\alphaup},\quad \Pq\,{=}\,\{\alphaup\,{\in}\,\Rad
 \mid \langle{\deltaup_{\Lq}\,|\,\alphaup}\rangle\geq{0}\},\quad\pt\,{=}\,\hg\,{+}\,{\sum}_{\alphaup\in\Pq}\,\gt^{\alphaup}.
\end{equation*}
By their definition, $\Pq$ is a parabolic subset of $\Rad$ and $\pt$ a complex parabolic  Lie subalgebra of $\gt$.
We need to prove that $\lt\,{\subseteq}\,\pt$ and $\pt\,{\subseteq}\,\qt\,{+}\,\sigmaup(\qt)$. \par
Having fixed a root $\alphaup\,{\in}\,\Rad$, let us denote by $\st_{\alphaup}$ the simple three dimensional
complex Lie subalgebra
$\slt_{\alphaup}\,{\coloneqq}\,\gt^{\alphaup}\,{\oplus}\,\gt^{-\alphaup}\,{\oplus}\,[\gt^{\alphaup},\gt^{-\alphaup}]$.\par
Consider first the case where $\alphaup\,{\in}\,\Lq$. For each root $\betaup\,{\in}\,\Rad$  we denote by
$\Sd_{\alphaup}(\betaup)\,{=}\,\{\betaup\,{+}j\,{\cdot}\,\alphaup\}_{-r\leq{j}\leq{s}}$ the maximal $\alphaup$-string 
through $\betaup$ in $\Rad$. 
Since $\lt\,{=}\,\Nt(\lt)$,
if $\betaup\,{\in}\,\Lq{\setminus}\{{\pm}\alphaup\}$ there is $j_{0}$ with ${-}r{\leq}j_{0}{\leq}0$,
such that $\betaup\,{+}j\,{\cdot}\,\alphaup\,{\in}\,\Lq$ for $j_{0}{\leq}j\,{\leq}\,s$. The trace
$\left(\alphaup\,{|}\,{\sum}_{j=-r}^{s}(\betaup\,{+}\,j\alphaup)\right)$ of of $H_{\alphaup}\,{\in}\,[\gt^{\alphaup},
\gt^{-\alphaup}]$ on $\Sd_{\alphaup}(\betaup)$ is zero. Hence 
$\left.\left\langle{\sum}_{j=j_{0}}^{s}(\betaup\,{+}\,j\alphaup)\right|\alphaup\right\rangle\geq{0}$ and
is strictly positive when $j_{0}\,{>}\,{-}r$. Since $\Lq$ is the union of the intersections of $\Lq$ with
the $\alphaup$-strings through its roots and its intersection with $\{{\pm}\alphaup\}$, we conclude that
$\langle\deltaup_{\Lq}\,{|}\,\alphaup\rangle\,{\geq}\,0$. This proves that $\lt\,{\subseteq}\,\pt$.\par
The argument above shows that $\langle\deltaup_{\Lq}\,{|}\,\alphaup\rangle\,{>}\,0$ when
$\alphaup\,{\in}\,\Qq^{n}\,{\cap}\,\stt(\Qq^{n})$, which is equivalent to the fact that
$\pt\,{\subseteq}\,\qt\,{+}\,\sigmaup(\qt)$. 
\end{proof}
By Proposition\,\ref{p4.81}, looking, in our quest, 
for larger leafs, we can reduce to complex foliations
associated to projections onto real generalised flag manifolds. 
It was noted in \cite{AMN06b} that one can find real parabolic subgroups $\mathbf{P}_{\sigmaup}$
of $\Gfs$ containing $\Kfs$ and such that the $\Gfs$-equi\-vari\-ant fibration 
\begin{equation}\sfM\,{\simeq}\,\Gfs{/}\Kfs\,{\to}\,\Gfs{/}\mathbf{P}_{\sigmaup}\end{equation}
has totally complex fibres. To this aim it suffices to fix an adapted Cartan subalgebra $\hst$, pick  
any $A$ in the facet describing the associated parabolic set $\Qq$ of $\qt$, and take the parabolic
set 
\begin{equation*}
 \Lq=\{\alphaup\in\Rad\mid (\alphaup+\stt(\alphaup))(A)\geq{0}\}. 
\end{equation*}
Then the  parabolic Lie subalgebra $\lt\,{=}\,\hg\,{+}\,{\sum}_{\alphaup\in\Lq}\gt^{\alphaup}$ of $\gt$
satisfies \eqref{e3.14}.\par
All parabolic $\lt$ satisfying \eqref{e3.14}, containing a same $\sigmaup$-invariant Cartan subalgebra of $\gt$,
can be described  by using 
a same root system $\Rad$.
\begin{lem}
Let 
$\Qq\,,\Pq\,{\in}\Pcr(\Rad)$  and
$\stt$ be an involution
of $\Rad$. Then 
\begin{align}\label{e4.44}
 \Pq\subseteq\Qq\cup\stt(\Qq)\; \Longleftrightarrow \; \Qq^{n}\cap\stt(\Qq^{n})\subseteq\Pq^{\,n},\\
 \label{e4.45}
\Qq\cap\stt(\Qq)\subseteq\Pq\; \Longleftrightarrow \; \Pq^{n}\subseteq\Qq^{n}\cup\stt(\Qq^{n}).
\end{align} 
\end{lem} 
\begin{proof}
 In fact, passing to complements, the first inclusion in \eqref{e4.44} is equivalent to
 $\Qq^{c}\,{\cap}\,\stt(\Qq^{c})\,{\subseteq}\,\Pq^{\,c}$, which, by changing the signs of roots, 
 is equivalent to the inclusion in the right
 hand side. Likewise, the first inclusion in \eqref{e4.45} is equivalent to $\Pq^{\,c}\,{\subseteq}\,
 \Qq^{c}{\cup}\,\stt(\Qq^{c})$, which, changing signs, becomes the right hand side of \eqref{e4.45}.
\end{proof} 
\begin{lem}\label{l4.83}
 Let $C$ be an $S$-fit chamber for $(\Qq\,,\stt)$, set $\Phi_{C}\,{=}\,\Qq^{n}\,{\cap}\,\Bz(C)$, take
 distinct roots 
 $\{\alphaup_{i}\}_{1{\leq}i{\leq}k}
 \,{\subseteq}\,\Phi_{C}^{\vee}$ and define $\Psi_{C}\,{=}\,\Phi_{C}\,{\cup}\,\{\alphaup_{i}\}_{1\leq{i}\leq{k}}$. Then 
\begin{equation*}\tag{$*$}
 \Qq_{\;\Psi_{C}}\cap\stt(\Qq_{\;\Psi_{C}})=\Qq\cap\stt(\Qq) \Longleftrightarrow \supp_{C}(\stt(\alphaup_{i}))\cap
 \Phi_{C}\neq\emptyset,\;\forall 1{\leq}i{\leq}k.
 \end{equation*}
\end{lem} 
\begin{proof} The fact that the right implies the left hand side of $(*)$ follows by the
characterisation of the polarization in Proposition\,\ref{p4.70}. Indeed $\qt_{[\sigmaup]}$ is the smallest
parabolic Lie subalgebra of $\gt$ containing $\qt\,{\cap}\,\sigmaup(\qt)$ and contained in $\qt\,{+}\,\sigmaup(\qt)$
and thus $\pt\,{\cap}\,\sigmaup(\pt)\,{=}\,\qt\,{\cap}\,\sigmaup(\qt)$ for all parabolic $\pt$ with
$\qt_{[\sigmaup]}\,{\subseteq}\,\pt\,{\subseteq}\,\qt$.
\par
To complete the proof, it suffices to notice that, if e.g. $\supp_{C}(\stt(\alphaup_{1}))\,{\subseteq}\,\Phi_{C}^{\vee}$,
then $({-}\alphaup_{1})\,{\in}\,(\Qq\,{\cap}\,\stt(\Qq)){\setminus}(\Qq_{\;\Psi_{C}}\,{\cap}\,\stt(\Qq_{\;\Psi_{C}})$. 
\end{proof}
\begin{exam} \label{ex4.84}
Consider the complex flag manifold $\sfF$ consisting of the pairs $(\ell_{1},\ell_{3})$,
with $\dim_{\C}\ell_{i}\,{=}\,i$ and $\ell_{1}\,{\subset}\,\ell_{3}\,{\subset}\,\C^{5}$. Fix  on $\C^{5}$ 
a Hermitian-symmetric
form $\bil$ of signature $(2,3)$. The group $\SU(2,3)$ of determinant one matrices leaving 
$\bil$ invariant is a real form of $\SL_{5}(\C)$ and acts on~$\sfF$. 
Fix a base of $\C^{5}$ in which $\e_{1},\e_{2},\e_{4},\e_{5}$ are $\bil$-isotropic vectors and 
$\langle\e_{1},\e_{5}\rangle$, $\langle\e_{2},\e_{4}\rangle$,
$\langle\e_{3}\rangle$ are pairwise $\bil$-orthogonal. The Cartan subalgebra of the diagonal matrices
of $\su(2,3)$ yield the root system $\Rad\,{=}\,\{{\pm}(\e_{i}{-}\e_{j})\,{\mid}\, 1{\leq}i{<}j{\leq}5\}$
and the real form  the involution
 $\stt(\e_{i})\,{=}\,{-}\e_{6-i}\; (1{\leq}i{\leq}5)$.
The minimal orbit $\sfM$ of $\SU(2,3)$ in $\sfF$ is Levi-nondegenerate, 
having at its point $(\langle\e_{1}\rangle,\langle\e_{1},\e_{2},\e_{3}\rangle)$
the $CR$-algebra $(\gs,\qt)$ 
described, by   the cross-marked $\Sigma$-diagram \par\smallskip
  
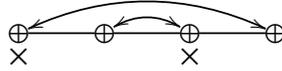
\begin{equation*}
  \xymatrix @M=0pt @R=2pt @!C=8pt{ \oplus \ar@{-}[r] \ar@/^12pt/@{<->}[rrr] & \oplus
  \ar@{-}[r] \ar@/^6pt/@{<->}[r] &\oplus 
 \ar@{-}[r] &\oplus\\
 \times&&\times}
\end{equation*}
Its polarization is totally real having cross-marked $\Sigma$-diagram \par\smallskip
  
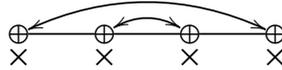
\begin{equation*}
  \xymatrix @M=0pt @R=2pt @!C=8pt{ \oplus \ar@{-}[r] \ar@/^12pt/@{<->}[rrr] & \oplus
  \ar@{-}[r] \ar@/^6pt/@{<->}[r] &\oplus 
 \ar@{-}[r] &\oplus\\
 \times&\times&\times&\times}
\end{equation*}
By Lemma\,\ref{l4.83} we obtain weaker $CR$ structures on $\sfM$ by taking any of the two diagrams 
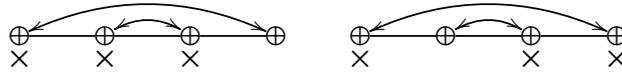
\begin{equation*}
  \xymatrix @M=0pt @R=2pt @!C=8pt{ \oplus \ar@{-}[r] \ar@/^12pt/@{<->}[rrr] & \oplus
  \ar@{-}[r] \ar@/^6pt/@{<->}[r] &\oplus 
 \ar@{-}[r] &\oplus\\
 \times&\times&\times}\qquad
  \xymatrix @M=0pt @R=2pt @!C=8pt{ \oplus \ar@{-}[r] \ar@/^12pt/@{<->}[rrr] & \oplus
  \ar@{-}[r] \ar@/^6pt/@{<->}[r] &\oplus 
 \ar@{-}[r] &\oplus\\
 \times&&\times&\times}
\end{equation*}
We obtain two complex foliations on $\sfM$, whose fibres correspond to the projections onto
the totally real $CR$-algebras corresponding  to the cross-marked diagrams\par\smallskip
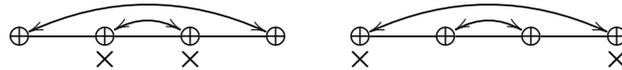
\begin{equation*}
  \xymatrix @M=0pt @R=2pt @!C=8pt{ \oplus \ar@{-}[r] \ar@/^12pt/@{<->}[rrr] & \oplus
  \ar@{-}[r] \ar@/^6pt/@{<->}[r] &\oplus 
 \ar@{-}[r] &\oplus\\
 &\times&\times}\qquad
  \xymatrix @M=0pt @R=2pt @!C=8pt{ \oplus \ar@{-}[r] \ar@/^12pt/@{<->}[rrr] & \oplus
  \ar@{-}[r] \ar@/^6pt/@{<->}[r] &\oplus 
 \ar@{-}[r] &\oplus\\
 \times&&&\times}
\end{equation*}
respectively. 
Geometrically, the fibre of the first consists, given an $\ell_{3}$ with $\ell_{3}^{\perp}\,{\subset}\,\ell_{3}$,
of all  lines $\ell_{1}$ contained in $\ell_{3}^{\perp}$. 
  The fibres of the second one
are obtained by fixing an isotropic line $\ell_{1}$ and taking all $3$-planes $\ell_{3}$ 
contained in $\ell_{1}^{\perp}$. In both cases the fibre is isomorphic to $\CP^{1}$.
\end{exam}
\begin{exam} \label{ex4.85}
Consider the complex flag manifold $\sfF$ consisting of the pairs $(\ell_{2},\ell_{3})$,
with $\dim_{\C}\ell_{i}\,{=}\,i$ and $\ell_{2}\,{\subset}\,\ell_{3}\,{\subset}\,\C^{5}$. Fix  on $\C^{5}$ 
a Hermitian-symmetric
form $\bil$ of signature $(1,4)$. The group $\SU(1,4)$ of determinant one matrices leaving 
$\bil$ invariant is a real form of $\SL_{5}(\C)$ and acts on~$\sfF$. 
Fix a base of $\C^{5}$ in which $\e_{1},\e_{5}$ are $\bil$-isotropic vectors and 
$\langle\e_{1},\e_{5}\rangle$, $\langle\e_{2}\rangle$, $\langle\e_{3}\rangle$, $\langle\e_{4}\rangle$,
are pairwise $\bil$-orthogonal. The Cartan subalgebra of the diagonal matrices
of $\su(1,4)$ yield the root system $\Rad\,{=}\,\{{\pm}(\e_{i}{-}\e_{j})\,{\mid}\, 1{\leq}i{<}j{\leq}5\}$
and the real form  the involution
 $\stt(\e_{1})\,{=}\,{-}\e_{5}$, $\stt(\e_{i})\,{=}\,{-}\e_{i},\; (2{\leq}{i}{\leq}4)$.
The minimal orbit $\sfM$ of $\SU(1,4)$ in $\sfF$ is Levi-nondegenerate, 
having at its point $(\langle\e_{1},\e_{2}\rangle,\langle\e_{1},\e_{2},\e_{3}\rangle)$
the $CR$-algebra $(\gs,\qt)$ 
described by   the cross-marked $\Sigma$-diagram 
  
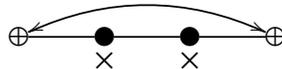
\begin{equation*}
  \xymatrix @M=0pt @R=2pt @!C=8pt{ \oplus \ar@{-}[r] \ar@/^12pt/@{<->}[rrr] & \medbullet
  \ar@{-}[r]  &\medbullet
 \ar@{-}[r] &\oplus\\
&\times&\times}
\end{equation*}
We obtain weaker $CR$ structures on $\sfM$ by taking its polarization, having cross-marked
$\Sigma$-diagram 
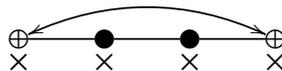
\begin{equation*}
  \xymatrix @M=0pt @R=2pt @!C=8pt{ \oplus \ar@{-}[r] \ar@/^12pt/@{<->}[rrr] & \medbullet
  \ar@{-}[r]  &\medbullet
 \ar@{-}[r] &\oplus\\
\times&\times&\times&\times}
\end{equation*}\par\smallskip
We obtain a complex foliation on $\sfM$ by taking the Levi-nondegenerate reduction of
the polarization, whose basis has the cross-marked $\Sigma$-diagram \par\smallskip
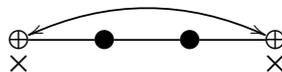
\begin{equation*}
  \xymatrix @M=0pt @R=2pt @!C=8pt{ \oplus \ar@{-}[r] \ar@/^12pt/@{<->}[rrr] & \medbullet
  \ar@{-}[r]  &\medbullet
 \ar@{-}[r] &\oplus\\
\times&&&\times}
\end{equation*}
\par\smallskip
Geometrically, the fibre through $(\ell_{2},\ell_{3})$ is obtained by setting 
$\ell_{1}\,{=}\,\ell_{2}\,{\cap}\,\ell_{2}^{\perp}$ 
and taking  all pairs $(\rq_{\,2},\rq_{\,3})$ with
$\ell_{1}\,{\subset}\,\rq_{\,2}\,{\subset}\,\rq_{\,3}\,{\subset}\ell_{1}^{\perp}$.
It has complex dimension three,  being 
isomorphic to the complete flag in $\C^{3}$.
\end{exam}
\begin{exam} \label{ex4.86}
Consider the complex flag manifold $\sfF$ consisting of the triples $(\ell_{1},\ell_{3},\ell_{4})$,
with $\dim_{\C}\ell_{i}\,{=}\,i$ and $\ell_{1}\,{\subset}\,\ell_{3}\,{\subset}\,\ell_{4}\,{\subset}\,\C^{6}$. 
Fix  on $\C^{6}$ 
a Her\-mit\-ian-symmetric
form $\bil$ of signature $(2,4)$. The group $\SU(2,4)$ of determinant one matrices leaving 
$\bil$ invariant is a real form of $\SL_{6}(\C)$ and acts on~$\sfF$. 
Fix a 
base $\e_{1},\hdots,\e_{6}$ 
of $\C^{6}$ in which $\e_{1},\e_{2},\e_{5},\e_{6}$ are $\bil$-isotropic and 
$\langle\e_{1},\e_{6}\rangle$, $\langle\e_{2}.\e_{5}\rangle$, $\langle\e_{3}\rangle$, $\langle\e_{4}\rangle$, 
are pairwise $\bil$-orthogonal. The Cartan subalgebra of the diagonal matrices
of $\su(2,4)$ yield the root system $\Rad\,{=}\,\{{\pm}(\e_{i}{-}\e_{j})\,{\mid}\, 1{\leq}i{<}j{\leq}6\}$
and the real form  the involution
 $\stt(\e_{1})\,{=}\,{-}\e_{6}$, $\stt(\e_{2})\,{=}\,{-}\e_{5}$, 
 $\stt(\e_{i})\,{=}\,{-}\e_{i},\, {i}{=}3,4$.
The minimal orbit $\sfM$ of $\SU(2,4)$ in $\sfF$ is Levi-nondegenerate, 
having at its point $(\langle\e_{1}\rangle, \langle\e_{1},\e_{2}.\e_{3}\rangle,\langle\e_{1},\e_{2},\e_{3},\e_{4}\rangle)$
the $CR$-algebra $(\gs,\qt)$ 
described by   the cross-marked $\Sigma$-diagram 
  
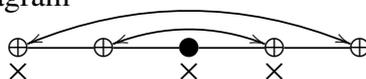
\begin{equation*}
  \xymatrix @M=0pt @R=2pt @!C=8pt{ \oplus \ar@{-}[r] \ar@/^14pt/@{<->}[rrrr] & \oplus
  \ar@/^7pt/@{<->}[rr]\ar@{-}[r] & \medbullet
  \ar@{-}[r]  &\oplus
 \ar@{-}[r] &\oplus\\
\times&&\times&\times}
\end{equation*}
We obtain weaker $CR$ structures on $\sfM$ by taking its polarization, having cross-marked
$\Sigma$-diagram 
  
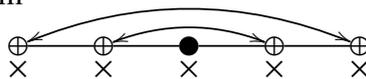
\begin{equation*}
  \xymatrix @M=0pt @R=2pt @!C=8pt{ \oplus \ar@{-}[r] \ar@/^14pt/@{<->}[rrrr] & \oplus
  \ar@/^7pt/@{<->}[rr]\ar@{-}[r] & \medbullet
  \ar@{-}[r]  &\oplus
 \ar@{-}[r] &\oplus\\
\times&\times&\times&\times&\times}
\end{equation*}
\par\smallskip
The Levi-nondegenerate reduction of its polarization has basis with cross-marked diagram 
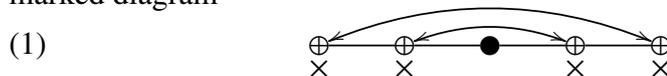
\begin{equation*}\tag{1}
   \xymatrix @M=0pt @R=2pt @!C=8pt{ \oplus \ar@{-}[r] \ar@/^14pt/@{<->}[rrrr] & \oplus
  \ar@/^7pt/@{<->}[rr]\ar@{-}[r] & \medbullet
  \ar@{-}[r]  &\oplus
 \ar@{-}[r] &\oplus\\
\times&\times&&\times&\times}
\end{equation*}\par\medskip
This foliation is not maximal, as maximal foliations have basis with cross-marked diagrams
\par\smallskip

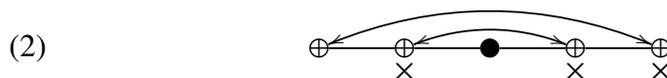
\begin{equation*}\tag{2}
   \xymatrix @M=0pt @R=2pt @!C=8pt{ \oplus \ar@{-}[r] \ar@/^14pt/@{<->}[rrrr] & \oplus
  \ar@/^7pt/@{<->}[rr]\ar@{-}[r] & \medbullet
  \ar@{-}[r]  &\oplus
 \ar@{-}[r] &\oplus\\
&\times&&\times&\times}
\end{equation*}\par\medskip
\begin{equation*}\tag{3}
   
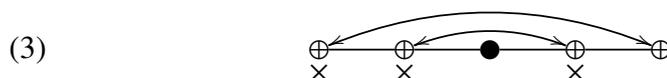
\xymatrix @M=0pt @R=2pt @!C=8pt{ \oplus \ar@{-}[r] \ar@/^14pt/@{<->}[rrrr] & \oplus
  \ar@/^7pt/@{<->}[rr]\ar@{-}[r] & \medbullet
  \ar@{-}[r]  &\oplus
 \ar@{-}[r] &\oplus\\
\times&\times&&\times&}
\end{equation*}
\par\smallskip
Geometrically, in the different cases 
the fibre through $(\ell_{1},\ell_{3},\ell_{4})$ is obtained in the following way:
\begin{itemize}
 \item[(1)] having fixed $\ell_{2},\ell_{4}$, the fibre consists of $(\ell_{2},\rq_{\,3},\ell_{4})$ with 
 $\ell_{4}^{\perp}\,{\subset}\,\rq_{\,3}\,{\subset}\ell_{4}$;
  \item[(2)] having fixed $\ell_{4}$, the fibre consists of $(\rq_{\,1},\rq_{\,3},\ell_{4})$ with 
 $\rq_{1}\,{\subset}\,\ell_{4}^{\perp}\,{\subset}\,\rq_{\,3}\,{\subset}\ell_{4}$;
  \item[(3)] having fixed $\ell_{1}$, the fibre consists of $(\ell_{1},\rq_{\,3},\rq_{4})$ with 
 $\ell_{1}\,{\subset}\,\rq_{\,3}\,{\subset}\,\rq_{\,4}{\subset}\,\ell_{1}^{\perp}$.
\end{itemize}
The fibres are isomorphic to $\CP^{1}$ in the first and to $\CP^{1}{\times}\,\CP^{1}$ in the last two cases.
\end{exam}
\begin{rmk}
 Lemmas\,\ref{l4.30},\,\ref{l4.83} can be used  
 to construct equivariant complex foliations. Starting with a given
 pair $(\Qq\,,\stt)$ we may proceed in the following way.
First we check if the pair is Levi-degenerate. If this is the case, 
the Levi-non\-de\-gen\-er\-ate reduction yields a complex
foliation with a $CR$-basis. If $(\Qq\,,\stt)$ is Levi-non\-de\-gen\-er\-ate, we choose an $S$-fit chamber
$C$ and take a new pair $(\Qq_{\;\Psi_{C}},\stt)$ with $\Psi_{C}$ obtained by adding to $\Phi_{C}$ 
roots $\alphaup$ of $\Phi^{\vee}_{C}$ with $\supp_{C}(\stt(\alphaup))\,{\cap}\,\Phi_{C}\,{\neq}\,\emptyset$.
A complex foliation is obtained by  
the Levi-non\-de\-gen\-er\-ate reduction of $(\Qq_{\;\Psi_{C}},\stt)$.
Iteration of this procedure yields after finitely many steps 
complex leafs on a compact real analytic homogeneous basis.
This was the construction of the \emph{real core} in~\cite{AMN06b}, where at each step the weakening of
the structure was obtained by polarization. Example\,\ref{ex4.85} illustrates this procedure, while in 
Example\,\ref{ex4.84} the weakening of the $CR$ structure does not  fit in the scheme of the real core.
\end{rmk}
\begin{prop}
 Let $(\gs,\qt)$ be a parabolic $CR$-algebra. If  
 $\qt\,{\cap}\,\sigmaup(\qt)$ is the unique complex Lie subalgebra of $\gt$ satisfying \eqref{e3.14},
then the polarization $(\gs,\qt_{[\sigmaup]})$ of $(\gs,\qt)$ is totally real. 
 \end{prop} 
\begin{proof}
 By Proposition\,\ref{p4.81}, 
 $\qt\,{\cap}\,\sigmaup(\qt)$ is parabolic. Since
 $\qt_{[\sigmaup]}$ is the smallest parabolic subalgebra containing $\qt\,{\cap}\,\sigmaup(\qt)$ and
 contained in $\qt\,{+}\,\sigmaup(\qt)$, we get that $\qt_{[\sigmaup]}\,{=}\,\qt\,{\cap}\,\sigmaup(\qt)$
 and therefore that $(\gs,\qt_{[\sigmaup]})$ is totally real. 
\end{proof}
\section{Levi-order of parabolic $CR$-algebras}
Let $(\gs,\qt)$ be a parabolic $CR$-algebra, $\hst$ an adapted Cartan subalgebra,
$\Qq$ the parabolic subset of $\Rad\,{=}\,\Rad(\gt,\hg)$ associated to $\qt$. 
We keep the notation of the previous sections.
Since all subalgebras in the formula below are $\hg$-regular, 
we have  Lie subalgebras, and corresponding root sets, decompositions 
\begin{equation*}
 \qt=\big(\qt\,{\cap}\,\sigmaup(\qt)\big)\oplus\big(\qt\,{\cap}\,\sigmaup(\qt^{c})\big),
 \quad \Qq=(\Qq\cap\,\stt(\Qq))\sqcup(\Qq\cap\,\stt(\Qc)).
\end{equation*}
To compute Levi-orders we may restrain to elements of $\qt\,{\cap}\,\sigmaup(\qt^{c})$ 
or, equivalently, work
with roots in $\Qq\,{\cap}\,\stt(\Qc)$. 
\subsection{Levi-sequences}
\begin{dfn}
Let $\alphaup\,{\in}\,\Qq\,{\cap}\,\stt(\Qc)$. A \emph{Levi-sequence} for
 $\alphaup$ is a sequence $(\betaup_{1},\hdots,\betaup_{q})$ satisfying 
\begin{equation} \label{e5.1}
\begin{cases}
 \betaup_{1},\hdots,\betaup_{q}\in\Qc\cap\,\stt(\Qq) ,\\
 \alphaup+{\sum}_{i=1}^{h}\betaup_{i}\,{\in}\,\Rad,\;\forall 1{\leq}h{\leq}q,\\
 \alphaup+{\sum}_{i=1}^{q}\betaup_{i}\in\Qc\cap\,\stt(\Qc).
\end{cases}
\end{equation}\par
A root 
$\alphaup$ admitting a Levi-sequence is said to have \emph{finite Levi-order}
and the minimal length  $\ord(\alphaup)$ of a Levi-sequence for $\alphaup$ is called 
its \emph{Levi-order}. \par
Otherwise, we say that $\alphaup$ has \emph{infinite} Levi-order.
\end{dfn}
\begin{lem}
 Let $\alphaup\,{\in}\,\Qq\,{\cap}\,\stt(\Qc)$ have finite Levi-order and let $(\betaup_{1},\hdots,\betaup_{q})$
 be a minimal Levi-sequence for $\alphaup$. Then\footnote{We indicate with $\Sb_{q}$ the permutations
 of $\{1,\hdots,q\}$.} 
\begin{equation}
 \label{e5.2} 
\begin{cases}
 (\betaup_{i_{1}},\hdots,\betaup_{i_{q}})\;\text{is a minimal Levi-sequence for $\alphaup$ for all $i\,{\in}\,\Sb_{q}$}\,
 ;
 \\ \text{if $q\,{\geq}\,2$, then}\;\;
 \betaup_{i_{1}}+\cdots+\betaup_{i_{h}}\notin\Rad,\;\forall 
 2\,{\leq}\,h\,{\leq}\,q\, ;\\
 \gammaup_{h}\,{=}\,\alphaup+{\sum}_{j=1}^{h}\betaup_{i}\in\Qq\cap\stt(\Qc),
 ,\;\;\forall 1{\leq}h{<}q\, .
\end{cases}
\end{equation}
\end{lem} 
\begin{proof} The fact that $\gammaup_{h}$ belongs to $\stt(\Qc)$ for all $h$ with 
$1{\leq}h{<}q$ follows because,
if one of them belongs to $\stt(\Qq)$, then also  
$\alphaup\,{+}\,{\sum}_{i=1}^{q}\betaup_{i}$
would belong to $\stt(\Qq)$.
Moreover, $\gammaup_{h}\,{\in}\,\Qq$ for $1{\leq}h{<}q$ 
and $\betaup_{1}{+}\,\cdots\,{+}\betaup_{h}\notin\Rad$ for $2\,{\leq}\,h{\leq}\,q$ 
by minimality. \par 
 Let $(X_{\betaup},H_{\betaup})_{\betaup\in\Rad}$ be a Chevalley system (cf. \cite[\S{1}]{MMN23}).
 The condition that $\gammaup_{q}\,{\in}\,\Qc\,{\cap}\,\stt(\Qc)$
is equivalent to 
\begin{equation*} \,
[X_{\alphaup},X_{\betaup_{1}},\hdots,X_{\betaup_{q}}]=[\hdots[[X_{\alphaup},X_{\betaup_{1}}],X_{\betaup_{2}}],
\hdots,X_{\betaup_{q}}]
 \notin\qt+\sigmaup(\qt).
\end{equation*}
The fact that permutations preserve  minimal Levi-sequences follows from 
\begin{equation*}
 [X_{\alphaup},\hdots,X_{\betaup_{j+1}},X_{\betaup_{j}},\hdots]=
  [X_{\alphaup},\hdots,X_{\betaup_{j}},X_{\betaup_{j+1}},\hdots]+
   [X_{\alphaup},\hdots,[X_{\betaup_{j+1}},X_{\betaup_{j}}],\hdots]
\end{equation*}
the last addendum on the right hand side belonging to $\qt\,{+}\,\stt(\qt)$ by minimality. 
Once we got invariance under permutations, minimality implies that no sums $\betaup_{i_{1}}{+}
\cdots\,{+}\,\betaup_{i_{h}}$ with distinct indices $1{\leq}i_{1}{<}\hdots{<}i_{h}{\leq}q$ can be a root.
\end{proof}
\begin{lem}\label{l3.2}
An $\alphaup\,{\in}\,\Qq^{r}\,{\cap}\,\stt(\Qc)$ has either Levi-order $1$ or ${+}\infty$.
\end{lem} 
\begin{proof}
 Assume that $\alphaup\,{\in}\,\Qq^{r}\,{\cap}\,\stt(\Qc)$ has finite Levi-order. In particular, by \eqref{e5.1}
 and \eqref{e5.2} we can find $\betaup\,{\in}\,\Qc\,{\cap}\,\stt(\Qq)$
 such that $\alphaup\,{+}\,\betaup$ is a root in $\stt(\Qc)$. If $\alphaup\,{\in}\,\Qq^{r}$,
 then $\alphaup{+}\betaup\,{\in}\,\Qc$ and hence $\ord(\alphaup)\,{=}\,1$.
\end{proof} 
\begin{cor}
If $(\Qq,\stt)$ is Levi nondegenerate, then its Levi order is either one, or the
maximum Levi order a root in $\Qq^{n}{\cap}\,\stt(\Qq^{c})$.\qed  
\end{cor}
\begin{dfn} We call a  sequence $(\betaup_{1},\hdots,\betaup_{q})$ in $\Rad$  
\emph{admissible} for $\alphaup\,{\in}\,\Rad$~if 
\begin{equation} \label{e5.3}
\begin{cases}
  {\sum}_{i=1}^{h}\betaup_{j_{i}}\notin\Rad\,{\cup}\,\{0\},&\;\text{if $h\,{\geq}2$} , \\[3pt]
 \alphaup+
{\sum}_{i=1}^{h}\betaup_{j_{i}} \in\Rad{\setminus}\{\alphaup\},&\;\text{if $h\,{\geq}1$} ,
\end{cases} \qquad \forall 1{\leq}j_{1}{<}\cdots{<}j_{h}{\leq}q,
\end{equation}
and define 
\emph{index} of $\alphaup$,  denoted by  $\muup(\alphaup)$, the maximal length
of its admissible sequences.
\end{dfn}
By  \eqref{e5.1} and \eqref{e5.2}, the \textit{index} of a root $\alphaup$  
yields 
bounds for its Levi-order. This index, for all different irreducible root systems,
 was computed in 
\cite[\S{2.5}]{MMN23}. Let us recall the procedure. \par
For a given $\alphaup\,{\in}\,\Rad$, we define
\begin{equation*}
 \Rad^{\!\text{add}}(\alphaup)\,{=}\,\{\betaup\,{\in}\,\Rad\,{\mid}\,\alphaup{+}\betaup\,{\in}\,\Rad\}.
\end{equation*}
For the irreducible root systems we use the notation of \cite[\S{2.4}]{MMN23}.
\subsubsection*{$\textsc{A}_{\ell}$} We can take $\alphaup\,{=}\,\e_{2}{-}\e_{1}$. Then 
\begin{equation*}\tag{$*\textsc{A}$}
 \Rad^{\!\text{add}}(\e_{2}{-}\e_{1})=\{\e_{1}{-}\e_{i},
 \mid 3{\leq}i{\leq}\ell{+}1\}
 \cup\{
 \e_{i}{-}\e_{2}\mid 3{\leq}i{\leq}\ell{+}1\}.
\end{equation*}
We have $\muup(\e_{2}{-}\e_{1})$ equal $0$ if $\ell\,{=}\,1$ and equal $1$ for $\ell\,{=}\,2$. 
If $\ell{\geq}3$, taking into account that an admissible sequence does contain at most one root from each
of the two sets in the right hand side of $(*\textsc{A})$,
 we find that $\muup(\e_{2}{-}\e_{1})\,{=}\,2$ and $(\e_{1}{-}\e_{3},\,\e_{4}{-}\e_{2})$
is a maximal length admissible sequence.
\subsubsection*{$\textsc{B}_{\ell{\geq}2}$}  If $\alphaup$ is short, we can take $\alphaup\,{=}\,{-}\e_{1}$. Then 
\begin{equation*}\tag{$*\textsc{B}$}
 \Rad^{\!\text{add}}({-}\e_{1})=\{{\pm}\e_{i}\mid 2{\leq}i{\leq}\ell\}\cup\{\e_{1}{\pm}\e_{i}\mid 2{\leq}i{\leq}\ell\}.
\end{equation*}
An admissible sequence for $({-}\e_{1})$ may contain at most one from the first and two from the second set in
the right hand side of $(*\textsc{B})$. We obtain 
\begin{itemize}
 \item If $\ell{=}2,$ then $\muup(-\e_{1})\,{=}\,2$ and $(\e_{1}{-}\e_{2},\,\e_{1}{+}\e_{2})$ is a maximal length
 admissible sequence;
 \item If $\ell{\geq}3$, then $\muup(-\e_{1})\,{=}\,3$ and $(\e_{1}{-}\e_{2},\,\e_{1}{+}\e_{2},\,\e_{3})$ is a maximal length
 admissible sequence.
\end{itemize} \par
If $\alphaup$ is a long root, then we can take $\alphaup\,{=}\,{-}\e_{1}{-}\e_{2}$.  Then
\begin{equation*}
 \tag{$*\! *\textsc{B}$} 
 \Rad^{\!\text{add}}({-}\e_{1}{-}\e_{2})=\{\e_{1},\,\e_{2}\}\cup\{\e_{1}{\pm}\e_{i}\mid 3{\leq}i{\leq}\ell\}
 \cup \{\e_{2}{\pm}\e_{i}\mid 3{\leq}i{\leq}\ell\}.
\end{equation*}
[The last two sets are empty if $\ell\,{=}\,2$.]\par
An admissible sequence may contain at most four terms, because each  root of $\Rad^{\text{add}}({-}\e_{1}{-}\e_{2})$ 
adds 
either one $\e_{1}$ or one $\e_{2}$.
 We obtain: 
\begin{itemize}
 \item if $\ell\,{=}\,2$, then $\muup({-}\e_{1}{-}\e_{2})\,{=}\,2$ and $(\e_{1},\e_{1})$ is an admissible sequence of 
 maximal length;
  \item if $\ell\,{\geq}\,3$, then $\muup({-}\e_{1}{-}\e_{2})\,{=}\,4$ and $(\e_{1},\e_{1},\,\e_{2}{-}\e_{3},
  \,\e_{2}{+}\e_{3})$ is an admissible sequence of 
 maximal length. Moreover, all sums of the four terms on a maximal length admissible sequence add to $\e_{1}{+}\e_{2}\,{=}\,{-}\alphaup$.
\end{itemize}
\subsubsection*{$\textsc{C}_{\ell{\geq}3}$} Let us start with the case of a long root: 
we can take $\alphaup\,{=}\,{-}2\e_{1}$. Then 
\begin{equation*}
 \tag{$*\textsc{C}$}
 \Rad^{\!\text{add}}({-}2\e_{1})\,{=}\,\{\e_{1}{\pm}\e_{i}\,{\mid}\, 2{\leq}i{\leq}\ell\}.
\end{equation*}
Since each term adds one $\e_{1}$, the maximal length of an admissible sequence is ${\leq}\,4$.
But we cannot take both $\e_{1}{-}\e_{i}$ and $\e_{1}{+}\e_{i}$ in a same admissible sequence.
In fact 
$\muup({-}2\e_{1})\,{=}\,2$; we can take e.g. the sequence $(\e_{1}{+}\e_{2},\e_{1}{+}\e_{3})$.
\par
We consider next the case of a short root. Take $\alphaup\,{=}\,{-}\e_{1}{-}\e_{2}$. Then 
\begin{equation*}\tag{$*\! *\textsc{C}$}
 \Rad^{\!\text{add}}({-}\e_{1}{-}\e_{2})=\{2\e_{1},2\e_{2}\}\cup\{{\pm}(\e_{1}{-}\e_{2})\}\cup
 \{\e_{i}{\pm}\e_{j}\mid i{=}1,2,\; 3{\leq}j{\leq}\ell\} .
\end{equation*}
An admissible sequence may contain both roots of the first
set on the right  hand side of $(**\textsc{C})$, but at most one from 
the second and two from the third, with different indices $i$. 
Moreover, if it contains a root $2\e_{i}$ ($i\,{=}\,1,2$)
of the first, it
does not contain a root from the second and from $\{\e_{i}{\pm}\e_{j}\,{\mid}\,3{\leq}j{\leq}\ell\}$. 
Admissible sequences containing a root from the second set have length $1$.
These arguments show that $\muup({-}\e_{1}{-}\e_{2}){=}2$: admissible sequences of maximal length
are e.g. $(2\e_{1},2\e_{2})$, $(2\e_{1},\e_{2}{+}\e_{3})$, $(\e_{1}{+}\,\e_{3},\e_{2}{+}\,\e_{3})$.
\subsubsection*{$\textsc{D}_{\ell{\geq}4}$} We can take e.g. $\alphaup\,{=}\,{-}\e_{1}{-}\e_{2}$. Then 
\begin{equation*}\tag{$*\textsc{D}$}
 \Rad^{\!\text{add}}({-}\e_{1}{-}\e_{2})=\{\e_{1}{\pm}\e_{i}\,{\mid}\,3{\leq}i{\leq}\ell\}\cup
 \{\e_{2}{\pm}\e_{i}\,{\mid}\,3{\leq}i{\leq}\ell\}.
\end{equation*}
An admissible sequence may contain at most two roots from each of the sets in the right
hand side of $(*\textsc{D})$. Thus $\muup({-}\e_{1}{-}\e_{2})$ is less or equal $4$ and actually
equals $4$, as the sequence $(\e_{1}{-}\e_{3},\e_{1}{+}\e_{3},\e_{2}{-}\e_{4},\e_{2}{+}\e_{4})$
is admissible. 
\subsubsection*{$\textsc{E}_{\ell}$} We take, as in \cite{MMN23}, 
$\Rad(\textsc{E}_{6})\,{\subset}\,\Rad(\textsc{E}_{7})
\,{\subset}\,\Rad(\textsc{E}_{8})$. We will reduce the discussion to the case $\textsc{E}_{8}$ and
we take $\alphaup\,{=}\,{-}\e_{1}{-}\e_{2}$. Then 
\begin{equation*}
 \tag{$*\textsc{E}$} \left\{ 
\begin{aligned}
 \Rad^{\!\text{add}}({-}\e_{1}{-}\e_{2})= \{\e_{1}{\pm}\e_{i}\,{\mid}\, 3{\leq}i{\leq}8\}\cup
 \{\e_{2}{\pm}\e_{i}\,{\mid}\, 3{\leq}i{\leq}8\}\cup\{{-}\zetaup_{\emptyset}\}
 \cup\{\zetaup_{1,2}\}\\
 \cup\,\{\zetaup_{1,2,i,j}\,{\mid}\, 3{\leq}i{<}j{\leq}8\}
 \cup\{{-}\zetaup_{i,j}\,{\mid}\, 3{\leq}i{<}j{\leq}8\}.
\end{aligned}\right.
\end{equation*}
Let us consider any admissible sequence $(\betaup_{1},\hdots,\betaup_{q})$ 
for ${-}\e_{1}{-}\e_{2}$ and let $q_{i}$ be the number of roots from each of the six sets in the right
hand side of $(*\textsc{E})$. Counting the coefficient of $\e_{i}$, for $i{=}1,2$, in $\alphaup{+}\betaup_{1}{+}
\cdots{+}\betaup_{q}$
we obtain  
\begin{equation*}\begin{cases}
 q_{1}+\tfrac{1}{2}(q_{3}+\cdots+q_{6})\leq 2,\\
  q_{2}+\tfrac{1}{2}(q_{3}+\cdots+q_{6})\leq 2,
 \end{cases}
\end{equation*}
showing that $q\,{=}\,q_{1}{+}\cdots{+}q_{6}{\leq}4$. \par
We get $\muup({-}\e_{1}{-}\e_{2})\,{=}\,4$,
since $(\e_{1}{-}\e_{3},\e_{1}{+}\e_{3},\e_{2}{-}\e_{4},\e_{2}{+}\e_{4})$ is admissible and of length $4$,
for all systems $\textsc{E}_{\ell}$ ($6{\leq}\ell{\leq}8$). We note that for an admissible system of
maximal length we get $\alphaup{+}\betaup_{1}{+}\betaup_{2}{+}\betaup_{3}{+}\betaup_{4}\,{=}\,{-}\alphaup$.
\subsubsection*{$\textsc{F}_{4}$} Consider first the case of a short root. We can take $\alphaup\,{=}\,{-}\e_{1}$.
Then 
\begin{equation*}\tag{$*\textsc{F}$}
 \Rad^{\!\text{add}}({-}\e_{1})=\{{\pm}\e_{i}\,{\mid}\,2{\leq}i{\leq}4\}
 \,{\cup}\,\{\tfrac{1}{2}(\e_{1}{\pm}\e_{2}{\pm}\e_{3}{\pm}\e_{4}\}\,{\cup}\,
 \{\e_{1}{\pm}\e_{i}\,{\mid}\,2{\leq}i{\leq}4\}.
\end{equation*}
Consider an admissible sequence $(\betaup_{1},\hdots,\betaup_{q})$ for ${-}\e_{1}$ and let
$q_{1},q_{2},q_{3}$ be the number of roots taken from each of the three sets in the right hand
side of $(*\textsc{F})$. Then $q_{1}{\leq}1$ and $\tfrac{1}{2}q_{2}{+}q_{3}{\leq}2$. Moreover,
$q_{1}{\cdot}q_{3}\,{=}\,0$. This implies that $\muup({-}\e_{1})\,{=}\,q_{1}{+}q_{2}{+}q_{3}{\leq}3$
and actually equality holds for an admissible system of maximal length, as we can take
e.g. $(\e_{1}{+}\e_{2},\,\e_{1}{-}\e_{2},\,\e_{3}).$ \par
As a long root we can take $\alphaup\,{=}\,{-}\e_{1}{-}\e_{2}$. Then 
\begin{equation*}\tag{$* \! *\!\textsc{F}_{4}$}
 \Rad^{\!\text{add}}({-}\e_{1}{-}\e_{2}){=}\{\e_{1}{\pm}\e_{i}\,{\mid}\,i{=}3,4\}{\cup}
 \{\e_{2}{\pm}\e_{i}\,{\mid}\,i{=}3,4\}{\cup}\{\tfrac{1}{2}(\e_{1}{+}\e_{2}{\pm}\e_{3}{\pm}\e_{4})\}.
\end{equation*}
Let  $(\betaup_{1},\hdots,\betaup_{q})$ be an admissible sequence
for ${-}\e_{1}{-}\e_{2}$ and
 $q_{1},q_{2},q_{3}$ the number of its roots in each of the three sets. We get 
\begin{equation*} 
\begin{cases}
 q_{1}+\tfrac{1}{2}q_{3}\leq{2},\\
 q_{2}+\tfrac{1}{2}q_{3}\leq{2}.
\end{cases}
\end{equation*}
This yields $q\,{=}\,q_{1}{+}q_{2}{+}q_{3}{\leq}4$. Hence $\muup({-}\e_{1}{-}\e_{2}){\leq}4$ and in fact
equality holds as $(\e_{1}{+}\e_{3},\,\e_{1}{-}\e_{3},\,\e_{2}{+}\e_{4},\,\e_{2}{-}\e_{4})$ is an
admissible sequence of length four. For an admissible sequence of length four we have
$\alphaup{+}\betaup_{1}{+}\cdots{+}\betaup_{4}{=}\,{-}\alphaup$. 
\subsubsection*{$\textsc{G}_{2}$} As a short root we can take $\alphaup\,{=}\,\e_{2}{-}\e_{1}$. We get 
\begin{equation*}\tag{$*\textsc{G}$}
 \Rad^{\!\text{add}}(\e_{2}{-}\e_{1})\,{=}\, \{\e_{1}{-}\e_{3}\}\,{\cup}\, \{\e_{3}{-}\e_{2}\}\,{\cup}\,
 \{2\e_{1}{-}\e_{2}{-}\e_{3}\}\,{\cup}\, \{\e_{1}{+}\e_{3}{-}2\e_{2}\}.
\end{equation*}
Let $(\betaup_{1},\hdots,\betaup_{q})$ be an admissible system and $q_{1},q_{2},q_{3},q_{4}$ be
the number of its
roots belonging to the corresponding set in the right hand side. Then 
\begin{equation*}
\begin{cases}
 q_{1}+2q_{3}+q_{4}\leq 2,\\
 q_{2}+q_{3}+2q_{4}\leq{2},
\end{cases}\;\; \Longrightarrow \;\; q_{1}{+}q_{2}{+}3q_{3}{+}3q_{4}\leq{4}.
\end{equation*}
We obtain that $q\,{=}\,q_{1}{+}q_{2}{+}q_{3}{+}q_{4}{\leq}2$ and in fact $\muup(\e_{2}{-}\e_{1})\,{=}\,2$,
as we can take e.g. $(\e_{1}{-}\e_{3},\, \e_{1}{+}\e_{3}{-}2\e_{2})$. We have $\alphaup{+}\betaup_{1}{+}
\betaup_{2}\,{=}\,
{-}\alphaup$.\par\smallskip
As a long root we can take $\alphaup\,{=}\,\e_{2}{+}\e_{3}{-}2\e_{1}$. Then 
\begin{equation*}\tag{$*\! *\textsc{G}$}
 \Rad^{\!\text{add}}(\e_{2}{+}\e_{3}{-}2\e_{1})=\{\e_{1}{-}\e_{2},\,\e_{1}{-}\e_{3}\}\,{\cup}\,
 \{\e_{1}{+}\e_{2}{-}2\e_{3},\, \e_{1}{+}\e_{3}{-}2\e_{2}\}.
\end{equation*}
As each root of an admissible system $(\betaup_{1},\hdots,\betaup_{q})$ adds one unit of $\e_{1}$,
we have $q{\leq}4$. In fact $\muup(\e_{2}{+}\e_{3}{-}2\e_{1})\,{=}\,4$, as we can take the sequence
$$(\e_{1}{-}\e_{2},\e_{1}{-}\e_{2},\e_{1}{-}\e_{2},\e_{1}{+}\e_{2}{-}2\e_{3}).$$ \par\medskip
We can summarize the discussion above in the following table:
\begin{equation*}
\begin{array}{| c | c | c | c |} 
\hline
\text{type} & \alphaup &  \betaup_{1},\hdots,\betaup_{q} & \alphaup{+}\betaup_{1}{+}\cdots
{+}\betaup_{q} \\
\hline\hline
\textsc{A}_{\ell} & \e_{2}{-}\e_{1} & \e_{1}{-}\e_{4},\; \e_{3}{-}\e_{2} & \e_{3}{-}\e_{4} \\
\hline\hline
\textsc{B}_{\ell} & {-}\e_{1} & \e_{1}{-}\e_{2},\;\e_{1}{+}\e_{2},\; \e_{3} & \e_{1}{+}\e_{3} \\
\hline
\textsc{B}_{\ell} & {-}(\e_{1}{+}\e_{2}) & \e_{1},\,\e_{1},\, \e_{2}{-}\e_{3},\;\e_{2}{+}\e_{3}& \e_{1}{+}\e_{2}\\
\hline\hline
\textsc{C}_{\ell} & {-}2\e_{1} & \e_{1}{+}\e_{2},\;\e_{1}{-}\e_{3} & \e_{2}{-}\e_{3}\\
\hline
\textsc{C}_{\ell} & {-}(\e_{1}{+}\e_{2}) & 2\e_{1},\; 2\e_{2} & \e_{1}{+}\e_{2} \\
\hline
\textsc{C}_{\ell} & {-}(\e_{1}{+}\e_{2}) & 2\e_{1},\; \e_{2}{+}\e_{3} & \e_{1}{+}\e_{3} \\
\hline
\textsc{C}_{\ell} & {-}(\e_{1}{+}\e_{2}) & \e_{1}{+}\e_{3},\; \e_{2}{+}\e_{4} & \e_{3}{+}\e_{4} \\
\hline\hline
\textsc{D}_{\ell} & {-}(\e_{1}{+}\e_{2}) & \e_{1}{-}\e_{3},\,\e_{1}{+}\e_{3},\, \e_{2}{-}\e_{4},\, \e_{2}{-}\e_{4} 
& \e_{1}{+}\e_{2}\\
\hline\hline
\textsc{E}_{\ell} &  {-}(\e_{1}{+}\e_{2}) & \e_{1}{-}\e_{3},\,\e_{1}{+}\e_{3},\, \e_{2}{-}\e_{4},\, \e_{2}{-}\e_{4} 
& \e_{1}{+}\e_{2}\\
\hline\hline
\textsc{F}_{4} & {-}\e_{1} & \e_{1}{-}\e_{2},\, \e_{1}{+}\e_{2},\, \e_{3} & \e_{1}{+}\e_{3}\\
\hline
\textsc{F}_{4} & {-}(\e_{1}{+}\e_{2}) & \e_{1}{-}\e_{3},\, \e_{1}{+}\e_{3},\, \e_{2}{-}\e_{4},\,\e_{2}{+}\e_{4}
& \e_{1}{+}\e_{2} \\
\hline \hline
\textsc{G}_{2} & \e_{2}{-}\e_{1} & \e_{1}{-}\e_{3},\; \e_{1}{+}\e_{3}{-}2\e_{2} & \e_{1}{-}\e_{2}\\
\hline 
\textsc{G}_{2} & \e_{2}{+}\e_{3}{-}2\e_{1} & \e_{1}{-}\e_{2},\; \e_{1}{-}\e_{2},\; \e_{1}{-}\e_{2},\; \e_{1}{+}\e_{2}{-}2\e_{3}
&
2\e_{1}{-}\e_{2}{-}\e_{3}\\
\hline 
 \end{array}
\end{equation*}
Thus we got 
\begin{prop} The following table yields the index of an $\alphaup$ belonging to an irreducible
root system.
\begin{equation*} 
\begin{array}{| c | c | c | c | c |}
 \hline 
 \text{type} & \alphaup & \ell{=}1 & \ell{=}2 & \ell{\geq}3 \\
 \hline
 \textsc{A}_{\ell} & \text{any} & 0 & 1 & 2\\
 \hline
 \textsc{B}_{\ell} & \text{short} & & 2 & 3 \\
 \hline
 \textsc{B}_{\ell}& \text{long} & & 2 & 4 \\
 \hline
 \textsc{C}_{\ell} &\text{any}& & & 2 \\
 \hline
 \textsc{D}_{\ell} &\text{any}&&& 4 \\
 \hline
 \textsc{E}_{\ell} &\text{any}&&& 4 \\
 \hline
 \textsc{F}_{4} & \text{short} &&& 3\\
 \hline
 \textsc{F}_{4} & \text{long} &&& 4\\
 \hline
 \textsc{G}_{2} & \text{short} &&2 & \\
 \hline
 \textsc{G}_{2} & \text{long} &&4 & \\
 \hline
\end{array}
\end{equation*}
\par
When the index of $\alphaup\,{\in}\,\Rad$ is four, we obtain 
$\alphaup{+}\betaup_{1}{+}\betaup_{2}{+}\betaup_{3}{+}\betaup_{4}\,{=}\,{-}\alphaup.$ \par
For a short root in $\textsc{G}_{2}$ we get $\alphaup{+}\betaup_{1}{+}\betaup_{2}\,{=}\,{-}\alphaup.$ 
\qed
\end{prop}
This yields (see \cite[Theorem\,2.21]{MMN21}): 
\begin{thm}
 Let $(\gs,\qt)$ be a parabolic $CR$-algebra with $\gt$ simple. Then 
\begin{itemize}
 \item If $\gt$ is either of type $\textsc{A}_{\ell}$, $\textsc{B}_{2}$  
 or $\textsc{C}_{\ell}$, then the Levi-order of
 $(\gs,\qt)$ is less or equal two;
  \item If $\gt$ is of one of the types $\textsc{B}_{\ell{\geq}3}$, 
  $\textsc{D}_{\ell\geq{4}}$, $\textsc{F}_{4}$, $\textsc{G}_{2}$,
  then the Levi-order of
 $(\gs,\qt)$ is less or equal three. 
 For each of these types there are $CR$-algebras with Levi-order $3$.
 \qed
\end{itemize}
\end{thm}
In \cite{MMN21} examples of $CR$-algebras of Levi-order $3$ 
are given for each type $\textsc{B}_{\ell{\geq}3}$, $\textsc{D}_{\ell\geq{4}}$,
$\textsc{G}_{2}$. We rehearse below an example for $\textsc{B}_{3}$,
give an extra example for the case $\textsc{F}_{4}$
and examples for $\textsc{B}_{4}$ and non maximal parabolic $\qt$.
\begin{exam}  Let $\qt$ be the maximal parabolic subalgebra 
of a semisimple complex
Lie algebra $\gt$ of type $\textsc{B}_{3}$. \par
With
$\Rad(\textsc{B}_{3})\,{=}\,\{{\pm}\e_{i}\,{\mid}\, i{=}1,2,3\}\,{\cup}\,\{{\pm}\e_{i}{\pm}\e_{j}\,{\mid}\,
1{\leq}i{<}j{\leq}3\}$, we take the parabolic
$\Qq\,{=}\,\{\alphaup\,{\in}\,\Rad(\textsc{B}_{3})\,{\mid}\,\chiup_{\Qq}(\alphaup)\,{\geq}0\}$ with 
$\chiup_{\Qq}(\e_{1}){=}1,$ $\chiup_{\Qq}(\e_{2}){=}1$, $\chiup_{\Qq}(\e_{3}){=}0$ and the involution $\stt$ of $\Rad$ with
$\stt(\e_{1}){=}{-}\e_{1}$, $\stt(\e_{2}){=}{-}\e_{3}$, $\stt(\e_{3}){=}{-}\e_{3}$.\par
The canonical chamber, with basis 
$\alphaup_{1}{=}\e_{1}{-}\e_{2}$, $\alphaup_{2}{=}\e_{2}{-}\e_{3}$, $\alphaup_{3}{=}\e_{3}$,
is $V$-fit for $(\Qq\,,\stt)$ and every $(\gs,\qt)$ with $\sigmaup\,{\in}\,\Invs$, which are 
Levi-non\-de\-gen\-er\-ate, 
having cross-marked diagram 
 
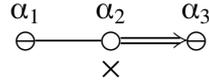
\begin{equation*}
  \xymatrix @M=0pt @R=4pt @!C=8pt{ 
  \alphaup_{1}&\alphaup_{2}&\alphaup_{3}\\
 \ominus 
 \ar@{-}[r] &\medcirc \ar@{=>}[r] &\ominus\\
 &\times}
\end{equation*}
From 
\begin{equation*} 
\begin{array}{ | c | c | c | c |}
 \hline
 \alphaup & \stt(\alphaup) & \chiup_{\Qq}(\alphaup) & \chiup_{\Qq}(\stt(\alphaup)) \\
 \hline
 \e_{1} & {-}\e_{1} & 1 & {-}1\;\, \\
 \hline
 \e_{2} & {-}\e_{3} & 1 & 0\\
 \hline
 \e_{3}& {-}\e_{2} & 0 & {-}1\;\, \\
 \hline
 \e_{1}{+}\e_{2} & {-}\e_{1}{-}\e_{3} & 2 & {-1}\;\; \\
 \hline
 \e_{1}{-}\e_{2} & \e_{3}{-}\e_{1} & 0 & {-}1\;\, \\
 \hline
 \e_{1}{+}\e_{3} & {-}\e_{1}{-}\e_{2} & 1 & {-}2\,\; \\
 \hline
 \e_{1}{-}\e_{3} & \e_{2}{-}\e_{1} & 1 & 0 \\
 \hline
 \e_{2}{+}\e_{3}& {-}\e_{1}{-}\e_{2} & 1 & {-}2\,\; \\
 \hline
 \e_{2}{-}\e_{3} & \e_{2}{-}\e_{3} & 1 & 1 \\
 \hline
\end{array}
\end{equation*}
we obtain that 
\begin{equation*} 
\begin{cases}
 \Qc\cap\,\stt(\Qc)=\{\e_{3}{-}\e_{2}\},\\
 \Qq\,{\cap}\,\stt(\Qc)=\{\e_{1},\e_{3},\e_{1}{+}\e_{2},\e_{1}{-}\e_{2}, \e_{1}{+}\e_{3},\e_{2}{+}\e_{3}\},\\
 \Qq^{n}\cap\stt(\Qc)=\{\e_{1},\e_{1}{+}\,\e_{2},\e_{1}{+}\,\e_{3},\e_{2}{+}\,\e_{3}\},\\
 \Qc\,{\cap}\,\stt(\Qq)=\{{-}\e_{1},{-}\e_{2},{-}\e_{1}{-}\e_{3}, {-}\e_{1}{+}\e_{3},{-}\e_{1}{-}\e_{2},{-}\e_{2}
 {-}\e_{3}\}.
\end{cases}
\end{equation*}
Let us compute Levi-orders it suffices to consider roots in $\Qq^{n}\,{\cap}\,\stt(\Qc)$. We get 
\begin{align*}
& \e_{1}+(\e_{3}{-}\e_{1})+({-}\e_{2})=\e_{3}{-}\e_{2} && \Rightarrow \ord(\e_{1})=2,\\
& (\e_{1}{+}\e_{2})+({-}\e_{2})+({-}\e_{2}){+}(\e_{3}{-}\e_{1})=\e_{3}{-}\e_{2}
&& \Rightarrow \ord(\e_{1}{+}\e_{2})=3,\\
& (\e_{1}{+}\e_{3}){+}({-}\e_{1}{-}\e_{2})=\e_{3}{-}\e_{2}
&& \Rightarrow \ord(\e_{1}{+}\e_{3})=1,\\
& (\e_{2}{+}\e_{3}){+}({-}\e_{2}){+}({-}\e_{2})=\e_{3}{-}\e_{2}
&& \Rightarrow \ord(\e_{2}{+}\e_{3})=2.
\end{align*}
The fact that $\ord(\e_{1}{+}\e_{2})\,{\geq}\,3$  follows because $\chiup_{\Qq}(\e_{1}{+}\e_{2})\,{=}\,2$
and all roots $\betaup$ in $\Qc\,{\cap}\,\stt(\Qq)$ that may be added to $\e_{1}{+}\e_{2}$ 
have $\chiup_{\Qq}(\betaup)\,{=}\,{-}1$.
\end{exam} 
\begin{exam}
  Let $\qt$ be the maximal parabolic subalgebra 
of a semisimple complex
Lie algebra $\gt$ of type $\textsc{F}_{4}$. The root system $\Rad(\textsc{F}_{4})$
is obtained by adding to $\Rad({B}_{4})$ 
the elements $\tfrac{1}{2}({\pm}\e_{1}{\pm}\e_{2}{\pm}\e_{3}{\pm}\e_{4})$. We found convenient to introduce
for these roots the notation 
\begin{align*}
 &\etaup_{0}{=}\,\tfrac{1}{2}(\e_{1}{+}\e_{2}{+}\e_{3}{+}\e_{4}), \quad \etaup_{i}{=}\,\etaup_{0}{-}\e_{i},\;
 1{\leq}i{\leq}4, \quad \etaup_{i,j}{=}\,\etaup_{0}{-}\e_{i}{-}\e_{j},\;
 1{\leq}i{<}j{\leq}4, \\
 &\etaup_{i,j,k}{=}\,\etaup_{0}{-}\e_{i}{-}\e_{j}{-}\e_{k},\;
 1{\leq}i{<}j{<}k{\leq}4,\qquad \etaup_{1,2,3,4}{=}\,{-}\etaup_{0}.
\end{align*}\par
 We consider the parabolic $\Qq$ in $\Rad(\textsc{F}_{4})$ with $\chiup_{\Qq}(\e_{1}){=}1$,
 $\chiup_{\Qq}(\e_{2}){=}1$, $\chiup_{\Qq}(\e_{3}){=}0$, $\chiup_{\Qq}(\e_{4}){=}2$ and the involution
 $\stt$ 
 with $\stt(\e_{1}){=}{-}\e_{1}$, 
 $\stt(\e_{2}){=}{-}\e_{3}$,
 $\stt(\e_{3}){=}{-}\e_{2}$, \par\noindent
 $\stt(\e_{4}){=} {-}\e_{4}.$
 Then 
 $ \alphaup_{1}{=}\e_{1}{-}\e_{2}$,  $\alphaup_{2}{=}\e_{2}{-}\e_{3}$, $\alphaup_{3}{=}\e_{3}$,
 $\alphaup_{4}{=}\tfrac{1}{2}(\e_{4}{-}\e_{1}{-}\e_{2}{-}\e_{3})$
 are the simple roots of a $V$-fit chamber $C$ for $(\Qq\,,\stt)$. We obtain indeed the cross-marked
 diagram  
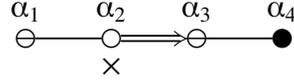
\begin{equation*}
  \xymatrix @M=0pt @R=4pt @!C=8pt{ 
  \alphaup_{1}&\alphaup_{2}&\alphaup_{3}&\alphaup_{4}\\
  \ominus \ar@{-}[r]  & \medcirc 
 \ar@{=>}[r] &\ominus \ar@{-}[r] &\medbullet\\
 &\times}
\end{equation*}
Since $\qt$ is maximal, it follows that $(\Qq\,,\stt)$ and 
$(\gs,\qt)$, for all $\sigmaup\,{\in}\,\Invs$,  is fundamental and 
Levi-nondegenerate.
From the table: 
\begin{equation*} 
\begin{array}{| c | c | c | c || c | c | c | c |}
 \hline
 \alphaup & \stt(\alphaup) & {\begin{smallmatrix} \chiup_{\Qq}(\alphaup)
 \end{smallmatrix}} &  {\begin{smallmatrix} \chiup_{\Qq}(\alphaup)
 \end{smallmatrix}}  &
 \alphaup & \stt(\alphaup) & {\begin{smallmatrix} \chiup_{\Qq}(\alphaup)
 \end{smallmatrix}} &  {\begin{smallmatrix} \chiup_{\Qq}(\alphaup)
 \end{smallmatrix}} \\
 \hline
\e_{1} & {-}\e_{1} & 1 & {-1} & \e_{4}{+}\e_{2} & {-}\e_{4}{-}\e_{3} & 3 & {-}2\,\; \\
 \hline
  \e_{2} & {-}\e_{3} & 1 & 0 & \e_{4}{-}\e_{2} & {-}\e_{4}{+}\e_{3} & 1 & -2\,\; \\
 \hline
  \e_{3}&{-}\e_{2}& 0 & {-1} &  \e_{4}{+}\e_{3} & {-}\e_{4}{-}\e_{2} & 2 & {-}3\,\; \\
 \hline
  \e_{4}&{-}\e_{4} & 2 & {-}2 &  \e_{4}{-}\e_{3} & {-}\e_{4}{+}\e_{2}& 2 & -1\,\; \\
 \hline
  \e_{1}{+}\e_{2} & {-}\e_{1}{-}\e_{3}& 2 & {-1}\;\; &  \etaup_{0} & \etaup_{1,2,3,4} & 2 & {-}2\,\; \\
 \hline
  \e_{1}{-}\e_{2}&{-}\e_{1}{+}\e_{3} & 0 & {-}1\;\; &  \etaup_{1} & \etaup_{2,3,4} & 1 & {-1}\;\; \\
 \hline
  \e_{1}{+}\e_{3} & {-}\e_{1}{-}\e_{2} & 1 & {-}2\,\; &  \etaup_{2} & \etaup_{1,2,4} & 1 & {-}2 \,\; \\
 \hline 
  \e_{1}{-}\e_{3} & {-}\e_{1}{+}\e_{2} & 1 & 0 &  \etaup_{3} & \etaup_{2,3,4} & 2 & {-1}\;\; \\
 \hline
  \e_{4}{+}\e_{1} & {-}\e_{4}{-}\e_{1} & 3 & {-}3\,\; &  \etaup_{1,2} & \etaup_{2,4} & 0 & {-1}\;\; \\
 \hline
  \e_{4}{-}\e_{1} & {-}\e_{4}{+}\e_{1} & 1 & {-1}\;\; &  \etaup_{1,3} & \etaup_{3,4} & 1 & 0 \\
 \hline
  \e_{2}{+}\e_{3} & {-}\e_{2}{-}\e_{3}& 1 & {-1}\;\; & \etaup_{2,3} & \etaup_{1,4} & 1 & {-1}\;\; \\
 \hline
  \e_{2}{-}\e_{3} & \e_{2}{-}\e_{3} & 1 & 1 &  \etaup_{1,2,3} & \etaup_{4} & 0 & 0 \\
 \hline
\end{array}
\end{equation*}
We have $\Qq^{c}{\cap}\,\stt(\Qq^{c})\,{=}\,\{\e_{3}{-}\e_{2}\}$. The
element $(\e_{1}{+}\e_{2})\,{\in}\,\Qn\,{\cap}\,\stt(\Qc)$ has Levi-order $3$.
Noting that $(\e_{3}{-}\e_{2}){-}(\e_{1}{+}\e_{2})$ is not a root of $\Rad(\textsc{F}_{4})$,
the elements which can compose a minimal Levi sequence for $(\e_{1}{+}\e_{2})$ belong to
\begin{equation*}
 \{\betaup\,{\in}\,\Qc\,{\cap}\,\stt(\Qq)\,{\mid}\,(\e_{1}{+}\e_{2}){+}\betaup\,{\in}\,\Qq\,{\cap}\,\stt(\Qc)\}
 =\{{-}\e_{2},\e_{3}{-}\,\e_{1}\},
\end{equation*}
and, since $(\e_{1}{+}\e_{2})$, $\e_{2}$ and $(\e_{3}{-}\e_{1})$ are linearly independent, 
\begin{equation*}
 (\e_{1}{+}\e_{2})+({-}\e_{2}){+}({-}\e_{2}){+}(\e_{3}{-}\e_{1})=\e_{3}{-}\e_{2}
\end{equation*}
is the unique possible minimal Levi-sequence for $(\e_{1}{+}\e_{2})$. \par
The largest root $(\e_{1}{+}\e_{4})$ is purely imaginary and belongs to $\Qn\,{\cap}\,\stt(\Qc)$.
Since $(\e_{3}{-}\e_{2})\,{-}\,(\e_{1}{+}\e_{4})\,{=}\,2\etaup_{1,2,4}$, with $\etaup_{1,2,4}{\in}\,\Qq^{c}{\cap}\,\stt(\Qn)$,
the Levi order of $(\e_{1}{+}\e_{4})$ is two. This shows that the Levi order of the $CR$ algebra can  be
different from that of the maximal root in $\Qn{\cap}\stt(\Qc)$.
\end{exam} 
\begin{exam} All examples of  Levi order three we gave so far have a maximal parabolic $\qt$. 
 Our further example, in which we consider again 
a Lie algebra of type $\textsc{F}_{4}$, produces a 
Levi-nondegenerate
 $CR$-algebra $(\gs,\qt)$ of Levi-order three 
 with a non maximal $\qt$. We keep the notation introduced in the previous
 example and consider the parabolic $\qt$ with parabolic set $\Qq\,{=}\,\{\alphaup\,{\in}\,\Rad(\textsc{F}_{4})
 \,{\mid}\,\chiup_{\Qq}(\alphaup)\,{\geq}\,0\}$ for $\chiup_{\Qq}(\e_{1}){=}1$, 
 $\chiup_{\Qq}(\e_{2}){=}1$, $\chiup_{\Qq}(\e_{3}){=}0$, $\chiup_{\Qq}(\e_{4}){=}4$, and the 
 involution $\stt$  with $\stt(\e_{1}){=}{-}\e_{4}$, $\stt(\e_{2}){=}{-}\e_{3}$, 
 $\stt(\e_{3}){=}{-}\e_{2}$, $\stt(\e_{4}){=}{-}\e_{1}$. 
 In the  basis 
  $ \alphaup_{1}{=}\e_{1}{-}\e_{2}$,  $\alphaup_{2}{=}\e_{2}{-}\e_{3}$, $\alphaup_{3}{=}\e_{3}$,
 $\alphaup_{4}{=}\tfrac{1}{2}(\e_{4}{-}\e_{1}{-}\e_{2}{-}\e_{3})$ ($V$\!{-}fit)
 and in the basis 
  $ \alphaup'_{1}{=}\e_{2}{-}\e_{1}$,  $\alphaup'_{2}{=}\e_{1}{+}\e_{3}$, $\alphaup'_{3}{=}{-}\e_{3}$,
 $\alphaup'_{4}{=}\tfrac{1}{2}(\e_{4}{-}\e_{1}{-}\e_{2}{+}\e_{3})$ ($S$\!{-}fit)
we get  the cross-marked diagrams
  
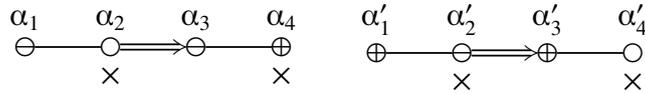
\begin{gather*}
  \xymatrix @M=0pt @R=4pt @!C=8pt{ 
  \alphaup_{1}&\alphaup_{2}&\alphaup_{3}&\alphaup_{4}\\
  \ominus \ar@{-}[r]  & \medcirc
 \ar@{=>}[r] &\ominus  \ar@{-}[r] &\oplus\\
 &\times&&\times}
 \qquad
  \xymatrix @M=0pt @R=4pt @!C=8pt{ 
  \alphaup'_{1}&\alphaup'_{2}&\alphaup'_{3}&\alphaup'_{4}\\
  \oplus \ar@{-}[r]  & \ominus
 \ar@{=>}[r] &\oplus  \ar@{-}[r] &\medcirc\\
 &\times&&\times}
\end{gather*}
The first diagram shows that
$(\Qq\,,\stt)$ is Levi-nondegenerate and the second, since  
$\stt(\alphaup'_{1})\,{=}\,\alphaup'_{1}{+}2\alphaup'_{2}{+}
4\alphaup'_{3}{+}2\alphaup'_{4}$, shows that $(\Qq\,,\stt)$ is
fundamental. We have the table:
\begin{equation*} 
\begin{array}{| c | c | c | c || c | c | c | c |}
 \hline
 \alphaup & \stt(\alphaup) & \begin{smallmatrix}\chiup_{\Qq}(\alphaup) 
 \end{smallmatrix}& \begin{smallmatrix}\chiup_{\Qq}(\stt(\alphaup)) \end{smallmatrix}
 & \alphaup & \stt(\alphaup) & \begin{smallmatrix}\chiup_{\Qq}(\alphaup) 
 \end{smallmatrix}& \begin{smallmatrix}\chiup_{\Qq}(\stt(\alphaup)) \end{smallmatrix}\\
 \hline
 \e_{1} & {-}\e_{4} & 1 & {-}4\,\; &  \e_{4}{+}\e_{2} & {-}\e_{1}{-}\e_{3} & 5 & {-1}\;\; \\
 \hline
 \e_{2} & {-}\e_{3} & 1 & 0 & \e_{4}{-}\e_{2} & \e_{3}{-}\e_{1} & 3 & -1\,\; \\
 \hline
 \e_{3}&{-}\e_{2}& 0 & {-1}\;\; & \e_{4}{+}\e_{3} & {-}\e_{1}{-}\e_{2} & 4 & {-}2\,\; \\
 \hline
 \e_{4}&{-}\e_{1} & 4 & {-1}\;\; &  \e_{4}{-}\e_{3} & \e_{2}{-}\e_{1}& 4 & 0\,\; \\
 \hline
 \e_{1}{+}\e_{2} & {-}\e_{4}{-}\e_{3}& 2 & {-}4\,\; &\etaup_{0} & \etaup_{1,2,3,4} & 3 & {-}3\,\; \\
 \hline
 \e_{1}{-}\e_{2}&\e_{4}{+}\e_{2} & 0 & {-}4\,\; & \etaup_{1} & \etaup_{1,2,3} & 2 & 1 \\
 \hline
 \e_{1}{+}\e_{3} & {-}\e_{2}{-}\e_{4} & 1 & {-}5\,\; & \etaup_{2} & \etaup_{1,2,4} & 2 & {-}3 \,\; \\
 \hline
 \e_{1}{-}\e_{3} & \e_{2}{-}\e_{4} & 1 & {-}3\,\; &  \etaup_{3} & \etaup_{1,3,4} & 3 & {-}2\,\; \\
 \hline
 \e_{4}{+}\e_{1} & {-}\e_{4}{-}\e_{1} & 5 & {-}5\,\; &  \etaup_{1,2} & \etaup_{1,2} & 1 & 1 \\
 \hline
 \e_{4}{-}\e_{1} & \e_{4}{-}\e_{1} & 3 & 3 \,\; &  \etaup_{1,3} & \etaup_{1,3} & 2 & 2 \\
 \hline
 \e_{2}{+}\e_{3} & {-}\e_{2}{-}\e_{3}& 1 & {-1}\;\; &  \etaup_{2,3} & \etaup_{1,4} & 2 & {-}2\,\; \\
 \hline
 \e_{2}{-}\e_{3} & \e_{2}{-}\e_{3} & 1 & 1 &  \etaup_{1,2,3} & \etaup_{1} & 2 & 2 \\
 \hline
\end{array}
\end{equation*}\par
In particular,  $\Qc{\cap}\,\stt(\Qc)\,{=}\,\{\e_{1}{-}\e_{4},\e_{3}{-}\e_{2},\etaup_{2,3,4},\etaup_{2,4},\etaup_{3,4},
\etaup_{4}\}$.\par
We claim that $\gammaup\,{=}\,(\e_{2}{+}\e_{4})$ has Levi-order three. Indeed, by subtracting 
$\gammaup$ from $\Qc\,{\cap}\,\stt(\Qc)$, as this set does not contain any root of 
$\Qc\,{\cap}\,\stt(\Qq)$, we show that $\ord(\gammaup)\,{>}\,1$. Then the candidate
roots of  $\Qc\,{\cap}\,\stt(\Qq)$ to be part of a minimal Levi sequence are those in 
$$
 \{\betaup\,{\in}\,\Qc\,{\cap}\,\stt(\Qq)\,{\mid}\,\betaup{+}\gammaup\,{\in}\,\Qq\,{\cap}\,\stt(\Qc)\}=
 \{{-}\e_{2},\e_{3}{-}\e_{4}\}
$$
and 
$
 \gammaup+({-}\e_{2}){+}({-}\e_{2}){+}(\e_{3}{-}\e_{4})=\e_{3}{-}\e_{2}
$
shows that $\ord(\gammaup)\,{=}\,3$.
\end{exam}

\begin{exam}
  Let $\qt$ be the non maximal parabolic subalgebra 
of a semisimple complex
Lie algebra $\gt$ of type $\textsc{B}_{4}$, associated with the parabolic set
$\Qq\,{=}\,\{\alphaup\,{\in}\,\Rad(\textsc{B}_{4})\,{\mid}\,\chiup_{\Qq}(\alphaup)\,{\geq}\,0\}$, for
$\chiup_{\Qq}(\e_{1}){=}2$, $\chiup_{\Qq}(\e_{2}){=}2$, $\chiup_{\Qq}(\e_{3}){=}1$, $\chiup_{\Qq}(\e_{4}){=}0$.
Consider the symmetry $\stt$ of the root system with 
$\stt(\e_{1})={-}\e_{2}$, $\stt(\e_{2})={-}\e_{1}$, $\stt(\e_{3})={-}\e_{4}$, $\stt(\e_{4})={-}\e_{3}$. 
The canonical chamber, with basis 
$\alphaup_{1}{=}\e_{1}{-}\e_{2}$, $\alphaup_{2}{=}\e_{2}{-}\e_{3}$, $\alphaup_{3}{=}\e_{3}{-}\e_{4}$,
$\alphaup_{4}{=}\e_{4}$ 
is $V$-fit for $(\Qq\,,\stt)$, while we obtain an $S$-fit chamber in the basis 
$\alphaup'_{1}{=}\e_{1}{-}\e_{3}$, $\alphaup'_{2}{=}\e_{3}{-}\e_{2}$, $\alphaup'_{3}{=}\e_{2}{+}\e_{4}$,
$\alphaup'_{4}{=}{-}\e_{4}$.
The corresponding  $S$ and $V$-diagrams are
 
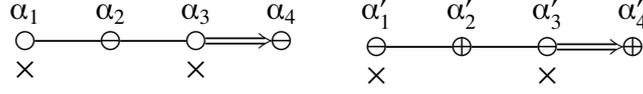
\begin{equation*}
  \xymatrix @M=0pt @R=4pt @!C=8pt{ 
  \alphaup_{1}&\alphaup_{2}&\alphaup_{3}&\alphaup_{4}\\
 \medcirc \ar@{-}[r] &\ominus 
 \ar@{-}[r] &\medcirc \ar@{=>}[r] &\ominus\\
 \times&&\times}\qquad 
   \xymatrix @M=0pt @R=4pt @!C=8pt{ 
  \alphaup'_{1}&\alphaup'_{2}&\alphaup'_{3}&\alphaup'_{4}\\
\ominus \ar@{-}[r] &\oplus 
 \ar@{-}[r] &\ominus \ar@{=>}[r] &\oplus\\
 \times&&\times}
\end{equation*}
The first shows that $(\Qq\,,\stt)$ is Levi-nondegenerate, $\stt(\alphaup'_{2})\,{=}\alphaup'_{1}{+}
\alphaup'_{2}{+}\alphaup'_{3}{+}2\alphaup'_{4}$ shows that $(\Qq\,,\stt)$ is fundamental. We get
\begin{equation*} 
\begin{array}{| c | c | c | c || c | c | c | c |}
 \hline
 \alphaup & \stt(\alphaup) &\begin{smallmatrix} \xiup_{\Qq}(\alphaup) 
 \end{smallmatrix}& \begin{smallmatrix}\xiup_{\Qq}(\stt(\alphaup)) \end{smallmatrix} &
 \alphaup & \stt(\alphaup) &\begin{smallmatrix} \xiup_{\Qq}(\alphaup) 
 \end{smallmatrix}& \begin{smallmatrix}\xiup_{\Qq}(\stt(\alphaup)) \end{smallmatrix} 
 \\
 \hline
 \e_{1} & {-}\e_{2} & 2 & {-1}\;\; &  \e_{1}{+}\e_{4} & {-}\e_{2}{-}\e_{3} & 2 & {-}2 \\
 \hline
 \e_{2} & {-}\e_{1} & 1 & {-}2\,\; &  \e_{1}{-}\e_{4} & {-}\e_{2}{+}\e_{3} & 2 & 0 \\
 \hline
 \e_{3}&{-}\e_{4}& 1 & 0 &  \e_{2}{+}\e_{3} & {-}\e_{1}{-}\e_{4}& 2 & {-}2\,\;\\
 \hline
 \e_{4}&{-}\e_{3} & 0 & {-1}\;\; &  \e_{2}{-}\e_{3} & {-}\e_{1}{+}\e_{4} & 0 & {-}2 \;\;\\
 \hline
 \e_{1}{+}\e_{2} & {-}\e_{1}{-}\e_{2}& 3 & {-}3\,\; &  \e_{2}{+}\e_{4} & {-}\e_{1}{-}\e_{3} & 1 & {-}3\,\; \\
 \hline
 \e_{1}{-}\e_{2}&\e_{1}{-}\e_{2} & 1 & 1 &  \e_{2}{-}\e_{4} & {-}\e_{1}{+}\e_{3} & 1 & {-1}\;\; \\
 \hline
 \e_{1}{+}\e_{3} & {-}\e_{2}{-}\e_{4} & 3 & {-1}\;\;&  \e_{3}{+}\e_{4} & {-}\e_{3}{-}\e_{4} & 1 & {-1}\;\; \\
 \hline
 \e_{1}{-}\e_{3} & {-}\e_{2}{+}\e_{4} & 1 & {-1}\;\; &  \e_{3}{-}\e_{4} & \e_{3}{-}\e_{4}& 1 & 1 \\
 \hline
\end{array}
\end{equation*}\par\smallskip
We claim that $\ord(\e_{1}{+}\e_{3})\,{=}\,3$. Indeed, $\Qq^{c}{\cap}\,\stt(\Qq^{c})\,{=}\,\{\e_{2}{-}\e_{1},\,
\e_{4}{-}\e_{3}\}$ and 
\begin{equation*}\tag{$*$}
 (\e_{2}{-}\e_{1})\,{-}\,(\e_{1}{+}\e_{3})\,{=}\,\e_{2}{-}\,2\e_{1}{-}\,\e_{3},\quad  ( \e_{4}{-}\e_{3})-(\e_{1}{+}\e_{3})
 \,{=}\, \e_{4}{-}\,\e_{1}{-}\,2\e_{3}.
\end{equation*}
Since the right hand sides of $(*)$ are not roots and the left hand sides are the only decompositions
of the right hand sides as sums of two roots,  the Levi order of $(\e_{1}{+}\e_{3})$ must be at least three
and indeed 
$$
 (\e_{1}{+}\e_{3}){+}({-}\e_{3}){+}({-}\e_{3}){+}(\e_{4}{-}\e_{1})\,{=}\,\e_{4}{-}\e_{3}
$$
shows that $\ord(\e_{1}{+}\e_{3})\,{=}\,3$. 
Therefore, for any $\sigmaup\,{\in}\,\Invs$, the
$CR$-algebra $(\gs,\qt)$ is fundamental, 
Levi-nondegenerate with Levi-order  three. 
\end{exam}

\subsection{Minimal type} Parabolic $CR$-algebras of minimal orbits have 
Levi-order less than or equal to two (see e.g. \cite{MMN21, mmn2022}). Since $CR$-algebras
associated to the same pair $(\Qq\,,\stt)$ have equal Levi-orders, this bound applies
to a larger set of orbits, that we called in \cite{MMN21} \emph{of minimal type}. 
\begin{dfn}\label{d5.3}
 A triple $(\gs,\qt,\hst)$, consisting of a parabolic $CR$-algebra $(\gs,\qt)$ and an adapeted
 Cartan subalgebra $\hst$ is said to be \emph{of the minimal type} if the associated
 pair $(\Qq\,,\stt)$ satisfies 
\begin{equation}\label{e5.5}
 \Qq^{n}\cap\stt(\Qq^{c})\subset\Rad^{\stt}_{\,\;\bullet}.
\end{equation}
\end{dfn}
\begin{rmk}
A major  flaw of this definition is its dependence  on the involution $\stt$ of the root
system, which in turn depends on the choice of the admissible Cartan subalgebra
$\hst$, as we illustrate 
by the following example.  
\begin{exam}\label{ex8.1}
 Let us consider the orbits of $\SU(3,3)$ in the Grassmannian $\sfG$ of $3$-planes of $\C^{6}$.
 The minimal orbit is a totally real submanifold of real dimension $9$, consisting of the totally
 isotropic $3$-planes of $\C^{6}$. \par
 Consider the intermediate orbit $\sfM,$ consisting
 of isotropic $3$-planes containing a $2$-plane of signature $(1,1).$ It is 
 a Levi-nondegenerate $CR$-sub\-man\-i\-fold
 of hypersurface type
 $(8,1)$. Let us give two alternative
 descriptions of $\sfM$, involving different choices of adapted Cartan subalgebras. \par
First we consider  a description utilising a 
 maximally vector 
 Cartan subalgebra. 
The corresponding involution $\stt$ of  
$\Rad\,{=}\,\{\pm(\e_{i}{-}\e_{j})\,{\mid}\,1{\leq}{i}{<}j{\leq}6\}$
is described by $\stt(\e_{i})\,{=}\,{-}\e_{i+1}$ for $i{=}1,3,5$ and  $\stt(\e_{i})\,{=}\,{-}\e_{i-1}$ for $i{=}2,4,6$.
Then $\Rad_{\,\;\bullet}^{\stt}\,{=}\,\emptyset$, while $\Qq^{n}{\cap}\,\stt(\Qq^{c}){=}\{\e_{i}{-}\e_{j}\,{\mid}\,
i{=}1,2\,j{=}5,6\}\,{\neq}\emptyset$ shows that \eqref{e5.5} is not satisfied. \par
We can choose  an adapted $\hg_{\stt'}$ corresponding to the involution $\stt'$ defined by
$\stt(\e_{1}){=}{-}\e_{6}$, $\stt(\e_{6}){=}{-}\e_{1}$, $\stt(\e_{i}){=}{-}\e_{i}$ for $2{\leq}i{\leq}5$,
by taking a basis of $\C^{6}$ in which, for a hermitian symmetric form of signature $(3,3)$, 
$\{\e_{1},\e_{6}\}$ is the isotropic basis of a 
hyperbolic $2$-plane, $\{\e_{2},\e_{4}\}$ the orthogonal basis of a positive $2$-plane
and $\{\e_{3},\e_{5}\}$ the orthogonal basis of a negative $2$-plane. In this case
$\Qq^{n}\,{\cap}\,\stt'(\Qq^{c}){=}\,\{\e_{i}{-}\e_{j}\,{\mid}\, i{=}2,3,\,j{=}4,5\}\,{\subset}\,\Rad^{\stt'}_{\,\;\bullet}$
and \eqref{e5.5} is satisfied.
\end{exam}
\end{rmk}
\subsection{Weak integrability}
We give in this subsection a new notion which generalises the one given in 
\cite{MMN21} and 
is independent of the choice of adapted Cartan subalgebras. 
\begin{prop}
 For a parabolic $CR$ algebra the following are equivalent 
\begin{equation}\label{e5.6}\vspace{-18pt}
\begin{cases}
(\mathrm{a})\; \qt\cap\sigmaup(\qt)+\qt^{n}+\sigmaup(\qt^{n}) \;\;\text{is a Lie subalgebra of $\gt$};\\
(\mathrm{b})\; [\qt^{n},\sigmaup(\qt^{n})]\subset \qt\cap\sigmaup(\qt)+\qt^{n}+\sigmaup(\qt^{n});\\
(\mathrm{c})\; \text{the polarized $(\gs,\qt_{[\sigmaup]})$ is integrable.}
 \end{cases}
\end{equation}\qed

\end{prop}
\begin{dfn}\label{d5.4} We call 
\emph{weakly integrable}  a parabolic $CR$ algebra $(\gs,\qt)$ satisfying the equivalent conditions
 \eqref{e5.6}.
\end{dfn} 
\begin{rmk} We denoted by $\qt_{[\sigmaup]}$ the parabolic subalgebra
$\qt\,{\cap}\,\sigmaup(\qt){+}\qt^{n}$, which is the lifted $CR$-structure of the
polarization of $(\gs,\qt)$. We showed in \S{\ref{s.pol}} how to compute polarization
by using the root system of the complexification of an adapted Cartan subalgebra and the simple roots
of an $S$-fit Weyl chamber. When $(\gs,\qt)$ is weakly integrable,
the right hand side of \eqref{e5.6}(b) is the lifted $CR$ structure $\qt'$ of the basis of the fundamental
reduction of $(\gs,\qt_{[\sigmaup]})$, otherwise is properly contained in $\qt'$.  
 \par
This follows because
 the nilradical of $\qt_{[\sigmaup]}$ is 
 $(\qt^{r}\,{\cap}\,\sigmaup(\qt^{n}))
 \,{\oplus}\,\qt^{n}$ and 
 $$\left(\qt_{[\sigmaup]}{\cap}\sigmaup(\qt_{[\sigmaup]})\right)\,{+}\,(\qt_{[\sigmaup]})^{n}
 +\sigmaup\left( (\qt_{[\sigmaup]})^{n}\right)=
 \qt_{[\sigmaup]}\,{+}\,\sigmaup(\qt_{[\sigmaup]}).$$ 
\end{rmk}
\begin{lem} Let $(\gs,\qt)$ be a parabolic $CR$-algebra, $\hst$ an adapted Cartan subalgebra and 
$\Qq$ the parabolic set of $\qt$ for the complexification $\hg$ of $\hst$. 
Then $(\gs,\qt)$ is weakly integrable if and only if 
\begin{equation}\label{e5.7}
 \{\alphaup-\betaup\in\Rad \mid \alphaup,\betaup\in\Qq^{n}\cap\stt(\Qq^{c})\}\subseteq (\Qq\cap\stt(\Qq))
 \cup\Qq^{n}\cup\stt(\Qq^{n}).
\end{equation}
\end{lem}
\begin{proof} Equation \eqref{e5.6}$(\mathrm{b})$ translates into 
\begin{equation*}\tag{$*$}
 \{\alphaup{+}\stt(\betaup)\,{\in}\,\Rad\,{\mid}\,\alphaup,\betaup\,{\in}\,\Qq^{n}\}\subseteq
( \Qq\,{\cap}\,\stt(\Qq))\cup\Qq^{n}\cup\stt(\Qq^{n}).
\end{equation*}
Let $\alphaup,\betaup\,{\in}\,\Qq^{n}$. If one of the two roots belongs to $\stt(\Qq)$,
then $\alphaup{+}\stt(\betaup)$ trivially belongs to the right hand side of $(*)$.
 Then it is necessary and sufficient to require that roots  which are sums
 $\alphaup{+}\stt(\betaup)$, with $\alphaup,\betaup\,{\in}\,\Qq^{n}\,{\cap}\,\stt(\Qq^{c})$,
belong to the right hand side of $(*)$, and
the equivalence with \eqref{e5.7} follows because $
\stt(\Qq^{n}\,{\cap}\,\stt(\Qq^{c}))\,{=}\,\{{-}\betaup\,{\mid}\,\betaup\,{\in}\,\Qq^{n}\,{\cap}\,\stt(\Qq^{c})\}.
$
\end{proof}
\begin{prop} If the triple $(\gs,\qt,\hst)$ is of the minimal type, then $(\gs,\qt)$ is
weakly integrable.
\end{prop}
\begin{proof}
This follows because 
 $\Rad^{\stt}_{\,\;\bullet}\,{\subseteq}\,(\Qq\,{\cap}\,\stt(\Qq))\,{\cup}\,\Qq^{n}\,{\cup}\,
 \stt(\Qq^{n})$. 
\end{proof} 
\begin{thm} The Levi-order of a  weakly integrable 
 parabolic $CR$-algebra  is either ${+}\infty$ or
less or equal two.
\end{thm}
\begin{proof} Let $(\gs,\qt)$ be a  weakly integrable parabolic $CR$ algebra,
$\hst$ an admissible Cartan subalgebra of $(\gs,\qt)$, $\hg$ the complexification of $\hst$ 
and 
$\Qq$ the parabolic set of $\qt$ in the root system $\Rad(\gt,\hg)$.\par
 Assume by contradiction 
 that there is a root $\alphaup\,{\in}\,\Qq\,{\cap}\,\stt(\Qq^{c})$ having Levi-order larger
 or equal three 
and let $(\betaup_{1},\betaup_{2},\hdots,\betaup_{r})$ be a minimal Levi-sequence in $\Qq^{c}{\cap}\,\stt(\Qq)$
for $\alphaup$. Since roots in $\Qq^{r}{\cap}\,\stt(\Qq^{c})$ have Levi order one, $\alphaup$ and
$\gammaup_{i}{=}\alphaup{+}\betaup_{i}$, 
for $1{\leq}i{\leq}r$, 
belong to $\Qq^{n}{\cap}\,\stt(\Qq^{c})$.
Then by \eqref{e5.7} the roots $\betaup_{i}{=}\gammaup_{i}{-}\alphaup$ belong to 
$(\Qq\cap\stt(\Qq))
 \cup\Qq^{n}\cup\stt(\Qq^{n})$ and hence to $\Qq^{c}{\cap}\,\stt(\Qq^{n})\,{=}\,\stt(\Qq^{n}{\cap}\,\stt(\Qq^{c}))$. 
By definition of  minimal 
type this implies that $\alphaup{+}\betaup_{1}{+}\betaup_{2}{+}{\cdots}{+}\betaup_{r}$
belongs to  $(\Qq\cap\stt(\Qq))
 \cup\Qq^{n}\cup\stt(\Qq^{n})$, contradicting the fact that $(\betaup_{1},\betaup_{2},\hdots,\betaup_{r})$
 is a Levi-sequence. The proof is complete.
\end{proof}
\begin{exam}
 Take in the root system  $\Rad{=}\{{\pm}\e_{1},{\pm}\e_{2}\}{\cup}\,\{{\pm}\e_{1}{\pm}\e_{2}\}$,
 of type $\textsc{B}_{2}$,
 the Borel subalgebra $\Qq\,{=}\,\Rad^{+}(C)\,{=}\{\e_{1},\e_{2},\e_{1}{\pm}\e_{2}\}$ 
 and the involution $\stt$ with $\stt(\e_{1}){=}{-}\e_{2},$ $\stt(\e_{2}){=}{-}\e_{1}$. 
 Then $\Qq\,{\cap}\,\stt(\Qq){\cup}\Qq^{n}{\cup}\stt(\Qq^{n})\,{=}\,\Rad(\textsc{B}_{2}){\setminus}\{\e_{2}{-}\e_{1}\}$
 is not a closed set of roots  and therefore $(\Qq,\stt)$ is not weakly integrable. 
\end{exam}
\begin{exam} We consider, for a root system $\Rad\,{=}\,\{{\pm}\e_{i}{\pm}\e_{j}\,{\mid}\,1{\leq}i{<}j{\leq}4\}$
of type~$\textsc{D}_{4}$, the pair $(\Qq\,,\stt)$ consisting of the parabolic $\Qq{=}\{\alphaup\,{\in}\,\Rad
\,{\mid}\,(\alphaup\,|\,\e_{1}{+}\e_{2}){\geq}0\}$ and the involution $\stt$ of $\Rad$ with
$\stt(\e_{1}){=}{-}\e_{3}$, $\stt(\e_{2}){=}\e_{4}$, $\stt(\e_{3}){=}{-}\e_{1}$, $\stt(\e_{4}){=}\e_{2}$. 
The canonical chamber with $\Bz(C)\,{=}\,\{\alphaup_{i}{=}\e_{i}{-}\e_{i+1}\,{\mid}\,1{\leq}i{\leq}3\}\,{\cup}\,
\{\alphaup_{4}{=}\e_{3}{+}\e_{4}\}$ is $V$-fit, the one with $\Bz(C')\,{=}\,\{\alphaup'_{1}{=}\e_{2}{-}\e_{1},\,
\alphaup_{2}'{=}\e_{1}{-}\e_{4},\, \alphaup'_{3}{=}\e_{4}{-}\e_{3},\, \alphaup'_{4}{=}\e_{4}{+}\e_{3}\}$
is $S$-fit and we obtain the diagrams 
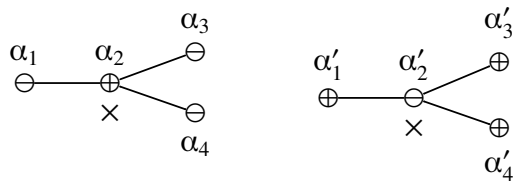
\begin{equation*}
  \xymatrix @M=0pt @R=4pt @!C=8pt{ 
  &&\alphaup_{3}&\\
 \alphaup_{1} &\alphaup_{2}& \ominus \\
  \ominus\ar@{-}[r] & \oplus \ar@{-}[ru] \ar@{-}[rd]\\
  &\times& \ominus\\ 
&& \alphaup_{4}}\quad
  \xymatrix @M=0pt @R=4pt @!C=8pt{ 
  &&\alphaup'_{3}&\\
 \alphaup'_{1} &\alphaup'_{2}& \oplus \\
  \oplus\ar@{-}[r] & \ominus \ar@{-}[ru] \ar@{-}[rd]\\
  &\times& \oplus\\ 
&& \alphaup'_{4}}
\end{equation*}
The first shows that $(\Qq\,,\stt)$ is Levi-nondegenerate. 
We get 
\begin{equation*} 
\begin{array}{| c | c | c | c || c | c | c | c |}
 \hline
 \alphaup & \stt(\alphaup) &\begin{smallmatrix} \xiup_{\Qq}(\alphaup) 
 \end{smallmatrix}& \begin{smallmatrix}\xiup_{\Qq}(\stt(\alphaup)) \end{smallmatrix} &
 \alphaup & \stt(\alphaup) &\begin{smallmatrix} \xiup_{\Qq}(\alphaup) 
 \end{smallmatrix}& \begin{smallmatrix}\xiup_{\Qq}(\stt(\alphaup)) \end{smallmatrix} 
 \\
 \hline
 \e_{1}{-}\e_{2}& \e_{3}{-}\e_{4} & 0 & 0 & \e_{2}{-}\e_{3} & \e_{1}{+}\e_{4} & 1 & 1\\
 \hline
 \e_{1}{+}\e_{2}&\e_{4}{-}\e_{3} & 2 & 0 & \e_{2}{+}\e_{3}& \e_{4}{-}\e_{1}& 1 & {-}1\,\; \\
\hline
\e_{1}{-}\e_{3}& \e_{1}{-}\e_{3} & 1 & 1 & \e_{2}{-}\e_{4}& \e_{4}{-}\e_{2} & 1 & {-}1\,\;\\
\hline
\e_{1}{+}\e_{3}& {-}\e_{1}{-}\e_{3}& 1 & {-}1\,\; & \e_{2}{+}\e_{4} & \e_{2}{+}\e_{4} & 1 & 1\\
\hline
\e_{1}{-}\e_{4} & {-}\e_{2}{-}\e_{3}& 1 & {-}1\,\; & \e_{3}{-}\e_{4} &{-}\e_{1}{-}\e_{2}& 0 &{-}2\,\;\\
\hline
\e_{1}{+}\e_{4} & \e_{2}{-}\e_{3} & 1 & 1 & \e_{3}{+}\e_{4} & \e_{2}{-}\e_{1} & 0 & 0 \\
\hline 
\end{array}
\end{equation*}
from which we obtain that 
\begin{equation*} 
 \Qq^{n}\cap\stt(\Qq^{c})=\{\e_{1}{+}\e_{3},\e_{1}{-}\e_{4},\e_{2}{+}\e_{3},\e_{2}{-}\e_{4}\}.
\end{equation*}
The roots which are differences of roots in $ \Qq^{n}{\cap}\,\stt(\Qq^{c})$ are 
{\small
\begin{equation*} 
\begin{array}
 {c c}
 (\e_{1}{+}\e_{3}){-}(\e_{1}{-}\e_{4}){=}\e_{3}{+}\e_{4}\,{\in}\,\Qq\,{\cap}\,\stt(\Qq), 
 &  (\e_{1}{+}\e_{3}){-}(\e_{2}{+}\e_{3}){=}\e_{1}{-}\e_{2}\,{\in}\,\Qq\,{\cap}\,\stt(\Qq),\\
 (\e_{1}{-}\e_{4}){-}(\e_{1}{+}\e_{3}){=}{-}\e_{3}{-}\e_{4}\,{\in}\,\Qq\,{\cap}\,\stt(\Qq),  &
 (\e_{1}{-}\e_{4}){-}(\e_{2}{-}\e_{4}){=}\e_{1}{-}\e_{2}\,{\in}\,\Qq\,{\cap}\,\stt(\Qq), \\
 (\e_{2}{+}\e_{3}){-}(\e_{1}{+}\e_{3}){=}\e_{2}{-}\e_{1}\,{\in}\,\Qq\,{\cap}\,\stt(\Qq), &
 (\e_{2}{+}\e_{3}){-}(\e_{2}{-}\e_{4}){=}\e_{3}{-}\e_{4}\,{\in}\,\Qq\,{\cap}\,\stt(\Qq),\\
 (\e_{2}{-}\e_{4}){-}(\e_{1}{-}\e_{4}){=}\e_{2}{-}\e_{1}\,{\in}\,\Qq\,{\cap}\,\stt(\Qq), &
 (\e_{2}{-}\e_{4}){-}(\e_{2}{+}\e_{3}){=}{-}\e_{3}{-}\e_{4}\,{\in}\,\Qq\,{\cap}\,\stt(\Qq).
\end{array}
\end{equation*}}
This shows that \eqref{e5.7} and therefore the weak integrability conditions \eqref{e5.6} are
satisfied. The roots $(\e_{1}{-}\e_{4})$ and $(\e_{2}{+}\e_{3})$ are not imaginary and
therefore the minimal type condition is not satisfied. \par
The pair $(\Qq\,,\stt)$ described above is associated with the $CR$-algebra 
$(\gs,\qt)$ 
of an orbit
of $\Ub_{4}(\Hb)$ in the Grassmannian of $\SO(8,\C)$
consisting of isotropic two-planes of an orthogonal structure of $\C^{8}$.
 The Satake diagram of the corresponding
real form is 
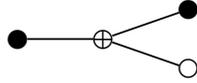
\begin{equation*}
  \xymatrix @M=0pt @R=4pt @!C=8pt{ 
& &\medbullet \\
  \medbullet\ar@{-}[r] & \oplus \ar@{-}[ru] \ar@{-}[rd]\\
  && \medcirc}
\end{equation*}
and all admissible Cartan subalgebras of $(\gs,\qt)$ may be obtained, modulo equivalence,  
by taking Cayley transforms
with respect to real roots in $\Rad^{\stt}_{\,\;\circ}$. From $\Qq^{r}{\cap}\,\Rad^{\stt}_{\,\;\circ}{=}\emptyset$
we conclude that no choice of an admissibe Cartan subalgebra of $(\gs,\qt)$ would lead to
an involution $\stt'$ for which $\Qq^{n}{\cap}\,\stt'(\Qq^{c})\,{\subset}\,\Rad^{\stt'}_{\,\;\bullet}$. 
This shows that Definition\,\ref{d5.4} is more general than 
requiring minimal type
(Definition\,\ref{d5.3}) for some choice of an adapted Cartan subalgebra.
\end{exam}
\subsection{The case of Borel $CR$-algebras}
 A \emph{Borel $CR$ algebra} 
 is a parabolic $CR$ algebra $(\gs,\qt)$ in which $\qt$ is a Borel subalgebra of $\gt$. \par
 All its admissible Cartan subalgebras are equivalent. Let $\Rad$ be the root system
for the complexification $\hg$ of an admissible Cartan subalgebra $\hst$ and 
$(\Qq\,,\stt)$ the pair associated  to $(\gs,\qt,\hst)$. Then
$\Qq\,{=}\,\Qq^{n}\,{=}\,\Rad^{+}(C)$ for a $C\,{\in}\,\Cd(\Rad)$, $\Phi_{C}{=}\,\Bz(C)$ and
$(\gs,\qt)$ is weakly integrable if and only if $\Qq\,{\cup}\,\stt(\Qq)$ is closed in $\Rad$.
Being a parabolic set containing $\Rad^{+}(C)$, it should be of the form
$\Qq_{\;\Psi_{C}}$,  for some $\Psi_{C}\,{\subseteq}\,\Bz(C)$. Let $\alphaup\,{\in}\,\Bz(C)$.
If $\stt(\alphaup)\,{\in}\,\Rad^{+}(C)$, then $({-}\alphaup)\,{\notin}\,\stt(\Qq)$ and hence $\alphaup\,{\in}\,\Psi_{C}$.
If $\stt(\alphaup)\,{\in}\,\Rad^{-}(C)$, then $({-}\alphaup)\,{\in}\,\stt(\Qq)$ and hence $\alphaup\,{\notin}\,\Psi_{C}$.
Therefore $\Qq\,{\cup}\,\stt(\Qq)$ is a closed set of roots if and only if, setting 
$\Psi_{C}{=}\{\alphaup\,{\in}\,\Bz(C)\,{\mid}\,
\stt(\alphaup)\,{\in}\,\Rad^{+}(C)\}$,
 it coincides with
the parabolic $\Qq_{\;\Psi_{C}}$.  This is equivalent to the fact 
that roots in $\Qq^{r}_{\;\Psi_{C}}{\cap}\,\Rad^{+}(C)$
have $C$-negative $\stt$-conjugate, i.e. that 
$\stt(\Qq^{n}_{\;\Psi_{C}})\,{=}\,\Qq^{n}_{\;\Psi_{C}}$. This means that  the pair $(\Qq_{\;\Psi_{C}},\stt)$
is totally real and  corresponds to the Levi-nondegenerate reduction of $(\gs,\qt)$.
Thus we obtined 
\begin{prop} A
 Borel $CR$-algebra is weakly integrable if and only if its Levi-nondegenerate reduction
 is totally real.\qed
\end{prop} 


\begin{prop}\label{p5.19} Let $\stt$ be an involution of $\Rad$ and 
$\sigmaup\,{\in}\,\Invs$. For $C\,{\in}\,\Cd(\Rad)$, denote by 
$
 \bt_{C}=\hg\oplus{\sum}_{\alphaup\in\Rad^{+}(C)}\gt^{\alphaup}
$
the corresponding Borel subalgebra. 
Let $C_{0}{\in}\,\Cd(\Rad)$.  
  The 
 following  are equivalent: 
\begin{enumerate}
 \item $\dim_{\R}(\gs,\bt_{C_{0}})=\min_{C\in\Cd(\Rad)}\dim_{\R}(\gs,\bt_{C})$;
 \item $C_{0}$ is an $S$\!-chamber for $\stt$;
  \item $\dim_{\R}(\gs,\bt_{C_{0}})=\tfrac{1}{2}(\#\Rad\,{+}\,\#\Rad_{\,\;\bullet}^{\stt})$;
   \item $\Rad^{+}(C_{0})\,{\cup}\,\Rad_{\,\;\bullet}^{\stt}$ is closed
 and $\stt$-invariant;
 \item $\Rad^{+}(C_{0})\,{\setminus}\,\Rad_{\,\;\bullet}^{\stt}$ is closed and 
 $\stt$-invariant;
 \end{enumerate}
 and imply that $(\gs,\bt_{C_{0}})$ is weakly integrable.
\end{prop} 
\begin{proof} If $C\,{\in}\,\Cd(\Rad)$, then by \eqref{e2.3}, 
\begin{equation}\label{e5.8}
 \dim_{\R}(\gs,\bt_{C})=\#\Rad - \#(\Rad^{+}(C)\,{\cap}\,\stt(\Rad^{+}(C))).
\end{equation}
Therefore  the real dimension of $(\gs,\bt_{C})$ is minimal when the number of roots of the intersection
$\Rad^{+}(C)\,{\cap}\,\stt(\Rad^{+}(C))$ is maximal. Since 
\begin{equation*}
 \Rad^{+}(C)\,{\cap}\,\stt(\Rad^{+}(C))\,{\cap}\,\Rad_{\,\;\bullet}^{\stt}=\emptyset,
\end{equation*}
this is true when $C$ is an $S$-chamber for $\stt$, i.e. when 
$\Rad^{+}(C){\setminus}\stt(\Rad^{+}(C))\,{\subset}\,\Rad_{\,\;\bullet}^{\stt}$. 
In this case 
\begin{equation*}
 \Rad^{+}(C)\,{\cap}\,\stt(\Rad^{+}(C)=\tfrac{1}{2}\big(\#\Rad\,{-}\,\#\Rad_{\,\;\bullet}^{\stt}\big).
\end{equation*}
This argument proves the equivalence of $(1)$, $(2)$ and $(3)$, while conditions $(4)$ and $(5)$ 
characterize the $S$-chambers and are therefore equivalent to $(2)$. 
Clearly the equivalent conditions $(1),\hdots,(5)$ imply that $\bt\,{+}\,\sigmaup(\bt)$ is a Lie
subalgebra of $\gt$, ansd therefore that $(\gs,\bt_{C_{0}})$ is weakly integrable.  
 \end{proof}
 \begin{exam} Condition (1) in the statement of Proposition\,\ref{p5.19} explains using the term
 \textit{minimal type} in Definition\,\ref{d5.3}. Let us illustrate by a further simple example the 
 broader scope of Definition\,\ref{d5.4}.
 \par
 Let $\gs\,{\simeq}\,\slt_{3}(\R)$ be the real form
\begin{equation*}
 \gs=\left.\left\{ \left( \begin{smallmatrix}
 \lambdaup & \bar{z}_{3} & \bar{z}_{2}\\
 z_{1} & {-}(\lambdaup\,{+}\bar{\,\lambdaup}) & \bar{z}_{1}\\
 z_{2} & z_{3} & \bar{\,\lambdaup}
\end{smallmatrix}\right) \;\right| \lambdaup,z_{1},z_{2},z_{3}\,{\in}\,\C\right\}
\end{equation*}
 of $\gt\,{=}\,\slt_{3}(\C)$ defined by the anti-$\C$-linear involution 
\begin{equation*}
 \sigmaup(X)=\Jd\,\bar{X}\,\Jd,\;\;\text{with}\;\; \Jd= 
\left(\begin{smallmatrix}
 0 & 0 & 1\\
 0 & 1 & 0\\
 1 & 0 & 0
\end{smallmatrix}\right).
\end{equation*}
Its diagonal matrices form its Cartan subalgebra
\begin{equation*}
 \hst=\left.\left\{\diag(\lambdaup,\,{-}(\lambdaup{+}\bar{\,\lambdaup}),\,\bar{\,\lambdaup})\right| 
 \lambdaup\,{\in}\,\C\right\}
\end{equation*}
and $\sigmaup$ defines on $\Rad$ the involution $\stt\,{=}\,\rtt{\e_{1}{-}\e_{3}}$.
\par
 The regular element $H\,{=}\,\diag(1,0,{-}1)$ defines the standard Weyl chamber $C$, whose corresponding
 Borel subalgebra $\bt_{C}$ consists of lower triangular matrices of $\slt_{3}(\C)$. The orbit of $\Gfs$ is open,
 since $\bt_{C}\,{+}\,\sigmaup(\bt_{C})\,{=}\,\gt$ (the minimal type condition is trivially satisfied)
 and is not the one with minimal dimension. 
The orbits of $\Gfs$ which are of minimal dimension among those 
having $CR$-algebras for which $\hst$  is adapted 
correspond to the choice of 
$\Rad^{+}(C_{i})\,{=}\,\{\alphaup\,{\in}\,\Rad\,{\mid}\,\alphaup(H_{i})\,{\geq}\,0\}$, for 
$H_{1}\,{=}\,\diag(1,{-}1,0)$ and $H_{2}\,{=}\,(0,1,{-}1)$,
having cross-marked $\Sigma$-diagrams 
 \begin{equation*}
 \xymatrix @M=0pt @R=2pt @!C=20pt{\oast\ar@{-}[r] & \medcirc\\
 \times &\times}\quad\text{and}\quad 
  \xymatrix @M=0pt @R=2pt @!C=20pt{\medcirc\ar@{-}[r] & \oast\\
 \times &\times}
\end{equation*}
If $\bt_{1},\bt_{2}$ are the Borel subalgebras corresponding to $H_{1}$ and $H_{2}$, 
the cross-marked $\Sigma$-diagrams for $(\gs,\bt_{1}{+}\sigmaup(\bt_{1}))$ and
$(\gs,\bt_{2}{+}\sigmaup(\bt_{2}))$ are 
 \begin{equation*}
 \xymatrix @M=0pt @R=2pt @!C=20pt{\oast\ar@{-}[r] & \medcirc\\
&\times}\quad\text{and}\quad 
  \xymatrix @M=0pt @R=2pt @!C=20pt{\medcirc\ar@{-}[r] & \oast\\
 \times }
\end{equation*}
respectively.
\end{exam}

\subsection{The case of general parabolic $CR$-algebras} Fix a parabolic $CR$-triple $(\gs,\qt,\hst)$,
with associated pair $(\Qq\,,\stt)$,
and consider the set $\mathcal{M}(\gs,\qt,\hst)$ of 
orbits of $\Gf_{\sigmaup}$ 
in $\sfF\,{=}\,\Gf/\Qf$ 
admitting a $CR$-algebra $(\gs,\qt)$ with adapted Cartan subalgebra $\hst$.
The Weyl group $\Wf(\Rad)$ acts transitively on $\mathcal{M}(\gs,\qt,\hst)$.
Orbits in $\mathcal{M}(\gs,\qt,\hst)$
have parabolic $CR$-triples $(\gs,\qt_{\wb},\hst)$, with
associated pairs  $(\wb(\Qq), \stt)$, with $\wb\,{\in}\,\Wf(\Rad)$. 
Denote by $\mathfrak{M}(\gs,\qt,\hst)$ the set of these \mbox{$CR$-triples}.
\begin{thm}\label{t5.21}
 Let $(\gs,\qt,\hst)$ be a parabolic $CR$-triple. If the real dimension of 
 the corresponding orbit is minimal in $\mathcal{M}(\gs,\qt,\hst)$, 
 then $(\gs,\qt,\hst)$ is of minimal type. 
\end{thm}
\begin{proof} Let $(\gs,\qt,\hst)$ be a parabolic $CR$-triple, with associated pair $(\Qq\,,\stt)$. 
The Cartan subalgebra $\hst$
is admissible also for the polarized $(\gs,\qt_{[\sigmaup]})$
and $(\gs,\qt)$ has the same real dimension of its polarized $(\gs,\qt_{[\sigmaup]})$.
Thus we can and will assume in the proof that $(\gs,\qt)$ is polarized.
By \eqref{e2.4} the real dimension  is minimal when 
the $CR$-codi\-men\-sion is maximal. We recall that
\begin{equation*}
 \text{$CR$-codim\,}(\gs,\qt)=\#(\Qq^{n}\cap\stt(\Qq^{n})).
\end{equation*}
Different orbits in $\mathcal{M}(\gs,\qt,\hst)$
have parabolic $CR$-triples $(\gs,\qt_{\wb},\hst)$, whose 
associated pairs  are $(\wb(\Qq), \stt)$, with $\wb$ in the Weyl group $\Wf(\Rad)$.  Setting
$\stt_{\wb}{=}\,\wb^{-1}{\circ}\,\stt\,{\circ}\,\wb$, we have 
\begin{equation*}
 \#(\wb(\Qq^{n})\cap\stt(\wb(\Qq^{n}))=\#(\Qq^{n}\cap\stt_{\wb}(\Qq^{n})).
\end{equation*}\par\smallskip
We will prove that, if $\Qq^{n}\,{\cap}\,\stt(\Qq^{c})\,{\not\subseteq}\,\Rad^{\stt}_{\,\;\bullet}$,
then the real dimension of $(\gs,\qt)$ is not minimal in $\mathfrak{M}(\gs,\qt,\hst)$.\par 
Let $C$ be an $S$-fit chamber for $(\gs,\qt,\hst)$. If $C$ is not an $S$\!-chamber, 
by \eqref{e4.40}
we can find
$\alphaup\,{\in}\,\Phi_{C}{\setminus}\Rad^{\stt}_{\,\;\bullet}$ with $\stt(\alphaup)\,{\in}\Qq^{c}$.
Set $\stt'\,{=}\,\rtt{\alphaup}{\circ}\,\stt\,{\circ}\,\rtt{\alphaup}$.
\par
We claim that \; \textsl{$\rtt{\alphaup}$ maps $\Qq^{n}\,{\cap}\,\stt(\Qq^{n})$
into $\Qq^{n}\,{\cap}\,\stt'(\Qq^{n})$.}\par 
If $\betaup\,{\in}\,\Qq^{n}$ and $\rtt{\alphaup}(\betaup)\,{\in}\,\Qq^{r}$, then 
$(\betaup|\alphaup)\,{>}\,0$, because $\rtt{\alphaup}(\betaup)\,{\in}\,\Rad^{+}(C)$, 
$\rtt{\alphaup}(\betaup)\,{\neq}\,\betaup$ and 
$\alphaup\,{\notin}\,\supp_{C}(\rtt{\alphaup}(\betaup))$,
 and $\stt(\rtt{\alphaup}(\betaup))\,{\in}\,\Qq^{r}$ because
$(\Qq\,,\stt)$ is polarized. Then  
\begin{equation*}
 \stt(\rtt{\alphaup}(\betaup))=\stt(\betaup)-\langle\betaup|\alphaup\rangle\stt(\alphaup)
 \Rightarrow\stt(\betaup)\,{=}\,\stt(\rtt{\alphaup}(\betaup))+\langle\betaup|\alphaup\rangle\stt(\alphaup)\,{\in}\,\Qq^{c}.
\end{equation*} \par 
This shows that $\rtt{\alphaup}(\Qq^{n}\,{\cap}\,\stt(\Qq^{n}))\,{\subseteq}\,\Qq^{n}$.\par
To prove  that $\rtt{\alphaup}(\Qq^{n}\,{\cap}\,\stt(\Qq^{n}))\,{\subseteq}\,\stt'(\Qq^{n})$,
we use the equalities
\begin{align*}
 &\stt'(\rtt{\alphaup}(\betaup))=\rtt{\alphaup}(\stt(\betaup))=\stt(\betaup)-\langle\stt(\betaup)|\alphaup\rangle\alphaup,\\
 &\stt(\stt'(\rtt{\alphaup}(\betaup)))=\betaup-\langle\stt(\betaup)|\alphaup\rangle\stt(\alphaup).
\end{align*}
The root in the first line is $C$-positive because $\stt(\betaup)$ is $C$-positive and distinct from $\alphaup$.
If $\langle\stt(\betaup)|\alphaup\rangle\,{\leq}\,0$, then $\stt'(\rtt{\alphaup}(\betaup))$ belongs to $\Qq^{n}$
being the sum of a root in $\Qq^{n}$ and a nonnegative multiple of the positive root $\alphaup$. If 
$\langle\stt(\betaup)|\alphaup\rangle\,{>}\,0$, then 
$ \stt(\stt'(\rtt{\alphaup}(\betaup)))$ belongs to $\Qq^{n}$, being the sum
of a root in $\Qq^{n}$ and a positive multiple of the  root $({-}\stt(\alphaup))$ of $\Qq^{n}$.
Since $(\Qq\,,\stt)$ is polarized, we conclude that its image by $\stt$, which is $\stt'(\rtt{\alphaup}(\betaup))$,
belongs to $\Qq^{n}\,{\cup}\,\Qq^{c}$ and hence, being positive, to $\Qq^{n}$. 
\par\smallskip
Since the map $\rtt{\alphaup}$ is injective, we get 
$\#(\Qq^{n}\,{\cap}\,\stt(\Qq^{n})\,{\leq}\,\#(\Qq^{n}\,{\cap}\,\stt'(\Qq^{n}))$.
\par
Let us show that in fact the inequality is
strict. 
 Since $\alphaup$ and $\stt(\alphaup)$ have the same length, $\langle\stt(\alphaup)|\alphaup\rangle$
 is either $0$ or ${\pm}1$. We get
\begin{equation*} 
\begin{array}{c c c c c}
 \langle\alphaup|\stt(\alphaup)\rangle\,{=}\,0 & \Rightarrow & \stt'(\alphaup)={-}\stt(\alphaup) &\Rightarrow
 & \alphaup,\,{-}\stt(\alphaup)\in\Qq^{n}\cap\stt'(\Qq^{n}),\\
  \langle\alphaup|\stt(\alphaup)\rangle\,{=}\,1 & \Rightarrow & \stt'(\alphaup)=\alphaup{-}\stt(\alphaup) &\Rightarrow
 & \alphaup,\,\alphaup{-}\stt(\alphaup)\in\Qq^{n}\cap\stt'(\Qq^{n}),\\
   \langle\alphaup|\stt(\alphaup)\rangle\,{=}\,{-}1 & \Rightarrow & \stt'({-}\stt(\alphaup))={-}\stt(\alphaup) &\Rightarrow
 & {-}\stt(\alphaup)\in\Qq^{n}\cap\stt'(\Qq^{n})\cap\Rad^{\stt'}_{\,\;\circ},\\
  \langle\alphaup|\stt(\alphaup)\rangle\,{=}\,{-}1 & \Rightarrow & \stt'(\alphaup)={-}\stt(\alphaup){-}\alphaup.&
 & 
\end{array}
\end{equation*}
\par
If $ \langle\alphaup|\stt(\alphaup)\rangle\,{\geq}\,0$, then $\alphaup$ is an element of $\Qq^{n}\,{\cap}\,
\stt'(\Qq^{n})$ which does not belong to $\rtt{\alphaup}(\Qq^{n}\,{\cap}\,\stt(\Qq^{n}))$. 
\par
When  $\langle\stt(\alphaup)|\alphaup\rangle\,{=}\,-1$,
we have $\rtt{\alphaup}(\stt({-}\alphaup))\,{=}\,{-}(\alphaup\,{+}\,\stt(\alphaup))$. \par
If this root belongs to
$\Qq^{r}$, then $({-}\stt(\alphaup))$, 
belonging to $\rtt{\alphaup}(\Qq^{r})$, is a root in $\Qq^{n}\,{\cap}\,\stt'(\Qq^{n})$ and 
not in 
$\rtt{\alphaup}(\Qq^{n}\,{\cap}\,\stt(\Qq^{n}))$. \par
Otherwise $\stt'(\alphaup)\,{=}\,{-}(\alphaup\,{+}\,\stt(\alphaup))$ belongs
to $\Qq^{n}$ and hence $\alphaup$ is a root of $\Qq^{n}\,{\cap}\,\stt'(\Qq^{n})$
which does not belong to $\rtt{\alphaup}(\Qq^{n}\,{\cap}\,\stt(\Qq^{n}))$.\par
This completes the proof.\end{proof}
\begin{exam} We consider in the projective plane $\CP^{2}$ 
the orbits of $\SL_{3}(\R)$. 
We choose on the root system 
$\Rad\,{=}\,\{{\pm}(\e_{i}{-}\e_{j}\,{\mid}\,1{\leq}i{<}j{\leq}3\}$ the involution $\stt$ with 
$\stt(\e_{1})\,{=}\,\e_{2}$, $\stt(\e_{2})\,{=}\,\e_{1}$, $\stt(\e_{3})\,{=}\,\e_{3}$. Then 
both the cross-marked diagrams 
  \begin{equation*}
 \xymatrix @M=0pt @R=2pt @!C=20pt{\oast\ar@{-}[r] & \medcirc\\
 \times }\quad\text{and}\quad 
  \xymatrix @M=0pt @R=2pt @!C=20pt{\medcirc\ar@{-}[r] & \oast\\
 \times }
\end{equation*}
describe orbits of the minimal type, although the first describes the open orbit $\CP^{2}{\setminus}\RP^{2}$,
the second the compact orbit $\RP^{2}$, having real dimensions $4$ and $2$, respectively.
\end{exam}
\subsection{Levi and contact-nondegeneracy} In this subsection we discuss relationships
between Levi and contact-nondegeneracy. 
\begin{lem}
 A $CR$ algebra $(\gs,\qt)$ and its Levi-nondegenerate reduction $(\gs,\qt_{\{\sigmaup\}})$
 have the same contact order. \qed
\end{lem}
\begin{prop}
 Let $(\gs,\qt)$ be a Levi-nondegenerate $CR$ algebra. If its Levi order is less than or equal to two,
 then it equals its contact-order. 
\end{prop} 
\begin{proof}
 Inded the contact order is less than or equal to the Levi order and, when the contact order in one,
 also the Levi order is one. 
\end{proof}

\providecommand{\bysame}{\leavevmode\hbox to3em{\hrulefill}\thinspace}
\providecommand{\MR}{\relax\ifhmode\unskip\space\fi MR }
\providecommand{\MRhref}[2]{%
  \href{http://www.ams.org/mathscinet-getitem?mr=#1}{#2}
}
\providecommand{\href}[2]{#2}

\end{document}